\documentclass[12pt,reqno,twoside,a4paper]{amsart}

\makeatletter
\@namedef{subjclassname@2020}{\textup{2020} Mathematics Subject Classification}
\makeatother

\usepackage{amsmath, amsthm, amssymb, amscd, amsfonts}
\usepackage{mathrsfs}
\usepackage{color}
\usepackage[all]{xy}             
 \usepackage[active]{srcltx}   
 \usepackage{hyperref}          
\font \smallrm=cmr10 at 10truept
\font \smallsl=cmsl10 at 10truept
 at 9truept
\font \smallbf=cmbx10 at 9truept
\font \ssmallrm=cmr10 at 9truept
\font \ssmallsl=cmsl10 at 9truept

\numberwithin{equation}{section}

\newcommand{\subu}[2]{{#1}_{\raise-2pt\hbox{$ \scriptstyle #2 $}}}
\newcommand{\subd}[3]{{#1}_{\raise-2pt\hbox{$ \scriptstyle #2 #3 $}}}

\newtheorem{lema}{Lemma}[subsection]
\newtheorem{theorem}[lema]{Theorem}
\newtheorem{cor}[lema]{Corollary}

\newtheorem{prop}[lema]{Proposition}

\theoremstyle{definition}
\newtheorem{definition}[lema]{Definition}
\newtheorem{exa}[lema]{Example}
\newtheorem{exas}[lema]{Examples}

\newtheorem{rmk}[lema]{Remark}
\newtheorem{rmks}[lema]{Remarks}
\newtheorem{free text}[lema]{}
\newtheorem{obs}[lema]{Observation}
\newtheorem{obs's}[lema]{Observations}

\theoremstyle{remark}

\newcommand \id{\operatorname{id}}
\newcommand \ad{\operatorname{ad}}

\newcommand \co{\operatorname{co}}

\newcommand \rk{\operatorname{rk}}

\newcommand \cop{\operatorname{cop}}

\newcommand \Hom{\operatorname{Hom}}
\newcommand \Alg{\operatorname{Alg}}

\newcommand \Ker{\operatorname{Ker}}
\newcommand \Prim{\operatorname{Prim}}

\newcommand \op{\operatorname{op}}

\newcommand \smallast {{\operatorname{\raise1,5pt\hbox{$ \scriptscriptstyle \ast $}}}}
\newcommand \smallp {{\operatorname{\raise1pt\hbox{$ \scriptscriptstyle + $}}}}
\newcommand \smallm {{\operatorname{\raise1pt\hbox{$ \scriptscriptstyle - $}}}}
\newcommand \smallpm {{\operatorname{\raise1pt\hbox{$ \scriptscriptstyle \pm $}}}}
\newcommand \smallmp {{\operatorname{\raise1pt\hbox{$ \scriptscriptstyle \mp $}}}}

\newcommand{\eps}{\epsilon}
\newcommand{\ot}{\otimes}
\newcommand{\com}{\Delta}

\newcommand{\R}{{\mathcal R}}

\newcommand{\E}{{\mathcal E}}
\newcommand{\F}{{\mathcal F}}

\newcommand{\D}{{\mathcal D}}
\newcommand{\G}{{\mathcal G}}
\newcommand{\Z}{{\mathcal Z}}
\newcommand{\uhat}{\widehat{U}}

\newcommand{\Lc}{{\mathcal L}}

\def \bq {\mathbf{q}}

\def \CC{\mathbb{C}}
\def \FF{\mathbb{F}}
\def \NN{\mathbb{N}}

\def \UU{\mathbb{U}}
\def \ZZ {\mathbb{Z}}
\def \k {\Bbbk}

\def \kq{{\Bbbk(q)}}
\def \kh{{\Bbbk[[\hbar]]}}
\def \khp{{\Bbbk(\hskip-1,7pt(\hbar)\hskip-1,7pt)}}

\def \Zqqm {{\mathbb{Z}\big[\hskip1pt q,q^{-1}\big]}}

\def\FF{\mathbb{F}}

\def \D {\mathcal{D}}
\def \G {\mathcal{G}}

\def \I {\mathcal{I}}

\def \K {\mathcal{K}}
\def \L {\mathcal{L}}

\def \SS {\mathcal{S}}

\def \JJ {\mathfrak{J}}

\def \Dpicc {{\scriptscriptstyle D}}
\def \Ppicc {{\scriptscriptstyle P}}
\def \Rpicc {{\scriptscriptstyle \R}}
\def \Thetapicc {{\scriptscriptstyle \Theta}}
\def \Psipicc {{\scriptscriptstyle \Psi}}

\def \erm {\mathrm{e}}

\def \frm {\mathrm{f}}

\def \hrm {\mathrm{h}}

\def \lieg{\mathfrak{g}}
\def \liem{\mathfrak{m}}

\def \liegd{\mathfrak{g}^{\raise-1pt\hbox{$ \Dpicc $}}}

\def \liegdP{\mathfrak{g}^{\raise-1pt\hbox{$ \Dpicc $}}_\Ppicc}
\def \liegRP{\mathfrak{g}^{\raise-1pt\hbox{$ \scriptscriptstyle \mathcal{R} $}}_\Ppicc}
\def \liegPsi{\mathfrak{g}^{\raise-1pt\hbox{$ \scriptscriptstyle \Psi $}}}
\def \liegdPsi{\mathfrak{g}^{\raise-1pt\hbox{$ {\scriptscriptstyle D,\Psi} $}}}

\def \liegdotdq{\dot{\lieg}_{\raise-2pt\hbox{$ \Dpicc , \hskip-0,9pt{\scriptstyle \bq} $}}}
\def \liegtildq{\tilde{\lieg}_{\raise-2pt\hbox{$ \Dpicc , \hskip-0,9pt{\scriptstyle \bq} $}}}
\def \liegdq{\lieg_{\raise-2pt\hbox{$ \Dpicc , \hskip-0,9pt{\scriptstyle \bq} $}}}
\def \liegdqcheck{\lieg_{\raise-2pt\hbox{$ \Dpicc , \hskip-0,9pt{\scriptstyle \check{\bq}} $}}}
\def \Gtildeqstar {\widetilde{G}^{\,\raise2pt\hbox{$ \scriptstyle * $}}_{\!\Dpicc , \hskip+0,9pt{\scriptstyle \bq}}}

\def \lieb{\mathfrak{b}}
\def \liebPhat{{\widehat{\mathfrak{b}}}^\Ppicc}
\def \liebP{\mathfrak{b}^\Ppicc}
\def \liebd{\mathfrak{b}^{\raise-1pt\hbox{$ \Dpicc $}}}
\def \lieh{\mathfrak{h}}
\def \liehd{\mathfrak{h}^{\raise-1pt\hbox{$ \Dpicc $}}}
\def \liek{\mathfrak{k}}
\def \lier{\mathfrak{r}}
\def \lies{\mathfrak{s}}
\def \lien{\mathfrak{n}}

\def \lieso{\mathfrak{so}}

\def \liez{\mathfrak{z}}

\def \uhg{U_\hbar(\hskip0,5pt\lieg)}

\def \uRPhg{U^{\,\R}_{\!P,\hskip0,7pt\hbar}(\hskip0,5pt\lieg)}
\def \URPhg{\mathbb{U}^{\,\R}_{\!P,\hskip0,7pt\hbar}(\hskip0,5pt\lieg)}

\def \calDRPhgd{\overrightarrow{\mathcal{D}}^{\,\R}_{\!P,\hskip0,7pt\hbar}(\hskip0,5pt\lieg)}
\def \calDRPhgs{\overleftarrow{\mathcal{D}}^{\,\R}_{\!P,\hskip0,7pt\hbar}(\hskip0,5pt\lieg)}
\def \DRPhgd{\overrightarrow{D}^{\,\R}_{\!P,\hskip0,7pt\hbar}(\hskip0,5pt\lieg)}
\def \DRPhgs{\overleftarrow{D}^{\,\R}_{\!P,\hskip0,7pt\hbar}(\hskip0,5pt\lieg)}
 
\def \caluRPhbp{\mathcal{U}^{\,\R}_{P,\hskip0,7pt\hbar}(\hskip0,5pt\lieb_+)}

\def \utildeRPhbp{\widetilde{U}^{\,\R}_{\!P,\hskip0,7pt\hbar}(\hskip0,5pt\lieb_+)}
\def \calutildeRPhbp{\widetilde{\mathcal{U}}^{\,\R}_{P,\hskip0,7pt\hbar}(\hskip0,5pt\lieb_+)}

\def \uRPhbp{U^{\,\R}_{\!P,\hskip0,7pt\hbar}(\hskip0,5pt\lieb_+)}

\def \caluRPhbm{\mathcal{U}^{\,\R}_{P,\hskip0,7pt\hbar}(\hskip0,5pt\lieb_-)}

\def \utildeRPhbm{\widetilde{U}^{\,\R}_{\!P,\hskip0,7pt\hbar}(\hskip0,5pt\lieb_-)}
\def \calutildeRPhbm{\widetilde{\mathcal{U}}^{\,\R}_{P,\hskip0,7pt\hbar}(\hskip0,5pt\lieb_-)}

\def \uRPhbm{U^{\,\R}_{\!P,\hskip0,7pt\hbar}(\hskip0,5pt\lieb_-)}

\def \caluRPhbpm{\mathcal{U}^{\,\R}_{P,\hskip0,7pt\hbar}(\hskip0,5pt\lieb_\pm)}

\def \utildeRPhbpm{\widetilde{U}^{\,\R}_{\!P,\hskip0,7pt\hbar}(\hskip0,5pt\lieb_\pm)}
\def \calutildeRPhbpm{\widetilde{\mathcal{U}}^{\,\R}_{P,\hskip0,7pt\hbar}(\hskip0,5pt\lieb_\pm)}

\def \uRPhbpm{U^{\,\R}_{\!P,\hskip0,7pt\hbar}(\hskip0,5pt\lieb_\pm)}
\def \uRPhbmp{U^{\,\R}_{\!P,\hskip0,7pt\hbar}(\hskip0,5pt\lieb_\mp)}

\def \uRPhbmp{U^{\,\R}_{\!P,\hskip0,7pt\hbar}(\hskip0,5pt\lieb_\mp)}
\def \caluRPhbmp{\mathcal{U}^{\,\R}_{P,\hskip0,7pt\hbar}(\hskip0,5pt\lieb_\mp)}

\def \uRPhnp{U^{\,\R}_{\!P,\hskip0,7pt\hbar}(\hskip0,5pt\lien_+)}
\def \uRPhnm{U^{\,\R}_{\!P,\hskip0,7pt\hbar}(\hskip0,5pt\lien_-)}
\def \uRPhnpm{U^{\,\R}_{\!P,\hskip0,7pt\hbar}(\hskip0,5pt\lien_\pm)}
\def \uRPhnmp{U^{\,\R}_{\!P,\hskip0,7pt\hbar}(\hskip0,5pt\lien_\mp)}

\def \uhh{U_\hbar(\hskip0,0pt\lieh)}

\def \fqgd{F_q[G_{\raise-2pt\hbox{$ \Dpicc $}}]}

\def \pf{\begin{proof}}
\def \epf{\end{proof}}

\newcommand\fT{\mathsf{T}}
\newcommand\fH{\mathsf{H}}

\theoremstyle{plain}

\begin{document}





\title[Formal multiparameter quantum groups]
 {Formal multiparameter quantum groups,  \\
deformations and specializations}

\author[g.~a.~garc{\'\i}a \ , \ \  f.~gavarini]
{Gast{\'o}n Andr{\'e}s Garc{\'\i}a${}^\flat$ \ ,  \ \ Fabio Gavarini$\,{}^\sharp$}

\address{\newline\noindent Departamento de Matem\'atica, Facultad de Ciencias Exactas
   \newline
 Universidad Nacional de La Plata   --- CMaLP-CIC-CONICET
   \newline
 1900 La Plata, ARGENTINA   \ --- \   {\tt ggarcia@mate.unlp.edu.ar}
 \vspace*{0.5cm}
   \newline
 Dipartimento di Matematica,
   \newline
 Universit\`a degli Studi di Roma ``Tor Vergata''
   \newline
 Via della ricerca scientifica 1   ---   I\,-00133 Roma, ITALY   \ --- \   {\tt gavarini@mat.uniroma2.it}}

\thanks{\noindent 2020 \emph{MSC:}\,  17B37, 17B62   ---
}


\begin{abstract}
 \medskip
   We introduce the notion of  {\sl formal multiparameter QUEA}   --- in short  {\sl FoMpQUEA}  ---
   as a straightforward generalization of Drinfeld's quantum group  $ \uhg \, $.
   Then we show that the class of FoMpQUEAs is closed under deformations by (``toral'')
   twists and deformations by (``toral'')  $ 2 $--cocycles:
   as a consequence, all ``multiparameter formal QUEAs'' considered so far are recovered, as falling within this class.
   In particular, we prove that any FoMpQUEA is isomorphic to a suitable deformation, by twist or by 2--cocycle, of Drinfeld's standard QUEA.
                                                                     \par
   We introduce also multiparameter Lie bialgebras (in short, MpLbA's), and we consider their deformations, by twist and by 2--cocycle.
   The semiclassical limit of every FoMpQUEA is a suitable MpLbA, and conversely each MpLbA can be quantized to a suitable FoMpQUEA.
   In the end, we prove that, roughly speaking, the two processes of ``specialization''   --- of a FoMpQUEA to a MpLbA ---
   and of ``deformation (by toral twist or toral 2--cocycle)''   --- at the level of FoMpQUEAs or of MpLbA's ---   do commute with each other.
\end{abstract}

%

{\ }

\vskip-91pt

   \centerline{ \smallrm  {\smallsl Annales de l'Institut Fourier\/}  {\smallbf 76}  (2026), no.~1, 51--167 }
                                                 \par
   \centerline{\smallrm {\smallbf DOI:}  10.5802/aif.3675   \ --- \   preprint  {\smallsl arXiv:2203.11023 [math.QA] (2022)}}
%
 \vskip3pt
   \centerline{\ssmallrm {\ssmallsl The original publication is available at\/}
\  https://aif.centre-mersenne.org/item/10.5802/aif.3675.pdf}

\vskip9pt   {\ }

\maketitle

\tableofcontents



\section{Introduction}  \label{sec: intro}

   Quantum groups can be thought of, roughly speaking,
   as Hopf algebras depending on one ``parameter'' such that, for a ``special value'' of this parameter,
   they turn isomorphic either to the universal enveloping algebra of some Lie algebra  $ \lieg $
   or to the function algebra of some algebraic group $ G \, $.  In the first case the quantum group is called
   ``quantized universal enveloping algebra'' (or QUEA in short) and in the second ``quantized function algebra''
   (or QFA in short).
                                           \par
   Quite soon, people also began to introduce new quantum groups depending on two or more parameters, whence the terminology ``multiparameter quantum groups'' came in use: see, e.g.,  \cite{BGH},  \cite{BW1},  \cite{BW2},  \cite{CM},  \cite{CV1},  \cite{CV2},  \cite{DPW},  \cite{GG1},  \cite{HLT},  \cite{HPR},  \cite{Jn},  \cite{Kh},  \cite{KT},  \cite{Ma},  \cite{OY},  \cite{Re},  \cite{Su},  \cite{Tk}   --- and the list might be longer.
   Nevertheless, one can typically describe a multiparameter quantum group so that one single parameter stands
   ``distinguished'', as the  \textsl{continuous\/}  one that can be specialized.
   The other parameters instead (seen as  \textsl{discrete\/})  parametrize different structures on a
   common ``socle'' underlying the semiclassical limit of the quantum group, that is
   achieved when the continuous parameter is specialized.  Indeed, this already occurs with one-parameter quantum groups:
   for example, the celebrated Drinfeld's QUEA  $ \uhg $  associated with a complex,
   finite-dimensional, semisimple Lie algebra  $ \lieg $  has a description where the  \textsl{continuous\/}  parameter
   $ \hbar $  bears the quantization nature of  $ \uhg \, $,  while other  \textsl{discrete\/}  parameters,
   namely the entries of the Cartan matrix of  $ \lieg \, $,  describe the Lie algebra structure on  $ \lieg $  itself.
 \vskip7pt
   In this paper we focus onto the study of multiparameter QUEAs; then it will be possible to
   realize a parallel study and to achieve the corresponding results for multiparameter QFA's by
   suitably applying duality.  Recall that QUEAs (and QFA's alike) are usually considered in two versions:
   the so-called ``formal'' one   --- dealing with topological Hopf algebras over  $ \kh $  ---
   and the ``polynomial'' one   --- dealing with Hopf algebras over a field  $ \mathbb{K} $  with some
   $ \, q \in \mathbb{K} \, $  entering the game as parameter.
%
%
 \vskip7pt
   One of the first general examples of multiparameter QUEA, hereafter mentioned as MpQUEA, was provided by Reshetikhin in  \cite{Re}.
   This extends Drinfeld's definition of  $ \uhg $  to a new object  $ U_\hbar^\Psipicc(\lieg) $
   that shares the same algebra structure of  $ \uhg $  but bears a new  \textsl{coalgebra\/}  structure,
   depending on a matrix  $ \Psi $  that collects the new,  \textsl{discrete\/}  parameters of  $ U_\hbar^\Psipicc(\lieg) \, $.
   At the semiclassical limit, these new parameters (hence  $ \Psi $)
   describe the new Lie coalgebra structure inherited by  $ \lieg $  from  $ U_\hbar^\Psipicc(\lieg) $  itself.
   Note that  $ U_\hbar^\Psipicc(\lieg) $  is defined from scratch as being the outcome of a  \textsl{deformation by twist\/}
   of Drinfeld's  $ \uhg \, $,  using a twist of a specific type (that we shall call ``toral'') defined via  $ \Psi \, $.
   It follows that the class of all Reshetikhin's MpQUEAs is stable under deformation by toral twists,
   i.e.\ any such deformation of an object of this kind is again an object of the same kind.
   Even more, this class is ``homogeneous'', in that each  $ U_\hbar^\Psipicc(\lieg) $  is nothing but a twist deformation of Drinfeld's  $ \uhg \, $.
                                                              \par
   With a parallel approach, a  \textsl{polynomial\/}  version of Reshetikhin's MpQUEAs was introduced and studied by Costantini-Varagnolo: see  \cite{CV1},  \cite{CV2},  and also  \cite{Ga1};  on the other hand, these works do not consider deformations.  Alternatively, using the duality with quantum coordinate algebras, two-parameters quantum envelopling algebras of polynomial type are considered in Dobrev-Parashar  \cite{DoP}  and in Dobrev-Tahri  \cite{DoT}.  The effect of the twist can be seen in the description of the coproduct after a change the presentation \`{a} la Drinfeld-Jimbo type.
                                                               \par
   In another direction, a different version of  \textsl{polynomial\/}  MpQUEA (still working over  $ \lieg $  as above),
   call it  $ U_{\mathbf{q}}(\lieg) \, $,  has been developed in the works of Andruskiewitsch-Schneider, Rosso,
   and many others   --- see for instance  \cite{AS1},  \cite{AS2},  \cite{HPR},  \cite{Ro}.  In this case,
   the ``multiparameter'' is cast into a matrix  $ \, \mathbf{q} = {\big(\, q_{ij} \big)}_{i, j \in I} \, $
   whose entries take part in the description of the  \textsl{algebra\/}  structure of  $ U_{\mathbf{q}}(\lieg) \, $.
   Under mild, additional conditions, this yields a very general family of MpQUEAs which
   is very well-behaved: in particular, it is stable under deformations by 2--cocycles of ``toral'' type.
   Even better, this family is ``homogeneous'', in that each  $ U_{\mathbf{q}}(\lieg) $
   is a 2--cocycle deformation of Jimbo-Lusztig's polynomial version  $ U_q(\lieg) $  of Drinfeld's  $ \uhg \, $.
                                                              \par
   Note that, in Hopf theory,  \textsl{twist\/}  and  \textsl{2--cocycle\/}  are notions dual to each other.
   Thus the constructions of MpQUEAs by Reshetikhin
   and by Andruskiewitsch-\break{}Schneider
 (besides the difference in being ``formal'' or ``polynomial'') are somehow  \textsl{dual\/}  to each other
 --- and, as such, seem definitely different from each other.
 \vskip7pt
   The purpose of this paper is to introduce a new notion of MpQUEA that encompass both Reshtikhin's
   one and Andruskiewitsch-Schneider's one.  Indeed, we achieve this goal introducing a new family
   of MpQUEAs which incorporates An\-druskie\-witsch-Schneider's one, hence in particular it
   includes Drinfeld's standard example (see  Definition \ref{def: Mp-Uhgd},
   Theorem \ref{thm: form-MpQUEAs_are_Hopf}  and  \S \ref{sec: constr-FoMpQUEAs}).
   We show that this new family is stable by toral 2--cocycle deformations
   (Theorem \ref{thm: 2cocdef-uPhgd=new-uPhgd}),  just as Andruskiewitsch-Schnei\-der's,  \textsl{and\/}
   it is also stable by toral twist deformations  (Theorem \ref{thm: twist-uRPhg=new-uRPhg}),
   hence it incorporates Reshetikhin's family as well.  In particular, we show that every MpQUEA of
   the Reshetikhin's family is actually isomorphic to one of the Andruskie\-witsch-Schneider's family,
   and viceversa: the isomorphism is especially meaningful in itself, in that it amounts to a suitable
   change of presentation via a well-focused change of generators  (see  Theorem \ref{thm: twist-uRPhg=new-uRPhg}).
   In this sense, we really end up with a single,  \textsl{homogeneous\/}  family
   --- not just a collage of two distinct families; this can be seen as a byproduct of the intrinsic
   ``self-duality'' of Drinfeld's standard  $ \uhg \, $.
                                                              \par
   For each one of these MpQUEAs, then, one can decide to focus the dependence on the  \textsl{discrete\/}
   multiparameters either on the coalgebra structure (which amounts to adopt Reshetikhin's point of view) or
   on the algebra structure (thus following Andruskiewitsch-Schneider's approach).
   In our definition we choose to adopt the latter point of view, as it is definitely closer to the classical
   Serre's presentation of  $ U(\lieg) $   --- or even to the presentation of Drinfeld's standard  $ \uhg $
   ---   where the  \textsl{discrete\/}  multiparameters given by the Cartan matrix entries rule the algebra structure.
 \vskip7pt
   Technically speaking, we adopt the setting and language of  \textsl{formal\/}  quantum groups,
   thus our newly minted objects are ``formal MpQUEAs'', in short ``FoMpQUEAs''.
   This is indeed a necessary option: in fact, the setup of polynomial MpQUEAs is well-suited
   when one deals with (toral) 2--cocycle deformations, but behaves quite poorly under deformations by (toral) twists.
   Roughly speaking, the toral part in a polynomial MpQUEA (in the sense of Andruskiewitsch-Schneider, say)
   happens to be too rigid, in general, under twist deformations; this is shown in our previous paper  \cite{GG2},
   where we pursued the same goal by means of ``polynomial MpQUEAs'', which eventually prove to be a somewhat less suitable tool.
                                                              \par
   Thus, one needs to allow a more flexible notion of ``toral part'' in our would-be MpQUEA
   in order to get a notion that is stable under deformation by (toral) twists.
   We obtain this by choosing to define our \textsl{formal\/}  MpQUEA as having a toral part with two distinguished sets of
   ``coroots'' and ``roots'', whose mutual interaction is encrypted in a ``multiparameter matrix''  $ P $
   whose role generalizes that of the Cartan matrix.  We formalize all this via the notion of  \textsl{realization\/}
   of the matrix  $ P $,  which is a natural extension of Kac' notion of realization of a generalized Cartan matrix
   (cf.\ Definition \ref{def: realization of P});  our FoMpQUEA then is defined much like Drinfeld's standard one,
   with the entries of  $ P $  playing the role of  \textsl{discrete\/}  multiparameters.
 \vskip7pt
   By looking at semiclassical limits, we find that our new class of FoMpQUEAs gives rise to a new
   family of multiparameter Lie bialgebras (in short MpLbA's) that come equipped with a
   presentation ``\`a la Serre'' in which the parameters   --- i.e., the entries of  $ P $,  again ---
   rule the Lie algebra structure (cf.\  \S \ref{MpLieBialgebras-double}).
   Again, we prove that this family is stable by deformations   --- in Lie bialgebra theoretical sense ---
   both via ``toral'' 2--cocycles and via ``toral'' twists (see Theorem \ref{thm: 2-cocycle-def-MpLbA}  and
   Theorem \ref{thm: twist-liegRP=new-liegR'P'}).  In particular, every such MpLbA admits an alternative
   presentation in which the Lie algebra structure stands fixed (always being ruled by a fixed generalized Cartan matrix)
   while the Lie  \textsl{coalgebra\/}  structure does vary according to the multiparameter matrix  $ P $.
   Like in the quantum setup, the isomorphism between the two presentations
%
%
 is quite meaningful, as it boils down to a well-chosen change of generators  (cf.\  Theorem \ref{thm: twist-liegRP=new-liegR'P'}).
 The very definition of these MpLbA's, as well as the just mentioned results about them,
 can be deduced as byproducts of those for FoMpQUEAs (via the process of specialization);
 otherwise, they can be introduced and proved directly; in short, we do both  (cf.\ \S \ref{sec: Mp Lie bialgebras}
 and  Theorem \ref{thm: semicl-limit FoMpQUEA}).  These MpLbA's were possibly known in literature,
 at least in part: yet our construction yields a new, systematic presentation of their whole family in its full extent,
 also proving its stability under deformations by both (toral) 2--cocycles and (toral) twists.
 \vskip7pt
   As a final, overall comment, we recall that a close relation between multiparameters and deformations is ubiquitous in several applications, e.g.\ in the classification of complex finite-dimensional pointed Hopf algebras over abelian groups  \cite{AS2},  \cite{AGI}   --- where deformations by 2--cocycle play a central role.  Moreover, MpQUEAs may also serve as interpolating objects in the study of the representation theory of quantum groups associated with Langlands dual semi-simple Hopf algebras  \cite{FH}   --- where deformations by twist instead are a key tool.
 \vskip7pt
   A last word about the organization of the paper.
                                                              \par
   In section  \ref{sec: Cartan-data_realiz's},  we introduce the ``combinatorial data''
   underlying our constructions of MpLbA's and FoMpQUEAs alike: the notion of  \textsl{realization\/}  of
   a multiparameter matrix, and the process of deforming realizations either by twists or by 2--cocycles.
                                                              \par
   In section  \ref{sec: Mp Lie bialgebras}  we introduce our MpLbA's and study their deformations by
   (toral) twists and by (toral) 2--cocycles.
                                                              \par
   Section  \ref{sec: form-MpQUEAs}  is dedicated to introduce our newly minted FoMpQUEAs,
   in particular using different, independent approaches, and to prove their basic properties.
                                                              \par
   With section  \ref{sec: deform's_FoMpQUEAs}  we discuss deformations of FoMpQUEAs by
   (toral) twists and by (toral) 2--cocycles: we prove that these deformations turn FoMpQUEAs into
   new FoMpQUEAs again, the case by twist being possibly the more surprising.
                                                              \par
   Finally, in section  \ref{sec: special-&-quantiz}  we perform specializations of FoMpQUEAs and
   look at their resulting semiclassical limit: we find that this limit is always a MpLbA (in short,
   by the very definition of MpLbA's), with the same multiparameter matrix  $ P $  as the FoMpQUEA
   it comes from.  Conversely, any possible MpLbA does arise as such a limit   --- in other words,
   any MpLbA has a FoMpQUEA which is quantization of it.  Then   --- more important ---
   we compare deformations (by toral twists or 2--cocycle) before and after specialization: the outcome is,
   in a nutshell, that  \textsl{``specialization and deformation (of either type) commute with each other''}
   (cf.\ Theorem \ref{thm: specializ twisted FoMpQUEA}  and  Theorem \ref{thm: specializ 2-cocycle FoMpQUEA}).
   In fact, this last result can be deduced also as a special instance of a more general one, which in turn is an outcome of a
   larger study about deformations (of either type) of formal quantum groups
   --- i.e., Drinfeld's-like QUEAs and their dual, the so-called QFSHA's ---   and of their semiclassical limits.
   This is a more general chapter in quantum group theory, with its own reasons of interest,
   thus we shall treat it in a separate publication --- cf.\ \cite{GG3}.

\vskip17pt

   \centerline{\ssmallrm ACKNOWLEDGEMENTS}
 \vskip3pt
%
   \indent   {\smallrm The authors thank Marco Farinati for several fruitful conversations.
   They also especially thank the Referee, whose shrewd remarks and comments helped a lot to improve the paper.}
                                                        \par
   {\smallrm This work was supported by CONICET, ANPCyT, Secyt (Argentina) and by INdAM/GNSAGA,
   ICTP and University of Padova (Italy), as well as by the MIUR
  {\smallsl Excellence Department Project MatMod@TOV (CUP E83C23000330006)}
 awarded to the Department of Mathematics, University of Rome ``Tor Vergata''}.

\bigskip
 \medskip

\section{Multiparameters and their realizations}  \label{sec: Cartan-data_realiz's}
 \vskip7pt
   In this section we fix the basic combinatorial data that we need later on.
   The definition of our multiparameter Lie bialgebras and formal multiparameter
   quantum groups requires a full lot of related material that we now present.
   In particular,  $ \, \NN = \{0, 1,\ldots\} \, $  and  $ \, \NN_+ := \NN \setminus \{0\} \, $,
   \,while  $ \k $  will be a field of characteristic zero.

\medskip

\subsection{Multiparameter matrices, Cartan data, and realizations}
  \label{MpMatrices, Cartan & realiz.'s}  {\ }
 \vskip7pt
   We introduce hereafter the ``multiparameters'',
   which we will use to construct (semi)classical and quantum objects as well.
   The theory can be developed more in general,
   but we stick to the case of ``Cartan type'' as more relevant to us; accordingly,
   this will keep us close to the common setup of Lie algebras of Kac--Moody type,
   in particular those whose Cartan matrix is symmetrisable.

\vskip9pt

\begin{free text}  \label{root-data_Lie-algs}
 {\bf Cartan data and associated Lie algebras.}  Hereafter we fix  $ \, n \in \NN_+ \, $  and  $ \, I := \{1,\dots,n\} \, $.
 Let  $ \, A := {\big(\, a_{ij} \big)}_{i, j \in I} \, $  be a generalized, symmetrisable Cartan matrix;
 then there exists a unique diagonal matrix
 $ \, D := {\big(\hskip0,7pt d_i \, \delta_{ij} \big)}_{i, j \in I} \, $  with positive integral,
 pairwise coprime entries such that  $ \, D A \, $  is symmetric.
 Let  $ \, \lieg =  \lieg_A \, $  be the Kac-Moody algebra over  $ \CC $
 associated with  $ A $  (cf.\ \cite{Ka});  we consider a split integral  $ \ZZ $--form  of  $ \lieg \, $,
 and for the latter the scalar extension from  $ \ZZ $  to any field  $ \k \, $:  by abuse of notation,
 {\it the resulting Lie algebra over  $ \k $  will be denoted by  $ \lieg $  again}.
                                                                  \par
   Let  $ \Phi $  be the root system  of  $ \lieg \, $,  with
   $ \, \Pi = \big\{\, \alpha_i \,\vert\, i\in I \,\big\} \, $  as a set of simple roots,
   $ \, Q = \bigoplus_{i \in I} \ZZ \, \alpha_i \, $  the associated root lattice,
   $ \Phi^+ $  the set of positive roots with respect to
   $ \Pi \, $,  $ \, Q^+ = \bigoplus_{i \in I} \NN \, \alpha_i \, $  the positive root (semi)lattice.
                                                                  \par
   Fix a Cartan subalgebra  $ \lieh $  of  $ \lieg \, $,  whose associated set of roots identifies with  $ \Phi \, $  (so  $ \, \k{}Q \subseteq \lieh^* \, $);
   then for all  $ \, \alpha \in \Phi \, $  we call  $ \lieg_\alpha $  the corresponding root space.
   Now set  $ \, \lieh' := \lieg' \cap \lieh \, $  where  $ \, \lieg' := [\lieg\,,\lieg] \, $  is the derived Lie subalgebra of  $ \lieg \, $:
   then  $ \, {\big( \lieh'\big)}^* = \k{}Q \subseteq \lieh^* \, $.  We fix a  $ \k $--basis
   $ \, \Pi^\vee := {\big\{\, \hrm_i := \alpha_i^\vee \,\big\}}_{i \in I} \, $  of  $ \lieh' $  so that  $ \, \big(\, \lieh \, , \Pi \, , \Pi^\vee \,\big) \, $
   is a  {\it realization\/}  of  $ A \, $,  as in  \cite[Chapter 1]{Ka};  in particular,  $ \, \alpha_i(\hrm_j) = a_{ji} \, $  for all  $ \, i , j \in I \, $.
                                                                \par
   Let  $ \lieh'' $  be any vector space complement of  $ \lieh' $  inside  $ \lieh \, $.
   Then there exists a unique symmetric  $ \k $--bilinear  pairing on  $ \lieh \, $,  denoted  $ (\,\ ,\ ) \, $,
   such that  $ \, (\hrm_i\,,\hrm_j) = a_{ij}\,d_j^{-1} \, $,  $ \, (\hrm_i\,,h''_2) = \alpha_i\big(h''_2\big) \, $
   and  $ \, (h''_1\,,h''_2) = 0 \, $,  for all  $ \, i, j \in I \, $,  $ \, h''_1, h''_2 \in \lieh'' \, $;
   in addition, this pairing is invariant and non-degenerate (cf.\ \cite[Chapter 2]{Ka}).
   By non-degeneracy, this pairing induces a  $ \k $--linear  isomorphism  $ \; t : \lieh^* \,{\buildrel \cong \over {\relbar\joinrel\longrightarrow}}\, \lieh \; $,
   \,and this in turn defines a similar pairing on  $ \, \lieh^* \, $,  again denoted  $ (\,\ ,\ ) \, $,
   via pull-back, namely  $ \, \big( t^{-1}(h_1) , t^{-1}(h_2) \big) := (h_1\,,h_2) \, $;  in particular, on simple roots this gives
   $ \, (\alpha_i \, , \alpha_j) := d_i\,a_{ij} \, $  for all  $ \, i, j \in I \, $.
   In fact, this pairing on  $ \lieh^* $  restricts to a (symmetric,
   $ \ZZ $--valued,  $ \ZZ $--bilinear)  pairing on  $ Q \, $;  note that, in terms of the latter pairing on  $ Q \, $,  one has
   $ \, d_i = (\alpha_i\,,\alpha_i) \big/ 2 \, $  and  $ \, a_{ij} = \frac{\,2\,(\alpha_i , \,\alpha_j)\,}{\,(\alpha_i , \,\alpha_i)\,} \, $  ($ \, i, j \in I \, $).
   Moreover  $ \; t : \lieh^* \,{\buildrel \cong \over {\relbar\joinrel\longrightarrow}}\, \lieh \; $  restricts to another isomorphism
   $ \; t' : {\big( \lieh' \big)}^* \,{\buildrel \cong \over {\relbar\joinrel\longrightarrow}}\, \lieh' \; $  for which we use notation
   $ \, t_\alpha := t'(\alpha) = t(\alpha) \, $.
 \vskip7pt
   Let  $ \lien_+ \, $,  resp.\  $ \lien_- \, $,  be the nilpotent subalgebra in  $ \lieg $
   containing all positive, resp.\  negative, root spaces, and set
   $ \, \lieb_\pm := \lieh \oplus \lien_\pm \, $  be the corresponding Borel subalgebras.
   There is a canonical, non-degenerate pairing between
   $ \lieb_+ $  and  $ \lieb_- \, $,  using which one can construct a
   {\sl Manin double\/}  $ \, \liegd = \mathfrak{b}_+ \oplus \mathfrak{b}_- \, $,
   automatically endowed with a structure of Lie bialgebra   --- roughly,  $ \liegd $  is like  $ \lieg $ {\sl but\/}
   with  {\sl two copies of\/}  $ \lieh $
   inside it (cf.\  \cite{CP}, \S 1.4),  namely  $ \, \lieh_+ := \lieh \oplus 0 \, $  inside  $ \lieb_+ $  and
   $ \, \lieh_- := 0 \oplus \lieh \, $  inside
   $ \lieb_- \, $;  accordingly, we set  $ \, \lieh'_+ := \lieh' \oplus 0 \, $  and  $ \, \lieh'_- := 0 \oplus \lieh' \, $.
   By construction both  $ \lieb_+ $  and
   $ \lieb_- $  lie in  $ \liegd $  as Lie sub-bialgebras.  Moreover, there exists a Lie bialgebra epimorphism
   $ \, \pi_{\liegd} \! : \liegd \!\relbar\joinrel\relbar\joinrel\twoheadrightarrow \lieg \; $
   which maps the copy of  $ \lieb_\pm $  inside
   $ \liegd $  identically onto its copy in  $ \lieg \, $.
 \vskip7pt
   For later use we fix generators  $ \, \erm_i , \hrm_i , \frm_i \, (\, i \in I \,) \, $  in  $ \lieg \, $  as in the usual Serre's presen\-tation of  $ \lieg \, $.  Moreover, for the corresponding elements inside  $ \, \liegd = \lieb_+ \oplus \lieb_- \, $  we adopt notation  $ \, \erm_i := (\erm_i , 0) \, $,  $ \, \hrm^+_i := (\hrm_i , 0) \, $,  $ \, \hrm^-_i := (0 , \hrm_i) \, $  and  $ \, \frm_i := (0 , \frm_i) \, $,  for all  $ \, i \in I \, $.
 {\sl Notice that\/}  we have by construction
  $$  \erm_i \in \lieg_{+\alpha_i}  \quad ,
   \qquad  \hrm_i = d_i^{-1} t_{\alpha_i} \in \lieh  \quad ,  \qquad  \frm_i \in \lieg_{-\alpha_i}
   \qquad \qquad  \forall \;\; i \in I  $$
\end{free text}

\smallskip

   In sight of applications to Lie theory, we introduce, mimicking  \cite[Ch.\ 1]{Ka},
   the notion of  \textsl{realization\/}  of a multiparameter matrix:

\smallskip

\begin{definition}  \label{def: realization of P}
 Let  $ \hbar $  be a formal variable, and  $ \kh $  the ring of formal power series in  $ \hbar $
 with coefficients in  $ \k \, $.  Let  $ \lieh $  be a free  $ \kh $--module  of finite rank, and pick subsets
 $ \; \Pi^\vee := {\big\{ T^+_i , T^-_i \big\}}_{i \in I} \! \subseteq \lieh \; $,
 and
 $ \; \Pi := {\big\{ \alpha_i \big\}}_{i \in I} \subseteq \lieh^* := \Hom_\kh\!\big(\, \lieh \, , \kh \big) \; $.
 For later use, we also introduce the elements  $ \, S_i := 2^{-1} \big(\, T^+_i + T^-_i \big) \, $  and
 $ \,  \varLambda_i := 2^{-1} \big(\, T^+_i - T^-_i \big) \, $  (for  $ \, i \in I \, $)  and the sets
 $ \; \Sigma := {\big\{ S_i \big\}}_{i \in I} \! \subseteq \lieh \; $  and
 $ \; \Lambda := {\big\{ \varLambda_i \big\}}_{i \in I} \! \subseteq \lieh \; $.
 \vskip3pt
   Let  $ \, P \in M_n\big(\kh\big) \, $  be any  $ (n \times n) $--matrix  with entries in  $ \kh \, $.
 \vskip5pt
   \textit{(a)}\;  We call the triple  $ \; \R \, := \, \big(\, \lieh \, , \Pi \, , \Pi^\vee \,\big) \; $
   a  \textsl{realization\/}  of  $ P \, $ over $\kh$,
   with  \textsl{rank\/}  defined as  $ \, \rk(\R) := \rk_\kh(\lieh) \, $,  \,if:
 \vskip3pt
   \quad \textit{(a.1)} \;  $ \; \alpha_j\big(\,T^+_i\big) \, = \, p_{\,ij} \; $,
$ \;\;\; \alpha_j\big(\,T^-_i\big) \, = \, p_{j\,i} \; $,  \quad  for all  $ \, i, j \in I \, $;
 \vskip3pt
   \quad \textit{(a.2)} \;  the set  $ \, \overline{\Sigma} := {\big\{\, \overline{S}_i := S_i \; (\text{\,mod\ } \hbar\,\lieh \,) \big\}}_{i \in I} \, $
   is  $ \Bbbk $--linearly  independent in  $ \, \overline{\lieh} := \lieh \Big/ \hbar\,\lieh \, $   ---
   \textsl{N.B.:}\, this is equivalent to saying that  $ \Sigma $  itself can be completed to a  $ \kh $--basis  of  $ \lieh \, $,
   hence in particular  $ \Sigma $  is  $ \kh $--linearly  independent in  $ \, \lieh \, $.
 \vskip5pt
   \textit{(b)}\;  We call a realization  $ \; \R \, := \, \big(\, \lieh \, , \Pi \, , \Pi^\vee \,\big) \; $
   of the matrix  $ P $,  respectively,
 \vskip3pt
   \quad \textit{(b.1)} \;  \textsl{straight\/}  if the set  $ \, \overline{\Pi} := {\big\{\, \overline{\alpha}_i := \alpha_i \; (\text{\,mod\ } \hbar\,\lieh^* ) \big\}}_{i \in I}  \, $
   is  $ \Bbbk $--linearly  independent in  $ \, \overline{\lieh^*} := \lieh^* \Big/ \hbar\,\lieh^* \, $
   ---  \textsl{N.B.:}\, this is equivalent to saying that  $ \Pi $ can be completed to a  $ \kh $--basis  of  $ \lieh^* \, $,
   thus in particular  $ \Pi $  is  $ \kh $--linearly  independent in  $ \, \lieh^* \, $;
 \vskip3pt
   \quad \textit{(b.2)} \;  \textsl{small\/}  if
   $ \,\; \textsl{Span}_\kh\big( {\{ S_i \}}_{i \in I} \big) \; = \; \textsl{Span}_\kh\big( {\big\{ T_i^+ , T_i^- \big\}}_{i \in I} \big) \;\, $;
 \vskip3pt
   \quad \textit{(b.3)} \;  \textsl{split\/}  if the set
   $ \, \overline{\Pi^\vee} := {\big\{\, \overline{T^\pm}_i := T^\pm_i \; (\text{\,mod\ } \hbar\,\lieh \,) \big\}}_{i \in I} \, $
   is  $ \Bbbk $--linearly  independent in  $ \, \overline{\lieh} := \lieh \Big/ \hbar\,\lieh \, $   ---   \textsl{N.B.:}\,
   this is equivalent to saying that  $ \Pi^\vee $  can be completed to a  $ \kh $--basis  of  $ \lieh \, $,
   hence in particular it is  $ \kh $--linearly  independent in  $ \, \lieh \, $;
 \vskip3pt
   \quad \textit{(b.4)} \;  \textsl{minimal\/}  if
   $ \,\; \textsl{Span}_\kh\big( {\big\{ T_i^+ , T_i^- \big\}}_{i \in I} \big) \, = \; \lieh \;\, $
   ---   \textsl{N.B.:}\,  in particular,  $ \R $  is  \textsl{split\/  {\rm and}  minimal\/}  if and only if
   $ \, {\big\{\, T_i^+ , T_i^- \big\}}_{i \in I} \, $  is a  $ \kh $--basis  of  $ \lieh \, $.
 \vskip5pt
   \textit{(c)}\;  For any pair of realizations  $ \; \R \, := \, \big(\, \lieh \, , \Pi \, , \Pi^\vee \,\big) \; $  and
   $ \; \dot{\R} \, := \, \big(\, \dot{\lieh} \, , \dot{\Pi} \, , {\dot{\Pi}}^\vee \,\big) \; $  of the same matrix  $ P $,
   a \textsl{(homo)morphism\/}  $ \, \underline{\phi} : \R \longrightarrow \dot{\R} \, $  is the datum of any
   $ \kh $--module morphism  $ \, \phi : \lieh \longrightarrow \dot{\lieh} \, $  such that
   $ \, \phi\big(T_i^\pm\big) = {\dot{T}}_{\sigma(i)}^\pm \, $
   (for all  $ \, i \in I \, $)  for some permutation  $ \, \sigma \in \mathbb{S}_{I} \, $   --- the symmetric group over  $ I $
   ---   hence, in particular,  $ \, \phi\big(\Pi^\vee\big) = {\dot{\Pi}}^\vee \, $,
   and also that  $ \, \phi^*\big(\dot{\Pi}\big) = \Pi \, $
   ---  \textsl{N.B.:}  realizations along with their morphisms form a category,
   in which the iso--/epi--/mono--morphisms are those morphisms
   $ \phi $  as above that actually are  $ \kh $--module  iso--/epi--/mono--morphisms.
 \vskip5pt
   \textit{(d)}\;  Let  $ \, A := {\big(\, a_{ij} \big)}_{i, j \in I} \in M_{n}(\Bbbk) \, $  be any symmetrisable generalized
   Cartan matrix, and  $ \, D := {\big(\hskip0,7pt d_i \, \delta_{ij} \big)}_{i, j \in I} \, $  the associated diagonal matrix,
   as in  \S \ref{root-data_Lie-algs}.  We say that a matrix  $ \, P \in M_n(\kh) \, $  is  \textsl{of Cartan type\/}
   with corresponding Cartan matrix  $ A $  if  $ \; P_s := 2^{-1} \big( P + P^{\,\scriptscriptstyle T} \big) = DA \; $.
 \vskip5pt
   \textit{N.B.:}\;  condition  \textit{(b.3)\/}  is equivalent  to requiring that
   $ \, \overline{\Sigma} \cup \overline{\Lambda} \, $  be  $ \Bbbk $--linearly  independent in
   $ \, \overline{\lieh} := \lieh \Big/ \hbar\,\lieh \, $;  \,in turn, this is equivalent to saying that
   $ \, \Sigma \cup \Lambda \, $  itself can be completed to a  $ \kh $--basis  of
   $ \lieh \, $,  hence in particular it is  $ \kh $--linearly  independent.  Similarly, the condition
   $ \, \phi\big(T_i^\pm\big) = {\dot{T}}_{\sigma(i)}^\pm \, $   ---  $ \, i \in I \, $,
   for some permutation  $ \, \sigma \in \mathbb{S}(I) \, $  ---   in  \textit{(c)\/}  can be replaced by
 $ \, \phi(S_i) = {\dot{S}}_{\sigma(i)} \, $  and  $ \, \phi(\varLambda_i) = {\dot{\varLambda}}_{\sigma(i)} \, $.
 \vskip5pt
   \textit{(e)}\;  In an entirely similar way, one may define realizations of a matrix
   $ \, P := {\big(\, p_{i,j} \big)}_{i, j \in I} \in M_n(\Bbbk) \, $  over a ground field  $ \Bbbk \, $.
   Such a realization  $ \, \R := \big(\, \lieh \, , \Pi \, , \Pi^\vee \,\big) \, $ consists of a  $ \Bbbk $--vector  space
   $ \lieh $  and distinguished subsets  $ \; \Pi^\vee := {\big\{ T^+_i , T^-_i \big\}}_{i \in I} \! \subseteq \lieh \; $
   and  $ \; \Pi := {\big\{ \alpha_i \big\}}_{i \in I} \subseteq \lieh^* := \Hom_\Bbbk\!\big(\, \lieh \, , \Bbbk \big) \; $:
   \,then condition  \textit{(a.1)\/}  reads the same, while  \textit{(a.2)\/}  instead says that the set
   $ \, {\big\{\, S_i = 2^{-1}\big(T^+_i + T^-_i\big) \,\big\}}_{i\in I} \, $  is linearly independent, and the
   \textsl{rank\/}  of the realization is  $ \, \rk(\R) := \dim_\Bbbk(\lieh) \, $.  Also,  $ \R $  is  \textsl{straight},
   resp.\  \textsl{split},  if  $ \Pi \, $,  resp.\  $ \Pi^\vee $,  is linearly independent.
                                                             \par
   Basing on the context, we shall possibly stress the ring we are working over, namely
   $ \kh $  for  $ \, P \in M_n\big(\kh\big) \, $  and the field  $ \Bbbk $  for  $ \, P\in M_n(\Bbbk) \, $.
\hfill  $ \diamondsuit $
\end{definition}

\vskip5pt

\begin{rmk}  \label{rmk: Kac'-realiz}
 In the present language, if  $ \, P = P^{\,\scriptscriptstyle T} \, $  is symmetric a  \textsl{realization\/}  of it in the sense of  \cite[Ch.\ 1, \S 1.1]{Ka},
 is also a realization, in the sense of  Definition \ref{def: realization of P},  of  $ P $  which has rank  $ \, 2n-r \, $,
 \,is  \textsl{straight\/}  and  \textsl{small\/}  with  $ \; \varLambda_i = 0 \, $  for all  $ \, i \in I \, $.
\end{rmk}

\vskip5pt

   The following consequence of the definitions yields another link with Kac' notion:

\vskip11pt

\begin{lema}  \label{lem: realiz_P => realiz_A}  {\ }
 Let  $ \, P \in M_n(\kh) \, $  be a matrix as above.  If  $ \; \R := \big(\, \lieh \, , \Pi \, , \Pi^\vee \,\big) \, $
 is a straight realization of  $ \, P $,  then the triple  $ \, \big(\, \lieh \, , \Pi \, , \Pi^\vee_S \big) \, $
 --- with  $ \, \Pi^\vee_S \! := {\{S_i\}}_{i \in I} \, $  ---   is a realization of  $ \, P_s := 2^{-1} \big( P + P^{\,\scriptscriptstyle T} \big) \, $
 --- over the ring  $ \kh $  ---   in the sense of  \cite{Ka}, Ch.\ 1, \S 1.1, but for condition (1.1.3).
\end{lema}

\vskip7pt

   Note that condition (1.1.3) in  \cite[Ch.\ 1, \S 1.1]{Ka},  is fulfilled whenever  $ \, \rk(\lieh) = 2\,n - \rk(P_s) \, $;
   \,in particular, we can always achieve that condition up to suitably enlarging or restricting  $ \lieh \, $.
   In any case, from now on with any straight realization of a matrix  $ P $  of Cartan type, for some Cartan matrix  $ A \, $,
   we shall always associate the realization of  $ \, P_s = DA \, $  given by  Lemma \ref{lem: realiz_P => realiz_A},
   hence also the corresponding realization of  $ A $  and then all the related data and machinery mentioned in  \S \ref{MpMatrices, Cartan & realiz.'s}.

\vskip9pt

   We need now a few technical results:

\vskip9pt

\begin{prop}  \label{prop: exist-realiz's}  {\ }
 \vskip3pt
   (a)\;  For every  $ \, P \in M_n\big(\kh\big) \, $  and every  $ \, \ell \geq 3\,n-\rk\big(P+P^{\,\scriptscriptstyle T}\big) \, $,
   \,there exists a straight split realization of  $ P $  with  $ \, \rk(\lieh) = \ell \, $,
 which
 is unique up to isomorphisms.
 \vskip3pt
   (b)\;  Claim (a) still holds true if we drop the condition ``straight'' and pick  $ \, \ell \geq 2\,n \, $.
\end{prop}

\pf
   \textit{(a)}\;  Let  $ \; r := \rk\big(P+P^{\,\scriptscriptstyle T}\,\big) \; $  and  $ \; \ell \geq 3\,n-r \; $  be fixed.  We set
  $$  S_i := 2^{-1} \big(\, T_i^+ + T_i^- \big) \;\; ,   \qquad   \varLambda_i :=
  2^{-1} \big(\, T_i^+ - T_i^- \big) \qquad \quad  \forall \;\; i \in I  $$
 for any choice of elements  $ \, T_i^\pm \, (i \in I) \, $  in any  $ \kh $--module  $ \lieh \, $;  \,then
 $ \; T_i^\pm = S_i \pm \varLambda_i \; $  for all  $ \, i \in I \, $,  \,so we have
 $ \, \textsl{Span}_\kh\Big( {\big\{ T_i^+ , T_i^- \big\}}_{i \in I}\Big) = \textsl{Span}_\kh\Big( {\big\{ S_i \, , \varLambda_i \big\}}_{i \in I}\Big) \, $.
 Therefore,
  \textsl{the existence of a split realization of  $ P $  amounts to the
  same as the existence of the slightly modified notion where:
 \vskip3pt
   --- instead of the  $ T_i^+ $'s  and the  $ T_i^- $'s  one considers the  $ S_i $'s and the  $ \varLambda_i $'s,
 \vskip3pt
   --- condition  \textit{(a.1)}  in  Definition 1\  is replaced by condition
 \vskip1pt
   \;\; \textit{(a.1+)}\;  $ \; \alpha_j(S_i) \, = \, 2^{-1} (p_{\,ij} + p_{\,ji}) \; $,
$ \;\; \alpha_j(\varLambda_i) \, = \, 2^{-1} (p_{\,ij} - p_{j\,i})  \qquad  \forall \;\; i, j \in I \, $;
 }
 \vskip5pt
   Therefore, we look now for such a ``realization'' in this alternative sense.
   We consider the matrices (respectively symmetric and antisymmetric)
  $$  P_s \, := \, 2^{-1} \big( P + P^{\,\scriptscriptstyle T} \big) \;\; ,   \qquad   P_a \, := \, 2^{-1} \big( P - P^{\,\scriptscriptstyle T} \big)  $$
%
%
 and then, reordering the indices in  $ I $  if necessary, we re-write the matrix  $ P_s $  in the block form
 $ \; P_s \, = \begin{pmatrix}
     P_s^{\,{}_{\scriptstyle \ulcorner}}  &  P_s^{\,{}_{\scriptstyle \urcorner}}   \\
     P_s^{\,\llcorner}  &  P_s^{\,\lrcorner}
            \end{pmatrix} \; $
 where  $ P_s^{\,{}_{\scriptstyle \ulcorner}} $ has size $r\times r$,  $ \, \rk\big( P_s^{\,{}_{\scriptstyle \ulcorner}} \big) = r \, $,
 \,and the other blocks have the corresponding sizes;  \textsl{according to the same reordering of the indices\/}
 (if any), we also re-write  $ P_a $  as
 $P_a \, = \begin{pmatrix}
     P_a^{\,{}_{\scriptstyle \ulcorner}}  &  P_a^{\,{}_{\scriptstyle \urcorner}} \,   \\
     P_a^{\,\llcorner}  &  P_a^{\,\lrcorner} \,
            \end{pmatrix} \; $
 with  $ P_a^{\,{}_{\scriptstyle \ulcorner}} $  of size  $ \, r \times r \, $,  \,and so on.
                                                               \par
   Now we consider the  $ \, \ell \times \ell \, $  matrix
 \begin{equation}  \label{eq: matrix G_P - Z}
    G_P  \; = \;  \begin{pmatrix}
  \; P_s^{\,{}_{\scriptstyle \ulcorner}}  &  P_s^{\,{}_{\scriptstyle \urcorner}}  &  0  &  0  &  0  &  0 \;   \\
  \; P_s^{\,\llcorner}  &  P_s^{\,\lrcorner}  &  \;\; I_{n-r}  &  0  &  0  &  0 \;   \\
  \; 0  &  \;\; I_{n-r}  &  0  &  0  &  0  &  0 \;   \\
  \; P_a^{\,{}_{\scriptstyle \ulcorner}}  &  P_a^{\,{}_{\scriptstyle \urcorner}}  &  0  &  \, I_r  &  0  &  0 \;   \\
  \; P_a^{\,\llcorner}  &  P_a^{\,\lrcorner}  &  0  &  0  &  \;\; I_{n-r}  &  0 \;   \\
  \; 0  &  0  &  0  &  0  &  0  &  \;\; I_{\ell-(3\,n-r)} \;
                  \end{pmatrix}
 \end{equation}
 that is non-degenerate, as  $ \; \det \big(G_P\big) = \pm \det \big(P_s^{\,{}_{\scriptstyle \ulcorner}}\big) \not= 0 \; $.
 Now, set  $ \; \lieh_\bullet := \kh^{3\,n-r} \, $,  \,fix as  $ S_i $'s,  respectively  $ \varLambda_i $'s  ($ \, i \in I \, $),
 the rows of  $ G_P $  (as vectors in  $ \lieh_\bullet $)  from  $ 1 $  to  $ n \, $, respectively from  $ \, 2\,n-r+1 \, $  to
 $ \, 3\,n-r \, $,  \,and fix as  $ \alpha_j $'s  ($ \, j \in I \, $)  the first  $ n $  linear coordinate functions on  $ \lieh_\bullet $
 (as vectors in  $ \lieh_\bullet^* \, $).  Now set  $ \, \Pi^\vee_{\scriptscriptstyle S, \varLambda} := {\big\{ S_i \, , \varLambda_i \big\}}_{i \in I} \, $
 and let  $ \lieh $  be the  $ \kh $--span  (inside  $ \lieh_\bullet $)  of the rows of  $ G_P \, $;  then the  $ \alpha_j $'s
 restrict to elements of  $ \lieh^* $  (that we still denote by  $ \alpha_j \, $)  hence we consider  $ \, \Pi := {\{\alpha_j\}}_{j \in I} \, $
 as a subset in  $ \lieh^* \, $.  Now the triple
 $ \; \R_{\scriptscriptstyle S, \varLambda} := \big(\, \lieh \, , \Pi \, , \Pi^\vee_{\scriptscriptstyle S, \varLambda} \,\big) \; $
 is a ``realization'' (in the present, modified sense) of  $ P $  which is straight split, thus proving the existence part of claim  \textit{(a)}.
 \vskip5pt
   As to uniqueness, we reverse the previous line of arguing.  Indeed, given a split ``realization'', in modified sense,
   $ \, \R_{\scriptscriptstyle S, \varLambda} := \big(\, \lieh \, , \Pi \, , \Pi^\vee_{\scriptscriptstyle S, \varLambda} \big) \, $  of  $ P $,
   we complete  $ \Pi^\vee_{\scriptscriptstyle S, \varLambda} $  to a  $ \kh $--basis  of  $ \lieh $  adding extra elements
   $ Y_1 \, , \dots , Y_{\ell-2n} \, $;  \,moreover, we define additional  $ \, \alpha_{n+1} , \dots , \alpha_\ell \in \lieh^* \, $
   such that the matrix of all values of the  $ \alpha_j $'s
 on the elements of the ordered basis
%
%
  $ \; \big\{\, S_1 , \dots , S_n \, , Y_1 \, , \dots , Y_{n-r} , \varLambda_1 , \dots , \varLambda_n \, , Y_{n-r+1} \, , \dots , Y_{\ell-2n} \,\big\} \; $
 is given by
 \begin{equation}  \label{eq: matrix isom-real's - Z}
    N_P  \; = \;  \begin{pmatrix}
  \; P_s^{\,{}_{\scriptstyle \ulcorner}}  &  P_s^{\,{}_{\scriptstyle \urcorner}}  &  0  &  0  &  0  &  0 \;   \\
  \; P_s^{\,\llcorner}  &  P_s^{\,\lrcorner}  &  \;\; I_{n-r}  &  0  &  0  &  0 \;   \\
  \; B_{\scriptscriptstyle <}  &  B_{\,\scriptscriptstyle >}  &  0  &  0  &  0  &  0 \;   \\
  \; P_a^{\,{}_{\scriptstyle \ulcorner}}  &  P_a^{\,{}_{\scriptstyle \urcorner}}  &  0  &  \, I_r  &  0  &  0 \;   \\
  \; P_a^{\,\llcorner}  &  P_a^{\,\lrcorner}  &  0  &  0  &  \;\; I_{n-r}  &  0 \;   \\
  \; D_{\scriptscriptstyle <}  &  D_{\,\scriptscriptstyle >}  &  0  &  0  &  0  &  \;\; I_{\ell-(3\,n-r)} \;
                  \end{pmatrix}
 \end{equation}
 which by construction is non-degenerate.
%
%
 Now, let us extend scalars for a while from  $ \kh $  to  $ \Bbbk(\!(\hbar)\!) \, $:  \,then by Gauss' elimination
 algorithm on the rows (involving in particular the first  $ r $  rows) we can modify the matrix  $ N_P $  in
 \eqref{eq: matrix isom-real's - Z}  till it gets a new form where  $ \, B_{\scriptscriptstyle <} = 0 \; $  and
 $ \, D_{\scriptscriptstyle <} = 0 \; $;  \,moreover, the ``new'' submatrix  $ B_{\,\scriptscriptstyle >} $  fulfills
  $$  \det\big(P_s^{\,{}_{\scriptstyle \ulcorner}}\big) \, \det (B_{\,\scriptscriptstyle >}) = \pm \det\big(N_P\big) \not= 0 \;\; ,
  \quad \text{hence} \quad  B_{\,\scriptscriptstyle >} \in \textsl{GL}_{n-r}\big(\,\Bbbk(\!(\hbar)\!)\big)  $$
 hence we can choose another basis  in  $ \, \textsl{Span}_{\Bbbk(\!(\hbar)\!)}\big(\, Y_1 \, , \dots , Y_{n-r} \big) \, $
 so to get  $ \, B_{\,\scriptscriptstyle >} = I_{n-r} \, $.  Then another Gauss' elimination process involving the rows from
 $ n+1 $  to  $ 2n-r $  allows us to modify the last  $ \ell-(3n-r) $  rows so as to get  $ \, D_{\scriptscriptstyle >} = 0 \, $.
                                                                      \par
   All this gives us a new split realization (in modified sense) of  $ P $  \textsl{over\/}  $ \Bbbk(\!(\hbar)\!) $
   and a specific basis, including the  $ S_i $'s  and the  $ \varLambda_i $'s,  of the  $ \Bbbk(\!(\hbar)\!) $--vector space
   $ \lieh_{\Bbbk(\!(\hbar)\!)} $  considered in it: eventually, taking as  $ \lieh $  the  $ \kh $--span
   of that basis we can easily read off that that ``\,realization over  $ \Bbbk(\!(\hbar)\!) \, $''  is a genuine realization \textsl{over\/}  $ \kh \, $,
   which is isomorphic to the original one, by construction (indeed, we only modified a direct sum complement of
   $ \, \textsl{Span}_{\Bbbk(\!(\hbar)\!)}\big( {\{ S_i \, , \varLambda_i \}}_{i \in I} \big) \, $  by a sheer rescaling, at most).
   But now, for this final realization the matrix  $ N_P $  in  \eqref{eq: matrix isom-real's - Z}
   takes the same form as  $ G_P $  in  \eqref{eq: matrix G_P - Z}:  so taking as
   $ \, \phi : \lieh \relbar\joinrel\longrightarrow \lieh \, $
   the isomorphism given by change of bases, we are done.
 \vskip7pt
   \textit{(b)}\;  As claim  \textit{(a)\/}  already guarantees the existence of straight realizations,
   the relevant part of claim  \textit{(b)\/}
   concerns the uniqueness, that is proved again like for  \textit{(a)},  up to minimal changes.
   Namely, instead of the matrix in
   \eqref{eq: matrix isom-real's - Z}  we deal with
  $$  N^{\,\prime}_{\!P}  \; = \;  \begin{pmatrix}
   \; P_s^{\,{}_{\scriptstyle \ulcorner}}  &  P_s^{\,{}_{\scriptstyle \urcorner}}  &  0  &  0  &  0  &  0 \;   \\
   \; P_s^{\,\llcorner}  &  P_s^{\,\lrcorner}  &  \;\; I_{n-r}  &  0  &  0  &  0 \;   \\
   \; P_a^{\,{}_{\scriptstyle \ulcorner}}  &  P_a^{\,{}_{\scriptstyle \urcorner}}  &  0  &  \, I_r  &  0  &  0 \;   \\
   \; P_a^{\,\llcorner}  &  P_a^{\,\lrcorner}  &  0  &  0  &  \;\; I_{n-r}  &  0 \;   \\
   \; D_{\scriptscriptstyle <}  &  D_{\,\scriptscriptstyle >}  &  0  &  0  &  0  &  \;\; I_{\ell-2\,n} \;
                  \end{pmatrix}  $$
 and then we observe that we can again reduce it to a similar matrix where  $ \, D_< = 0 \, $
 --- acting by Gauss' elimination on the rows, exploiting the
 nonsingular square submatrix  $ P_s^{\,{}_{\scriptstyle \ulcorner}} $  ---
 and  $ \, D_> = 0 \, $ --- where we perform another Gauss' elimination  \textsl{on the columns\/}
 (which in the end amounts to modifying the  $ \alpha_j $'s),
 exploiting the nonsingular square submatrix  $ I_{\ell - 2\,n} \, $.
\epf

\vskip5pt

\begin{rmk}
 It follows from definitions that a  \textsl{necessary\/}  condition for a small realization of any  $ \, P \in M_n\big(\kh\big) \, $
 to exist is  $ \, \rk\!\big( P_s \,\big| P_a \big) = \rk(P_s) \, $.
 Conversely, with much the same arguments used in the proof of
 Proposition \ref{prop: exist-realiz's},  we can prove that such a condition is also
 \textsl{sufficient},  as the following holds true, indeed:
\end{rmk}

\vskip5pt

\begin{prop}  \label{prop: exist-realiz's_small}
 If  $ \, P \in M_n\big(\kh\big) \, $  is such that  $ \, \rk\!\big( P_s \,\big| P_a \big) = \rk(P_s) \, $,
 then, for all  $ \, \ell \geq 2\,n-\rk(P_s) \, $,
 there exists a straight small realization of  $ \, P $  with  $ \, \rk(\lieh) = \ell \, $,
\,and such a realization is unique up to isomorphisms.
\end{prop}

\vskip5pt

   After this existence results concerning realizations of special type, we can achieve a
   more general result with two additional steps.
   The first one tells in short that every realization can be ``lifted'' to a \textsl{split\/}  one:

\vskip13pt

\begin{lema}  \label{lemma: split-lifting}
 Let  $ \; \R := \big(\, \lieh \, , \Pi \, , \Pi^\vee \,\big) \; $  be a realization of  $ \, P \in M_n\big(\kh\big) \, $.
 Then there exists a  \textsl{split}  realization  $ \; \dot{\R} := \big(\, \dot{\lieh} \, , \dot{\Pi} \, , {\dot{\Pi}}^\vee \,\big) \, $
 of the same matrix  $ P $  and an epimorphism of realizations  $ \; \underline{\pi} : \dot{\R} \relbar\joinrel\twoheadrightarrow \R \; $
 such that, if  $ \, \lieh_{{}_T} := \textsl{Span}\Big( {\big\{ T_i^\pm \big\}}_{i \in I} \Big) \, $  and
 $ \, \dot{\lieh}_{{}_T} := \textsl{Span}\Big( {\big\{ \dot{T}_i^\pm \big\}}_{i \in I} \Big) \, $,  then  $ \, \underline{\pi} \, $
 induces an isomorphism  $ \; \pi_* : \dot{\lieh} \Big/ \dot{\lieh}_{{}_T} \cong \lieh \Big/ \lieh_{{}_T} \; $.
                                                                               \par
   If in addition  $ \, \R $  is  \textsl{straight},  resp.\ \textsl{minimal},  then a  \textsl{split}  realization  $ \, \dot{\R} $
   as above can be found that is  \textsl{straight},  resp.\ \textsl{minimal},  as well.
\end{lema}

\pf
 We proceed in two steps, first working over scalar extensions from  $ \kh $  to  $ \Bbbk(\!(\hbar)\!) $  and then
 ``pulling back'' our result to the original setup.  To this end, hereafter, for any  $ \kh $--module  $ \mathfrak{m} $
 we write  $ \, \mathfrak{m}^{(\hbar)} := \Bbbk(\!(\hbar)\!) \otimes_\kh \mathfrak{m} \, $.
                                                                  \par
   Let  $ \, \lieh_{{}_T} := \textsl{Span}\Big( {\big\{ T_i^+ , T_i^- \big\}}_{i \in I} \Big) \, $.  Then  $ \lieh_{{}_T}^{(\hbar)} $
   embeds into  $ \lieh^{(\hbar)} $  and the latter splits into  $ \; \lieh^{(\hbar)} = \lieh_{{}_T}^{(\hbar)} \oplus \lieh'_\circ \; $
   for some  $ \Bbbk(\!(\hbar)\!) $--submodule  $ \lieh'_\circ $  in  $ \lieh^{(\hbar)} \, $.  Now fix formal vectors
   $ \, \dot{T}_i^\pm \, (\, i \in I \,) \, $,  the free  $ \khp $--module  $ \lieh'_{{}_T} $  with  $ \Bbbk(\!(\hbar)\!) $--basis
   $ \, {\Pi'}^\vee \! := {\big\{ \dot{T}_i^+ , \dot{T}_i^- \big\}}_{i \in I} \, $, and  the
   $ \Bbbk(\!(\hbar)\!) $--module  epimorphism
   $ \; \pi' : \lieh'_\oplus := \lieh'_{{}_T} \oplus \lieh'_\circ \!\relbar\joinrel\relbar\joinrel\twoheadrightarrow \lieh^{(\hbar)} \; $
   given by  $ \, \pi'\big( \dot{T}_i^\pm \big) := T_i^\pm \; $  ($ \, i \in I \, $)  and
   $ \, \pi{\big|}_{\lieh_\circ} \! := \textsl{id}_{\lieh_\circ} \, $.
   If we let  $ \alpha^{(\hbar)}_j $  be the natural scalar extension of  $ \alpha_j $  ($ \, j \in I \, $),
   then every such  $ \alpha^{(\hbar)}_j $
   is a  $ \Bbbk(\!(\hbar)\!) $--linear  function from  $ \lieh^{(\hbar)} $  to  $ \khp $,
   and altogether the  $ \alpha^{(\hbar)}_j $'s
   are linearly independent over  $ \khp $  if the  $ \alpha_j $'s  are;  \,therefore, the set
   $ \, \Pi' := {\big\{\, \alpha'_j := \alpha_j^{(\hbar)} \circ \pi' \,\big\}}_{j \in I} \, $  lies inside the  $ \khp $--dual
   module of  $ \lieh'_\oplus \, $,  and it is also  $ \khp $--linearly  independent if  $ \Pi $  is   --- i.e., if we are in the  \textsl{straight\/}  case.
                                                                                 \par
   Now look at  $ \lieh $  embedded into  $ \lieh^{(\hbar)} $  and set
 $ \, \dot{\lieh} := {(\pi')}^{-1}(\lieh) \, $,
 $ \, {\dot{\Pi}}^\vee := {\big(\pi'{\big|}_{\dot{\lieh}}\big)}^{-1}\big(\Pi^\vee\big) = {\Pi'}^\vee \, $,
 $ \, \dot{\Pi} := {\big(\pi'{\big|}_{\dot{\lieh}}\big)}^{\!*}(\Pi) =
 {\Big\{ {\big(\pi'{\big|}_{\dot{\lieh}}\big)}^{\!*}\Big(\alpha_j^{(\hbar)}\Big) :=
 \alpha_j^{(\hbar)} \circ \pi'{\big|}_{\dot{\lieh}} = \alpha'_j{\big|}_{\dot{\lieh}} \,\Big\}}_{j \in I} \, $;
 then  $ \dot{\lieh} $  is a free  $ \kh $--module  containing  $ \dot{\Pi}^{\vee} $
 and such that  $ \, \dot{\lieh}^{(\hbar)} = \lieh'_\oplus \, $,
 while  $ \dot{\Pi} $  is a subset in the  $ \kh $--dual  of  $ \dot{\lieh} \, $,
 that is even  $ \kh $--linearly  independent if  $ \Pi $
 is   --- i.e., if we are in the  \textsl{straight\/}  case.  Even more, looking in depth we find that
 $ \; \dot{\R} \, := \, \big(\, \dot{\lieh} \, , \dot{\Pi} \, , {\dot{\Pi}}^\vee \,\big) \; $  is indeed a
 \textsl{split\/}  realization of the matrix  $ P $   --- which is also  \textsl{straight},  resp.\ \textsl{minimal}
 if the original  $ \R $  is straight,  resp.\ minimal ---   that together with the epimorphism
 $ \, \pi := \pi'{\big|}_{\dot{\lieh}} : \dot{\lieh} \relbar\joinrel\relbar\joinrel\twoheadrightarrow \lieh \, $
 yields all that is prescribed in the claim.  Indeed, we only have to point out the last step,
 noting that  $ \pi $  induces an isomorphism
 $ \; \pi_* : \dot{\lieh} \Big/ \pi^{-1}(\lieh_{{}_T}) \cong \lieh \Big/ \lieh_{{}_T} \; $
 and then observing that, by construction, we have
 $ \, \pi^{-1}(\lieh_{{}_T}) = \dot{\lieh}_{{}_T} \, $.
\epf

\vskip5pt

   A last result concerns morphisms between realizations.

\vskip13pt

\begin{lema}  \label{lemma: ker-morph's_realiz's}
 Let  $ \; \hat{\R} \, := \, \big(\, \hat{\lieh} \, , \hat{\Pi} \, , \hat{\Pi}^\vee \,\big) \; $  and
 $ \; \check{\R} \, := \, \big(\, \check{\lieh} \, , \check{\Pi} \, , \check{\Pi}^\vee \,\big) \; $
 be two realizations of a same  $ \, P \in M_n\big(\kh\big) \, $,  and let
 $ \; \underline{\phi} : \hat{\R} \longrightarrow \check{\R} \; $
 be a morphism between them.  Then
 $ \; \Ker\big(\, \phi : \hat{\lieh} \longrightarrow \check{\lieh} \,\big) \, \subseteq \,
 \bigcap\limits_{j \in I} \Ker(\hat{\alpha}_j) \;\, $.
\end{lema}

\pf
 Since  $ \, \hat{\alpha}_j = \phi^*(\check{\alpha}_j) = \check{\alpha}_j \circ \phi \, $  ($ \, j \in I \, $)
 by assumption, for all  $ \, k \in \Ker(\phi) \, $
 we have  $ \; \hat{\alpha}_j(k) = (\check{\alpha}_j \circ \phi)(k) = \check{\alpha}_j \big(\phi(k)\big) =
 \check{\alpha}_j(0) = 0 \; $  ($ \, j \in I \, $),  whence the claim.
\epf

\vskip9pt

\begin{rmk}
 Working with matrices in  $ \, M_n(\Bbbk) \, $  and realizations of them over
 $ \Bbbk $  (for some field  $ \Bbbk \, $),
 all the previous constructions still make sense; some results
 (e.g.,  Proposition \ref{prop: exist-realiz's}  and  Lemma \ref{lemma: split-lifting})
 even get stronger and/or have simpler proofs.
\end{rmk}

\medskip

\subsection{Twist deformations of multiparameters and realizations}
  \label{subsec: twisted-realiz}  {\ }
 \vskip7pt

   In this subsection we introduce the notion of  \textit{deformation by twist\/}  of realizations,
   which will be needed later when dealing with deformations of multiparameter
   Lie bialgebras and formal multiparameter quantum universal enveloping algebras.

\vskip9pt

\begin{free text}  \label{twist-deform.'s x mpmatr.'s & realiz.'s}
 {\bf Deforming realizations (and matrices) by twist.}  Fix  $ \, P := {\big(\, p_{i,j} \big)}_{i, j \in I} \in M_n\big(\kh\big) \, $
 and a realization  $ \; \R \, := \, \big(\, \lieh \, , \Pi \, , \Pi^\vee \,\big) \, $,
 \,possibly (up to changing minimal details in what follows) over  $ \Bbbk $  if  $ \, P \in M_n(\Bbbk) \, $;
 in particular  $ \, d_i := p_{ii}/2 \, $  for all  $ \, i \in I \, $.
                                                                       \par
   Recall that  $ \lieh $  is, by assumption, a free  $ \kh $--module  of finite rank  $ \, t := \rk(\lieh) \, $.
   We fix in  $ \lieh $  any  $ \kh $--basis  $ \, {\big\{ H_g \big\}}_{g \in \G} \, $,
   where  $ \G $  is an index set with  $ \, |\G| = \rk(\lieh) = t \, $.
                                                                       \par
   Fix an antisymmetric square matrix  $ \; \Phi = \big( \phi_{gk} \big)_{g, k \in \G} \in \lieso_t\big(\kh\big) \; $
   --- indeed, we might work with any  $ \, \Phi \in M_t\big(\kh\big) \, $,
   \,but at some point we should single out its antisymmetric part
   $ \, \Phi_a := 2^{-1} \big( \Phi - \Phi^{\scriptscriptstyle T} \,\big) \, $
   which would be all that matters.  We define the  \textsl{twisted ``distinguished toral elements'' (or ``coroots'')}
\begin{equation}  \label{eq: T-phi}
  T^\pm_{{\scriptscriptstyle \Phi},\ell} \; := \;
  T^\pm_\ell \, \pm {\textstyle \sum\limits_{g,k=1}^t} \alpha_\ell(H_g) \, \phi_{kg} \, H_k
\end{equation}
 As a matter of notation, let
 $ \, \fT := \bigg(\, {\displaystyle {\underline{T}^+ \atop \underline{T}^-}} \bigg) \, $
 be the  $ (2n \times 1) $--matrix  given by the column vectors
 $ \, \underline{T}^\pm = {\big(\, T_i^\pm \big)}_{i \in I} \, $.  Similarly, let
 $ \; \fT_{{\scriptscriptstyle \Phi}} :=
 \bigg(\, {\displaystyle {\underline{T}^+_{\,\scriptscriptstyle \Phi} \atop \underline{T}^-_{\,\scriptscriptstyle \Phi}} }\bigg) \; $
 be the  $ (2n \times 1) $--matrix  given by the (superposed) two column vectors
 $ \, \underline{T}^\pm_{\,\scriptscriptstyle \Phi} = {\big(\, T_{{\scriptscriptstyle \Phi}, i}^\pm \big)}_{i \in I} \, $,
 \,and  $ \fH $  the column vector  $ \, \fH := {\big(\, H_g \big)}_{g \in \G} \, $.
 Moreover, denote by  $ \mathfrak{A} $  the  $ (n \times t) $--matrix  with entries in  $ \kh $  given by
 $ \; \mathfrak{A} := {\big(\, \alpha_\ell(H_g) \big)}_{\ell \in I}^{g \in \G} \; $,  \,and set
 $ \; \mathfrak{A}_\bullet  :=  \begin{pmatrix}
        +\mathfrak{A} \,  \\
        -\mathfrak{A} \,
                        \end{pmatrix} \; $
 --- a matrix of size  $ (2n \times t) \, $.
%
%
%
%
 \vskip3pt
   Now, using matrix notation we have
  $ \; \fT_{\scriptscriptstyle \Phi} \, = \, \fT - \, \mathfrak{A}_\bullet \, \Phi \, \fH \; $.
 Eventually, define also
\begin{equation}  \label{def-P_Phi}
  P_{\scriptscriptstyle \Phi} \, =
   \, {\big(\, p^{\scriptscriptstyle \Phi}_{i,j} \big)}_{i, j \in I} \, :=
   \, P - \mathfrak{A} \, \Phi \, \mathfrak{A}^{\,\scriptscriptstyle T}
\end{equation}

\vskip5pt

   Now, using the above notation, a direct computation yields
  $$  \displaylines{
   S_{{\scriptscriptstyle \Phi},i} := \, 2^{-1} \big(\, T^+_{{\scriptscriptstyle \Phi},i} + T^-_{{\scriptscriptstyle \Phi},i} \big) \,
   = \, 2^{-1} \big(\, T^+_i + T^-_i \big) = S_i   \qquad \qquad  \forall \;\; i \in I  \cr
   \qquad   \alpha_j\big(\,T^+_{{\scriptscriptstyle \Phi},i}\,\big) \, = \, p^{\scriptscriptstyle \Phi}_{i,j}  \quad ,
   \qquad  \alpha_j\big(\,T^-_{{\scriptscriptstyle \Phi},i}\,\big) \, = \, p^{\scriptscriptstyle \Phi}_{j,i}
   \qquad \qquad \qquad  \forall \;\; i, j \in I  }  $$
 so that the triple  $ \; \R_\Phi := \big(\, \lieh_\Phi \, , \Pi_\Phi \, , \Pi^\vee_{\scriptscriptstyle \Phi} \,\big) \; $
 with  $ \, \lieh_\Phi := \lieh \, $,  $ \, \Pi_\Phi := \Pi \, $  and
 $ \, \Pi^\vee_{\scriptscriptstyle \Phi} := \big\{\, T^+_{{\scriptscriptstyle \Phi},i} \, ,
 T^-_{{\scriptscriptstyle \Phi},i} \,\big|\, i \in I \,\big\} \, $,  \,is a realization of the matrix
 $ \, P_{\scriptscriptstyle \Phi} = {\big(\, p^{\scriptscriptstyle \Phi}_{i,j} \big)}_{i, j \in I} \, $
 --- as in  Definition \ref{def: realization of P};  also, by construction
 $ \R_{\scriptscriptstyle \Phi} $  is also \textsl{straight},  resp.\  \textsl{small},
 if and only if such is  $ \R \, $.  Moreover,  $ P_{\scriptscriptstyle \Phi} $
 is the sum of  $ P $  plus an antisymmetric matrix, so
 \textsl{the symmetric part of  $ P_{\scriptscriptstyle \Phi} $  is the same as  $ P $},
 i.e.\  $ \, {\big( P_{\scriptscriptstyle \Phi} \big)}_s = P_s \, $.
 In particular, if  $ P $  is of Cartan type, then so is  $ P_{\scriptscriptstyle \Phi} \, $,
 and they are associated with the same Cartan matrix. In short, we get:
\end{free text}

\vskip7pt

\begin{prop}  \label{prop: twist-realizations}
 With notation as above, the following holds true:
 \vskip3pt
   (a)\,  the matrix  $ \; P_{\scriptscriptstyle \Phi} \, := \, P -  \, \mathfrak{A} \, \Phi \, \mathfrak{A}^{\,\scriptscriptstyle T} \; $
obeys  $ \, {\big( P_{\scriptscriptstyle \Phi} \big)}_s = P_s \, $;  \,in particular, if
$ P $  is of Cartan type, then so is  $ P_{\scriptscriptstyle \Phi} \, $,
and they are associated with the same Cartan matrix.
 \vskip3pt
   (b)\,  the triple  $ \, \R_{\scriptscriptstyle \Phi} := \big(\, \lieh_\Phi \! := \lieh \, ,
   \, \Pi_\Phi \! := \Pi \, , \, \Pi^\vee_{\scriptscriptstyle \Phi} \! :=
   {\big\{ T^+_{{\scriptscriptstyle \Phi},i} \, , T^-_{{\scriptscriptstyle \Phi},i} \big\}}_{i \in I} \,\big) \, $
   is a realization of  $ \, P_{\scriptscriptstyle \Phi} \, $;  \,moreover,  $ \R_\Phi $  is  \textsl{straight},
   resp.\  \textsl{small},  if and only if such is  $ \R \, $.
\end{prop}

\vskip7pt

\begin{definition}  \label{def:twisted-realization}
 The realization  $ \, \R_\Phi := \big(\, \lieh \, , \Pi \, , \Pi^\vee_{\scriptscriptstyle \Phi} \,\big) \, $
 of the matrix  $ \, P_{\scriptscriptstyle \Phi} = {\big(\, p^{\scriptscriptstyle \Phi}_{i,j} \big)}_{i, j \in I} $  \,
 is called a  \textit{twist deformation\/}  (via  $ \Phi \, $)  of the realization
 $ \; \R = \big(\, \lieh \, , \Pi \, , \Pi^\vee \,\big) \, $  of  $ P \, $.
                                                                    \par
   Similarly, the matrix  $ P_{\scriptscriptstyle \Phi} $  is called a  \textit{twist deformation\/}
   of the matrix  $ P $.   \hfill  $ \diamondsuit $
\end{definition}

\vskip7pt

\begin{rmks}  \label{rmks: transitivity-twist}
 \textit{(a)}\,  Observe that, by the very definition of twisting one has that
  $$  \big(P_{\scriptscriptstyle \Phi}\big)_{\scriptscriptstyle \Phi'} = P_{\scriptscriptstyle \Phi + \Phi'}  \qquad \text{ and }
  \qquad  \big(\R_{\scriptscriptstyle \Phi}\big)_{\scriptscriptstyle \Phi'} = \R_{\scriptscriptstyle \Phi + \Phi'}
  \qquad  \text{ for all }  \quad  \Phi \, , \, \Phi' \in \lieso_t\big(\kh\big)  $$
 Therefore, the additive group  $ \, \lieso_t\big(\kh\big) \, $  acts on the set of
 (multiparameter) matrices of size  $ \, n := |I| \, $  with fixed symmetric part,
 as well as on the set of their realizations of (any) fixed rank.  \textsl{When two matrices
 $ \, P , P' \in M_n\big(\kh\big) \, $  belong to the same orbit of this  $ \lieso_t\big(\kh\big) $--action  we say that
 \textit{$ P $  and  $ P' $  are twist equivalent}}.
 \vskip5pt
   \textit{(b)}\,  It follows from  Proposition \ref{prop: twist-realizations}\textit{(a)\/}
   that if two multiparameter matrices  $ P $  and  $ P' $  are twist equivalent,
   then their symmetric part is the same, that is  $ \, P_s = P'_s \, $.
   Next result shows the converse holds true as well,  \textit{up to taking the group
   $ \lieso_t\big(\kh\big) $  big enough, namely with  $ \, t \geq 3\,n-\rk(P_s) \, $}.
\end{rmks}

\vskip3pt

\begin{lema}  \label{lemma: twist=sym}
 With notation as above, let  $ \, P , P' \in M_n\big(\kh\big) \, $,
 \,and consider the aforementioned action by twist on  $ M_n\big(\kh\big) $
 by any additive group  $ \lieso_t\big(\kh\big) $  with  $ \, t \geq 3\,n-\rk(P_s) \, $.
 Then  $ P $  and  $ P' \, $  are twist equivalent if and only if  $ \, P_s = P'_s \, $.
\end{lema}

\pf
 We have to prove the ``if'' part of the statement, so we assume that  $ \, P_s = P'_s \, $
 and we look for $ \, \Phi \in \lieso_t\big(\kh\big) \, $  such that  $ \, P' = P_\Phi \, $,  \,that is
 $ \; P' = P - \, \mathfrak{A} \, \Phi \, \mathfrak{A}^{\,\scriptscriptstyle T} \, $.
                                                                  \par
   By assumption  $ \, P' = P + \Lambda \, $  with  $ \, \Lambda := P' - P \, $  antisymmetric, and we want
\begin{equation}  \label{eq: Lambda = APhiAt}
 \Lambda  \; = \;  -  \, \mathfrak{A} \, \Phi \, \mathfrak{A}^{\,\scriptscriptstyle T}
\end{equation}
 for some  $ \; \Phi \in \lieso_t\big(\kh\big) \, $   --- in other words, we have to show that the equation
 \eqref{eq: Lambda = APhiAt}  in the indeterminate  $ \Phi $  has a solution.
                                                                         \par
   By  Proposition \ref{prop: exist-realiz's}\textit{(a)},  there exists a straight split realization of  $ P \, $,  say
   $ \, \R = \big(\, \lieh \, , \Pi \, , \Pi^\vee \,\big)  \, $,  of rank  $ t \, $.  By the straightness assumption, the
   $ \alpha_\ell $'s  are linearly independent in  $ \lieh^* $,  while the  $ H_g $'s  form a basis of  $ \lieh \, $,
   \,so the matrix  $ \; \mathfrak{A} := {\big(\, \alpha_\ell(H_g) \big)}_{\ell \in I}^{g \in \G} \; $
   has rank  $ \, |I| = n \, $;  \,therefore, we can write it as a block matrix
   $ \; \mathfrak{A} = \big(\, G \,\big|\, Q \,\big) \, M_\sigma \; $  where  $ G $,  $ Q $  and  $ M_\sigma $
   are matrices of size  $ \, n \times n \, $,  $ n \times (t-n) \, $  and  $ \, t \times t \, $,
   respectively, and moreover  $ G $  is  \textsl{invertible\/}  and  $ M_\sigma $  is a  \textsl{permutation matrix}.
   Then  \eqref{eq: Lambda = APhiAt}  reads
\begin{equation}  \label{eq: Lambda = APhiAt (2)}
   \Lambda  \; = \;  - \,  \, \big(\, G \,\big|\, Q \,\big) \, M_\sigma \, \Phi \, M_\sigma^{\,\scriptscriptstyle T} \,
   \bigg(\! {{\,G^{\,\scriptscriptstyle T}\,} \over {\,Q^{\,\scriptscriptstyle T}\,}} \!\bigg)
\end{equation}
 Now let us replace the indeterminate matrix  $ \Phi $  with
 $ \; \Psi \, := \, M_\sigma \, \Phi  \, M_\sigma^{\,\scriptscriptstyle T} \; $
 and accordingly let us read  \eqref{eq: Lambda = APhiAt (2)}  as an equation in  $ \Psi \, $,  namely
\begin{equation}  \label{eq: Lambda = APsiAt}
   \Lambda  \; = \;  -  \, \big(\, G \,\big|\, Q \,\big) \, \Psi \,
   \bigg(\! {{\,G^{\,\scriptscriptstyle T}\,} \over {\,Q^{\,\scriptscriptstyle T}\,}} \!\bigg)
\end{equation}
 then writing the latter in block form as
 $ \; \displaystyle{ \Psi \, = \, \bigg(\! {{\; A \,|\, B \;} \over {{\; C \hskip2,8pt|\, D \;}}} \!\bigg) } \; $
 where  $ A \, $,  $ B \, $,  $ C $  and  $ D $  has size  $ \, n \times n \, $,  $ \, n \times (t-n) \, $,  $ \, (t-n) \times n \, $
 and  $ \, (t-n) \times (t-n) \, $,  respectively, we see at once that a possible solution for  \eqref{eq: Lambda = APsiAt}  is
 $ \; \displaystyle{ \Psi \, = \bigg(\! {{\; A \,|\hskip2,8pt 0 \;} \over {{\;\hskip2,1pt 0 \hskip3,5pt|\hskip2,8pt 0 \;}}} \!\bigg) } \; $
 with  $ \; A := G^{-1} \Lambda \, G^{-t} \; $.  Thus  \eqref{eq: Lambda = APsiAt}
 has a solution, hence \eqref{eq: Lambda = APhiAt}  has one too, q.e.d.
\epf

\vskip7pt

\begin{rmk}
  A similar notion of twist-equivalence of matrices can be found in  \cite{AS1}  and  \cite{Ro}
  in relation with matrices corresponding to diagonal braidings.
\end{rmk}

\vskip7pt

   Next proposition ``upgrades'' the previous result to the level of realizations.

\vskip11pt

\begin{prop}  \label{prop: realiz=twist-standard}
 Let  $ P $  and  $ P' $  be two matrices in  $ M_n\big(\kh\big) $  with the same symmetric part, i.e.\ such that  $ \, P_s = P'_s \; $.
 \vskip5pt
   (a)\,  Let  $ \, \R $  be a  \textsl{straight}  realization of  $ P $ of rank  $ t \, $.
  Then there exists a matrix  $ \, \Phi \in \lieso_t\big(\kh\big) \, $  such that  $ \, P' = P_{\scriptscriptstyle \Phi} \, $
  and the corresponding realization  $ \, \R_{\scriptscriptstyle \Phi} \, $  is straight.
                                                            \par
   In short, if  $ \, P'_s = P_s \, $  then from any  \textsl{straight}  realization of  $ P $  we can obtain by
   twist deformation a straight realization (of the same rank) for  $ P' $,  and viceversa.
 \vskip5pt
   (b)\,  Let  $ \, \R $  and  $ \R' $  be  \textsl{straight small}  realizations of  $ \, P $  and  $ P' \, $,  such that  $ \, \rk(\R) = \rk(\R') =: t \, $.
   Then there exists a matrix
 $ \, \Phi \in \lieso_t\big(\kh\big)
 \, $  such that  $ \, \R' \cong \R_{\scriptscriptstyle \Phi} \, $.
                                                            \par
   In short, if  $ \, P'_s = P_s \, $  then any straight small realization of  $ P $  is isomorphic to a twist deformation of a
   straight small realization of  $ P' $  of same rank, and viceversa.
 \vskip5pt
   (c)\,  Every  \textsl{straight small}  realization  $ \, \R $  of  $ P $  is isomorphic to some
   twist deformation of the standard realization of  $ P_s $  of the same rank as  $ \, \R \, $,
   as in  Lemma \ref{lem: realiz_P => realiz_A}.
\end{prop}

\pf
 \textit{(a)}\,
 We can resume the same argument used in the proof of  Lemma \ref{lemma: twist=sym}
 to show that there exists a suitable  $ \, \Phi \in
%
 \lieso_t\big(\kh\big)  \, $  such that  $ \, P' = P_{\scriptscriptstyle \Phi} \, $,  \,the difference being only
 that now the starting point is the given realization  $ \R $  of  $ P \, $.
 But then  Proposition \ref{prop: twist-realizations}  ensures also that
 $ \, \R_\Phi := \big(\, \lieh \, , \Pi \, , \Pi^\vee_{\scriptscriptstyle \Phi} \,\big) \, $  is a realization of
 $ \, P' = P_{\scriptscriptstyle \Phi} \, $,  \,which is straight because  $ \R $  is.
 \vskip5pt
   \textit{(b)}\,  This follows from claim  \textit{(a)},  along with the uniqueness
   (up to isomorphisms) of straight small realizations   --- cf.\ Proposition \ref{prop: exist-realiz's}\textit{(a)}.
 \vskip5pt
   \textit{(c)}\,  This follows as an application of claim  \textit{(b)}, taking  $ \, P' := P_s \, $  and  $ \, \R' := \big(\, \lieh \, , \, \Pi \, , \, \Pi^\vee_S \,\big) \, $
   --- assuming  $ \, \R = \big(\, \lieh \, , \, \Pi \, , \, \Pi^\vee \,\big) \, $  ---
   as the standard Kac' realization over  $ \kh $  which is straight and small, see
   Remark \ref{rmk: Kac'-realiz}  and  Lemma \ref{lem: realiz_P => realiz_A}.
\epf

\vskip11pt

\begin{free text}  \label{further_stability}
 {\bf Stability issues for twisted realizations.}  Keep notation as above; in particular, from
 \S \ref{subsec: twisted-realiz}  we consider
 $ \, \underline{T}^\pm = {\big(\, T_i^\pm \big)}_{i \in I} \, $
 and  $ \; \fT := \bigg(\, {\displaystyle {\underline{T}^+ \atop \underline{T}^-}} \bigg) \, $,
 \,we fix a vector of  $ \kh $--basis  elements
 $ \, \fH := {\big(\, H_g \big)}_{g \in \G} \, $
 for  $ \lieh $  and accordingly we set
 $ \; \mathfrak{A} := {\big(\, \alpha_\ell(H_g) \big)}_{\ell \in I}^{g \in \G} \; $
 and  $ \; \mathfrak{A}_\bullet  :=  \begin{pmatrix}
       +\mathfrak{A} \,  \\
       -\mathfrak{A} \,
                        \end{pmatrix} \; $.
 Finally, given  $ \, \Phi \in
%
\lieso_t\big(\kh\big)  \, $
%
%
 we consider the new strings of ``coroot vectors''
 $ \, \underline{T}^\pm_{\,\scriptscriptstyle \Phi} = {\big(\, T_{{\scriptscriptstyle \Phi}, i}^\pm \big)}_{i \in I} \, $
 and  $ \; \fT_{{\scriptscriptstyle \Phi}} :=
 \bigg(\, {\displaystyle {\underline{T}^+_{\,\scriptscriptstyle \Phi} \atop \underline{T}^-_{\,\scriptscriptstyle \Phi}} }\bigg) \; $
 that are linked to the old coroot vectors   --- as in  \S \ref{twist-deform.'s x mpmatr.'s & realiz.'s}  ---   by the formulas
 \begin{equation}  \label{eq: T->T_Phi}
   \underline{T}^\pm_{\,\Phi} \, := \,
%
\underline{T}^\pm \mp \mathfrak{A} \, \Phi \, \fH   \quad ,
   \qquad  \fT_{\scriptscriptstyle \Phi} \, = \,
%
\fT - \, \mathfrak{A}_\bullet \, \Phi \, \fH
 \end{equation}
   \indent   Eventually, recall also the notation
   $ \,\; P_{\scriptscriptstyle \Phi} \, = \,  {\big(\, p^{\scriptscriptstyle \Phi}_{i,j} \big)}_{i, j \in I} \, := \,
%
P - \, \mathfrak{A} \, \Phi \, \mathfrak{A}^{\,\scriptscriptstyle T}  \;\, $.
 \vskip7pt
   From  Proposition \ref{prop: twist-realizations}  we have that the class of (small) minimal realizations is stable under twist deformations.
   We look now instead at the  \textsl{split\/}  case.
 \vskip3pt
   Assume that the realization  $ \; \R := \big(\, \lieh \, , \Pi \, , \Pi^\vee \,\big) \; $  of  $ P $  is  \textsl{split},
   i.e.\ the  $ T^\pm_i $'s  are part of a  $ \kh $--basis  of  $ \lieh \, $.  From  \eqref{eq: T->T_Phi}
   we see that we cannot give for granted the same property for the  $ T^\pm_{{\scriptscriptstyle \Phi}, i} $'s,
   hence we cannot say either that the realization  $ \; \R_\Phi := \big(\, \lieh \, , \Pi \, , \Pi^\vee_{\scriptscriptstyle \Phi} \,\big) \; $  of
   $ P_\Phi $  be split as well   --- in fact, all that depends on the matrix  $ \, \mathfrak{A}_\bullet \Phi \, $.
   We shall now discuss this issue in detail in a more restricted setting.
 \vskip7pt
   We assume now that  $ \R $  is  \textsl{split minimal},  so  $ \,  {\big\{ T_i^+ , T_i^- \big\}}_{i \in I} \, $  is a $ \kh $--basis  of  $ \lieh $
   --- cf.\ Definition \ref{def: realization of P}\textit{(d)}.  Again from  Definition \ref{def: realization of P},  let us consider the elements
 $ \; S_i := \, 2^{-1} \big(\, T^+_i + T^-_i \,\big) \; $  and  $ \; \varLambda_i := \, 2^{-1} \big(\, T^+_i - T^-_i \,\big) \; $
 --- for all  $ \, i \in I \, $  ---   and similarly
 $ \; S_{{\scriptscriptstyle \Phi},i} := \, 2^{-1} \big(\, T^+_{{\scriptscriptstyle \Phi},i} + T^-_{{\scriptscriptstyle \Phi},i} \,\big) \; $
 and
 $ \; \varLambda_{{\scriptscriptstyle \Phi},i} :=
 \, 2^{-1} \big(\, T^+_{{\scriptscriptstyle \Phi},i} - T^-_{{\scriptscriptstyle \Phi},i} \,\big) \; $
 --- for all  $ \, i \in I \, $;  \,set also
 $ \, \underline{S} := 2^{-1} \, \big(\, \underline{T}^+ \! + \underline{T}^- \,\big) \, $  and
 $ \, \underline{\varLambda} := 2^{-1} \, \big(\, \underline{T}^+ \! - \underline{T}^- \,\big) \, $,
 \,and similarly
 $ \, \underline{S}_{\,\scriptscriptstyle \Phi} := 2^{-1} \, \big(\, \underline{T}^+_{\,\scriptscriptstyle \Phi} \! +
 \underline{T}^-_{\,\scriptscriptstyle \Phi} \,\big) \, $  and
 $ \, \underline{\varLambda}_{\,\scriptscriptstyle \Phi} := 2^{-1} \,
 \big(\, \underline{T}^+_{\,\scriptscriptstyle \Phi} \! - \underline{T}^-_{\,\scriptscriptstyle \Phi} \,\big) \, $.
 In matrix terms, we have
  $$  \bigg(\, {\displaystyle {\underline{S} \atop \underline{\varLambda}}} \,\bigg) \, = \,
   \begin{pmatrix}
      +2^{-1} I_n  &  +2^{-1} I_n  \cr
      +2^{-1} I_n  &  -2^{-1} I_n
   \end{pmatrix} \!
 \bigg(\, {\displaystyle {\underline{T}^+ \atop \underline{T}^-}} \bigg)  \;\; ,
 \quad  \bigg(\, {\displaystyle {\underline{S}_{\,\scriptscriptstyle \Phi}
 \atop \underline{\varLambda}_{\,\scriptscriptstyle \Phi}}} \bigg) \, = \,
   \begin{pmatrix}
      +2^{-1} I_n  &  +2^{-1} I_n  \cr
      +2^{-1} I_n  &  -2^{-1} I_n
   \end{pmatrix} \!
 \bigg(\, {\displaystyle {\underline{T}^+_{\,\scriptscriptstyle \Phi} \atop \underline{T}^-_{\,\scriptscriptstyle \Phi}}} \bigg)  $$
 and conversely
  $$  \bigg(\, {\displaystyle {\underline{T}^+ \atop \underline{T}^-}} \bigg) \, = \,
   \begin{pmatrix}
      +I_n  &  +I_n  \cr
      +I_n  &  -I_n
   \end{pmatrix}
 \bigg(\, {\displaystyle {\underline{S} \atop \underline{\varLambda}}} \,\bigg)  \quad ,
 \qquad  \bigg(\, {\displaystyle {\underline{T}^+_{\,\scriptscriptstyle \Phi}
 \atop \underline{T}^-_{\,\scriptscriptstyle \Phi}}} \bigg) \, = \,
   \begin{pmatrix}
      +I_n  &  +I_n  \cr
      +I_n  &  -I_n
   \end{pmatrix}
 \bigg(\, {\displaystyle {\underline{S}_{\,\scriptscriptstyle \Phi}
 \atop \underline{\varLambda}_{\,\scriptscriptstyle \Phi}}} \bigg)  $$
 In particular, we have  $ \; \textsl{Span}_{\,\kh}\Big( {\big\{ S_i \, ,
 \varLambda_i \big\}}_{i \in I} \Big) = \textsl{Span}_{\,\kh}\Big( {\big\{\, T^+_i , T^-_i \,\big\}}_{i \in I} \Big) \; $
 and similarly
 $ \; \textsl{Span}_{\,\kh}\Big( {\big\{ S_{\scriptscriptstyle \Phi,i} \, ,
 \varLambda_{\scriptscriptstyle \Phi,i} \big\}}_{i \in I} \Big) = \textsl{Span}_{\,\kh}\Big( {\big\{\, T^+_{\scriptscriptstyle \Phi,i} \, ,
 T^-_{\scriptscriptstyle \Phi,i} \,\big\}}_{i \in I} \Big) \; $.
 \vskip3pt
   Now, with respect to the previous analysis we pick our fixed  $ \kh $--basis  of  $ \lieh $  to be
   $ \, {\{H_g\}}_{g \in \G} := {\big\{\, T^+_i , T^-_i \,\big\}}_{i \in I} \, $:  \,then  $ \, \mathfrak{A} \, $  reads as a block
   $ (n \times 2\,n) $--matrix  $ \, \mathfrak{A} = \big(\, P^{\,\scriptscriptstyle T} \; P \,\big) \, $,
   \,hence the first identity in  \eqref{eq: T->T_Phi}  yields, via straightforward computations,
  $$  \underline{S}_{\,\scriptscriptstyle \Phi} = \, \underline{S}  \quad ,
\qquad  \underline{\varLambda}_{\,\scriptscriptstyle \Phi} = \, \underline{\varLambda} -
%
  \big(\, P^{\,\scriptscriptstyle T} \; P \,\big) \, \Phi
  \ \begin{pmatrix}
      +I_n  &  +I_n  \cr
      +I_n  &  -I_n
   \end{pmatrix}
 \bigg(\, {\displaystyle {\underline{S} \atop \underline{\varLambda}}} \,\bigg)  $$
 which in matrix terms reads
\begin{equation}  \label{eq: S-varLambda ---> SPhi-varLambdaPhi}
  \bigg(\, {\displaystyle {\underline{S}_{\,\scriptscriptstyle \Phi} \atop \underline{\varLambda}_{\,\scriptscriptstyle \Phi}}} \bigg)  \; = \;
   \begin{pmatrix}
      I_n  &  0_n  \cr
%
 -B'  &  \big(\, I_n -
%
 B''  \,\big)
   \end{pmatrix}
 \bigg(\, {\displaystyle {\underline{S} \atop \underline{\varLambda}}} \,\bigg)
\end{equation}
 where  $ B' $  and  $ B'' $  are blocks in the matrix
 $ \; \big(\, P^{\,\scriptscriptstyle T} \; P \,\big) \,
%
\Phi
\begin{pmatrix}
      +I_n  &  +I_n  \cr
      +I_n  &  -I_n
   \end{pmatrix}  = \big(\, B' \; B'' \,\big) \; $,
%
%
 \,i.e.\ they are the  $ (n \times n) $--matrices
  $ \,\; B' \, := \, P^{\,\scriptscriptstyle T} \big(\, \Phi^{+,+} \! + \Phi^{+,-}  \,\big) + P \, \big(\, \Phi^{-,+} \! + \Phi^{-,-}  \,\big) \; $
  and  $ \; B'' \, := \, P^{\,\scriptscriptstyle T} \big(\, \Phi^{+,+} \! - \Phi^{+,-}  \,\big) + P \, \big(\, \Phi^{-,+} \! - \Phi^{-,-}  \,\big) \; $,
 with notation as follows: we write  $ \Phi $  in block form
 $ \; \Phi := \left( \begin{smallmatrix}
     \Phi^{++}  &  \Phi^{+-}  \\
     \Phi^{-+}  &  \Phi^{--}
                       \end{smallmatrix} \right) \; $
 with  $ \, \Phi^{\varepsilon_1,\varepsilon_2} = {\big( \phi_{ij}^{\varepsilon_1,\,\varepsilon_2} \big)}_{i,j \in I} \, $
 for all  $ \, \varepsilon_i, \varepsilon_j \in \{+,-\} \, $.
%
%
 \vskip3pt
   Now, it is clear that the set  $ {\big\{\, T^+_{\scriptscriptstyle \Phi,i} \, , T^-_{\scriptscriptstyle \Phi,i} \,\big\}}_{i \in I} $
   is  $ \kh $--linearly  independent if and only if
 $ \; \textsl{Span}_{\,\kh}\Big( {\big\{\, T^+_{\scriptscriptstyle \Phi,i} \, , T^-_{\scriptscriptstyle \Phi,i} \,\big\}}_{i \in I} \Big) =
 \textsl{Span}_{\,\kh}\Big( {\big\{\, T^+_i , T^-_i \,\big\}}_{i \in I} \Big) \; $,
 and the latter is true if and only if
 $ \; \textsl{Span}_{\,\kh}\Big( {\big\{ S_{\scriptscriptstyle \Phi,i} \, ,
 \varLambda_{\scriptscriptstyle \Phi,i} \big\}}_{i \in I} \Big) =
 \textsl{Span}_{\,\kh}\Big( {\big\{ S_i \, , \varLambda_i \big\}}_{i \in I} \Big) \; $.
 But the latter holds true, by  \eqref{eq: S-varLambda ---> SPhi-varLambdaPhi},
 if and only if the matrix  $ \, \big(\, I_n - \,B'' \,\big) \, $  is invertible in  $ M_n\big(\kh\big) \, $.
 Finally, from the explicit form of  $ B'' $,  we find the following criterion:
 \vskip5pt
%
   \textit{Assume that the realization  $ \, \R := \big(\, \lieh \, , \Pi \, , \Pi^\vee \,\big) \, $  of  $ P $  be split minimal.
   If the matrix
 $ \,\; M^{\,\scriptscriptstyle \Phi}_{\scriptscriptstyle P}  := \,
 I_n -  \, P^{\,\scriptscriptstyle T} \big(\,\Phi^{+,+} - \Phi^{+,-} \big) \, - \, P \, \big(\, \Phi^{-,+} - \Phi^{-,-}  \big) \;\, $
 is invertible in  $ M_n\big( \kh \big) \, $,  \,then the realization
 $ \; \R_\Phi := \big(\, \lieh \, , \Pi \, , \Pi^\vee_{\scriptscriptstyle \Phi} \,\big) \, $  of  $ \, P_\Phi $
 is split minimal too (and viceversa). }
 \vskip5pt
   As an outcome, this proves that the twist deformation of a split realization
   may be not split (counterexamples do exist, see below), hence the subclass
   of all split realizations is  \textsl{not\/}  stable under twist deformations.
\end{free text}

\smallskip

\begin{exas}  \label{examples: special-cases}
%
%
 Note that  $ \, M^{\,\scriptscriptstyle \Phi}_{\scriptscriptstyle P} \, $  has the form
  $ \,\; M^{\,\scriptscriptstyle \Phi}_{\scriptscriptstyle P} \, = \, I_n -  \, \big(\, P^{\,\scriptscriptstyle T} \; P \,\big) \, \Phi \bigg( {\displaystyle{+I_n \atop -I_n}} \bigg) \; $.
 Using this, we may find examples where  $ M^{\,\scriptscriptstyle \Phi}_{\scriptscriptstyle P} $  is invertible or
not.  For instance:
 \vskip3pt
   \textit{(a)}\;  if  $ \, \Phi = 0_{\,2n}  \, $,  \,then  $ \, M^{\,\scriptscriptstyle \Phi}_{\scriptscriptstyle P} = I_n \, $
   is invertible;
 \vskip3pt
   \textit{(b)}\;  if  $ \; \big( P^{\,\scriptscriptstyle T} \; P \,\big) \, \Phi \bigg( {\displaystyle{+I_n \atop -I_n}} \bigg) \; $
   is nilpotent, then  $ \, M^{\,\scriptscriptstyle \Phi}_{\scriptscriptstyle P} \, $  is clearly invertible;
 \vskip3pt
   \textit{(c)}\;  Let
 $ \, \Phi =
     \begin{pmatrix}
        \, 0  &  \varPhi \,  \\
        \, -\varPhi^{\,\scriptscriptstyle T}  &  0 \,
     \end{pmatrix} \, $
 with  $ \, \varPhi \in M_n\big(\kh\big) \, $.  Then
 $ \; M^{\,\scriptscriptstyle \Phi}_{\scriptscriptstyle P} \, =
 \, I_n + P^{\,\scriptscriptstyle T} \, \varPhi + P \, \varPhi^{\,\scriptscriptstyle T} \; $.
 If we take  $ \, \varPhi \, $  antisymmetric, we find
 $ \; M^{\,\scriptscriptstyle \Phi}_{\scriptscriptstyle P} = I_n - 2 \, P_a \, \varPhi \; $  where
 $ \, P_a := 2^{-1} \big( P - P^{\,\scriptscriptstyle T} \big) \, $  is the antisymmetric part of  $ P \, $.
 Hence, by taking the canonical multiparameter  \, $ P = DA \, $,  we get
 $ \, M^{\,\scriptscriptstyle \Phi}_{\scriptscriptstyle P} = I_n \, $.
 On the other hand, there are plenty of examples such that  $ \, \big( I_n - 2 \, P_a \, \varPhi \big) \, $
 is non-invertible; for example, take  $ n $  even,  $ \varPhi $  invertible and antisymmetric, and
 $ \, P_a = 2^{-1} \varPhi^{-1} \, $.  In addition, in this last case is
 $ \, M^{\,\scriptscriptstyle \Phi}_{\scriptscriptstyle P} = 0 \, $,  \,hence
  $$  \textsl{Span}_{\,\kh}\Big( {\big\{ S_{\scriptscriptstyle \Phi,i} \, ,
  \varLambda_{\scriptscriptstyle \Phi,i} \big\}}_{i \in I} \Big)  \; = \;
  \textsl{Span}_{\,\kh}\Big( {\big\{ S_i \big\}}_{i \in I} \Big) \; \subsetneqq \;  \textsl{Span}_{\,\kh}\Big( {\big\{ S_i \, ,
  \varLambda_i \big\}}_{i \in I} \Big)  $$
 By our analysis, this proves that  \textit{in this case the realization  $ \R_\Phi $
 is definitely \textsl{not}  split}   --- on the contrary it is (small) minimal.
\end{exas}

\medskip

\subsection{2-cocycle deformations of multiparameters and realizations}  \label{subsec: 2-cocycle-realiz}  {\ }
 \vskip7pt
   In this subsection we introduce the notion of  \textit{deformation by 2--cocycles\/}
   of realizations (as well as of multiparameters),
   which is dual to that of deformation by twist.

\vskip7pt

\begin{free text}  \label{2-cocycle-deform.'s x mpmatr.'s & realiz.'s}
 {\bf Deforming realizations (and matrices) by  $ 2 $--cocycles.}  Fix again
 $ \, P := {\big(\, p_{i,j} \big)}_{i, j \in I} \in M_n\big( \kh \big) \, $  and a realization
 $ \; \R \, := \, \big(\, \lieh \, , \Pi \, , \Pi^\vee \,\big) \; $  of it, setting  $ \, d_i := p_{ii}/2 \, $  for all
 $ \, i \in I \, $ and  $ \, D_{\scriptscriptstyle P} := \text{\sl diag}\big(d_1\,,\dots,d_n\big) \, $.
 We consider special deformations of realizations, called ``2-cocycle deformations''.
 To this end, like in  \S \ref{subsec: twisted-realiz},  we fix in  $ \lieh $  a  $ \kh $--basis
 $ \, {\big\{ H_g \big\}}_{g \in \G} \, $,  where  $ \G $  is an index set with  $ \, |\G| = \rk(\lieh) = t \, $.
 \vskip5pt
   Let  $ \, \chi : \lieh \times \lieh \relbar\joinrel\longrightarrow \kh \, $  be any  $ \kh $--bilinear  map:
   note that it bijectively corresponds to some
   $ \, X = {\big(\, \chi_{g{}\gamma} \big)}_{g , \gamma \in \G} \in M_t\big(\kh\big) \, $  via
   $ \, \chi_{g{}\gamma} = \chi(H_g\,,H_\gamma) \, $.  We assume that  $ \chi $ is antisymmetric, which means
   $ \, \chi^{\scriptscriptstyle T}(x,y) = -\chi(x,y) \, $  where
   $ \, \chi^{\scriptscriptstyle T}(x,y) := \chi(y,x) \, $,  \,for all  $ \, x, y \in \lieh \, $;
   \,this is equivalent to saying that  $ X $  is antisymmetric, i.e.\  $ \, X \in \lieso_t\big(\kh\big) \, $.
   We denote by  $ \, \text{\sl Alt}_{\,\kh}\big(\, \lieh \times \lieh \, , \, \kh \,\big) \, $  the set of all antisymmetric,
   $ \kh $--bilinear  maps from  $\, \lieh \times \lieh \, $  to  $ \, \kh \, $.
   We assume also that  $ \chi $  obeys
\begin{equation}  \label{eq: condition-chi}
  \chi(S_i \, ,\,-\,)  \, = \,  0  \, = \,  \chi(\,-\,,S_i)   \qquad \qquad \forall \;\; i \in I
\end{equation}
 where  $ \, S_i := 2^{-1} \big(\, T^+_i + T^-_i \big) \, $  for all  $ \, i \in I \, $.  In particular,
 this implies  (for  $ \, i \in I \, $,  $ \, T \in \lieh \, $)  that
  $ \; \chi\big( +T_i^+ , \, T \,\big)  \, = \,  \chi\big( -T_i^- , \, T \,\big) \; $,  $ \; \chi\big(\, T \, , +T_i^+ \big)  \, = \,  \chi\big(\, T \, , -T_i^- \big) \; $,
 \,hence
\begin{equation}  \label{eq: condition-chi-chit}
 +\chi\big(\, \text{--} \, , T_i^+ \big) = -\chi\big(\, \text{--} \, , T_i^- \big)   \qquad \qquad  \forall \;\; i \in I
\end{equation}
 For later use,  \textit{we introduce also the notation}
\begin{equation}  \label{eq: def-Alt_S}
 \text{\sl Alt}_{\,\kh}^{\,S}(\lieh)  \, := \,
 \Big\{\, \chi \in \text{\sl Alt}_{\,\kh}\big(\, \lieh \times \lieh \, , \, \kh \,\big) \,\Big|\; \chi \text{\ obeys\ }  (\ref{eq: condition-chi}) \,\Big\}
\end{equation}
 and to each  $ \, \chi \in  \text{\sl Alt}_{\,\kh}^{\,S}(\lieh)  \, $  we associate
  $ \, \mathring{X} := {\Big( \mathring{\chi}_{i{}j} = \chi\big(\,T_i^+,T_j^+\big) \!\Big)}_{i, j \in I} \! \in \lieso_n\big(\kh\big) \, $.
 \vskip5pt
  Basing on the above, we define
  $$
  P_{(\chi)}  := \,  P \, + \, \mathring{X}  \, = \,  {\Big(\, p^{(\chi)}_{i{}j} := \,  p_{ij} + \mathring{\chi}_{i{}j} \Big)}_{\! i, j \in I}  \;\, ,  \!\quad
   \Pi_{(\chi)}  := \,  {\Big\{\, \alpha_i^{(\chi)}  := \,
   \alpha_i \pm \chi\big(\, \text{--} \, , T_i^\pm \big)  \Big\}}_{i \in I}
   $$
\end{free text}

\vskip7pt

   We are now ready for our key result on  $ 2 $--cocycle  deformations.

\vskip13pt

\begin{prop}  \label{prop: 2cocdef-realiz}
 Keep notation as above. Then:
 \vskip5pt
   (a)\;  $ \, P_{(\chi)} \, := \, P \, + \, \mathring{X} \; $  obeys  $ \, {\big( P_{(\chi)} \big)}_s = P_s \, $;
   \,in particular, if  $ P $  is of Cartan type, then so is  $ P_{(\chi)} \, $,
   and they are associated with the same Cartan matrix.
 \vskip5pt
   (b)\;  the triple  $ \, \R_{(\chi)} \, = \, \big(\, \lieh \, , \Pi_{(\chi)} \, , \Pi^\vee \,\big) \, $
   is a realization of the matrix  $ \, P_{(\chi)} \, $,  \,which is minimal, resp.\ split, if so is  $ \, \R \, $.
\end{prop}

\pf
   \textit{(a)}\,  This is obvious, as  $ \mathring{X} $ is antisymmetric.
 \vskip7pt
   \textit{(b)}\,  Since the set  $ \Pi^\vee $  does not change, condition  \textit{(a.2)\/}  is trivially satisfied.
   In particular,  $ \R_{(\chi)} $  is minimal, resp.\ split, if so is  $ \R \, $.
   The conditions on  \textit{(a.1)\/}  follows easily by definition and  \eqref{eq: condition-chi-chit}:  namely,
\begin{align*}
  \alpha_j^{(\chi)} \big(\,T^+_i\big)  \;  &  = \;  \alpha_j\big(\,T^+_i\big) + \chi\big(\,T_i^+,T_j^+\big)  \; = \;  p_{\,ij} + \mathring{\chi}_{ij}  \; = \;  p^{(\chi)}_{i{}j}  \\
  \alpha_j^{(\chi)} \big(\,T^-_i\big)  \;  &  = \;  \alpha_j\big(\,T^-_i\big) + \chi\big(\,T_i^-,T_j^+\big)  \; = \;  p_{\,ji} - \chi\big(\,T_j^+,T_i^-\big)  \; =  \\
  &  \qquad \qquad \qquad   = \;  p_{\,ji} + \chi\big(\,T_j^+,T_i^+\big)  \; = \;  p_{\,ji} + \mathring{\chi}_{ji}  \; = \;  p^{(\chi)}_{j{}i}
\end{align*}
 for all  $ \, i , j \in I \, $.  This shows that  $ \R_{(\chi)} $  is a realization of  $ \, P_{(\chi)} \, $,  \,q.e.d.
\epf

\vskip7pt

\begin{definition}\label{def:2-cocycle-realization}
 The realization  $ \, \R_{(\chi)} = \big(\, \lieh \, , \Pi_{(\chi)} \, , \Pi^\vee \,\big) \, $  of
 $ \, P_{(\chi)} = {\big(\, p^{(\chi)}_{i{}j} \big)}_{\!i, j \in I} \, $  is called a
 \textit{2--cocycle deformation\/}  of the realization
 $ \; \R = \big(\, \lieh \, , \Pi \, , \Pi^\vee \,\big) \, $  of  $ P \, $.
                                                                    \par
   Similarly, the matrix  $ P_{(\chi)} $  is called a  \textit{2--cocycle deformation\/}  of the matrix  $ P $.   \hfill  $ \diamondsuit $
\end{definition}

\vskip7pt

\begin{rmks}  \label{rmks: cocycle-deform-realiz}
 \textit{(a)}\,  The very definitions give
  $$  \big(P_{(\chi)}\big)_{(\chi)'} = P_{(\chi + \chi')}  \quad \text{ and }  \quad
  \big(\R_{(\chi)}\big)_{(\chi')} = \R_{(\chi + \chi')}  \qquad  \text{ for all }
  \;\;  \chi \, , \, \chi' \in \text{\sl Alt}_{\,\kh}^{\,S}(\lieh)  $$
 Thus, the additive group  $ \, \text{\sl Alt}_{\,\kh}^{\,S}(\lieh) \, $  acts on the set of
 (multiparameter) matrices of size  $ \, n := |I| \, $
 with fixed symmetric part, as well as on the set of their realizations of (any) fixed rank.
 \textsl{When two matrices  $ P $  and  $ P' $  in  $ \, M_n\big(\kh\big) \, $
 belong to the same orbit of this  $ \text{\sl Alt}_{\,\kh}^{\,S}(\lieh) $--action,
 we say that  \textit{$ P $  and  $ P' $  are 2--cocycle equivalent}}.
 \vskip5pt
   \textit{(b)}\,  It follows from  Proposition \ref{prop: 2cocdef-realiz}\textit{(a)\/}
   that if two multiparameter matrices  $ P $  and  $ P' $  are  $ 2 $--cocycle  equivalent,
   then their symmetric part is the same, i.e.\  $ \, P_s = P'_s \, $.
   As a consequence of the next result, the converse holds true as well
   (cf.\  Lemma \ref{lemma: 2-cocycle=sym}  below),  \textit{under mild, additional assumptions}.
\end{rmks}

\vskip7pt

   Next result concerns the aforementioned  $ \text{\sl Alt}_{\,\kh}^{\,S}(\lieh) $--action
   on realizations; indeed, up to minor details it can be seen as the
   ``$ 2 $--cocycle  analogue'' of  Proposition \ref{prop: realiz=twist-standard}:

\vskip13pt

\begin{prop}  \label{prop: mutual-2-cocycle-def}
 Let  $ \, P, P' \in M_n\big(\kh\big) \, $  be two matrices  with the same symmetric part,
 i.e.\ such that  $ \, P_s = P'_s \, $.  Moreover, let  $ \, \R $  be a  \textsl{split}  realization of  $ P $.
 \vskip5pt
   (a)\,  There exists a map $ \, \chi \in \text{\sl Alt}_{\,\kh}^{\,S}(\lieh) \, $
   such that  $ \, P' = P_{(\chi)} \, $  and the realization
   $ \, \R_{(\chi)}\,= \, \big(\, \lieh \, , \Pi_{(\chi)} \, , \Pi^\vee \,\big) \, $  of  $ \, P' = P_{(\chi)} \, $
   is  \textsl{split}.  In a nutshell, if  $ \, P'_s = P_s \, $  then from any  \textsl{split}  realization of  $ P $
   we can obtain a split realization (of the same rank) of  $ P' $  by 2--cocycle deformation, and viceversa.
 \vskip5pt
 (b)\,  Assume in addition that  $ \, \R $  be  \textsl{minimal}.
 Then  $ \, \R $  is isomorphic to a 2--cocycle deformation of the split minimal realization of  $ \, P_s \, $.
\end{prop}

\pf
   \textit{(a)}\,  Since  $ P $  and  $ P' $  share the same symmetric part, we have
   $ \, \varLambda := P' - P \, \in \, \lieso_n\big(\kh\big) \, $,  \,so
 $ \; P' = P + \mathring{X} \; $  with  $ \; \mathring{X} = {\big(\, \mathring{\chi}_{i{}j} \big)}_{i, j \in I} := \varLambda \; $.
 Let  $ \, \lieh'' := \textsl{Span}_\kh \big( {\big\{ T_i^+ \big\}}_{i \in I} \big) \, $.
 Then  $ \mathring{X} $  defines a unique antisymmetric,  $ \kh $--bilinear  map
  $$  \chi'' \in \text{\sl Alt}_{\,\kh}\big(\, \lieh'' \times \lieh'' \, , \, \kh \,\big)  \qquad \text{with}
  \qquad  \chi''\big(\,T_i^+,T_j^+\big) := \mathring{\chi}_{i{}j}   \quad  \forall \; i, j \in I  $$
 Imposing  \eqref{eq: condition-chi-chit},  this  $ \chi'' $  extends to a map    --- non-unique, in general ---
 $ \; \chi \in \text{\sl Alt}^{\,S}_{\,\kh}\big(\, \lieh \times \lieh \, , \, \kh \,\big) \, $,
 \,obeying  \eqref{eq: condition-chi-chit}  and such that  $ \, \chi{\Big|}_{\lieh'' \times \lieh''} \! = \chi'' \, $.  Now choosing
 \;  $ \Pi_{(\chi)} := \, {\Big\{\, \alpha_i^{(\chi)} := \, \alpha_i \pm \chi\big(\, \text{--} \, , T_i^\pm \big) \Big\}}_{i \in I} \, \subseteq \; \lieh^* \, $,
 \,we get, thanks to  Proposition \ref{prop: 2cocdef-realiz},  that
 $ \, \R_{(\chi)} \, := \, \big(\, \lieh \, , \Pi_{(\chi)} \, , \Pi^\vee \,\big) \, $
 is a split realization of  $ \, P' = P_{(\chi)} \, $,  \,q.e.d.
 \vskip5pt
   \textit{(b)}\, Let us write the split minimal realization of  $ P_s \, $  as
   $ \; \R_{st} \,= \, \big(\, \lieh \, , \Pi_{st} \, , \Pi_{st}^\vee \,\big) \; $,  \,with
   $ \, \Pi_{st}^\vee\, = \, {\big\{\, T_i^\pm \,\big\}}_{i \in I} \, $  and
   $ \, \Pi_{st} = {\big\{\, \alpha_i^{(st)} \,\big\}}_{i \in I} \, $.  Since
   $ \, P = P_s + P_a \, $   --- with  $ \, P_a := 2^{-1} \big( P - P^{\scriptscriptstyle T} \big) \, $
   --- applying the arguments in \textit{(a)\/}  above we fix the matrix
   $ \, \mathring{X} := P_a =  {\big(\, \mathring{\chi}_{ij} \,\big)}_{i , j \in I} \in \lieso_n\big(\kh\big)  \, $
   and  $ \, \chi \in \text{\sl Alt}^{\,S}_{\,\kh}\big(\, \lieh \times \lieh \, , \, \kh \,\big) \; $
   obeying  \eqref{eq: condition-chi-chit};  moreover, for all  $ \, i \in I \, $  we set
 $ \; \alpha_i^{(\chi)} := \, \alpha_i^{(st)} \pm \chi\big(\, \text{--} \, , T_i^\pm \big)  \; $.
 As  $ \R $  is split minimal and
  $$  \alpha_j^{(\chi)\!}\big(T_i^\pm\big)  =  \alpha_j^{(st)\!}\big(T_i^\pm\big) + \chi \big(\, T_i^\pm , T_j^+ \big) =
  {\big(P_s\big)}_{ij} \! + \mathring{\chi}_{ij}  =  {\big(P_s\big)}_{ij} \! + \! {\big(P_a\big)}_{ij}  =  p_{ij}  $$
 for all  $ \, i , j \in I \, $,  \,we get  $ \, \alpha_j^{(\chi)} = \alpha_j \, $  for all  $ \, j \in I \, $.
 Thus the realization  $ \big(\R_{st}\big)_{\!(\chi)} $  obtained from the
 $ 2 $--cocycle  deformation of  $ \R $  afforded by  $ \chi $  does coincide with  $ \R \, $.
 Finally, the assumption ``split minimal'' implies  $ \, \rk(\lieh) = 2\,n \, $,
 \,hence the uniqueness property in  Proposition \ref{prop: exist-realiz's}\textit{(b)\/}  gives
 $ \; \R \, \cong \,  \R_{st}\, $  as desired.
\epf

\vskip9pt

   As a byproduct, we find the following ``2-cocycle counterpart'' of  Lemma \ref{lemma: twist=sym}:

\vskip11pt

\begin{lema}  \label{lemma: 2-cocycle=sym}
 With notation as above, let  $ \, P , P' \in M_n\big(\kh\big) \, $.  Then  $ P $  and  $ P' \, $  are 2-cocycle equivalent
 --- for the aforementioned 2-cocycle action on  $ M_n\big(\kh\big) $  of some additive group
%
 $ \lieso_t\big(\kh\big) \, $   ---   if and only if  $ \, P_s = P'_s \, $.
\end{lema}

\pf
 The ``if'' part is  Proposition \ref{prop: 2cocdef-realiz},  so we are left to prove the ``only if''.
 By the existence result for realizations (cf.\  Proposition \ref{prop: exist-realiz's}),
 we can pick a realization  $ \R $  of  $ P $  of rank  $ \, \rk(\R) = t \, $:  \,then  Proposition \ref{prop: mutual-2-cocycle-def}\textit{(a)\/}
 applies, and we are done.
\epf

\vskip13pt

\begin{obs}  \label{obs: two deform's x P & R}
 To sum up, we wish to stress the following, remarkable fact.  Consider two matrices
 $ \, P , P' \in M_n\big(\kh\big) \, $
 with the same symmetric part  $ \, P_s = P'_s \, $,  and a realization
 $ \, \R = \big(\, \lieh \, , \Pi \, , \Pi^\vee \,\big) \, $  of  $ P $
 that is  \textsl{split\/}  and  \textsl{straight}.  Then, by
 Proposition \ref{prop: realiz=twist-standard}  and  Proposition \ref{prop: mutual-2-cocycle-def},
 one can construct  \textsl{two\/}  realizations  $ \R_{\scriptscriptstyle \Phi} $  and
 $ \R_{(\chi)} $  of  $ P' $  by a twist deformation,
 respectively a 2-cocycle deformation, of  $ \R $  that affects only the coroot set  $ \Pi^\vee $
 or the root set  $ \Pi \, $,  respectively; in particular,
 $ \R_{\scriptscriptstyle \Phi} $  is still straight (yet possibly not split) and  $ \R_{(\chi)} $
 is still split (yet possibly not straight),
 while both have the same rank as  $ \R \, $.
\end{obs}

\bigskip

\section{Multiparameter Lie bialgebras and their deformations}  \label{sec: Mp Lie bialgebras}

\vskip13pt

   In this section we introduce multiparameter Lie bialgebras, i.e.\ Lie bialgebra
   structures on a given vector space that depend on a multiparameter, and their deformations.
   Indeed, these will be the semiclassical objects corresponding to the specialization of
   our formal multiparameter quantum enveloping algebras at  $ \, \hbar = 0 \, $.

\medskip

 \subsection{Lie bialgebras and their deformations}  \label{subsec: Lie-bialg's & deform's}  {\ }
 \vskip7pt
   We recall hereafter a few notions concerning Lie bialgebras and their deformations;
   all this is classic, so we rely on references for a more detailed treatment.

\vskip9pt

\begin{free text}  \label{gen's on LbA's}
 {\bf Generalities on Lie bialgebras.}  A  \textsl{Lie bialgebra\/}  is any triple
 $ \, \big(\, \lieg \, ; \, [\,\ ,\,\ ] \, , \, \delta \,\big) \, $  such that  $ \lieg $  is a  $ \Bbbk $--module
 --- for some ground field  $ \Bbbk $  ---   $ \, [\,\ ,\,\ ] \, $  is a Lie bracket on  $ \lieg $
 (making the latter into a Lie algebra), the map  $ \, \delta : \lieg \longrightarrow \lieg \wedge \lieg \, $
 is a  \textit{Lie cobracket\/}  on  $ \lieg $  (making it into a  \textit{Lie coalgebra},
 i.e.\  $ \, \delta^* : \lieg^* \wedge \lieg^* \longrightarrow \lieg^* \, $  is a Lie algebra bracket on  $ \lieg^* \, $),
 and the two structures are linked by the constraint that  $ \delta $  is a  $ 1 $--cocycle
 --- for the Chevalley-Eilenberg cohomology of the Lie algebra  $ \, \big(\, \lieg \, ; \, [\,\ ,\,\ ] \,\big) \, $
 with coefficients in  $ \, \lieg \wedge \lieg \, $.  As a matter of notation, we set
 $ \, x \wedge y := 2^{-1} (x \otimes y - y \otimes x) \, $  and thus we identify  $ \, \lieg \wedge \lieg \, $
 with the subspace of antisymmetric tensors in
 $ \, \lieg \otimes \lieg \, $. Moreover, we loosely use a Sweedler's-like notation
 $ \, \delta(x) = x_{[1]} \otimes x_{[2]} \, $  for any  $ \, x \in \lieg \, $.

   For example, the compatibility condition between both structures reads
\begin{equation}  \label{eq:compat-bracket-cobracket}
   \begin{aligned}
      \delta([x,y])  &  \; = \;  \ad_x\big(\delta(y)\big) - \ad_y\big(\delta(x)\big)  \; =  \\
                     &  \; = \;  \big[x,y_{[1]}\big] \otimes y_{[2]} + y_{[1]} \otimes \big[x,y_{[2]}\big] -
                     \big[y,x_{[1]}\big] \otimes x_{[2]} - x_{[1]} \otimes \big[y,x_{[2]}\big]
   \end{aligned}
\end{equation}
   \indent   When  $ \, \big(\, \lieg \, ; \, [\,\ ,\,\ ] \, , \, \delta \,\big) \, $  is a Lie bialgebra,
   the same holds for  $ \, \big(\, \lieg^* \, ; \, \delta^* , \, {[\,\ ,\,\ ]}^* \,\big) \, $
   --- up to topological technicalities, if  $ \lieg $  is infinite-dimensional ---   which is called the  \textit{dual\/}
   Lie bialgebra to  $ \, \big(\, \lieg \, ; \, [\,\ ,\,\ ] \, , \, \delta \,\big) \, $.
                                                             \par
   We shall usually denote a Lie bialgebra simply by  $ \lieg \, $,  \,hence its dual by  $ \lieg^* $.
                                                             \par
   We need some more notation.  Given  $ \, r = r_{1} \otimes r_{2} \, $  and
   $ \, s = s_{1} \otimes s_{2} \, $  in  $ \, \lieg \otimes \lieg \, $
   --- and similarly in  $ \, \lieg \wedge \lieg \, $  ---   we write $ \; r_{2,1} := r_{2} \otimes r_{1} \; $  and
  $$  [\![ r , s ]\!] \, := \,  \big[r_{1},s_{1}\big] \otimes r_{2} \otimes s_{2} + r_{1} \otimes
  \big[r_{2},s_{1}\big] \otimes s_{2} +
  r_{1} \otimes s_{1} \otimes \big[r_{2},s_{2}\big]  $$
%
 which in compact form reads
  $$  [\![ r , s ]\!] \, := \,  [r_{1,2},s_{1,3}] + [r_{1,2},s_{2,3}] + [r_{1,3},s_{2,3}]  $$
   \indent   Further details can be found in  \cite{CP},  \cite{Mj},  and references therein.
\end{free text}

\vskip9pt

\begin{free text}  \label{deformations of LbA's}
 {\bf Deformations of Lie bialgebras.}  A general theory of  \textit{deformations\/}  for Lie bialgebras exists, which clearly springs up as a sub-theory of that of Lie algebras: see, for instance,  \cite{CG},  \cite{MW},  and references therein.  In the present work, we are mainly interested in two special kinds of deformations, which we now define, where either the Lie cobracket or the Lie bracket alone is deformed, leaving the ``other side'' of the overall structure untouched.
 \vskip3pt
   We begin by deforming the Lie cobracket.  Let  $ \big(\, \lieg \, ; \, [\,\ ,\,\ ] \, , \, \delta \,\big) \, $
   be a Lie bialgebra.  Let then  $ \, c \in \lieg \otimes \lieg \, $   --- identified with a  $ 0 $--cochain  ---   be such that
\begin{equation}  \label{eq: twist-cond_Lie-bialg}
   \ad_x\!\big( (\id \otimes \delta)(c) + \text{c.p.} + [\![\, c \, , c \,]\!] \,\big)  \; = \; 0  \;\; ,
   \quad  \ad_x\!\big( c + c_{{}_{\,2,1}} \big) \, = \, 0   \qquad   \forall \;\; x \in \lieg
\end{equation}
 where  $ \, \ad_x \, $  denotes the standard adjoint action of  $ x $  and ``\,c.p.''
 means ``cyclic permutations (on the tensor factors of the previous summand)''.
                                                             \par
   Then  \textit{the map  $ \, \delta^{\,c} : \lieg \relbar\joinrel\longrightarrow \lieg \wedge \lieg \, $  defined by
\begin{equation}  \label{eq: def_twist-delta}
  \delta^{\,c}  \, := \;  \delta - \partial(c) \;\; ,  \qquad  \text{i.e.\ \ \ }
  \delta^{\,c}(x) \, := \, \delta(x) - \ad_x(c)   \qquad \forall \;\; x \in \lieg
\end{equation}
 is a new Lie cobracket on the Lie algebra  $ \big(\, \lieg \, ; \, [\,\ ,\,\ ] \,\big) \, $  making
 $ \, \big(\, \lieg \, ; \, [\,\ ,\ ] \, , \, \delta^{c} \,\big) \, $  into a new Lie bialgebra}  (cf.\ \cite[Theorem 8.1.7]{Mj}).
\end{free text}

\vskip9pt

\begin{definition}  \label{def: twist-deform_Lie-bialg's}
 Every  $ \, c \in \lieg \otimes \lieg \, $  that obeys  \eqref{eq: twist-cond_Lie-bialg}  is called a  \textit{twist\/}
 of the Lie bialgebra  $ \lieg \, $,  \,and the associated Lie bialgebra
 $ \; \lieg^{\,c} \, := \, \big(\, \lieg \, ; \, [\,\ ,\,\ ] \, , \, \delta^{\,c} \,\big) \; $  is called a \textit{deformation by twist}
 (or  ``\textit{twist deformation\/}'')  of the original Lie bialgebra  $ \lieg \, $.   \hfill  $ \diamondsuit $
\end{definition}

\vskip7pt

   Now we go and deform the Lie bracket.  Let again  $ \big(\, \lieg \, ; \, [\,\ ,\,\ ] \, , \, \delta \,\big) \, $
   be a Lie bialgebra.  Let now  $ \, \chi \in \Hom_\Bbbk\!\big(\, \lieg \otimes \lieg \, , \Bbbk \,\big) \, $
   and identify  $ \; \Hom_\Bbbk\!\big(\, \lieg \otimes \lieg \, , \Bbbk \,\big) \, = \, {( \lieg \otimes \lieg )}^* \, = \,
   \lieg^*\otimes \lieg^* \; $   --- up to technicalities in the infinite-dimensional case (yet the outcome is always the same).
   Then condition  \eqref{eq: twist-cond_Lie-bialg}  with  $ \lieg^* $  replacing  $ \lieg $  and  $ \chi $  in the role of  $ c $
   reads
\begin{equation}  \label{eq: cocyc-cond_Lie-bialg}
 \begin{aligned}
   \ad_\psi\!\big(\, \partial_*(\chi) + {[\![\, \chi \, , \chi \,]\!]}_* \,\big)  \; = \; 0  \;\; ,
   \quad  \ad_\psi\!\big(\, \chi + \chi_{{}_{2,1}} \big) \, = \, 0
   \quad \qquad   \forall \;\; \psi \in \lieg^*
 \end{aligned}
\end{equation}
 where  $ \, \chi_{{}_{2,1}} := \chi^{\scriptscriptstyle T} \, $,  $ \, \partial_* \, $
 is the coboundary map for the Lie algebra  $ \lieg^* $  and similarly the symbol  $ \, {[\![\,\ ,\,\ ]\!]}_* \, $
 has the same meaning as above but with respect to  $ \lieg^* $.
 \vskip5pt
   For example, the condition  $ \; \ad_\psi\!\big( \chi + \chi_{{}_{2,1}} \big) \, = \, 0 \; $  for all
   $ \, \psi \in \lieg^* \, $  reads
  $$  \psi(x_{[1]}) \big(\, \chi(x_{[2]}, y) + \chi(y,x_{[2]}) \big) +
  \psi(y_{[1]}) \big(\, \chi(x,y_{[2]}) + \chi(y_{[2]},x) \big) = 0   \qquad  \forall\ x , y \in \lieg  $$
 This is clearly satisfied, for instance, whenever  $ \chi $  is antisymmetric, i.e.\ it is a  $ 2 $--cochain
 for the usual Lie algebra cohomology.
 \vskip5pt
   Then  \textit{the map  $ \; {[\,\ ,\ ]}_\chi : \lieg \wedge \lieg \relbar\joinrel\longrightarrow \lieg \, $  defined by
\begin{equation}  \label{eq: def_cocyc-bracket}
  {[x,y]}_\chi  \, := \;  [x,y] \, + \, \chi\big(x_{[1]},y\big) \, x_{[2]} \, - \, \chi\big(y_{[1]},x\big) \, y_{[2]}
  \qquad \qquad \forall \;\; x, y \in \lieg
\end{equation}
 is a new Lie bracket on the Lie coalgebra  $ \big(\, \lieg \, ; \, \delta \,\big) \, $  making
 $ \, \big(\, \lieg \, ; \, {[\,\ ,\ ]}_\chi \, , \, \delta \,\big) \, $  into a new Lie bialgebra}
 (cf.\ \cite[Exercise 8.1.8]{Mj}).

\vskip11pt

\begin{definition}  \label{def: cocyc-deform_Lie-bialg's}
 Every  $ \, \chi \in \Hom_\Bbbk\!\big(\, \lieg \wedge \lieg \, , \Bbbk \,\big) \, $  that obeys
 \eqref{eq: cocyc-cond_Lie-bialg}  is called a  \textit{2--cocy\-cle\/}  of the Lie bialgebra  $ \lieg \, $,
 \,and the Lie bialgebra  $ \, \lieg_\chi \, := \, \big(\, \lieg \, ; \, {[\,\ ,\ ]}_\chi \, , \, \delta \,\big) \, $
 is called a  \textit{deformation by 2--cocycle}  (or  ``\textit{2--cocycle deformation\/}'')  of the Lie bialgebra
 $ \lieg \, $.   \hfill  $ \diamondsuit $
\end{definition}

\vskip9pt

   At last, we point out that the two notions of ``twist'' and of  ``$ 2 $--cocycle''  for Lie bialgebras,
   as well as the associated deformations, are so devised as to be dual to each other.
   The following result then holds, whose proof is left to the reader:

\vskip11pt

\begin{prop}  \label{prop: duality-deforms x LbA's}
 Let  $ \lieg $  be a Lie bialgebra, and  $ \lieg^* $  the dual Lie bialgebra.
 \vskip3pt
   {\it (a)}\,  Let  $ c $  be a twist for  $ \lieg \, $,  and  $ \chi_c $  the image of  $ c $  in  $ {\big( \lieg^* \otimes \lieg^* \big)}^* $
   for the natural composed embedding
   $ \, \lieg \otimes \lieg \lhook\joinrel\relbar\joinrel\longrightarrow \lieg^{**} \otimes \lieg^{**} \lhook\joinrel\relbar\joinrel\longrightarrow {\big( \lieg^* \otimes \lieg^* \big)}^* \, $.
   Then  $ \chi_c $  is a 2--cocycle  for  $ \lieg^* \, $,  and there exists a canonical isomorphism
   $ \, {\big( \lieg^* \big)}_{\chi_c} \cong {\big( \lieg^c \big)}^* \, $.
 \vskip3pt
   {\it (b)}\,  Let  $ \chi $  be a 2--cocycle  for  $ \lieg \, $;  assume that  $ \lieg $  is finite-dimensional, and let  $ c_{\,\chi} $  be the image of  $ \, \chi $
   in the natural identification  $ \, {(\lieg \otimes \lieg)}^* = \lieg^* \otimes \lieg^* \, $.  Then  $ c_{\,\chi} $  is a twist for  $ \lieg^* \, $,
   and there exists a canonical isomorphism  $ \, {\big( \lieg^* \big)}^{c_{\,\chi}} \cong {\big( \lieg_\chi \big)}^* \, $.
\qed
\end{prop}

\medskip

\subsection{Multiparameter Lie bialgebras (=MpLbA's)}  \label{subsec: MpLbA's}  {\ }
 \vskip7pt
   Let  $ \, A := {\big(\, a_{i,j} \big)}_{i, j \in I} \, $  be some fixed generalized symmetrizable Cartan matrix, and let
   $ \, P := {\big(\, p_{i,j} \big)}_{i, j \in I} \in M_n(\Bbbk) \, $  be a matrix of Cartan type with associated Cartan matrix
   $ A \, $:  about the latter, hereafter we refer to the notions in  Definition \ref{def: realization of P}
   and all what follows in  \S \ref{sec: Cartan-data_realiz's},  \textsl{but\/}  working now with  $ \Bbbk $
   as ground ring instead of  $ \kh \, $.  Thus  $ \, P + P^{\,\scriptscriptstyle T} = 2\,D\,A \, $,  \,i.e.\
   $ \, p_{ij} + p_{ji} = 2\,d_i\,a_{ij} \, $  for all  $ \, i, j \in I \, $,  which implies  $ \, p_{ii} = 2\,d_i \not= 0 \, $  for all
   $ \, i \in I \, $.  Let  $ \, \R := \big(\, \lieh \, , \Pi \, , \Pi^\vee \,\big) \, $  be a split minimal realization of  $ P $,
   as in  Definition \ref{def: realization of P}\textit{(b.4)}  ---   so  $ \lieh $  is free over  $ \Bbbk $  with
   $ {\big\{ T_i^+ , T_i^- \big\}}_{i \in I} \, $  as  $ \Bbbk $--basis.
                                                                 \par
   Out of these data, we introduce the so-called ``multiparameter Borel Lie bialgebras''  $ \liebP_\pm $
   and a suitable Lie bialgebra pairing among them; then out of this pairing we construct the associated
   \textsl{Manin double},  that is a suitable, canonical structure of Lie bialgebra onto
   $ \, \liebP_+ \oplus \liebP_- \, $  depending on that of  $ \liebP_\pm $  and on the pairing itself.
   Our recipe follows in the footsteps of Halbout's construction in  \cite{Hal},
   that we are just slightly generalizing: indeed, all proofs in  \cite{Hal}  easily adapt to the present situation,
   the only assumptions which really are relevant in the calculations being that  $ \; (\alpha_i\,,\alpha_i) = 2\,d_i = p_{ii} \; $
   and  $ \; (\alpha_i\,,\alpha_j) + (\alpha_j\,,\alpha_i) = d_i a_{ij} + d_j a_{ji} = p_{ij} + p_{ji} \; $  ($ \, i , j \in I \, $).
 \vskip3pt
   \textsl{N.B.:}  as a matter of notation, as we are dealing with  $ \Bbbk $  rather than  $ \kh \, $,
   \,comparing with  \S \ref{def: realization of P}  we identify the space  $ \lieh $  with
   $ \overline{\lieh} \, $,  the roots  $ \alpha_j $  with  $ \overline{\alpha}_j \, $,  \,etc.

\vskip11pt

\begin{free text}  \label{pre-Borel MpLieBial}
 {\bf Pre-Borel multiparameter Lie bialgebras.}  We define the
 {\sl positive, resp.\ negative, pre-Borel multiparameter Lie bialgebra\/}  with multiparameter  $ P $
 as being the free Lie algebra over  $ \Bbbk \, $,  denoted by  $ \liebPhat_+ \, $,  resp.\ by  $ \liebPhat_- \, $,
 with generators  $ \, T^+_i $,  $ E_i \, $,  resp.\  $ T^-_i $,  $ F_i \, $  ($ \, i \in I \, $).
 Moreover, we give  $ \liebPhat_+ \, $,  resp.\   $ \liebPhat_- \, $,
 the unique structure of Lie bialgebra over  $ \k $  whose Lie cobracket is uniquely defined
 --- still using shorthand notation  $ \; x \wedge y \, := \, 2^{-1} (x \otimes y - y \otimes x) \; $ ---   by
  $$  \displaylines{
   \phantom{\text{resp.\ by}}  \quad  \delta\big(\,T\big) \, = \, 0  \; ,   \;\;
   \delta\big(E_i\big) \, = \, -2\,T^+_i \!\wedge E_i \, = \, -\big(\, T^+_i \otimes E_i - E_i \otimes T^+_i \,\big)
   \quad \hskip5pt  \forall \;\;  i \in I  \;  \cr
   \text{resp.\ by}  \quad  \delta\big(\,T\big) \, = \, 0  \; ,   \;\;
   \delta\big(F_i\big) \, = \, +2\,T^-_i \!\wedge F_i \, = \, +\big(\, T^-_i \otimes F_i - F_i \otimes T^-_i \,\big)
   \quad \hskip7pt  \forall \;\;  i \in I  \;  }  $$
 \vskip5pt
   One can prove   --- like in  \cite{Hal}  ---   that there exists a Lie bialgebra pairing
  $$  \langle\;\ ,\,\ \rangle  \, : \, \liebPhat_+ \times \liebPhat_- \relbar\joinrel\relbar\joinrel\longrightarrow \k  $$
 uniquely given   --- for all  and  $ \, i \, , j \in I \, $  ---   by
  $$  \displaylines{
   \big\langle\, T_i^+ , T_j^- \,\big\rangle  \, = \,  p_{ij} \, = \, \alpha_i(\,T_j^-\,) \, = \, \alpha_j(\,T_i^+\,)  \quad ,   \qquad
   \big\langle\, T_i^+ , F_j \,\big\rangle  \, = \,  0  \, = \,  \big\langle\, E_i \, , T_j^- \,\big\rangle  \cr
   \big\langle\, E_i \, , F_j \,\big\rangle  \, = \,  \delta_{i{}j} \, p_{ii}^{-1}  \, = \,  \delta_{i{}j} \, {2\,d_i}^{-1} }  $$
\end{free text}

\vskip11pt

\begin{free text}  \label{Borel MpLieBial}
 {\bf Borel multiparameter Lie bialgebras.}
 We introduce a Lie ideal  $ \mathfrak{l}_\pm $  of  $ \liebPhat_\pm $  as follows.
 On the one hand,  $ \mathfrak{l}_+ $  is the Lie ideal generated by the elements
   $$  \displaylines{
   T^+_{i,j} \, := \, \big[\, T_i^+ , T_j^+ \big]\; ,
 \quad  E^{(T)}_{i,j} \, := \,\big[\,  T_i^+ , E_j \, \big]\,  - \, \alpha_j(T_i^+) \, E_j   \hskip21pt  \forall \; i \, , \, j \in I  \cr
  \qquad   E_{i,j}  \; := \,  {\big(\textsl{ad}\,(E_i)\big)}^{1-a_{ij}}(E_j)   \qquad \qquad  \forall \; i \not= j  }  $$
 on the other hand,  $ \mathfrak{l}_- $  is the Lie ideal generated by the elements
   $$  \displaylines{
   T^-_{i,j} \, := \, \big[\, T_i^- , T_j^- \big]\; ,
 \quad  F^{(T)}_{i,j} \, := \,\big[\,  T_i^- , F_j \, \big]\, + \, \alpha_j(T_i^-) \, F_j   \hskip21pt  \forall \; i \, , \, j \in I  \cr
  \qquad   F_{i,j}  \; := \,  {\big(\textsl{ad}\,(F_i)\big)}^{1-a_{ij}}(F_j)   \qquad \qquad  \forall \; i \not= j  }  $$
   \indent   Now, acting once again like in  \cite{Hal},  one sees that
 \textit{$ \mathfrak{l}_+ $  is contained in the left radical and  $ \mathfrak{l}_- $  is contained
 in the right radical of the pairing
 $ \; \langle\;\ ,\,\ \rangle  \, : \, \liebPhat_+ \times \liebPhat_- \relbar\joinrel\relbar\joinrel\longrightarrow \k \; $
 introduced above}.  This has two consequences:
 \vskip3pt
   \textsl{(1)}\;  first,  $ \mathfrak{l}_\pm $  is in fact a Lie bi-ideal in the Lie bialgebra  $ \liebPhat_\pm \, $,
   \,hence either quotient  $ \; \liebP_\pm := \liebPhat_\pm \big/ \mathfrak{l}_\pm \; $
   is a quotient  \textsl{Lie bialgebra\/}  indeed: we call  $ \liebP_+ \, $,  resp.\  $ \liebP_- \, $
   the  \textsl{positive},  resp.\ the  \textsl{negative},  \textsl{Borel multiparameter Lie bialgebra\/}
   with multiparameter  $ P \, $;
 \vskip3pt
   \textsl{(2)}\;  second, the (non-degenerate) Lie bialgebra pairing
 $ \; \langle\;\ ,\,\ \rangle  \, : \, \liebPhat_+ \times \liebPhat_- \relbar\joinrel\relbar\joinrel\longrightarrow \k \; $
 boils down to a (possibly degenerate) Lie bialgebra pairing
 $ \; \langle\;\ ,\,\ \rangle  \, : \, \liebP_+ \times \liebP_- \relbar\joinrel\relbar\joinrel\longrightarrow \k \; $
 of (multiparameter) Borel Lie bialgebras.
\end{free text}

\vskip9pt

\begin{free text}  \label{MpLieBialgebras-double}
 {\bf Multiparameter Lie bialgebras as doubles.}  Still following  \cite{Hal},
 we proceed now to introduce our multiparameter Lie bialgebras, in two steps.
 \vskip3pt
   \textsl{$ \underline{\text{The split minimal case}} $:}
   By general theory we can use the two Lie bialgebras  $ \liebP_+ $  and  $ \liebP_- $
   and the pairing among them to define their  \textsl{generalized double\/}  (as in  \cite[\S 8.3]{Mj}).
   Namely, we endow  $ \, \liegdP := \liebP_+ \oplus \liebP_- \, $  with a Lie algebra structure described
   in the same way as in  \cite[Th{\'e}or{\`e}me 1.5]{Hal},
   and the unique Lie coalgebra structure such that
   $ \, {\big( \liebP_+ \big)}^{\text{cop}} \!\lhook\joinrel\longrightarrow \liegdP \!\longleftarrow\joinrel\rhook \liebP_- \, $
   are Lie coalgebra embeddings, where a superscript ``cop''
   means that we are taking the opposite Lie coalgebra structure   --- cf.\  \cite[Ch.\ 1, \S 1.4]{CP},
   for further details, or even  \cite[\S 2.3]{ApS}  (and references therein),
   for a quick recap of this topic and its generalizations.
   This makes  $ \liegdP $  into a Lie bialgebra; in addition, when the pairing
 $ \; \langle\,\ ,\ \rangle  \, : \, \lieb_+ \times \lieb_- \relbar\joinrel\longrightarrow \k \; $
 is non-degenerate, the Lie bialgebra  $ \liegdP $  is even  \textsl{quasitriangular}.
                                                         \par
   Finally, from the previous description of  $ \liebP_+ $  and  $ \liebP_- $  one
   also deduces an explicit presentation for  $ \liegdP \, $.  Namely,  $ \liegdP $
   is the Lie bialgebra generated (as Lie algebra) by the  $ \Bbbk $--subspace  $ \lieh $
   together with elements  $ E_i $  and  $ F_i $  ($ \, i \in I \, $),  with relations
  $$  \displaylines{
   \big[\, T' , T'' \,\big] \, = \, 0  \;\; ,   \quad   \big[\, T , E_j \,\big] \, - \, \alpha_j(\,T) \, E_j \, = \, 0  \;\; ,
   \quad   \big[\, T , F_j \,\big] \, + \, \alpha_j(\,T) \, F_j \, = \, 0  \cr
   {\big(\textsl{ad}\,(E_i)\big)}^{1-a_{ij}}(E_j) \, = \, 0  \;\; ,   \;\;\;
   {\big(\textsl{ad}\,(F_i)\big)}^{1-a_{ij}}(F_j) \, = \, 0  \;\; ,   \;\;\;  \big[\, E_i \, , F_j \,\big] \, = \,
   \delta_{ij} \, {{\;T_i^+ \! + T_i^-\;} \over {\;2\,d_i\;}}  }  $$
 --- for all  $ \; T' , T'' \, , T \in \lieh \, $,  $ \; i , j \in I \, $,  $ \; i \neq j \, $,  \;with
 Lie bialgebra structure given on generators   --- for all  $ \, T \in \lieh \, $  and all  $ \, i \in I \, $  ---   by
  $$  \delta\big(\,T\big) \, = \, 0  \quad ,   \qquad  \delta\big(E_i\big) \, = \, 2 \; T^+_i \!\wedge E_i
   \quad ,   \qquad  \delta\big(F_i\big) \, = \,  2 \; T^-_i \!\wedge F_i  $$
   \indent   As a last remark, we stress that in  $ \liegdP $  the Lie algebra structure does
   depend on the multiparameter  $ P $,  while the Lie coalgebra structure
   \hbox{is actually independent of it.}
%
 \vskip3pt
   \textsl{$ \underline{\text{The general case}} $:}  Let now  $ P $
   be again a multiparameter (of Cartan type) and  $ \; \R \, := \, \big(\, \lieh \, , \Pi \, , \Pi^\vee \,\big) \; $
   be  \textsl{any\/}  realization of it   --- not necessarily split nor minimal.
   Then we define a Lie bialgebra  $ \liegRP $  generalizing the notion of  $ \liegdP $  introduced above.
                                                                      \par
   Namely, as a Lie  \textsl{algebra\/}  we define  $ \liegRP $  by generators and relations with
   (formally) the very same presentation as for  $ \liegdP $  right above.  The Lie  \textsl{coalgebra\/}
   also has the same form, but we introduce it indirectly.
                                                                      \par
   First of all, using  Lemma \ref{lemma: split-lifting}  we fix a realization
   $ \; \dot{\R} \, := \, \big(\, \dot{\lieh} \, , \dot{\Pi} \, , {\dot{\Pi}}^\vee \,\big) \; $  of  $ P $  that is
   \textsl{split}, and we also fix  $ \, \dot{\lieh}_{{}_T} := \textsl{Span}\Big( {\big\{ T_i^\pm \big\}}_{i \in I} \Big) \, $
   inside  $ \dot{\lieh} \, $.  Then we consider also
   $ \; \mathring{\R} \, := \, \big(\, \dot{\lieh}_{{}_T} , \mathring{\Pi} \, , \Pi^\vee \,\big) \; $   --- where
   $ \, \mathring{\Pi} := \Big\{\, \mathring{\alpha}_i := \alpha_i{\big|}_{\dot{\lieh}_{{}_T}} \Big\}_{i \in I} \, $
   ---   that is again a realization of  $ P $,  which is now  \textsl{split and minimal}.  Out of  $ \dot{\R} $  and
   $ \mathring{\R} $  we define two Lie  \textsl{algebras}   --- denoted
   $ \, {\mathfrak{g}^{\raise-1pt\hbox{$ \scriptscriptstyle \dot{\mathcal{R}} $}}_\Ppicc} \, $,
   resp.\  $ \, {\mathfrak{g}^{\raise-1pt\hbox{$ \scriptscriptstyle \mathring{\mathcal{R}} $}}_\Ppicc} \, $
   ---   via an explicit presentation, which is again (formally) like the one given above for
   $ \liegdP \, $,  \,up to the obvious changes.  Clearly, the inclusion  $ \, \dot{\lieh}_{{}_T} \subseteq \dot{\lieh} \, $
   induces an embedding of Lie algebras
 $ \, {\mathfrak{g}^{\raise-1pt\hbox{$ \scriptscriptstyle \mathring{\mathcal{R}} $}}_\Ppicc}
 \lhook\joinrel\relbar\joinrel\longrightarrow
 {\mathfrak{g}^{\raise-1pt\hbox{$ \scriptscriptstyle \dot{\mathcal{R}} $}}_\Ppicc} \, $.
                                                                      \par
   Now, as  $ \mathring{\R} $  is  \textsl{split minimal},  the Lie algebra
   $ \, {\mathfrak{g}^{\raise-1pt\hbox{$ \scriptscriptstyle \mathring{\mathcal{R}} $}}_\Ppicc} \, $
   actually is one of the form  $ \liegdP \, $,  \,and as such it bears a structure of Lie  \textsl{bialgebra\/}  as given above.
   But then, it follows by construction that there is a unique way to extend the Lie cobracket of
   $ \, {\mathfrak{g}^{\raise-1pt\hbox{$ \scriptscriptstyle \mathring{\mathcal{R}} $}}_\Ppicc} \, $  to
   $ \, {\mathfrak{g}^{\raise-1pt\hbox{$ \scriptscriptstyle \dot{\mathcal{R}} $}}_\Ppicc} \, $
   in such a way that the embedding
 $ \, {\mathfrak{g}^{\raise-1pt\hbox{$ \scriptscriptstyle \mathring{\mathcal{R}} $}}_\Ppicc}
 \lhook\joinrel\relbar\joinrel\longrightarrow {\mathfrak{g}^{\raise-1pt\hbox{$ \scriptscriptstyle
 \dot{\mathcal{R}} $}}_\Ppicc} \, $
 mentioned above is actually one of Lie  \textsl{bialgebras}.  In short,
 $ \, {\mathfrak{g}^{\raise-1pt\hbox{$ \scriptscriptstyle \dot{\mathcal{R}} $}}_\Ppicc} \, $
 bears a Lie bialgebra structure that is again described by the very same formulas as for  $ \liegdP \, $,
 up to replacing everywhere  $ \lieh $  with  $ \dot{\lieh} \, $.
                                                                      \par
   Finally, again by Lemma \ref{lemma: split-lifting}  there exists also an epimorphism of realizations
   $ \; \underline{\pi} : \dot{\R} \relbar\joinrel\twoheadrightarrow \R \; $.  Then, from the presentation of both
   $ \, {\mathfrak{g}^{\raise-1pt\hbox{$ \scriptscriptstyle \dot{\mathcal{R}} $}}_\Ppicc} \, $  and  $ \, \liegRP \, $,
   \,this  $ \underline{\pi} $  induces an epimorphism of Lie algebras
   $ \; \L_{\underline{\pi}} : {\mathfrak{g}^{\raise-1pt\hbox{$ \scriptscriptstyle \dot{\mathcal{R}} $}}_\Ppicc} \relbar\joinrel\twoheadrightarrow \liegRP \; $,
   such that
   $ \Ker\big(\L_{\underline{\pi}}\,\big) $
is generated by
   $ \, \Ker\big(\, \pi : \dot{\lieh} \relbar\joinrel\relbar\joinrel\twoheadrightarrow \lieh \,\big) \, $,
and the latter lies in the center of  $ {\mathfrak{g}^{\raise-1pt\hbox{$ \scriptscriptstyle \dot{\mathcal{R}} $}}_\Ppicc} \, $,
\,by definitions and by  Lemma \ref{lemma: ker-morph's_realiz's};  moreover, the Lie cobracket of
$ {\mathfrak{g}^{\raise-1pt\hbox{$ \scriptscriptstyle \dot{\mathcal{R}} $}}_\Ppicc} $  kills  $ \Ker(\pi) \, $,
so the latter is a Lie  \textsl{bi-ideal}  in the Lie bialgebra  $ {\mathfrak{g}^{\raise-1pt\hbox{$ \scriptscriptstyle \dot{\mathcal{R}} $}}_\Ppicc} \, $.
Thus  $ \liegRP $  inherits via  $ \L_{\underline{\pi}} $  a  \textsl{quotient Lie bialgebra structure\/}  from
$ {\mathfrak{g}^{\raise-1pt\hbox{$ \scriptscriptstyle \dot{\mathcal{R}} $}}_\Ppicc} \, $,
again described by the same formulas given above for  $ \liegdP \, $.
 \vskip3pt
   Every such Lie bialgebra  $ \liegRP $  will be called  \textit{multiparameter Lie bialgebra},  in short  \textit{MpLbA}.
   In addition, we say that the MpLbA  $ \liegRP $  is  \textsl{straight},  or \textsl{small},  or \textsl{minimal},  or  \textsl{split},
   if such is  $ \R \, $,  \,and we define  \textsl{rank\/}  of  $ \liegRP $  as  $ \; \rk\!\big( \liegRP \big) := \rk(\R) = \rk_\Bbbk(\lieh) \, $.
 \vskip3pt
   For later use, we remark that every  $ \liegRP $  has two obvious  \textsl{triangular decompositions}
\begin{equation}  \label{eq: triang-decomp's_Lie-bialg's}
  \liegRP  \; = \;  \lien_+ \oplus \lieh \oplus \lien_-  \quad ,  \qquad \qquad  \liegRP  \; = \;  \lien_- \oplus \lieh \oplus \lien_+
\end{equation}
 as a direct sum of Lie algebras, where  $ \lien_+ \, $,  resp.\  $ \lien_- \, $,  is the Lie subalgebra of  $ \liegRP $
 generated by all the  $ E_i $'s,  resp.\ all the  $ F_i $'s.
\end{free text}

\vskip3pt

   The following result points out the fact that the dependence of MpLbA's on
   realizations (for a common, fixed multiparameter matrix) is functorial:

\vskip13pt

\begin{prop}  \label{prop: functor_R->liegRP}
 Let  $ \, P \in M_n(\Bbbk) \, $.  If both  $ \R' $  and  $ \, \R'' $  are realizations of  $ P $
 and  $ \, \underline{\phi} : \R' \relbar\joinrel\relbar\joinrel\longrightarrow \R'' \, $
 is a morphism between them, then there exists a unique morphism of Lie bialgebras
 $ \; \Lc_{\underline{\phi}} : \lieg^{\scriptscriptstyle \R'}_\Ppicc \!\relbar\joinrel\longrightarrow
 \lieg^{\scriptscriptstyle \R''}_\Ppicc \; $  that extends the morphism
 $ \, \phi : \lieh' \!\relbar\joinrel\longrightarrow \lieh'' \, $  given by  $ \underline{\phi} \, $;  \,moreover,
 $ \, \Lc_{\underline{\id}_\R} = \id_{\liegRP} \, $  and
 $ \; \Lc_{\underline{\phi}' \circ\, \underline{\phi}} = \Lc_{\underline{\phi}'} \circ \Lc_{\underline{\phi}} \; $
 (whenever  $ \, \underline{\phi}' \circ\, \underline{\phi} \, $  is defined).
                                                                                \par
   Thus, the construction  $ \, \R \mapsto \liegRP \, $   --- for any fixed  $ P $  ---   is functorial in  $ \R \, $.
                                                                              \par
   Moreover, if  $ \underline{\phi} $  is an epimorphism, resp.\ a monomorphism, then
   $ \Lc_{\underline{\phi}} $  is an epimorphism, resp.\ a monomorphism, as well.
   Finally, for any morphism  $ \, \underline{\phi} : \R' \relbar\joinrel\longrightarrow \R'' \, $,  \,the kernel
   $ \, \Ker\big(\Lc_{\underline{\phi}}\big) $  of  $ \, \Lc_{\underline{\phi}} $  coincides with  $ \, \Ker(\phi) \, $,
   and the latter is central in  $ \, \lieg^{\scriptscriptstyle \R'}_\Ppicc \, $.
                                                                                \par
   In particular, when  $ \underline{\phi} \, $,  and hence also  $ \Lc_{\underline{\phi}} \, $,
   is an epimorphism, we have   --- setting  $ \; \liek := \Ker(\phi) \, $  ---
   a central exact sequence of Lie bialgebras
  $$  0 \,\relbar\joinrel\relbar\joinrel\longrightarrow\, \liek
  \,\relbar\joinrel\relbar\joinrel\longrightarrow\, \lieg^{\scriptscriptstyle \R'}_\Ppicc
  \,\;{\buildrel {\Lc_{\underline{\phi}}} \over {\relbar\joinrel\relbar\joinrel\longrightarrow}}\;\,
  \lieg^{\scriptscriptstyle \R''}_\Ppicc \,\relbar\joinrel\relbar\joinrel\longrightarrow\, 0  $$
\end{prop}

\pf
 The existence of  $ \Lc_{\underline{\phi}} $  is obvious, as well as all the other claims;
 we only spend a moment on the centrality of  $ \Ker(\phi) \, $.
 Lemma \ref{lemma: ker-morph's_realiz's}  imply
 $ \; \Ker(\phi) \, \subseteq \, \bigcap\limits_{j \in I} \Ker(\alpha'_j) \; $;
 \,then from the relations among the generators of
 $ \lieg^{\scriptscriptstyle \R'}_\Ppicc $  (cf.\  \S \ref{MpLieBialgebras-double})
 we get that each element in  $ \Ker(\phi) $  commutes with all generators of
 $ \lieg^{\scriptscriptstyle \R'}_\Ppicc \, $,  thus  $ \Ker(\phi) $  is central.
\epf

\vskip9pt

\begin{cor}  \label{cor: isom_R -> isom_liegRP}
 With notation as above, if  $ \; \R' \cong \R'' \, $  then
 $ \; \lieg^{\scriptscriptstyle \R'}_\Ppicc \cong \lieg^{\scriptscriptstyle \R''}_\Ppicc \; $.
 \vskip3pt
   In particular, all MpLbA's built upon split realizations, respectively small realizations,
   of the same matrix  $ P $  and sharing the same rank of  $ \, \lieh \, $,  are isomorphic to each other,
   hence they are independent (up to isomorphisms) of the specific realization,
   but only depend on  $ P $  and on the rank of  $ \, \lieh \, $.
\end{cor}

\pf
 This follows at once from  Proposition \ref{prop: functor_R->liegRP}  together with the uniqueness result in
 Proposition \ref{prop: exist-realiz's}  and  Proposition \ref{prop: exist-realiz's_small}.
\epf

\vskip9pt

\begin{rmk}
 We expect that our definition (and construction) of MpLbA's, as well as the related results presented below,
 can be extended to the case where the symmetrizable generalized Cartan matrix  $ A $
 is replaced by a more general symmetrizable Borcherds-Cartan matrix, as in  \cite{Bor}.
 However, due to additional technical difficulties, we do not pursue such a goal in the present paper.
\end{rmk}

\medskip

\begin{free text}  \label{constr.-db-cr-sums}
 {\bf Construction via double cross sums.}  In this subsection we give another construction of our MpLbA's,
 as suitable  \textit{double cross sums\/};  the latter can be seen as a semiclassical version of the double cross products
 of FoMpQUEAs given in  \S \ref{constr.-double-crossproduct}.  We follow  \cite[Section 8.3]{Mj}  for the exposition.
\end{free text}

\vskip7pt

\begin{definition}  \label{def:matched-pair-LieA}
 \cite[Definition 8.3.1]{Mj}
 Two Lie algebras  $ \, (\lieg,\liem) \, $  form a  \textsl{right-left matched pair\/}  if there is a right action of
 $ \lieg $  on  $ \liem $  and a left action of  $ \liem $  on  $ \lieg \, $,  denoted
  $$  \triangleleft \, : \liem \otimes \lieg \longrightarrow \liem   \qquad \text{and} \qquad
  \triangleright \, : \liem \otimes \lieg \longrightarrow \lieg  $$
 obeying the following conditions (for all  $ \, m , n \in \liem \, $  and  $ \, x , y \in \lieg \, $):
\begin{align*}
   [m,n] \triangleleft x  &  \; = \;  [m \triangleleft x , n] \, + \, [m \, , n \triangleleft x] \, + \, n \triangleleft (m \triangleright x) \, - \, m \triangleleft (n \triangleright x)  \\
   m \triangleright [x,y]  &  \; = \;  [m \triangleright x , y] \, + \, [x , m \triangleright y] \, + \, (m \triangleleft y) \triangleright x \, - \, (m \triangleleft x) \triangleright y
\end{align*}
\end{definition}

\vskip7pt

   After the previous definition, the key fact is the following result:

\vskip7pt

\begin{prop}  \cite[Proposition 8.3.2]{Mj}
%
%
 \vskip3pt
   (a)\,  Given a matched pair of Lie algebras  $ \, (\lieg,\liem) \, $,  there exists a Lie algebra,
   called  \textsl{double cross sum}  Lie algebra and denoted by  $ \, \lieg \bowtie \liem \, $,
   \,whose socle is the vector space  $ \, \lieg \oplus \liem $  and whose Lie bracket is
   (for all  $ \, x , y \in \lieg \, $  and  $ \, m , n \in \liem \, $)
  $$  [x \oplus m , y \oplus n]  \; = \;  \big( [x,y] + m \triangleright y - n \triangleright x \big) \, \oplus \, \big( [m,n] + m \triangleleft y - n \triangleleft x \big)  $$
 \vskip3pt
   (b)\,  Conversely, if the direct sum space  $ \, \lieh := \lieg \oplus \liem \, $
   bears a structure of Lie bialgebra such that the copies of\/  $ \lieg $  and\/  $ \liem $
   inside it are Lie subalgebras, then  $ \, (\lieg , \liem) \, $  is a matched pair of Lie algebras whose
   associated double cross sum obeys  $ \, \lieg \bowtie \liem \cong \lieh \, $,  \,i.e.\ it is isomorphic to  $ \lieh \, $.
   The required actions are recovered from the identities
  $$  \big[ j(m) , i(x) \big]  \; = \;  i(m \triangleright x) + j(m \triangleleft x)   \eqno \forall \;\; x \in \lieg \, , m \in \liem  \qquad  $$
 where  $ \; i : \lieg \longrightarrow \lieg \oplus \liem =: \lieh \; \big(\, y \mapsto i(y) := (y,0_\liem) \big) \; $
 and likewise  $ \; j : \liem \longrightarrow \lieg \oplus \liem =: \lieh $  $ \; \big(\, n \mapsto j(n) := (0_\lieg,n) \big) \; $  are the natural Lie algebra monomorphisms.
\qed
\end{prop}

\vskip9pt

   In order to extend the notion of matched pair to Lie bialgebras, it is necessary to have a
   compatibility of the left-right actions with the Lie coalgebra structures.
   Thus assume now that  $ \lieg $  and  $ \liem $  are Lie bialgebras: we say that  \textsl{$ \liem $
   is a right  $ \lieg $--module  Lie coalgebra\/}  if it is a right  $ \lieg $--module  and in addition
   (for  $ \, m \in \liem \, $  and  $ \, x \in \lieg \, $)  one has
  $$
  \delta_\liem(m \triangleleft x)  \; = \;  \big( m_{[1]} \triangleleft x \big) \otimes m_{[2]} \, +
  \, m_{[1]} \otimes \big( m_{[2]} \triangleleft x \big)  \; =: \;  \delta_\liem(m) \triangleleft x
  $$
the notion of left  $ \liem $--module  Lie coalgebra is defined analogously.

\vskip9pt

\begin{prop}  \cite[Proposition 8.3.4]{Mj}
 \vskip3pt
   Let  $ \, (\lieg , \liem) \, $  be a matched pair of Lie algebras, with both\/  $ \lieg $  and\/  $ \liem $
   being Lie bialgebras and with  $ \triangleleft \, $,  resp.\  $ \triangleright \, $,  making  $ \lieg $  into a left
   $ \liem $--module Lie coalgebra, resp.\  $ \liem $  into a right  $ \lieg $--module Lie coalgebra,
   such that, for all  $ \, m \in \liem \, $  and  $ \, x \in \lieg \, $,  we have
%
%
  $$
  \displaylines{
   \quad   0  \; = \;  m \triangleleft \delta_\lieg(x) \, + \, \delta_\liem(m) \triangleright x  \; =   \hfill  \cr
   \hfill   = \;  \big( m \triangleleft x_{[1]} \big) \otimes x_{[2]} \, + \, x_{[1]} \otimes \big( m \triangleleft x_{[2]} \big) \, + \,
   \big( m_{[1]} \triangleright x \big) \otimes m_{[2]} \, + \, m_{[1]} \otimes \big( m_{[2]} \triangleright x \big)   \quad  }
   $$
 Then the direct sum Lie coalgebra structure makes  $ \, \lieg \bowtie \liem \, $  into a  \textsl{Lie bialgebra},
 which is called the  \textsl{double cross sum Lie bialgebra}.   \qed
\end{prop}

\vskip9pt

\begin{free text}  \label{MpLieBialgebras-db-cr-sums}
 {\bf Multiparameter Lie bialgebras as double cross sums.}  Let  $ \, A := {\big(\, a_{i,j} \big)}_{i, j \in I} \, $
 be some fixed generalized symmetrizable Cartan matrix, and let  $ \, P := {\big(\, p_{i,j} \big)}_{i, j \in I} \in M_n(\Bbbk) \, $
 be a matrix of Cartan type with associated Cartan matrix  $ A \, $.  Then one defines, as in Subsection
 \ref{Borel MpLieBial},  the Borel multiparameter Lie bialgebras  $ \liebP_+ $  and  $ \liebP_- \, $,
 which are dually paired by a Lie bialgebra pairing, that we denote hereafter by
 $ \; \langle\;\ ,\,\ \rangle  \, : \, \liebP_+ \times \liebP_- \relbar\joinrel\relbar\joinrel\longrightarrow \k \; $.
 Using this pairing one may define left-right actions
\begin{align*}
   \triangleleft \,  &  : \liebP_+ \otimes {\big( \liebP_- \big)}^{\op} \relbar\joinrel\longrightarrow \liebP_+  &
\triangleright \,  &  : \liebP_+ \otimes {\big( \liebP_- \big)}^{\op} \relbar\joinrel\longrightarrow {\big( \liebP_- \big)}^{\op}  \\
                     &  m \triangleleft \, x  \, := \,  m_{[1]} \, \big\langle m_{[2]} , x \big\rangle  &
                   &  m \triangleright \, x  \, := \,  x_{[1]} \, \big\langle m , x_{[2]} \big\rangle
\end{align*}
 for all  $ \, m \in \liebP_+ \, $  and  $ \, x \in {\big( \liebP_- \big)}^{\op} \, $,  cf.\  \cite[Example 8.3.7]{Mj}.
 Then these Borel multiparameter Lie bialgebras form a matched pair  $ \, \Big( {\big(\liebP_- \big)}^{\op} , \liebP_+ \Big) \, $,
 \,whence the double cross sum Lie bialgebra  $ {\big( \liebP_- \big)}^{\op} \bowtie \liebP_+ \, $  is defined.
 By the very construction, there exist also Lie bialgebra monomorphisms
  $$
  {\big( \liebP_- \big)}^{\op} \lhook\joinrel\relbar\joinrel\relbar\joinrel\relbar\joinrel\longrightarrow {\big( \liebP_- \big)}^{\op}
  \bowtie \liebP_+ \!\longleftarrow\joinrel\relbar\joinrel\relbar\joinrel\relbar\joinrel\rhook \liebP_+
  $$
   \indent   An entirely similar construction can be made transposing the opposite Lie algebra structure on
   $ \liebP_- $  to the co-opposite Lie coalgebra structure on  $ \liebP_+ \, $:  in other words, the pairing
   $ \, \langle \;\ ,\,\ \rangle \, $  induces skew-parings both on  $ {\big( \liebP_- \big)}^{\op} \otimes \liebP_+ \, $
   and on  $ \, \liebP_- \otimes {\big( \liebP_+ \big)}^{\cop} \, $.  In this case, we get the matched pair of Lie bialgebras
   $ \, \Big( \liebP_- , {\big( \liebP_+ \big)}^{\cop} \Big) \, $  and the double cross sum  $ \, \liebP_- \bowtie {\big( \liebP_+ \big)}^{\cop} \, $;
   \,the latter also admits the Lie bialgebra monomorphisms
  $$
  \liebP_- \lhook\joinrel\relbar\joinrel\relbar\joinrel\relbar\joinrel\longrightarrow \liebP_- \bowtie {\big( \liebP_+ \big)}^{\cop}
  \!\longleftarrow\joinrel\relbar\joinrel\relbar\joinrel\relbar\joinrel\rhook {\big( \liebP_+ \big)}^{\cop}
  $$
 Morevoer, by the very definition, this double cross sum is isomorphic to the Drinfeld double  $ \liegdP $
 as defined in \S \ref{MpLieBialgebras-double},  that is
  $ \,\; \liebP_- \bowtie {\big( \liebP_+ \big)}^{\cop}  \, \cong \,  \liegdP \;\, $.
\end{free text}

\medskip

\subsection{Deformations of MpLbA's by toral twists}  \label{subsec: tor-twist def's_mp-Lie_bialg's}
 Let  $ \liegRP $  be a MpLbA as in  \S \ref{MpLieBialgebras-double}  above;
 then  $ \lieh $  is a free  $ \k $--module  of finite rank  $ \, t := \rk(\lieh) \, $:
 \,we fix in it a  $ \k $--basis  $ \, {\big\{ H_g \big\}}_{g \in \G} \, $,  where
 $ \G $  is an index set with  $ \, |\G| = \rk(\lieh) =: t \, $.
                                                       \par
   We begin introducing the so-called ``toral'' twists for  $ \liegRP \, $.

\vskip9pt

\begin{definition}  \label{def: toral twists x MpLbA's}
 For any  \textsl{antisymmetric\/}  matrix  $ \; \Theta = \big( \theta_{gk} \big)_{g, k \in \G} \in \lieso_t(\Bbbk) \; $,  \;we set
\begin{equation}  \label{eq: def_twist_Lie-bialg}
  j_\Thetapicc  \; := \;  {\textstyle \sum_{g,k=1}^t} \theta_{gk} \, H_g \otimes H_k  \; \in \;
  \lieh \otimes \lieh  \; \subseteq \;  \lieg \otimes \lieg
\end{equation}
 and we call this  $ \, j_\Thetapicc \, $  \textit{the toral twist\/}  (or  \textit{``twist of toral type''\/})
 associated with  $ \, \Theta\, $.   \hfill  $ \diamondsuit $
\end{definition}

\vskip7pt

   \textsl{$ \underline{\text{N.B.}} $:}\,  When  $ \lieg $  is a simple Lie algebra, in the classification of  \cite{ESS}  via Belavin-Drinfeld triples the above twist is associated with the empty datum; moreover, it turns out to be a semiclassical limit of a twist for  $ \uhg \, $,  \,see  \S \ref{sec: special-&-quantiz}.

\vskip7pt

   Next result   --- which explains our use of terminology ---   follows by construction;
   in particular, it makes use of the antisymmetry condition on  $ \Theta \, $.

\vskip11pt

\begin{lema}  \label{lemma: toral twist x MpLbA's}
 For any matrix  $ \; \Theta = \big( \theta_{gk} \big)_{g, k \in \G} \in \lieso_t(\Bbbk) \; $,  \,the element
 $ \, j_\Thetapicc \, $  given in  Definition \ref{def: toral twists x MpLbA's}  is a  \textsl{twist\/}
 element for the Lie bialgebra  $ \liegRP \, $,  \,in the sense of  \eqref{eq: twist-cond_Lie-bialg}.
\end{lema}

\vskip9pt

   Concerning deformations of MpLbA's by toral twists, our main result is the next one.
   To settle its content, let  $ \, P\in M_t(\Bbbk) \, $  be a multiparameter matrix of Cartan type with
   associated Cartan matrix  $ A \, $,  let  $ \R $  be a realization of it, and let  $ \liegRP $
   be the associated multiparameter Lie bialgebra; then, for any given antisymmetric matrix
   $ \; \Theta = \big( \theta_{gk} \big)_{g, k \in \G} \in \lieso_t(\Bbbk) \; $,  \,let
 $ \; j_\Thetapicc \, := \, {\textstyle \sum\limits_{g,k=1}^t} \theta_{gk} \, H_g \otimes H_k \; $
be  the associated twist as in  \eqref{eq: def_twist_Lie-bialg}.
Moreover, we consider the ``deformed'' multiparameter matrix
 $ \; P_\Thetapicc = \, {\big(\, p^\Thetapicc_{i,j} \big)}_{i, j \in I} \, := \,
 P - \mathfrak{A} \, \Theta \, \mathfrak{A}^{\,\scriptscriptstyle T} \; $
 as in  \eqref{def-P_Phi}   --- again of Cartan type, the same as  $ P $  ---
 and its corresponding ``deformed'' realization
 $ \, \R_\Thetapicc := \big(\, \lieh \, , \Pi \, , \Pi^\vee_\Thetapicc \! :=
 {\big\{ T^+_{\Thetapicc,i} \, , T^-_{\Thetapicc,i} \big\}}_{i \in I} \,\big) \, $.

\vskip11pt

\begin{theorem}  \label{thm: twist-liegRP=new-liegR'P'}
 There exists a Lie bialgebra isomorphism
  $ \; f^\Thetapicc_{\scriptscriptstyle P} : \lieg_{P_\Theta}^{\R_\Thetapicc} \, {\buildrel \cong \over
 {\lhook\joinrel\relbar\joinrel\relbar\joinrel\relbar\joinrel\twoheadrightarrow}} \,
{\big( \liegRP \big)}^{j_\Thetapicc} \; $
 (notation as above, with in right-hand side the twist deformation
 $ \, {\big( \liegRP \big)}^{j_\Thetapicc} $  of  $ \, \liegRP $  by  $ j_\Thetapicc $  occurs) given by
  $ \; E_i \, \mapsto \, E_i \, , \; T \, \mapsto \, T \, $  and  $ \; F_i \, \mapsto \, F_i \; $
 for all  $ \, i \in I \, $,  $ \, T \in \lieh \, $.
 \vskip3pt
   In particular, the class of all MpLbA's of any fixed Cartan type and of fixed rank is stable by toral twist deformations.
   Moreover, inside it the subclass of all such MpLbA's associated with  \textsl{straight},  resp.\  \textsl{small},
   realizations is stable as well.
\end{theorem}

\pf
 By  \S \ref{MpLieBialgebras-double},  the Lie algebra structure in  $ {\big( \liegRP \big)}^{j_\Thetapicc} $
 is the same as in  $ \liegRP \, $,  and the latter only depends on the $ \alpha_j $'s  and the sums
 $ \, S_j := 2^{-1}\big(T_i^+ \! + T_i^-\big) \, $  ($ \, j \in I \, $).  Now, both the  $ \alpha_j $'s  and the  $ S_j $'s
 \textsl{do not change\/}  (see above)  when we pass from  $ \liegRP $  to  $ \lieg_{P_\Theta}^{\R_\Thetapicc} $
 or viceversa; therefore, the formulas in the claim  (mapping each generator of  $ \lieg_{P_\Theta}^{\R_\Thetapicc} $
 onto the same name generator of  $ \, \liegRP = {\big( \liegRP \big)}^{j_\Thetapicc} \, $)  provide an isomorphism of Lie algebras.
                                                                  \par
   Now consider the toral twist
  $ \; j_\Thetapicc := {\textstyle \sum_{g,k=1}^t} \theta_{gk} \, H_g \otimes H_k \; $
 given in  \eqref{eq: def_twist_Lie-bialg}.  By  \eqref{eq: def_twist-delta}
  $$  \delta^{j_\Thetapicc}(x)  \; := \;  \delta(x) - \ad_x\!\big(j_\Thetapicc\big)  \; = \;
  \delta(x) - {\textstyle \sum_{g,k=1}^t} \theta_{gk} \,
  \big( \big[x,H_g\big] \otimes H_k + H_g \otimes \big[x,H_k\big] \big)  $$
 for all  $ \, x \in \lieg \, $.  Now take  $ \, x := E_\ell \, $  ($ \, \ell \in I \, $):  \,then the previous formula reads
  \eject

  $$  \displaylines{
   \delta^{j_\Thetapicc}(E_\ell)  \; := \;  \delta(E_\ell) - {\textstyle \sum_{g,k=1}^t} \, \theta_{gk} \,
   \big( \big[E_\ell \, , H_g\big] \otimes H_k + H_g \otimes \big[E_\ell \, , H_k\big] \big)  \; =   \hfill  \cr
   \quad   = \;  \delta(E_\ell) - {\textstyle \sum_{g,k=1}^t} \, \theta_{gk} \,
   \big(\! -\alpha_\ell(H_g) E_\ell \otimes H_k - H_g \otimes \alpha_\ell(H_k) E_\ell \,\big)  \; =   \hfill  \cr
   = \;  T^+_\ell \otimes E_\ell - E_\ell \otimes T^+_\ell + {\textstyle \sum_{g,k=1}^t} \, \theta_{gk} \,
   \big(\, \alpha_\ell(H_g) E_\ell \otimes H_k + \alpha_\ell(H_k) H_g \otimes E_\ell \,\big)  \; =   \quad  \cr
   \hfill   = \;  \Big(\, T^+_\ell + {\textstyle \sum_{g,k=1}^t} \, \theta_{kg} \, \alpha_\ell(H_g) H_k \Big) \otimes E_\ell \, -
   \, E_\ell \otimes \Big(\, T^+_\ell + {\textstyle \sum_{g,k=1}^t} \, \theta_{kg} \, \alpha_\ell(H_g) H_k \Big)  \; =  \cr
   \hfill   = \;  T^+_{\Thetapicc,\ell} \otimes E_\ell \, - \, E_\ell \otimes T^+_{\Thetapicc,\ell}  \; = \; 2 \;
   T^+_{\Thetapicc,\ell} \wedge \! E_\ell  }  $$
 hence in short we get  $ \; \delta^{j_\Thetapicc}(E_\ell) \, = \, 2 \; T^+_{\Thetapicc,\ell} \wedge \! E_\ell  \; $.
 Similar computations give
$$
   \delta^{j_\Thetapicc}(E_i) = 2 \, T^+_{\Thetapicc,i} \wedge \! E_i  \; ,   \!\quad
 \delta^{j_\Thetapicc}(T) = 0  \; ,   \!\quad
 \delta^{j_\Thetapicc}(F_i) = 2 \, T^-_{\Thetapicc,i} \wedge \! F_i  \; ,   \quad  \forall \; i \in I \, , \, T \!\in \lieh
$$
 This means that, through the Lie algebra isomorphism  $ f_\Ppicc^\Thetapicc \, $,
 the Lie coalgebra structure of  $ {\big( \liegRP \big)}^{j_\Thetapicc} $  is described on
 generators exactly like that of  $ \lieg_{P_\Theta}^{\R_\Thetapicc} \, $,  with the new ``coroots''
 $ T^\pm_{\Thetapicc,i} $  ($ \, i \in I \, $)  in  $ {\big( \liegRP \big)}^{j_\Thetapicc} $  playing the role of the coroots
 $ T^\pm_i $  ($ \, i \in I \, $)  in  $ \lieg_{P_\Theta}^{\R_\Thetapicc} \, $.  Therefore
 $ \; f^\Thetapicc_{\scriptscriptstyle P} : \lieg_{P_\Theta}^{\R_\Thetapicc}
 \relbar\joinrel\relbar\joinrel\longrightarrow {\big( \liegRP \big)}^{j_\Thetapicc} \; $
 is an isomorphism of Lie bialgebras, q.e.d.
\epf

\vskip7pt

   In fact,  \textsl{the previous result can be reversed},  somehow.
   Namely, our next result shows, in particular, that \textit{every straight small MpLbA
   can be realized as a toral twist deformation of the ``standard'' MpLbA  $ \lieg^\Dpicc_{\scriptscriptstyle DA} $}
   (as in  \S \ref{MpLieBialgebras-double})   ---  cf.\ claim  \textit{(c)\/}  below.

\vskip9pt

\begin{theorem}  \label{thm: MpLbA=twist-gD}
 With assumptions as above, let  $ P $  and  $ P' $  be two matrices of Cartan type with the
 same associated Cartan matrix $ A \, $,  \,i.e.\ such that  $ \, P_s = P'_s \, $.
 \vskip3pt
   (a)\,  Let  $ \, \R $  be a  \textsl{straight}  realization of  $ P $  and let  $ \lieg_\Ppicc^{\scriptscriptstyle \R} $
   be the associated MpLbA.  Then there exists a matrix  $ \, \Theta \in \lieso_t(\Bbbk) \, $
   such that  $ \, P' = P_\Thetapicc \, $,  the corresponding $ \R_\Thetapicc $
   is a  \textsl{straight}  realization of  $ \, P' = P_\Thetapicc \, $,  and for the twist element
   $ j_\Thetapicc $  as in  \eqref{eq: def_twist_Lie-bialg}  we have
  $$  \lieg_{\scriptscriptstyle P'}^{\scriptscriptstyle \R_\Thetapicc}  \; \cong \;
  {\big(\, \lieg_\Ppicc^{\scriptscriptstyle \R} \big)}^{j_\Theta}  $$
   \indent   In a nutshell, if  $ \, P'_s = P_s \, $  then from any straight MpLbA over  $ P $
   we can obtain by toral twist deformation a straight MpLbA (of the same rank) over  $ P' \, $.
 \vskip3pt
   (b)\,  Let  $ \R $  and  $ \R' $  be  \textsl{straight small}  realizations of  $ P $  and  $ P' $
respectively, with  $ \, \rk(\R) = \rk(\R') =: t \, $,  and let  $ \liegRP $  and
$ \lieg_{\scriptscriptstyle P'}^{\scriptscriptstyle \R'} $  be the associated MpLbA's.
Then there exists a matrix  $ \, \Theta \in \lieso_t(\Bbbk) \, $  such that for the twist element
$ j_\Thetapicc $  as in  \eqref{eq: def_twist_Lie-bialg}  we have
  $$  \lieg_{\scriptscriptstyle P'}^{\scriptscriptstyle \R'}  \; \cong \;  {\big(\, \liegRP \big)}^{j_\Theta}  $$
   \indent   In a nutshell, if  $ \, P'_s = P_s \, $  then any straight small MpLbA over
   $ P' $  is isomorphic to a toral twist deformation of any straight small MpLbA over  $ P $  of the same rank.
 \vskip3pt
   (c)\,  Every straight small MpLbA is isomorphic to some toral twist deformation of the
   \textsl{standard}  MpLbA  $ \lieg^\Dpicc_{\scriptscriptstyle DA} $  (over  $ \, DA = P_s \, $,  cf.\
   \S \ref{MpLieBialgebras-double})  of the same rank.
\end{theorem}

\pf
 \textit{(a)}\,  By  Theorem \ref{thm: twist-liegRP=new-liegR'P'}  it is enough to find
 $ \, \Theta \in \lieso_t(\Bbbk)  \, $  such that  $ \, P' = P_\Thetapicc \, $,  \,that is
 $ \; P' = P - \mathfrak{A} \, \Theta \, \mathfrak{A}^{\,\scriptscriptstyle T} \, $;
 \,but this is guaranteed by  Lemma \ref{lemma: twist=sym},  so we are done.
 \vskip3pt
   \textit{(b)}\,  This follows from claim  \textit{(a)},  along with the uniqueness of straight small realizations
   --- cf.\ Proposition \ref{prop: exist-realiz's}\textit{(b)}  ---   and  Proposition \ref{prop: functor_R->liegRP}.
 \vskip3pt
   \textit{(c)}\,  This follows from  \textit{(b)},  taking as  $ \lieg_{\scriptscriptstyle P'}^{\scriptscriptstyle \R'} $
   the given straight small MpLbA and as  $ \liegRP $  the ``standard'' MpLbA  $ \lieg^\Dpicc_\Ppicc $  over  $ \, P := DA = P'_s \, $.
\epf

\vskip7pt

\begin{obs's}  \label{obs: parameter_(P,Phi) x MpLbA's}  {\ }
 \vskip2pt
   \textit{(a)}\,  Theorems \ref{thm: twist-liegRP=new-liegR'P'}  and  \ref{thm: MpLbA=twist-gD}  have the following interpretation.
   Our MpLbA's  $ \, \liegRP \, $  depend on the multiparameter  $ P \, $;  \,at a further level, once we perform onto
   $ \, \liegRP \, $  a deformation by toral twist the outcome
   $ \; \lieg^{\scriptscriptstyle \R}_{\Ppicc,\Thetapicc} := {\big(\, \liegRP \big)}^{j_\Theta} \; $
   depends on  \textsl{two\/}  multipara\-meters, namely  $ P $  \textsl{and\/}  $ \Theta \, $.
   Thus all these  $ \lieg^{\scriptscriptstyle \R}_{\Ppicc,\Thetapicc} $'s form a seemingly richer family of
   ``twice-multiparametric'' MpLbA's.  Nonetheless,  Theorem \ref{thm: twist-liegRP=new-liegR'P'}
   above proves that this  \textsl{coincides\/}  with the family of all MpLbA's, though the latter seems smaller.
                                                              \par
%
%
   In short,  Theorems \ref{thm: twist-liegRP=new-liegR'P'}  and  \ref{thm: MpLbA=twist-gD}  show the following.  The dependence of the Lie bialgebra structure of  $ \lieg^{\scriptscriptstyle \R}_{\Ppicc,\Thetapicc} $  on the ``double parameter''  $ (P\,,\Theta) $  is ``split'' in the algebraic structure   --- ruled by  $ P $  ---   and in the coalgebraic structure   --- ruled by  $ \Theta \, $.  Now,  Theorems \ref{thm: twist-liegRP=new-liegR'P'}
   and  \ref{thm: MpLbA=twist-gD}  enable us to ``polarize'' this dependence so to codify it either entirely within the algebraic structure (while the coalgebraic one is reduced to a ``canonical form'')   --- so the single multiparameter  $ P_\Theta $  is enough to describe it ---   or entirely within the coalgebraic structure (with the algebraic one being reduced to the ``standard'' one)   --- so the one multiparameter  $ \Theta_P $  is enough to describe it.
 \vskip2pt
   \textit{(b)}\,  As we saw at the end of  \S \ref{further_stability},  the (sub)class of  \textsl{split\/}
   realizations is  \textsl{not closed\/}  under twist deformation; as a consequence,
   the subclass of all MpLbA's that are ``split'' is not closed either under twist deformation.
\end{obs's}

\medskip

\subsection{Deformations of MpLbA's by toral  $ 2 $--cocycles}  \label{subsec: tor-2cocyc def's_mp-Lie_bialg's}
 Let again  $ \liegRP $  be a MpLbA as in  \S \ref{MpLieBialgebras-double},  and keep notation as above.
 Dually to what we did before, we consider now the so-called ``toral'' 2--cocycles for  $ \liegRP \, $.

\vskip9pt

\begin{definition}  \label{def: toral_2-cocyc x MpLbA's}
 Fix an antisymmetric  $ \Bbbk $--linear  map  $ \, \chi : \lieh \otimes \lieh \relbar\joinrel\longrightarrow \Bbbk \, $
 such that
\begin{equation}  \label{eq: condition-eta}
  \chi(S_i \, ,\,-\,)  \, = \,  0  \, = \,  \chi(\,-\,,S_i)   \qquad \qquad \forall \;\; i \in I
\end{equation}
 where  $ \, S_i := 2^{-1} \big(\, T^+_i + T^-_i \big) \, $  for all  $ \, i \in I \, $  (cf.\ Definition \ref{def: realization of P});
 in other words, we have  $ \, \chi \in \text{\sl Alt}_{\,\kh}^{\,S}(\lieh) \, $,  cf.\ \eqref{eq: def-Alt_S}.
 Moreover, let
 $ \,\; \pi^{\liegRP}_\lieh : \liegRP \relbar\joinrel\twoheadrightarrow \lieh \;\, $
 be the canonical projection induced by any one of the triangular decompositions in
 \eqref{eq: triang-decomp's_Lie-bialg's}.  We define
$$
  \chi_\lieg := \chi  \circ \Big(\, \pi^{\liegRP}_\lieh \otimes \pi^{\liegRP}_\lieh \Big) :
  \liegRP \otimes \liegRP \relbar\joinrel\relbar\joinrel\twoheadrightarrow \lieh \otimes \lieh
  \relbar\joinrel\longrightarrow \Bbbk   \qquad
$$
 and we call it  \textit{the toral 2--cocycle\/}  (or  \textit{``2--cocycle  of toral type''\/})
 \hbox{associated with  $ \, \chi \, $.   \hfill  $ \diamondsuit $}
\end{definition}

\vskip9pt

   Next result follows at once by construction, and explains our use of terminology:

\vskip11pt

\begin{lema}  \label{lemma: toral 2-cocyc x MpLbA's}
 For any  antisymmetric
 $ \Bbbk $--linear  map  $ \, \chi : \lieh \otimes \lieh \longrightarrow \Bbbk \, $  obeying  \eqref{eq: condition-eta},
  $$  \chi_\lieg := \chi \circ \Big( \pi^{\liegRP}_\lieh \otimes \pi^{\liegRP}_\lieh \,\Big) : \liegRP
  \otimes \liegRP \relbar\joinrel\relbar\joinrel\twoheadrightarrow \lieh \otimes \lieh \relbar\joinrel\longrightarrow \Bbbk  $$
 is a  \textsl{2--cocycle\/}  map for the Lie bialgebra  $ \liegRP \, $,  \,in the sense of  \eqref{eq: cocyc-cond_Lie-bialg}.
\end{lema}

\pf

 We have to check that  $ \, \chi_\lieg \in \Hom_\Bbbk\big(\, \liegRP \otimes \liegRP \, , \Bbbk \,\big) \, $
 satisfies  \eqref{eq: cocyc-cond_Lie-bialg},  that is
  $$  \begin{aligned}
   \ad_\psi\!\big(\, \partial_*(\chi_\lieg) - {[\![\, \chi_\lieg \, , \chi_\lieg \,]\!]}_* \,\big)  \; = \; 0  \;\; ,
   \;\quad  \ad_\psi\!\big(\, \chi_\lieg + {(\chi_\lieg)}_{\!{}_{2,1}} \big) \, = \, 0
   \quad \qquad   \forall \;\; \psi \in (\liegRP)^*
      \end{aligned}  $$
 Since  $ \, \chi \, $ is antisymmetric, we have that  $ \, \chi_\lieg + {(\chi_\lieg)}_{\!{}_{2,1}} = 0 \, $,
 \,hence the second condition is trivially satisfied.  On the other hand, the first equality follows from the
 fact that actually one has  $ \; \partial_*(\chi_\lieg) = 0 \; $  and  $ \; {[\![\, \chi_\lieg \, , \chi_\lieg \,]\!]}_* = 0 \; $.
 To see it, let us describe  $ \, \partial_*(\chi_\lieg) \, $,
 $ \, {[\![\, \chi_\lieg \, , \chi_\lieg \,]\!]}_* \in {\big(\liegRP \otimes \liegRP \otimes \liegRP)}^* \, $
 explicitly.  For  $ \, x, y, z \in \liegRP \, $  we have
  \eject

\begin{align*}
  \partial_*(\chi_\lieg)(x,y,z)  &  \; = \;  \big((\id \otimes \delta_*)(\chi_g) + \text{c.p.} \big)(x,y,z)  \; = \;
  \chi_\lieg\big( x \otimes [y,z] \big) + \text{c.p.}  \; =  \\
  &  \; = \;  \chi_\lieg\big( x , [y,z] \big) + \chi_\lieg\big( z , [x,y] \big) + \chi_\lieg\big( y , [z,x] \big)  \\
  {[\![\, \chi_\lieg \, , \chi_\lieg \,]\!]}_*(x,y,z)  &  \, = \,
  \chi_\lieg(x_{[1]},y) \chi_\lieg(x_{[2]},z) + \chi_\lieg(y_{[1]},z) \chi_\lieg(y_{[2]},x) + \chi_\lieg(z_{[1]},x) \chi_\lieg(z_{[2]},y)
\end{align*}
Using  \eqref{eq: triang-decomp's_Lie-bialg's},
 for  $ \, g \in \liegRP \, $  write  $ \, g = g_+ + g_0 + g_- \, $  with  $ \, g_\pm \in \lien_\pm \, $
 and  $ \, g_0 \in \lieh \, $.  As  $ \chi_g $  is defined through the  map
  $ \pi^{\liegRP}_\lieh \, $,  \,the map  $ \partial_*(\chi_\lieg) $
 vanishes when evaluated at elements in  $ \lien_+ $  or  $ \lien_- \, $.  Moreover, since the bracket on  $ \lieh $
 is trivial and  $ \, [\, \lieh , \lien_\pm \,] \subseteq \lien_\pm \, $, we get
  $$  \partial_*(\chi_{\lieg})(x,y,z)  \; = \;  \chi\big( x_0 \, , \, [\,y_+,z_-] + [\,y_-,z_+] \big) \, + \, \text{c.p.}  $$
 But  $ \, [\lien_+,\lien_-] \, $  is contained in the Lie subalgebra spanned by the  $ S_i $'s  ($ \, i \in I \, $),
 so we eventually get  $ \; \partial_*(\chi_\lieg)(x,y,z) = 0 \; $   --- for all  $ \, x, y, z \in \liegRP \, $  ---
 by condition  \eqref{eq: condition-eta}.
                                                    \par
   To prove that  $ \, {[\![\, \chi_\lieg \, , \chi_\lieg \,]\!]}_*(x,y,z) = 0 \, $,
   \,we use the Lie coalgebra structure.
   From  \eqref{eq:compat-bracket-cobracket}  it follows that  $ \, \delta(\lien_\pm) \subseteq \liebP_\pm \otimes \lien_\pm \, $.
   Since  $ \, \delta(\lieh) = 0 \, $  by definition and  $ \, \partial_*(\chi_\lieg) \, $  vanishes on  $ \lien_\pm \, $,
   \,we get that each summand of  $ \, {[\![\, \chi_\lieg \, , \chi_\lieg \,]\!]}_*(x,y,z) \, $  is zero,  \,q.e.d.
\epf

\vskip7pt

   The second result is the dual analog of  Theorem \ref{thm: twist-liegRP=new-liegR'P'}
   and of  Theorem \ref{thm: MpLbA=twist-gD}.

\vskip7pt

   We start with some preliminaries.  Let  $ \, P\in M_t(\Bbbk) \, $
   be a multiparameter matrix of Cartan type with associated Cartan matrix  $ A \, $,
   let  $ \R $  be a realization of it, and let  $ \liegRP $  be the associated multiparameter Lie bialgebra;
   then, given any  $ \, \chi \in \text{\sl Alt}_{\,\kh}^{\,S}(\lieh) \, $  as in  \eqref{eq: def-Alt_S},  let
   $ \, \chi_\lieg : \liegRP \!\otimes \liegRP \longrightarrow \Bbbk \, $  be the  $ 2 $--cocycle  map for
   $ \liegRP $  as in Lemma \ref{lemma: toral 2-cocyc x MpLbA's}.

\vskip5pt

   Consider the antisymmetric matrix
   $ \, \mathring{X} := {\Big(\, \mathring{\chi}_{i{}j} = \chi\big(\,T_i^+\,,T_j^+\big) \Big)}_{i, j \in I} \! \in \lieso_n(\Bbbk) \, $.
   By  Proposition \ref{prop: 2cocdef-realiz}  we have a matrix  $ P_{(\chi)} $  and a realization  $ \R_{(\chi)} $  of it,
   given by
  $$
  P_{(\chi)}  := \,  P \, + \, \mathring{X}  \, = \,  {\Big(\, p^{(\chi)}_{i{}j} := \,  p_{ij} + \mathring{\chi}_{i{}j} \Big)}_{\! i, j \in I}  \;\; ,  \quad
   \Pi_{(\chi)}  := \,  {\Big\{\, \alpha_i^{(\chi)}  := \,
   \alpha_i \pm \chi\big(\, \text{--} \, , T_i^\pm \big)  \Big\}}_{i \in I}
   $$
 In particular, if  $ P $  is of Cartan type, then so is  $ P_{(\chi)} $  as well, and they are associated with the
 same Cartan matrix.

\vskip13pt

\begin{theorem}  \label{thm: 2-cocycle-def-MpLbA}
 Keep notation as above.
 \vskip2pt
   (a)\,  There exists a Lie bialgebra isomorphism
  $ \; f^{(\chi)}_{\scriptscriptstyle P} : \lieg_{P_{(\chi)}}^{\R_{(\chi)}} {\buildrel \cong \over
{\lhook\joinrel\relbar\joinrel\relbar\joinrel\relbar\joinrel\twoheadrightarrow}} {\big( \liegRP \big)}_{\chi_\lieg} \; $
 --- where  $ \, {\big( \liegRP \big)}_{\chi_\lieg} = \big(\, \liegRP \, ; \, {[\,\ ,\ ]}_{\chi_\lieg} \, , \, \delta \,\big) \, $
 is the 2--cocycle deformation of  $ \, \liegRP $  as in Definition \ref{def: cocyc-deform_Lie-bialg's}.
 \vskip3pt
   In particular, the class of all MpLbA's of any fixed Cartan type and of fixed rank is stable by toral
   2--cocycle deformations.  Moreover, inside it the subclass of all such MpLbA's associated with
   \textsl{split},  resp.\  \textsl{minimal},  realizations is stable as well.
 \vskip4pt
   (b)\,  Let  $ P $  and  $ P' $  be two matrices of Cartan type with the same associated Cartan matrix $ A \, $,
   \,i.e.\ such that  $ \, P_s = P'_s \, $.  Then the following holds:
 \vskip3pt
   \, (b.1)\,  Let  $ \, \R $  be a  \textsl{split}  realization of  $ P $  and let  $ \lieg_\Ppicc^{\scriptscriptstyle \R} $
   be the associated MpLbA.  Then there is  $ \, \chi \in \text{\sl Alt}_{\,\Bbbk}^{\,S}(\lieh) \, $ such that
   $ \, P' = P_{(\chi)} \, $,  the corresponding $ \R_{(\chi)} $  is a  \textsl{split}  realization of  $ \, P' = P_{(\chi)} \, $,
   and for the 2--cocycle  $ \chi_\lieg $  as in Definition \ref{def: toral_2-cocyc x MpLbA's}  we have
  $$  \lieg_{\scriptscriptstyle P'}^{\scriptscriptstyle \R_{(\chi)}}  \; \cong \;
  {\big(\, \lieg_\Ppicc^{\scriptscriptstyle \R} \big)}_{\chi_{{}_\lieg}}  $$
   \indent   In a nutshell, if  $ \, P'_s = P_s \, $  then from any split MpLbA over  $ P $
   we can obtain by toral 2--cocycle deformation a split MpLbA (of the same rank) over  $ P' \, $.
 \vskip3pt
   \, (b.2)\,  Let  $ \R $  and  $ \R' $  be  \textsl{split minimal}  realizations of  $ P $  and  $ P' $
respectively, and let  $ \liegRP $  and  $ \lieg_{\scriptscriptstyle P'}^{\scriptscriptstyle \R'} $
be the associated MpLbA's.  Then there exists  $ \, \chi \in \text{\sl Alt}_{\,\Bbbk}^{\,S}(\lieh) \, $
such that for the 2--cocycle  $ \chi_\lieg $  as in Definition \ref{def: toral_2-cocyc x MpLbA's}  we have
  $$  \lieg_{\scriptscriptstyle P'}^{\scriptscriptstyle \R'}  \; \cong \;  {\big(\, \liegRP \big)}_{\chi_{\lieg}}  $$
   \indent   In a nutshell, if  $ \, P'_s = P_s \, $  then any split minimal MpLbA over  $ P' $
   is isomorphic to a toral 2--cocycle deformation of any split minimal MpLbA over  $ P $.
 \vskip3pt
   (b.3)\,  Every split minimal MpLbA is isomorphic to some toral 2--cocycle deformation
   of the Manin double  $ \, \lieg^\Dpicc_{\scriptscriptstyle DA} := \lieb_+ \oplus \lieb_- \, $
   associated with $ \, DA \, $  endowed with the canonical Lie bialgebra structure given in  \S \ref{MpLieBialgebras-double}.
%
%
\end{theorem}

\smallskip

\pf
 \textit{(a)}\,  By  \S \ref{deformations of LbA's},  the Lie coalgebra structure in
 $ {\big( \liegRP \big)}_{\chi_\lieg} $  is the same as in  $ \liegRP \, $,  and the latter coincides with the one in
 $ \lieg_{P_{(\chi)}}^{\R_{(\chi)}} $.  In particular, we have a Lie coalgebra isomorphism among them.
 With respect to the Lie algebra structure, we know that the Lie bracket in  $ {\big( \liegRP \big)}_{\chi_\lieg} $
 is a deformation of that of  $ \liegRP $  according to  \eqref{eq: def_cocyc-bracket}.
 Let us see that the modified defining relations in  $ {\big( \liegRP \big)}_{\chi_\lieg} $  coincide with the ones in
 $ \lieg_{P_{(\chi)}}^{\R_{(\chi)}} \, $,  cf.\  \S \ref{MpLieBialgebras-double};
 this will imply that both objects are isomorphic as Lie bialgebras.
                                                                           \par
 As the Lie cobracket on the Cartan subalgebra  $ \lieh $  of  $ \lieg_{P_{(\chi)}}^{\R_{(\chi)}} $  is trivial, we get that
  $$  {[T',T'']}_{\chi_{{}_\lieg}}  \, = \;  [T',T''] \, - \, T'_{[1]} \, \chi_\lieg\big(T'_{[2]},T''\big) \, +
  \, T''_{[1]} \, \chi_\lieg\big(T''_{[2]},T'\big)  \, = \,  [T',T'']  \, = \,  0  $$
 for all  $ \, T', T'' \in \lieh \, $.  Take now  $ \, i \in I \, $:  since
 $ \, \delta(T) = 0 \, $,  $ \, \delta(E_i) = 2 \, T_i^+ \!\wedge E_i \, $  and
 $ \, \delta(F_i) = 2 \, T_i^- \!\wedge F_i \, $,  \,we get that
\begin{align*}
   {[T,E_i]}_{\chi_\lieg}  &  \, = \,  [T,E_i] - E_i \, \chi_\lieg\big(T_i^+,T\big)  \, = \,
   \alpha_i(T) E_i + \chi\big(T, T_i^+\big) E_i  \, = \, \alpha_i^{(\chi)}(T)  \, E_i  \\
   {[T,F_i]}_{\chi_\lieg}  &  \, = \,  [T,F_i] - F_i \, \chi_\lieg\big(T_i^-,T\big)  \, = \,
   -\alpha_i(T) F_i + \chi_{\lieg}\big(T, T_i^-\big) \, F_i  \, =  \\
   &  = \,  -\Big(\alpha_i(T) F_i - \chi\big(T,T_i^-\big)\Big) \, F_i  \, = \,  -\alpha_i^{(\chi)}(T)
 \, F_i  \\
   {[E_i,F_j]}_{\chi_\lieg}  &  \, = \,  [E_i,F_j] + E_i \, \chi_\lieg\big(T_i^+,F_j\big) \, - \, F_j \, \chi_\lieg\big(T_j^-,E_i\big)  \, = \,  [E_i,F_j]  \, = \,
   \delta_{ij} \, {{\;T_i^+ \! + T_i^-\;} \over {\;2\,d_i\;}}
\end{align*}
 Note that only the relations involving the roots are changed by the cocycle.
 Now, with respect to the Serre relations, as the Lie subalgebras  $ \lien_\pm $
 are contained in the right and left radical of  $ \chi_\lieg \, $,
 analogous calculations as above yield that
\begin{align*}
   {\big(\textsl{ad}_{\chi_\lieg}\,(E_i)\big)}^{1-a_{ij}}(E_j)  &  \, = \,  {\big(\textsl{ad}\,(E_i)\big)}^{1-a_{ij}}(E_j)  \, = \,  0  \\
   {\big(\textsl{ad}_{\chi_\lieg}\,(F_i)\big)}^{1-a_{ij}}(F_j)  &  \, = \,  {\big(\textsl{ad}\,(F_i)\big)}^{1-a_{ij}}(F_j)\,  = \,  0
\end{align*}
Here  $ \textsl{ad}_{\chi_\lieg} $  denotes the adjoint action with respect to the deformed bracket
$ \, {[-,-]}_{\chi_\lieg} \, $.
 \vskip7pt
   \textit{(b.1)}\, By claim  \textit{(a)},  it is enough to find an antisymmetric  $ \Bbbk $--linear
  $ \, \chi : \lieh \otimes \lieh \longrightarrow \Bbbk \, $  obeying  \eqref{eq: condition-eta} such that
  $ \, P' = P_{(\chi)} \, $;
  this is guaranteed by  Proposition \ref{prop: mutual-2-cocycle-def}.
 \vskip5pt
   \textit{(b.2)}\,  This follows from claim  \textit{(b.1)},  along with the uniqueness of split realizations
   --- cf.\ Proposition \ref{prop: exist-realiz's}\textit{(b)}.
 \vskip5pt
   \textit{(b.3)}\,  This follows as an application of claim  \textit{(b.2)},  taking as
   $ \lieg_{\scriptscriptstyle P'}^{\scriptscriptstyle \R'} $
   the given split minimal MpLbA and for  $ \liegRP $  the ``standard''  MpLbA  $ \lieg^\Dpicc_\Ppicc $  over
   $ \, P := DA = P'_s \, $
   (cf.\ \S \ref{MpLieBialgebras-double}),  which by definition is straight and split minimal.
\epf

\bigskip

\section{Formal multiparameter QUEAs (=FoMpQUEAs)}  \label{sec: form-MpQUEAs}

\vskip7pt

%
%
%
%

   This section is devoted to introduce formal multiparameter quantized universal enveloping algebras
   (in short, FoMpQUEAs) and to study their deformations.
 \vskip15pt

\subsection{The Hopf algebra setup}  \label{sec: Hopf-setup}  {\ }
 \vskip7pt
   Our main references for the theory of Hopf algebras are  \cite{Mo} and  \cite{Ra}.
   For topological or  $ \hbar $--adically  complete Hopf algebras see for instance  \cite{Ks},  \cite{CP},  \cite{KS}.

\vskip11pt


\begin{free text}  \label{Hopf notation}

 {\bf Hopf notation.}
 Let us fix our notation for Hopf algebra theory (mainly standard, indeed).
 The comultiplication is denoted  $ \com $, the counit  $\epsilon  \, $ and the antipode
 $ \SS \, $;  for the first, we use the Heyneman-Sweedler notation, namely  $ \, \com(x) = x_{(1)} \otimes x_{(2)} \, $.
                                                              \par
   Hereafter by  $ \k $  we denote the ground ring of our algebras, coalgebras, etc.  In any coalgebra  $ C \, $,
   the set of group-like elements is denoted by  $ G(C) \, $;  also, we denote by  $ \, C^+ := \Ker(\epsilon) \, $
   the augmentation ideal.  If  $ \, g, h \in G(C) \, $,  the set of  $ (g,h) $--primitive elements is defined to be
 $ \; P_{g,h}(C)  \, := \,  \big\{\, x \in C \,\vert\, \com(x) = x \ot g + h \ot x \,\big\} \; $.
In case $C$ is a bialgebra, we write $\Prim(C) = P_{1,1}(C)$ for the space of primitive elements.
\par
   For a Hopf algebra $ H $ (or just bialgebra),  we write  $ H^{\text{op}} \, $,  resp.\  $ H^{\text{cop}} \, $,
   for the Hopf algebra (or bialgebra) given by taking in  $ H $  the opposite product, resp.\ coproduct.
                                                              \par
   Given a Hopf algebra map  $ \, \pi : H \longrightarrow K \, $,  then  $ H $  is a left and right  $ K $--comodule,
 with structure maps
$ \; \lambda := (\pi \otimes \id) \com : H \longrightarrow K \otimes H \; $,
  $ \; \rho := (\id \otimes \pi) \com : H \longrightarrow H \otimes K \; $.
The space of left and right  \textit{coinvariants\/}  then is defined, respectively, by
\begin{align*}
   {}^{\co\!K}\!H  &  \; := \;  {}^{\co \pi}H  \, = \,  \big\{\, h \in H \,\big|\, (\pi \otimes \id) \big(\com(h)\big) = 1 \otimes h \,\big\}  \\
   H^{\,\co\!K}  &  \; := \;  H^{\co \pi}  \, = \,  \big\{\, h \in H \,\big|\, (\id \otimes \pi) \big(\com(h)\big) = h \otimes 1 \,\big\}
\end{align*}

\vskip9pt

   We recall the (essentially standard) notion of  \textit{skew-Hopf pairing\/}  between two Hopf algebras
   and the construction of the Drinfeld's double.
\end{free text}

\vskip9pt

\begin{definition}  \label{def_(skew-)Hopf-pairing}\cite[\S 2.1]{AY}
 Given two Hopf algebras  $ H $  and  $ K $  with bijective antipode over the ring  $ \k \, $,  a  $ \k $--linear  map
 $ \, \eta : H \otimes_\k K \longrightarrow \k \, $  is called a {\sl skew-Hopf pairing\/}  (between  $ H $  and  $ K \, $)
 if, for all
 $ \, h \in H \, $,  $ \, k \in K \, $,  one has
 \begin{eqnarray}
   &  \eta\big( h \, , \, k' \, k'' \big)  \, = \;  \eta\big( h_{(1)} \, , \, k' \big) \, \eta\big( h_{(2)} \, , \, k'' \big)   \label{eqn:skew-Hpair-1}  \\
   &  \eta\big( h' \, h'' \, , \, k \big)  \, = \;  \eta\big( h' , \, k_{(2)} \big) \, \eta\big( h'' , \, k_{(1)} \big)   \label{eqn:skew-Hpair-2}  \\
   &  \eta\big( h \, , 1 \big) \, = \, \epsilon(h)  \;\; ,   \qquad   \eta\big( 1 \, , k \big) \, = \, \epsilon(k)   \label{eqn:skew-Hpair-3}  \\
   &  \eta\big( \SS^{\pm 1}(h) \, , k \,\big) \,  =  \; \eta\big( h \, , \, \SS^{\mp 1}(k) \big)   \label{eqn:skew-Hpair-4}
  \end{eqnarray}
 \vskip1pt
  \textsl{Note\/}  that the map  $ \eta $  turns out to be convolution invertible: its inverse is
given by  $ \; \eta^{-1}(h,k) = \eta(h,\SS(k)) = \eta(\SS^{-1}(h),k) \; $  for all  $ \, h \in H \, $  and  $ \, k \in K \, $.

\vskip5pt

  In this setup, the  {\it Drinfeld double},  or  \textit{``quantum double''},  $ D(H,K,\eta) \, $
  is the quotient algebra  $ \; T(H \oplus K) \Big/ \I \; $,  \,where  $ \I $  is the (two-sided) ideal generated by the relations
  $$  1_{H}  = \,  1  \, = \,  1_K  \,\; ,  \;\quad  a \otimes b  \, = \, a\,b  \,\; ,  \;\quad
   x_{(1)} \otimes y_{(1)} \ \eta(y_{(2)},x_{(2)})  \, = \,  \eta(y_{(1)},x_{(1)})\ y_{(2)} \otimes x_{(2)}  $$
 for all  $ \, a \, , b \in H \, $  or  $ \, a \, , b \in K \, $  and  $ \, x \in K \, , y \in H \, $.
 This is also endowed with a standard Hopf algebra structure, for which  $ H $  and  $ K $  are Hopf  $ \k $--subalgebras.
 \hfill  $ \diamondsuit $
\end{definition}

\vskip9pt

\begin{free text}  \label{Topological issues}
{\bf Topological issues.}  We will often deal
 with  \textit{topological\/}  Hopf algebras, namely Hopf algebras over the ring  $ \kh $
 of formal power series over a field  $ \Bbbk $  in a formal variable  $ \hbar \, $.
 This ring carries a natural topology, called the  \textit{$ \hbar $--adic  topology},
 coming from the so-called  \textit{$ \hbar $--adic  norm\/}  with respect to which it is complete, namely
  $$  \big\|\, a_n \hbar^n + a_{n+1} \hbar^{n+1} + \cdots \big\|  \; := \;  C^{-n}   \qquad  \big(\, a_n \not= 0 \,\big)  $$
 where $ \, C > 1 \, $  is any fixed constant in  $ \mathbb{R} \, $.
 In this sense, we shall consider \textit{topological\/}  $ \kh $--modules  and the
 \textit{completed\/}  tensor products among them, which we denote by
 $ \, \widehat{\otimes}_\kh \, $  or simply by  $ \, \widehat{\otimes} \, $.
 For any  $ \Bbbk $--vector  space  $ V \, $,  set
  $$  V[[\hbar]]  \; := \;  \Big\{\, {\textstyle \sum_{n \geq 0}} v_n \hbar^n \,\Big|\, v_n \in V \, , \; \forall \, n \geq 0 \,\Big\}  $$
 then  $ V[[\hbar]] $  is a complete  $ \kh $--module.  We call a topological  $ \kh $--module  \textit{topologically free\/}
 if it is isomorphic to  $ V[[\hbar]] $  for some  $ \Bbbk $--vector  space  $ V \, $.  For two topologically free modules
 $ V[[\hbar]] $  and  $ W[[\hbar]] $  one has that
 $ \; V[[\hbar]] \,\widehat{\otimes}\, W[[\hbar]] \, \cong \, \big(V \otimes W\big)[[\hbar]] \, $,
 \;see  \cite[Proposition XVI.3.2]{Ks}.  All completed tensor products between  $ \kh $--modules
 then will be denoted simply by  $ \otimes \, $,
 unless we intend to stress the topological aspect.  In particular, we make no distinction on the notation
 between Hopf algebras and topological
 Hopf algebras; we assume it is well-understood from the context.
\end{free text}

\vskip11pt

\begin{free text}  \label{defs_Hopf-algs}
 {\bf Hopf algebra deformations.}
 There exist two standard methods to deform Hopf algebras, leading to so-called
 ``$ 2 $--cocycle  deformations'' and to ``twist deformations'': hereafter we recall both procedures,
 adapting them to the setup of  \textsl{topological\/}  Hopf algebras,
 then later on we apply them to formal quantum groups.

\vskip11pt

 \textsl{$ \underline{\text{Twist deformations}} $:}\,
 Let  $ H $  be a Hopf algebra (over a commutative ring), and let
 $ \, \F \in H \otimes H \, $  be an invertible element in  $ H^{\otimes 2} $
(later called a ``{\it twist\/}'')
 such that
  $$
  \F_{1{}2} \, \big( \Delta \otimes \text{id} \big)(\F)  \, = \,  \F_{2{}3} \, \big( \text{id} \otimes \Delta \big)(\F)  \quad ,
  \qquad  \big( \epsilon \otimes \text{id} \big)(\F)  =  1  =  \big( \text{id} \otimes \epsilon \big)(\F)
  $$
 Then  $ H $  bears a second Hopf algebra structure, denoted  $ H^\F $
 and called  \textit{twist deformation\/}  of the old one, with the old product,
 unit and counit, but with new ``twisted'' coproduct  $ \Delta^\F $ and antipode  $ \SS^\F $  given by
$$
    \qquad   \Delta^{\!\F}(x)  \, := \,  \F \, \Delta(x) \, \F^{-1}  \quad ,  \qquad  \SS^\F(x) \, := \, v\,\SS(x)\,v^{-1}
    \quad \qquad  \forall \;\; x \in H
$$
 where  $ \, v := \sum_\F \SS(f'_1)\,f'_2 \, $   --- with  $ \, \sum_\F f'_1 \otimes f'_2 = \F^{-1} \, $  ---
 is invertible in  $ H $  (see,  \cite[ \S 4.2.{\sl E}]{CP},  for further details).  \textsl{When  $ H $
 is in fact a  \textit{topological}  Hopf algebra}   --- meaning that, in particular, its coproduct  $ \Delta $
 takes values into  $ \, H \otimes H \, $  where now  ``$ \, \otimes \, $''  stands for a suitable topological tensor product ---
 then the same notions still make sense, and the related results apply again, up to properly reading them.

\vskip11pt

 \textsl{$ \underline{\text{Cocycle deformations}} $:}\,
 Let $ \, \big( H, m\,, 1\,, \Delta\,, \epsilon \big) \, $  be a bialgebra over a ring  $ \k \, $.
 A convolution invertible linear map  $ \sigma $  in  $ \, \Hom_{\Bbbk}(H \otimes H, \k \,) \, $
 is called a (normalized) {\it Hopf 2-cocycle\/}
(or just a  ``$ 2 $--cocycle''  if no confusion arises) if
  $$  \sigma(b_{(1)},c_{(1)}) \, \sigma(a,b_{(2)}c_{(2)})  \; = \;  \sigma(a_{(1)},b_{(1)}) \, \sigma(a_{(2)}b_{(2)},c)  $$
and  $ \, \sigma (a,1) = \eps(a) = \sigma(1,a) \, $  for all  $ \, a, b, c \in H \, $,  see  \cite[Sec.\ 7.1]{Mo}.
                                                                     \par
   Using a  $ 2 $--cocycle  $ \sigma $  it is possible to define a new algebra structure on  $ H $
   by deforming the multiplication.
   Indeed, define  $ \, m_{\sigma} = \sigma * m * \sigma^{-1} : H \otimes H \longrightarrow H \, $  by
  $$  m_{\sigma}(a,b)  \, = \,  a \cdot_{\sigma} b  \, =  \,
  \sigma(a_{(1)},b_{(1)}) \, a_{(2)} \, b_{(2)} \, \sigma^{-1}(a_{(3)},b_{(3)})  \eqno \forall \;\, a, b \in H  \quad  $$
 If in addition  $ H $  is a Hopf algebra with antipode  $ \SS \, $,  then define also
$ \, \SS_\sigma : H \longrightarrow H \, $  as  $ \, \SS_{\sigma} : H \longrightarrow H \, $
where
  $$  \SS_{\sigma}(a)  \, = \,  \sigma(a_{(1)},\SS(a_{(2)})) \, \SS(a_{(3)}) \, \sigma^{-1}(\SS(a_{(4)}),a_{(5)})
  \eqno \forall \;\, a \in H  \quad  $$
 It is known  that  $ \, \big( H, m_\sigma, 1, \Delta, \eps \big) \, $  is in turn a bialgebra, and
 $ \, \big( H, m_\sigma, 1, \Delta, \eps, \SS_\sigma \big) \, $  is a Hopf algebra: we shall call such
 a new structure on  $ H $  a  \textit{cocycle deformation\/}  of the old one,
 and we shall graphically denote it by  $ H_\sigma \, $;  see  \cite{Doi}  for more details.
\end{free text}

\vskip11pt

%

\begin{free text}  \label{deforms-vs-duality}
 {\bf Deformations and duality.}\,
 The two notions of  ``$ 2 $--cocycle''  and of ``twist'' are so devised as to be dual to each other
 with respect to Hopf duality.  The proof of the following result (an exercise in Hopf theory) is left to the reader:

\vskip11pt

\begin{prop}  \label{prop: duality-deforms}
 Let  $ H $  be a Hopf algebra over a field, and  $ H^* $  its linear dual.
 \vskip3pt
   {\it (a)}\,  Let  $ \F $  be a twist for  $ H \, $,  and  $ \sigma_{{}_\F} $  the image of  $ \F $  in $ {(H \otimes H)}^* $
   for the natural composed embedding
   $ \, H \otimes H \lhook\joinrel\relbar\joinrel\longrightarrow H^{**}
   \otimes H^{**} \lhook\joinrel\relbar\joinrel\longrightarrow {\big( H^* \otimes H^* \big)}^* \, $.
   Then  $ \sigma_{{}_\F} $  is a 2--cocycle for  $ H^* \, $,  and there exists a canonical isomorphism
   $ \, {\big( H^* \big)}_{\sigma_{{}_\F}} \cong {\big( H^\F \,\big)}^* \, $.
 \vskip3pt
   {\it (b)}\,  Let  $ \sigma $  be a 2--cocycle for  $ H \, $;  assume that
   $ H $  is finite-dimensional, and let  $ \F_\sigma $
   be the image of  $ \sigma $  in the natural identification  $ \, {(H \otimes H)}^* = H^* \otimes H^* \, $.
   Then  $ \F_\sigma $  is a twist for  $ H^* \, $,  and there exists a canonical isomorphism
   $ \, {\big( H^* \big)}^{\F_\sigma} \cong {\big( H_\sigma \big)}^* \, $.
\qed
\end{prop}
\end{free text}

\medskip

%

\begin{free text}  \label{q-numbers}
 {\bf Some  $ q $--numbers.}\,
%
%
 Let  $ \Zqqm $  be the ring of Laurent polynomials with integral coefficients in the indeterminate  $ q \, $.  For every  $ \, n \in \NN_{+} \, $ we define
  $$  \displaylines{
   {(0)}_q  \, := \;  1 \;\; ,  \quad
    {(n)}_q  \, := \;  \frac{\,q^n -1\,}{\,q-1\,}  \; = \;  1 + q + \cdots + q^{n-1}  \; =
    \; {\textstyle \sum\limits_{s=0}^{n-1}} \, q^s  \qquad  \big(\, \in \, \ZZ[q] \,\big)  \cr
   {(n)}_q!  \, := \;  {(0)}_q {(1)}_q \cdots {(n)}_q  := \,  {\textstyle \prod\limits_{s=0}^n}
{(s)}_q  \;\; ,  \quad
    {\binom{n}{k}}_{\!q}  \, := \;  \frac{\,{(n)}_q!\,}{\;{(k)}_q! {(n-k)}_q! \,}  \qquad
\big(\, \in \, \ZZ[q] \,\big)  \cr
   {[0]}_q  := \,  1  \; ,  \;\;
    {[n]}_q  := \,  \frac{\,q^n -q^{-n}\,}{\,q-q^{-1}\,}  =
\, q^{-(n-1)} + \cdots + q^{n-1}  =  {\textstyle \sum\limits_{s=0}^{n-1}}
\, q^{2\,s - n + 1}  \!\quad  \big( \in \Zqqm \,\big)  \cr
   {[n]}_q!  \, := \;  {[0]}_q {[1]}_q \cdots {[n]}_q  = \,
{\textstyle \prod\limits_{s=0}^n} {[s]}_q  \;\; ,  \quad
    {\bigg[\, {n \atop k} \,\bigg]}_q  \, := \;  \frac{\,{[n]}_q!\,}{\;{[k]}_q! {[n-k]}_q!\,}
\qquad  \big(\, \in \, \Zqqm \,\big)  }  $$
 \vskip5pt
\noindent
   In particular, we have the identities
  $$  {(n)}_{q^2} = q^{n-1} {[n]}_q \;\; ,  \qquad  {(n)}_{q^2}! =
q^{\frac{n(n-1)}{2}} {[n]}_q  \;\; ,
\qquad  {\binom{n}{k}}_{\!\!q^{2}} = q^{k(n-k)} {\bigg[\, {n \atop k} \,\bigg]}_q  \;\; .  $$
 Moreover, for any field  $ \FF $  we can think of Laurent polynomials as functions on  $ \FF^\times $,
 hence for any  $ \, q \in \FF^\times \, $  we shall read every symbol above as a suitable element in  $ \, \FF \, $.
\end{free text}

\medskip

\subsection{Formal multiparameter QUEAs}  \label{subsec: FoMpQUEAs}
 {\ }
 \vskip11pt
We introduce now the notion of  {\sl formal multiparameter quantum universal enveloping algebra}
--- or just ``FoMpQUEA'', in short.
                                                                \par
   Hereafter,  $ \Bbbk $  is a field of characteristic zero,
   $ \kh $  the ring of formal power series in  $ \hbar $  with coefficients in  $ \Bbbk \, $.
   In any topological,  $ \kh $--adically  complete  $ \kh $--algebra  $ \mathcal{A} \, $,
   if  $ \, X \in \mathcal{A} \, $  we use the standard notation
   $ \; e^{\hbar\,X} := \exp(\hbar\,X) \, = \, {\textstyle \sum\limits_{n=0}^{+\infty}} \, \hbar^n X^n \big/ n! \; \in \, \mathcal{A} \; $.

\vskip9pt

\begin{definition}  \label{def: params-x-Uphgd}
 Let  $ \, A := {\big(\, a_{i,j} \big)}_{i, j \in I} \, $  be some fixed generalized symmetrizable Cartan matrix, and let
 $ \, P := {\big(\, p_{i,j} \big)}_{i, j \in I} \in M_{n}\big( \kh \big) \, $  be a matrix of Cartan type associated with
 $ A $  as in the sense of Definition \ref{def: realization of P}\textit{(d)},  that is
 $ \, P + P^{\,\scriptscriptstyle T} = 2\,D\,A \, $,  i.e.\  $ \, p_{ij} + p_{ji} = 2\,d_i\,a_{ij} \, $  for all  $ \, i, j \in I \, $,
 which implies  $ \, p_{ii} = 2\,d_i \not= 0 \, $  for all  $ \, i \in I \, $.
                                                                 \par
   We define in  $ \kh $  the following elements:  $ \, q := e^\hbar = \exp(\hbar) \in \kh \, $,
   $ \, q_i := e^{\hbar\,d_i} \, \big(\! = q^{d_i} \big) \, $,  $ \, q_{ij} := e^{\hbar\,p_{ij}} \, \big(\! = q^{p_{ij}} \big) \; $
   for all  $ \, i, j \in I \, $,  and also  $ \; q_{ij}^{1/2} := e^{\hbar\,p_{ij}/2} \; $  for all  $ \, i, j \in I \, $.
   \textsl{In particular we have  $ \, q_{ii}^{1/2}= e^{\hbar{}d_i} = q_i \, $  and  $ \; q_{ij} \, q_{ji} = q_{ii}^{\,a_{ij}} \; $
   for all  $ \, i, j \in I \, $}.   \hfill  $ \diamondsuit $
\end{definition}

\vskip5pt

   We can now define our FoMpQUEAs, using notation as in  Definition \ref{def: params-x-Uphgd}  above.

\vskip9pt

\begin{definition}  \label{def: Mp-Uhgd}
 Let  $ \, A := {\big(\, a_{i,j} \big)}_{i, j \in I} \, $  be a generalized symmetrizable Cartan matrix, and
 $ \, P := {\big(\, p_{i,j} \big)}_{i, j \in I} \in M_{n}\big( \kh \big) \, $  a matrix of Cartan type associated with  $ A $.
 We fix a realization  $ \, \R := \big(\, \lieh \, , \Pi \, , \Pi^\vee \,\big) \, $  of  $ P $  as in  Definition \ref{def: realization of P}.
 \vskip3pt
   {\it (a)}\,  We define the  {\sl formal multiparameter quantum universal enveloping algebra}
   --- in short  {\sl formal MpQUEA, or simply FoMpQUEA}  ---
   with multiparameter  $ P $  and realization  $ \R $
 as follows.  It is the unital, associative, topological,  $ \hbar $--adically  complete  $ \kh $--algebra  $ \uRPhg $  generated
 by the  $ \kh $--submodule  $ \lieh $  together with elements
 $ \, E_i \, $,  $ F_i \, $  (for all  $ \, i \in I \, $),
 with relations
 (for all  $ \, T , T' , T'' \in \lieh \, $,  $ \, i \, , j \in I \, $)
 \vskip-9pt
\begin{equation}  \label{eq: comm-rel's_x_uPhg}
 \begin{aligned}
   T \, E_j \, - \, E_j \, T  \, = \,  +\alpha_j(T) \, E_j  \;\; ,  \qquad  T \, F_j \, - \, F_j \, T  \; = \;  -\alpha_j(T) \, F_j   \hskip33pt  \\
   T' \, T''  \; = \;  T'' \, T'  \;\; ,  \qquad   E_i \, F_j \, - \, F_j \, E_i  \; = \;
   \delta_{i,j} \, {{\; e^{+\hbar \, T_i^+} - \, e^{-\hbar \, T_i^-} \;} \over {\; q_i^{+1} - \, q_i^{-1} \;}}   \hskip37pt  \\
   \sum\limits_{k = 0}^{1-a_{ij}} (-1)^k {\left[ { 1-a_{ij} \atop k }
\right]}_{\!q_i} q_{ij}^{+k/2\,} q_{ji}^{-k/2} \, E_i^{1-a_{ij}-k} E_j E_i^k  \; = \;  0   \qquad  (\, i \neq j \,)   \hskip25pt  \\
   \sum\limits_{k = 0}^{1-a_{ij}} (-1)^k {\left[ { 1-a_{ij} \atop k }
\right]}_{\!q_i} q_{ij}^{+k/2\,} q_{ji}^{-k/2} \, F_i^k F_j F_i^{1-a_{ij}-k}  \; = \;  0   \qquad  (\, i \neq j \,)   \hskip25pt
 \end{aligned}
\end{equation}
 \vskip-1pt
   We say that the FoMpQUEA  $ \uRPhg $  is  \textsl{straight},  or  \textsl{small},
   or \textsl{minimal},  or  \textsl{split},
   if such is  $ \R \, $;  also, we define the  \textsl{rank\/}  of  $ \uRPhg $  as
   $ \; \rk\!\big( \uRPhg \big) := \rk(\R) = \rk_\kh(\lieh) \, $.
 \vskip3pt
   {\it (b)}\,  We define the  {\sl Cartan subalgebra\/}
   $ U_{P,\hbar}^{\,\R}(\lieh) \, $,  or just  $ U_\hbar(\lieh) \, $,  of a FoMpQUEA  $ \uRPhg $
   as being the unital,  $ \hbar $--adically  complete topological  $ \kh $--subalgebra  of  $ \uRPhg $
   generated by the  $ \kh $--submodule  $ \lieh \, $.
 \vskip3pt
   {\it (c)}\,  We define the  {\sl positive, resp.\ the negative, Borel subalgebra\/}  $ \uRPhbp $,  resp.\  $ \uRPhbm $,
 of  $ \, \uRPhg $  to be
 the unital,  $ \hbar $--adically  complete topological  $ \kh $--subalge\-bra  of
 $ \uRPhg $  generated by  $ \lieh $  and all the  $ E_i $'s,  resp.\  by  $ \lieh $  and all the  $ F_i $'s  ($ \, i \in I \, $).
 \vskip5pt
   {\it (d)}\,  We define the  {\sl positive, resp.\ negative, nilpotent subalgebra\/}  $ \uRPhnp $,  resp.\  $ \uRPhnm $,
   of a FoMpQUEA  $ \uRPhg $  to be the unital,  $ \hbar $--adically  complete topological  $ \kh $--subalgebra
   of  $ \uRPhg $  generated by the  $ E_i $'s,  resp.\ the  $ F_i $'s,  with  $ \, i \in I \, $.   \hfill  $ \diamondsuit $
\end{definition}

\vskip9pt

   The following two results underscore that the dependence of FoMpQUEAs on realizations
   (which includes that on the multiparameter matrix) is functorial:

\vskip11pt

\begin{prop}  \label{prop: functor_R->uRPhg}
 Let  $ \, P \in M_{n}\big( \kh \big) \, $.  If both  $ \R' $  and  $ \, \R'' $  are realizations of  $ P $  and
 $ \, \underline{\phi} : \R' \relbar\joinrel\relbar\joinrel\longrightarrow \R'' \, $  is a morphism between them,
 then there exists a unique morphism of unital topological  $ \, \kh $--algebras
 $ \; U_{\underline{\phi}} : U^{\,\R'}_{\!P,\hbar}(\lieg) \relbar\joinrel\relbar\joinrel\longrightarrow U^{\,\R''}_{\!P,\hbar}(\lieg) \; $
 that extends the morphism  $ \, \phi : \lieh' \relbar\joinrel\relbar\joinrel\longrightarrow \lieh'' \, $
 of  $ \, \kh $--modules  given by
 $ \underline{\phi} \, $;  \,moreover,  $ \, U_{\underline{\id}_\R} = \id_{\uRPhg} \, $  and
 $ \; U_{\underline{\phi}' \circ\, \underline{\phi}} = U_{\underline{\phi}'} \circ U_{\underline{\phi}} \; $
 (whenever  $ \, \underline{\phi}' \circ\, \underline{\phi} \, $  is defined).
                                                                                \par
   Thus, the construction  $ \, \R \mapsto \uRPhg \, $  --- for any fixed  $ P $  ---   is functorial in  $ \R \, $.
                                                                              \par
   Moreover, if  $ \underline{\phi} $  is an epimorphism, resp.\ a monomorphism, then  $ U_{\underline{\phi}} $
   is an epimorphism, resp.\ a monomorphism, as well.
                                                                              \par
   Finally, for any morphism  $ \, \underline{\phi} : \R' \relbar\joinrel\relbar\joinrel\relbar\joinrel\longrightarrow \R'' \, $,
   \,the kernel  $ \, \Ker\big(U_{\underline{\phi}}\big) $  of  $ \, U_{\underline{\phi}} $  is the two-sided ideal in
   $ \, U^{\,\R'}_{\!P,\hbar}(\lieg) $  generated by  $ \, \Ker(\phi) \, $,
   and the latter is central in  $ \, U^{\,\R'}_{\!P,\hbar}(\lieg) \, $.
\end{prop}

\pf

 Everything is obvious, we only spend some words on the centrality of  $ \Ker(\phi) \, $.
 Lemma \ref{lemma: ker-morph's_realiz's}  gives  $ \; \Ker(\phi) \, \subseteq \, \bigcap\limits_{j \in I} \Ker(\alpha'_j) \; $;
 \,then  \eqref{eq: comm-rel's_x_uPhg}  implies that each element in  $ \Ker(\phi) $  commutes with all generators of
 $ \, U^{\,\R'}_{\!P,\hbar}(\lieg) \, $,  so  $ \Ker(\phi) $  is central in  $ \, U^{\,\R'}_{\!P,\hbar}(\lieg) \, $.
\epf

\vskip7pt

\begin{cor}  \label{cor: isom_R -> isom_uRPhg}
 With notation as above, if  $ \; \R' \cong \R'' \, $  then  $ \; U^{\,\R'}_{P,\hbar}(\lieg) \cong U^{\,\R''}_{P,\hbar}(\lieg) \; $.
                                                     \par
   In particular, all FoMpQUEAs built upon split realizations, respectively small
%
%
 realizations, of the same matrix  $ P $  and sharing the same rank of  $ \, \lieh $  are isomorphic to each other,
 hence they are independent (up to isomorphisms)  of the specific realization, but only depend on
 $ P $  and on the rank of  $ \, \lieh \, $.
\end{cor}

\pf
 This follows at once from  Proposition \ref{prop: functor_R->uRPhg}  together with the uniqueness result in
 Proposition \ref{prop: exist-realiz's}  and  Proposition \ref{prop: exist-realiz's_small}.
\epf

\vskip9pt

   We conclude this subsection with an important structure result, namely the ``triangular decomposition''
for FoMpQUEAs.  We begin with some preliminaries.
\vskip13pt

\begin{definition}  \label{def: Uotimes}
 Let  $ \, A := {\big(\, a_{i,j} \big)}_{i, j \in I} \, $  be a generalized symmetrizable Cartan matrix, and
 $ \, P := {\big(\, p_{i,j} \big)}_{i, j \in I} \in M_{n}\big( \kh \big) \, $  a matrix of Cartan type associated with  $ A $.
 We fix a realization  $ \, \R := \big(\, \lieh \, , \Pi \, , \Pi^\vee \,\big) \, $  of  $ P $  as in  Definition \ref{def: realization of P}.
 \vskip5pt
   {\it (a)}\,  We define  $ \uhat^+ $,  resp.\  $ \uhat^- $,  to be the unital, associative, topological,
   $ \hbar $--adically  complete  $ \kh $--algebra  with generators  $ E_i \, (i \in I\,) \, $,  resp.\  $ F_i \, (i \in I\,) \, $,  \,and relations
%
%
  $$  \displaylines{
   u^E_{ij}  \, := \,  {\textstyle \sum\limits_{k = 0}^{1-a_{ij}}} (-1)^k {\left[ { 1-a_{ij} \atop k } \right]}_{\!q_i} q_{ij}^{+k/2\,}
  q_{ji}^{-k/2} \, E_i^{1-a_{ij}-k} E_j E_i^k  \,\; = \;\,  0   \qquad  (\, \forall \; i \neq j \,)  \cr
   u^F_{ij}  \, := \,  {\textstyle \sum\limits_{k = 0}^{1-a_{ij}}} (-1)^k {\left[ { 1-a_{ij} \atop k } \right]}_{\!q_i} q_{ij}^{+k/2\,}
  q_{ji}^{-k/2} \, F_i^k F_j F_i^{1-a_{ij}-k}  \,\; = \;\,  0   \qquad  (\, \forall \; i \neq j \,)  }  $$
 \vskip5pt
   {\it (b)}\,  We define  $ \, \uhat^0 \, $  to be the unital, associative, commutative, topological,
   $ \hbar $--adically  complete  $ \kh $--algebra  generated by  $ \lieh \, $.  In other words, it is
   $ \, \uhat^0 := \widehat{S}_\hbar(\lieh) \, $,  the  $ \hbar $--adic  completion of the symmetric
   $ \kh $--algebra over the  $ \kh $--module  $ \lieh \, $.
 \vskip5pt
   {\it (c)}\,  We define  $ \; \overrightarrow{U}^{\R,\otimes}_{\!P,\,\hbar}(\lieg) :=
   \uhat^- \,\widehat{\otimes}_\kh\, \uhat^0 \,\widehat{\otimes}_\kh\, \uhat^+ \, $,  \,and we introduce notation
  $$  \displaylines{
   \uhat^-_\otimes := \uhat^- \otimes \kh \otimes \kh  \;\; ,  \!\!\quad
   \uhat^0_\otimes(\lieh) := \kh \otimes \uhat^0 \otimes \kh  \;\; ,  \!\!\quad
   \uhat^+_\otimes := \kh \otimes \kh \otimes \uhat^+  \cr
   F^\otimes := F \otimes 1 \otimes 1 \; ,  \,\; H^\otimes := 1 \otimes H \otimes 1 \; ,  \,\;
E^\otimes := 1 \otimes 1 \otimes E   \quad\;  \forall \; F \in \uhat^- , H \in \uhat^0 , E \in \uhat^+  }  $$
   \indent   Similarly, we set  $ \; \overleftarrow{U}^{\R,\otimes}_{\!P,\,\hbar}(\lieg) :=
   \uhat^+ \,\widehat{\otimes}_\kh\, \uhat^0 \,\widehat{\otimes}_\kh\, \uhat^- \, $.
\end{definition}

\vskip9pt

   The following, key technical result is also interesting in itself:

\vskip11pt

\begin{lema}  \label{lemma: alg-struct-on-Uotimes(g)}
 There exists on  $ \, \overrightarrow{U}^{\R,\otimes}_{\!P,\,\hbar}(\lieg) $
 a unique structure of unital, associative, topological,  $ \hbar $--adically  complete  $ \kh $--algebra  such that
 $ \, \uhat^-_\otimes \, $,  $ \uhat^0_\otimes $  and  $ \, \uhat^+_\otimes $
 are all\/  $ \kh $--subalgebras  in  $ \, U^{\R,\otimes}_{\!P,\,\hbar}(\lieg) \, $,  \,and moreover
  $$  \displaylines{
   F_i^\otimes \cdot T^\otimes \; = \; F_i \otimes T \otimes 1  \;\; ,   \quad
   T^\otimes \cdot E_j^\otimes \; = \; 1 \otimes T \otimes E_j  \;\; ,  \quad
   F_i^\otimes \cdot E_j^\otimes \; = \; F_i \otimes 1 \otimes E_j  \cr
   T^\otimes \cdot F_i^\otimes  \; = \;  F_i \otimes T \otimes 1 \, - \, \alpha_i(T)\,F_i^\otimes  \quad ,
   \qquad   E_j^\otimes \cdot T^\otimes  \; = \;  1 \otimes T \otimes E_j \, -
   \, \alpha_j(T)\,E_j^\otimes  \cr
   E_j^\otimes \cdot F_i^\otimes \; = \; F_i \otimes 1 \otimes E_j \, + \, \delta_{i{}j} \,
   1\ot {{\;\;e^{+\hbar\,T_i^+} - e^{-\hbar\,T_i^-}\;} \over {e^{+\hbar\,d_i} - e^{-\hbar\,d_i}}} \ot 1 }  $$
 \vskip3pt
   An entirely similar claim holds true for  $ \, \overleftarrow{U}^{\R,\otimes}_{\!P,\,\hbar}(\lieg) :=
   \uhat^+ \,\widehat{\otimes}_\kh\, \uhat^0 \,\widehat{\otimes}_\kh\, \uhat^- \, $.
\end{lema}

\begin{proof}
   It is enough to prove the statement about  $ \overrightarrow{U}^{\R,\otimes}_{\!P,\,\hbar}(\lieg) \, $.
                                                          \par
   We introduce a structure of an  $ \hbar $--adically  complete, topological  $ \kh $--algebra  $ \overrightarrow{U}^{\R,\otimes}_{\!P,\,\hbar}(\lieg) $  as required by hands, somehow.  First, we assume that in this algebra the submodules
 $ \uhat^-_\otimes \, $,  $ \uhat^0_\otimes \, $  and  $ \uhat^+_\otimes \, $
 sit as  $ \kh $--subalgebras   --- there is no obstruction to such a requirement.  After this, the structure will be uniquely determined once we fix the products among elements in any two (different) of these subalgebras.  Moreover, as the subalgebra
 $ \, \uhat^-_\otimes \, $,  \,resp.\  $ \, \uhat^0_\otimes(\lieh) \, $,  \,resp.\  $ \, \uhat^+_\otimes \, $,
 \,is (topologically) generated by the  $ F_i^\otimes $'s  ($ \, i \in I \, $),  resp.\  the
 $ T^\otimes $'s  ($ \, T \in \lieh \, $),  resp.\ the  $ E_j^\otimes $'s  ($ \, j \in I \, $),
 \,it is enough to fix the products among any two such generators (from different sets).
 Eventually, recall that the  $ F_i $'s,  resp.\ the  $ E_j $'s,  are indeed generators for  $ \uhat^- $,
 resp.\ for  $ \uhat^+ $,  \,which are only subject to the ``quantum Serre relations'' in
 Definition \ref{def: Uotimes}\textit{(b)},  while the  $ T $'s  are \textsl{``commutative free''}
 --- but for the fact that they are related by obvious, built-in relations such as  $ \, T = c'\,T' + c''T'' \, $
 inside  $ \lieh $  implies  $ \, T = c'\,T' + c''T'' \, $  in  $ \uhat^0_\otimes \, $,  as  $ \lieh $  is naturally mapped (linearly) into  $ \uhat^0_\otimes \, $.
 Thus, one can define the values of the product among
 $ F_i^\otimes $,  $ E_j^\otimes $'  and  $ T^\otimes $
  \textsl{in any possible way as soon as all ``quantum Serre relations'' among the  $ F_i^\otimes $'s
  and among the  $ E_j^\otimes $'s,  as well as the ``obvious relations'' among the  $ T $'s  from\/  $ \lieh $}
  --- namely the ``linear relations'' (such as  $ \, T = c'\,T' + c''T'' \, $)
  and the commutation relations (of the form  $ \, T' T'' = T'' T' \, $)
   \textsl{are respected}.
                                                              \par
   By the above discussion, the following choices
  $$  \displaylines{
   F_i^\otimes \cdot T^\otimes \; := \; F_i \otimes T \otimes 1  \;\; ,
   \quad   T^\otimes \cdot E_j^\otimes \; := \; 1 \otimes T \otimes E_j  \;\; ,  \quad
   F_i^\otimes \cdot E_j^\otimes \; := \; F_i \otimes 1 \otimes E_j  \cr
   T^\otimes \cdot F_i^\otimes  := \,  F_i \otimes T \otimes 1 \, - \, \alpha_i(T) \, F_i \otimes 1 \otimes 1  \; ,
   \!\!\!\quad   E_j^\otimes \cdot T^\otimes  := \,  1 \otimes T \otimes E_j \, - \, \alpha_j(T) \, 1 \otimes 1 \otimes E_j  \cr
   E_j^\otimes \cdot F_i^\otimes \; := \; F_i \otimes 1 \otimes E_j \, + \,
   \delta_{i\,j} \cdot 1 \otimes {{\;\; e^{+\hbar\,T_i^+} \! - e^{-\hbar\,T_i^-} \;} \over {e^{+\hbar\,d_i} - e^{-\hbar\,d_i}}} \otimes 1  }  $$
 for the values of the product among two generators
 --- from different subalgebras  $ \, \uhat^-_\otimes \, $,  $ \, \uhat^0_\otimes \, $  or
 $ \, \uhat^+_\otimes \, $  ---
 are enough to determine a unique algebra structure as required:
 we only have still to check that, using these defining formulas
 for the product, the relations mentioned above among generators
 are respected.
 \vskip3pt
   First of all, we consider all linear relations and commutation relations among the  $ T $'s:
   in this case, the check is entirely trivial.
 \vskip3pt
   Second, we consider the case of quantum Serre relations.  Concerning them,
   \textsl{in order to have more readable formulas, we simplify notation\/}  (with a slight abuse) by writing, instead of
   ``$ \, F^\otimes \, $''  ($ \, \forall \; F \in \uhat^- \, $)  just  ``$ F \, $''  again, and similarly  ``$ H \, $''  instead of
   ''$ \, H^\otimes \, $''  ($ \, \forall \; H \in \uhat^0 \, $)  and  ``$ E \, $''  instead of  ''$ \, E^\otimes \, $''  ($ \, \forall \; E \in \uhat^+ \, $).
                                                                              \par
   Our goal now  is to check that the multiplication defined by the previous formulas
   ``respects'' the quantum Serre relations, which boils down to verify the following:
   \textit{all products between a factor chosen in  $ \, \big\{\, F_i \, , T , E_j \,|\, i \, , j \in I , \, T \in \lieh \,\big\} \, $
   and another (in either order) chosen in  $ \, \big\{\, u^E_{ij} \, , u^F_{ij} \,\big|\, i \, , j \in I , \, i \not= j \,\big\} \, $  is  \textsl{zero}.}
 \vskip3pt
   Clearly all products of type  $ \, u^F_{ij} \cdot F_\ell \, $  and  $ \, F_\ell \cdot u^F_{ij} \, $,
   \,resp.\ $ \, u^E_{ij} \cdot E_t \, $  and  $ \, E_t \cdot u^E_{ij} \, $,
   \,are zero because so they are in the subalgebra  $ U^- $, resp.\  $ U^+ $.  The non-trivial cases are
  $$  T \cdot u^E_{ij} \;\; ,  \!\quad  u^E_{ij} \cdot T \; ,  \!\quad  T \cdot u^F_{ij} \;\; ,
  \!\quad  u^F_{ij} \cdot T \; ,  \!\quad  u^E_{ij} \cdot F_\ell \;\; ,  \!\quad  F_\ell \cdot u^E_{ij} \;\; ,
  \!\quad  u^F_{ij} \cdot E_t \;\; ,  \!\quad  E_t \cdot u^F_{ij}  $$
 but among these, four cases are indeed almost trivial, as definitions give
  $$  \displaylines{
   T \cdot u^E_{ij} \, = \, 1 \otimes T \otimes u^E_{ij} \, = \, 1 \otimes T \otimes 0 \, = \, 0  \,\; ,
   \!\!\quad  u^F_{ij} \cdot T \, = \, u^F_{ij} \otimes T \otimes 1 \, = \, 0 \, \otimes T \otimes 1 \, = \, 0  \cr
   F_\ell \cdot u^E_{ij} \, = \, F_\ell \otimes 1 \otimes u^E_{ij} \, = \, F_\ell \otimes 1 \otimes 0 \, = \, 0  \,\; ,
   \!\!\quad  u^F_{ij} \cdot E_t \, = \, u^F_{ij} \otimes 1 \otimes E_t \, = \, 0 \, \otimes 1 \otimes E_t \, = \, 0  }  $$
   \indent   Eventually, the remaining, really non-trivial cases are the following four
  $$  u^E_{ij} \cdot T \;\; ,  \qquad  T \cdot u^F_{ij} \;\; ,  \qquad  u^E_{ij} \cdot F_\ell \;\; ,  \qquad  E_t \cdot u^F_{ij}  $$
 that we now go and analyze in detail.
 \vskip3pt
   Let us consider the product  $ \, T \cdot u^F_{ij} \, $:  \,straightforward calculations give
  $$  \displaylines{
   T \cdot u^F_{ij}  \; = \;
   T \, \bigg(\, {\textstyle \sum\limits_{k = 0}^{1-a_{ij}}} {(-1)}^k
   {\displaystyle {\left[ { 1-a_{ij} \atop k } \right]}_{\!q_i}} q_{ij}^{+k/2\,} q_{ji}^{-k/2} \, F_i^k \, F_j \, F_i^{1-a_{ij}-k} \bigg)  \; =   \hfill  \cr
   = \;  {\textstyle \sum\limits_{k = 0}^{1-a_{ij}}} {(-1)}^k
   {\displaystyle {\left[ { 1-a_{ij} \atop k } \right]}_{\!q_i}} q_{ij}^{+k/2\,} q_{ji}^{-k/2} \, T \, F_i^k \, F_j \, F_i^{1-a_{ij}-k}  \; =   \hfill  \cr
   \hfill   =  {\textstyle \sum\limits_{k = 0}^{1-a_{ij}}} {(-1)}^k
   {\displaystyle {\left[ { 1-a_{ij} \atop k } \right]}_{\!q_i}} \! q_{ij}^{+k/2\,} q_{ji}^{-k/2} \,
   F_i^k \, F_j \, F_i^{1-a_{ij}-k} \otimes \Big( T - \big( \big( 1 - a_{ij} \big) \alpha_i + \alpha_j\big)(T) \Big) \otimes 1  \, =  \cr
   \hfill   = \,  u^F_{ij} \otimes \Big( T - \big( \big( 1 - a_{ij} \big) \alpha_i + \alpha_j\big)(T) \Big) \otimes 1  \,
   = \,  0 \otimes \Big( T - \big( \big( 1 - a_{ij} \big) \alpha_i + \alpha_j\big)(T) \Big) \otimes 1  \, = \,  0  }  $$
 which is good.  The product  $ \, u^E_{ij} \cdot T \, $  is dealt with in a similar way.
 \vskip3pt
   Let us now consider the product  $ \, u^E_{ij} \cdot F_\ell \, $:  \,again,
   direct calculations yield different results, depending on whether  $ \, \ell \in \{i\,,j\} \, $  or not.  The first possible case is
  $$  \displaylines{
   \fbox{$ \, \ell \not\in \{i\,,j\} $}  \;\; \Longrightarrow \;\;
      u^E_{ij} \cdot F_\ell  \; =  \bigg(\, {\textstyle \sum\limits_{k = 0}^{1-a_{ij}}} {(-1)}^k
      {\displaystyle {\left[ { 1-a_{ij} \atop k } \right]}_{\!q_i}} \! q_{ij}^{+k/2\,} q_{ji}^{-k/2} \, E_i^{1-a_{ij}-k} \, E_j \, E_i^k \bigg) F_\ell  \; =   \hfill  \cr
   = \;  F_\ell \otimes 1 \otimes \bigg(\, {\textstyle \sum\limits_{k = 0}^{1-a_{ij}}} {(-1)}^k
   {\displaystyle {\left[ { 1-a_{ij} \atop k } \right]}_{\!q_i}} \! q_{ij}^{+k/2\,} q_{ji}^{-k/2} \, E_i^{1-a_{ij}-k} \, E_j \, E_i^k \bigg) \;  =  \cr
   \hfill   = \;  F_\ell \otimes 1 \otimes u^E_{ij}  \; = \;  F_\ell \otimes 1 \otimes 0  \; = \; 0  }  $$
 which stands good!  The second case is
  $$  \displaylines{
   \fbox{$ \, \ell = j \, $}  \; \Longrightarrow \;
      u^E_{ij} \cdot F_\ell  \; =  \bigg(\, {\textstyle \sum\limits_{k = 0}^{1-a_{ij}}} {(-1)}^k
      {\displaystyle {\left[ { 1-a_{ij} \atop k } \right]}_{\!q_i}} \! q_{ij}^{+k/2\,} q_{ji}^{-k/2} \,
      E_i^{1-a_{ij}-k} \, E_j \, E_i^k \bigg) F_j  \; =   \hfill  \cr
   = \;  {\textstyle \sum\limits_{k = 0}^{1-a_{ij}}} {(-1)}^k
   {\displaystyle {\left[ { 1-a_{ij} \atop k } \right]}_{\!q_i}} \! q_{ij}^{+k/2\,} q_{ji}^{-k/2} \, E_i^{1-a_{ij}-k} \cdot E_j \cdot F_j \cdot E_i^k \;  =   \hfill  \cr
   \hfill   = \;  {\textstyle \sum\limits_{k = 0}^{1-a_{ij}}} {(-1)}^k
   {\displaystyle {\left[ { 1-a_{ij} \atop k } \right]}_{\!q_i}} \! q_{ij}^{+k/2\,} q_{ji}^{-k/2} \, E_i^{1-a_{ij}-k}
   \cdot \bigg( F_j \cdot E_j + {{\; e^{+\hbar \, T_j^+} - e^{-\hbar \, T_j^-} \;}
   \over {e^{+\hbar \, d_j} - e^{-\hbar \, d_j}}} \bigg) \cdot E_i^k \;  =  \cr
   \hfill   = \;  F_j \cdot u^E_{ij} \, + \, {\textstyle \sum\limits_{k = 0}^{1-a_{ij}}} {(-1)}^k
   {\displaystyle {\left[ { 1-a_{ij} \atop k } \right]}_{\!q_i}} \! q_{ij}^{+k/2\,} q_{ji}^{-k/2} \, E_i^{1-a_{ij}-k}
   \cdot {{\; e^{+\hbar \, T_j^+} - e^{-\hbar \, T_j^-} \;} \over {e^{+\hbar \, d_j} - e^{-\hbar \, d_j}}} \cdot E_i^k  \; =  \cr
   \qquad   = \;  F_j \otimes 1 \otimes u^E_{ij} \; +   \hfill  \cr
   \qquad \qquad   + \,  {\textstyle \sum\limits_{k = 0}^{1-a_{ij}}} {(-1)}^k
   {\displaystyle {\left[ { 1-a_{ij} \atop k } \right]}_{\!q_i}} \! q_{ij}^{+k/2\,} q_{ji}^{-k/2} \, q_{ji}^{k-1+a_{ij}}
   \cdot {{e^{+\hbar \, T_j^+}} \over {\; e^{+\hbar \, d_j} - e^{-\hbar \, d_j} \;}} \cdot E_i^{1-a_{ij}}  \; -   \hfill  \cr
   \qquad \qquad \qquad   - \,  {\textstyle \sum\limits_{k = 0}^{1-a_{ij}}} {(-1)}^k
   {\displaystyle {\left[ { 1-a_{ij} \atop k } \right]}_{\!q_i}} \! q_{ij}^{+k/2\,} q_{ji}^{-k/2} \, q_{ij}^{1-a_{ij}-k}
   \cdot {{e^{-\hbar \, T_j^-}} \over {\; e^{+\hbar \, d_j} - e^{-\hbar \, d_j} \;}} \cdot E_i^{1-a_{ij}}  \; =   \hfill  \cr
   \hfill   = \;  F_j \otimes 1 \otimes u^E_{ij} \, + \, 1 \otimes {{\; C^+_{ij}(q) \,
   e^{+\hbar \, T_j^+} \! - C^-_{ij}(q) \, e^{-\hbar \, T_j^-} \;} \over {e^{+\hbar \, d_j} - e^{-\hbar \, d_j}}} \otimes E_i^{1-a_{ij}}  }  $$
 where in the last line we have  $ \, F_j \otimes 1 \otimes u^E_{ij} = F_j \otimes 1 \otimes 0 = 0 \, $,  \,and
  $$  \displaylines{
   C^+_{ij}(q)  \; := \;  {\textstyle \sum\limits_{k = 0}^{1-a_{ij}}} {(-1)}^k
   {\displaystyle {\left[ { 1-a_{ij} \atop k } \right]}_{\!q_i}} \! q_{ij}^{+k/2\,} q_{ji}^{-k/2} \, q_{ji}^{k-1+a_{ij}}  \; =   \hfill  \cr
   \hfill   = \,  {\textstyle \sum\limits_{k = 0}^{1-a_{ij}}} {(-1)}^k
   {\displaystyle {\left[ { 1-a_{ij} \atop k } \right]}_{\!q_i}} \!\! {\big( q_{ij} \, q_{ji} \big)}^{+k/2} \, q_{ji}^{-1+a_{ij}}  \, = \;  q_{ji}^{a_{ij}-1}
   {\textstyle \sum\limits_{k = 0}^{1-a_{ij}}} {(-1)}^k {\displaystyle {\left[ { 1-a_{ij} \atop k } \right]}_{\!q_i}} q_i^{\,k a_{ij}}  \; = \;  0  }  $$
 where the very last identity follows from the general, combinatorial  $ q $--identity  (see for example \cite[Chapter 0]{Ja},  or  \cite[\S 1.3.4]{Lu})
  $$  {\textstyle \sum\limits_{k = 0}^N} {(-1)}^k {\displaystyle {\left[ { N \atop k } \right]}_{\!q_i}} q_i^{\,k (1-N)}  \; = \;  0   \eqno \forall \; N \in \NN_+   \qquad  $$
 In a parallel way we get  $ \, C^-_{ij}(q) = 0 \, $,
 \,hence from the above analysis we conclude that  $ \, u^E_{ij} \cdot F_\ell = u^E_{ij} \cdot F_j = 0 \, $
 whenever  $ \, \ell = j \, $.  The third and last case is when  $ \, \ell = i \, $.
 To deal with that, let us notice that standard computations give us, for all  $ \, n \in \NN \, $,
  $$
  \displaylines{
   E_i^n \cdot F_i  \; = \;  F_i \, E_i^n \, + \, \big[ E_i^n , F_i \big]  \; = \;
   F_i \, E_i^n \, + \, {\textstyle\sum\limits_{\ell=0}^n} \, E_i^\ell \, \big[E_i\,,F_i\big] \, E_i^{n-1-\ell}  \; =   \hfill  \cr
   \qquad   = \;  F_i \, E_i^n \, + \, {\textstyle\sum\limits_{\ell=0}^{n-1}} \, E_i^\ell \, {{\; e^{+\hbar \, T_i^+} -
   e^{-\hbar \, T_i^-} \;} \over {e^{+\hbar \, d_i} - e^{-\hbar \, d_i}}} \, E_i^{n-1-\ell}  \; =   \hfill  \cr
   \qquad \qquad   = \;  F_i \cdot E_i^n \, + \, {\; {\textstyle\sum_{\ell=0}^{n-1}} \, q_{i{}i}^{-\ell} \,
   {e^{+\hbar \, T_i^+} \, - \, {\textstyle\sum_{\ell=0}^{n-1}} \, q_{i{}i}^{+\ell} \,
   {e^{-\hbar \, T_i^-}} \;} \over {q_i^{+1} - \, q_i^{-1}}} \cdot E_i^{n-1}  \; =   \hfill  \cr
   \qquad \qquad \qquad   = \;  F_i \cdot E_i^n \, + \, {[n]}_{q_i} \, {\;\, q_i^{-n+1} \,
   {e^{+\hbar \, T_i^+} - \, q_i^{+n-1} \, {e^{-\hbar \, T_i^-}} \;} \over {q_i^{+1} - \, q_i^{-1}}} \cdot E_i^{n-1}   \hfill  }
   $$
 Using this, we compute still more, and get, for all  $ \, r , s \in \NN \, $,
  $$
  \displaylines{
   E_i^r \, E_j \, E_i^s \cdot F_i  \; = \;  E_i^r \, E_j \, \Big( F_i \cdot E_i^s \, + \,
   {[s]}_{q_i} \, {\;\, q_i^{-s+1} \, {e^{+\hbar \, T_i^+} - \, q_i^{+s-1} \, {e^{-\hbar \, T_i^-}} \;}
   \over {q_i^{+1} - \, q_i^{-1}}} \cdot E_i^{s-1} \,\Big)  \; =   \hfill  \cr
   \hskip37pt   = \;  E_i^r \cdot F_i \cdot E_j \, E_i^s \, + \, {[s]}_{q_i} \, E_i^r \, E_j \cdot {{\;\, q_i^{-s+1} \,
   e^{+\hbar \, T_i^+} - \, q_i^{+s-1} \, e^{-\hbar \, T_i^-} \;} \over {q_i^{+1} - \, q_i^{-1}}} \cdot E_i^{s-1}  \; =   \hfill  \cr
   \hskip74pt   = \;  F_i \cdot E_i^r \, E_j \, E_i^s \, + \, {[r]}_{q_i} \, {{\;\, q_i^{-r+1} \,
   e^{+\hbar \, T_i^+} - \, q_i^{+r-1} \, e^{-\hbar \, T_i^-} \;} \over {q_i^{+1} - \, q_i^{-1}}} \cdot E_i^{r-1} \, E_j \, E_i^s \; +   \hfill  \cr
   \hskip111pt   + \;\, {[s]}_{q_i} \, {{\;\, q_i^{-s+1-2r} q_{ij}^{-1} \,
   e^{+\hbar \, T_i^+} - \, q_i^{+s-1+2r} q_{ji}^{+1} \, e^{-\hbar \, T_i^-} \;} \over {q_i^{+1} - \, q_i^{-1}}} \cdot E_i^r \, E_j \, E_i^{s-1}   \hfill  }
   $$
 Applying this result, we get the following:
  $$  \displaylines{
   u^E_{ij} \cdot F_i  \; =  {\textstyle \sum\limits_{r+s=1-a_{ij}}} \hskip-5pt {(-1)}^s
   {\displaystyle {\left[ { 1-a_{ij} \atop s } \right]}_{\!q_i}} \! q_{ij}^{+s/2\,} q_{ji}^{-s/2} \, E_i^r \, E_j \, E_i^s \cdot F_i  \; =   \hfill  \cr
   \hskip15pt   =  {\textstyle \sum\limits_{r+s=1-a_{ij}}} \hskip-5pt {(-1)}^s {\displaystyle
   {\left[ { 1-a_{ij} \atop s } \right]}_{\!q_i}} \! q_{ij}^{+s/2\,} q_{ji}^{-s/2} \cdot
   \bigg( F_i \cdot E_i^r \, E_j \, E_i^s \;\; +   \hfill  \cr
   \hskip30pt   + \;\; {[r]}_{q_i} \, {{\;\, q_i^{-r+1} \, e^{+\hbar \, T_i^+} - \, q_i^{+r-1} \,
   e^{-\hbar \, T_i^-} \;} \over {q_i^{+1} \! - q_i^{-1}}} \cdot E_i^{r-1} \, E_j \, E_i^s \;\; +   \hfill  \cr
   \hskip45pt   + \;\; {[s]}_{q_i} \, {{\;\, q_i^{-s+1-2r} q_{ij}^{-1} \, e^{+\hbar \, T_i^+} - \,
   q_i^{+s-1+2r} q_{ji}^{+1} \, e^{-\hbar \, T_i^-} \;} \over {q_i^{+1} \! - q_i^{-1}}} \cdot E_i^r \, E_j \, E_i^{s-1} \bigg)  \;\; =   \hfill  \cr
   \hskip5pt   = \;\;  F_i \cdot \hskip-5pt {\textstyle \sum\limits_{r+s=1-a_{ij}}} \hskip-5pt {(-1)}^s
   {\displaystyle {\left[ { 1-a_{ij} \atop s } \right]}_{\!q_i}} \! q_{ij}^{+s/2\,} q_{ji}^{-s/2} \, E_i^r \, E_j \, E_i^s \;\; +   \hfill  \cr
   \hskip10pt   +  {\textstyle \sum\limits_{h+k=-a_{ij}}} \hskip-1pt {{{\; (-1)}^k \, \varGamma_{h,k}^+ \;}
   \over {\; q_i^{+1} \! - q_i^{-1} \;}} \; e^{+\hbar \, T_i^+} \! \cdot E_i^h \, E_j \, E_i^k  \; +
   {\textstyle \sum\limits_{h+k=-a_{ij}}} \hskip-1pt {{{\; (-1)}^k \, \varGamma_{h,k}^- \;}
   \over {\; q_i^{+1} \! - q_i^{-1} \;}} \; e^{-\hbar \, T_i^-} \! \cdot E_i^h \, E_j \, E_i^k  }  $$
 where the coefficients  $ \, \varGamma_{h,k}^+ \, $  and  $ \, \varGamma_{h,k}^- \, $  are given by
\begin{equation*}  \label{eq: Gamma_hk^pm}
  \hskip-7pt   \varGamma_{h,k}^\pm   \; = \;  q_{ij}^{+k/2} \, q_{ji}^{-k/2} \, q_i^{\mp h}
  \left( {\left[ {{h+k+\!1} \atop k} \right]}_{\!q_i} \!\! {\big[h+\!1\big]}_{\!q_i} \, - \, {\left[ {{h+k+\!1}
  \atop {k+\!1}} \right]}_{\!q_i} \!\! {\big[k+\!1\big]}_{\!q_i} \right)  \; = \;\,  0
\end{equation*}
   \indent   Plugging this result in the previous formulas, we get
  $$  \displaylines{
   \fbox{$ \, \ell = i \, $}  \,\; \Longrightarrow \;\,
      u^E_{ij} \cdot F_\ell  \,\; =  \;\,  F_i \, \cdot \hskip-5pt
      {\textstyle \sum\limits_{r+s=1-a_{ij}}} \hskip-7pt {(-1)}^s {\displaystyle {\left[ { 1-a_{ij} \atop s } \right]}_{\!q_i}}
      \! q_{ij}^{+s/2\,} q_{ji}^{-s/2} \, E_i^r \, E_j \, E_i^s \; + \; 0 \; + \; 0  \;\; =   \hfill  \cr
   \hfill   = \;\;  F_i \, u^E_{ij} \; + \; 0 \; + \; 0  \;\; = \;\;  F_i \otimes 1 \otimes u^E_{ij}  \;\; =
   \;\;  F_i \otimes 1 \otimes 0  \;\; = \;\; 0  }  $$
 which makes the job.  The case of  $ \, E_t \cdot u^F_{ij} \, $  is entirely similar, hence we are done.
\end{proof}

\vskip9pt

   The next result shows that
   \textit{$ \overrightarrow{U}^{\R,\otimes}_{\!P,\,\hbar}(\lieg) $  and  $ \overleftarrow{U}^{\R,\otimes}_{\!P,\,\hbar}(\lieg) $
   are nothing but different, explicit realizations of our FoMpQUEA  $ \uRPhg \, $;}
 \,moreover, from this we deduce an explicit description of the nilpotent, Borel and Cartan quantum subalgebras in  $ \uRPhg \, $.

\vskip11pt

\begin{theorem}  \label{thm: Uotimes-cong-uRPhg}  {\ }
 \vskip3pt
   (a)\,  There exist natural isomorphisms of topological\/  $ \kh $--algebras
  $$
  \displaylines{
   \overrightarrow{U}^{\R,\otimes}_{\!P,\,\hbar}(\lieg)  \, := \,  \uhat^- \, \widehat{\otimes}_\kh\, \uhat^0 \, \widehat{\otimes}_\kh\, \uhat^+  \,\; \cong \;\,  \uRPhg  \cr
   \overleftarrow{U}^{\R,\otimes}_{\!P,\,\hbar}(\lieg)  \, := \,  \uhat^+ \, \widehat{\otimes}_\kh\, \uhat^0 \, \widehat{\otimes}_\kh\, \uhat^-  \,\; \cong \;\, \uRPhg  }
   $$
 induced by multiplication in  $ \uRPhg \, $.
 \vskip3pt
   (b)\,  With notation as in  Definition \ref{def: Mp-Uhgd},
   the isomorphisms in claims (a) above induce by restriction similar isomorphisms
  $$  \displaylines{
   \uRPhnm \,\; \cong \;\, \uhat^-  \quad ,  \qquad  U_\hbar(\lieh) \,\; \cong \;\, \uhat^0  \quad ,  \qquad
   \uRPhnp \,\; \cong \;\, \uhat^+  \cr
   \uhat^\pm \,\widehat{\otimes}_\kh\, \uhat^0  \; \cong \;  \uRPhbpm  \quad ,   \qquad
 \uhat^0 \,\widehat{\otimes}_\kh\, \uhat^\pm  \; \cong \;  \uRPhbpm  }  $$
 It follows then that  $ \uRPhnpm \, $,  $ U_\hbar(\lieh) \, $,  and  $ \uRPhbpm $
 admit the obvious descriptions (in particular, by generators and relations)
   inherited from
  Definition \ref{def: Uotimes}.
\end{theorem}

\begin{proof}
 \textit{(a)}\,  The two cases are similar, so we prove the claim only for  $ \overrightarrow{U}^{\R,\otimes}_{\!P,\,\hbar}(\lieg) \, $.
%
%
 \vskip3pt
   Consider the  $ \kh $--algebra  $ \; \overrightarrow{U}^{\R,\otimes}_{\!P,\,\hbar}(\lieg)
   := \uhat^- \,\widehat{\otimes}\, \uhat^0 \,\widehat{\otimes}\, \uhat^+ \; $  of  Lemma \ref{lemma: alg-struct-on-Uotimes(g)}.
   By construction, it is (topologically) generated by the elements  $ \, F_i^\otimes $,
   $ T^\otimes $,  $ E_j^\otimes $  ($ \, i \, , j \in I \, , T \in \lieh \, $),  \,and
   \textsl{these generators obey the same relations (up to ``inserting/removing'' any super/sub-script  ``$ \, \otimes $'')
   as the analogous generators of\/}  $ \uRPhg \, $.  Therefore, there exists an epimorphism of topological  $ \kh $--algebras
  $$  \pi \, : \, \uRPhg \relbar\joinrel\relbar\joinrel\relbar\joinrel\relbar\joinrel\twoheadrightarrow
U^{\R,\otimes}_{\!P,\,\hbar}(\lieg) \; ,  \qquad  F_i \mapsto F_i^\otimes \; ,  \;\;  T \mapsto T^\otimes \; ,
\;\;  E_j \mapsto E_j^\otimes  \;\quad  \big(\, \forall \; i \, , j \, ,  T \,\big)  $$
   \indent   On the other hand, for each  $ \, \bullet \in \{-\,,0\,,+\} \, $  there is a morphism
   $ \, \uhat^\bullet \,{\buildrel {\eta_\bullet} \over {\relbar\joinrel\relbar\joinrel\longrightarrow}}\, \uRPhg \, $
   mapping every generator of  $ \uhat^\bullet $  onto the corresponding generator in  $ \uRPhg \, $.
   Composing these with ``threefold-multiplication''  $ m_3 $  in  $ \uRPhg $  we obtain a morphism
  $$
  \mu_3 := m_3 \circ \big( \eta_- \otimes \eta_0 \otimes \eta_+ \big) :
  U^{\R,\otimes}_{\!P,\,\hbar}(\lieg) := \uhat^- \,\widehat{\otimes}\, \uhat^0 \,\widehat{\otimes}\, \uhat^+ \!\relbar\joinrel\relbar\joinrel\relbar\joinrel\longrightarrow \uRPhg
  $$
 of topological  $ \kh $--modules.
 Now, by construction, the map  $ \mu_3 $  is inverse to  $ \pi \, $,  so the latter is bijective,
 hence it is a  $ \kh $--algebra  isomorphism.  As on both sides we consider  $ \hbar $--adic  topology,
 this  $ \pi $  is then an homeomorphism of topological spaces too, so it is an isomorphism of topological  $ \kh $--algebras.
 Therefore  $ \, \mu_3 \, $,  being the inverse of  $ \pi \, $,  \,is an isomorphism of topological  $ \kh $--algebras  too, q.e.d.
 \vskip5pt
   \textit{(b)}\,  This follows quite easily from claim  \textit{(a)}.
\end{proof}

\vskip9pt

\begin{obs}  \label{obs: Uotimes = qdouble & Bosonization}
 The proof above relies on an  \textit{ad-hoc\/}  argument which has roots on well-known Hopf theoretic constructions:
 \textit{bosonization\/}  and  \textit{lifting}.  Indeed, the algebras  $ \uhat^\pm $  admit a \textit{braided\/}
 Hopf algebra structure whose comultiplication is defined by setting the generators  $ E_i \, $,  resp.\  $ F_i \, $,
 to be primitive elements, for all  $ \, i \in I \, $.  Hereafter, by braided Hopf algebra we mean a
 Hopf algebra in a braided tensor category; in the present case, the category is the category
 $ {}_{\uhat^0}^{\uhat^0}\mathcal{Y\,D} $  of Yetter-Drinfeld modules over the topological Hopf algebra
 $ \uhat^0 \, $.
%
%
 Given a Hopf algebra  $ B $ in a category of Yetter-Drinfeld modules  $ {}_H^H\mathcal{Y\,D} $
 over a Hopf algebra  $ H $,  there is a process to construct a usual Hopf algebra, called the
 ``Radford biproduct'' or  \textit{bosonization\/}:  it is kind of a semidirect product and coproduct,
 discovered by Radford and interpreted categorically by Majid.  As a vector space, the bosonization
 $ \, B \# H \, $  is just the vector space  $ \, B \otimes H \, $.  In our case, the (completed) tensor product
 $ \, \uhat^- \otimes \uhat^+ \, $  of the braided Hopf algebras is again a braided Hopf algebra, and the bosonization
 $ \, \big( \uhat^- \otimes \uhat^+ \big) \, \# \, \uhat^0 \, $  is a topological, complete Hopf  $ \kh $--algebra.
 By construction, it can be presented by the generators  $ \, T \in \lieh \, $  and  $ E_i \, $,  $ F_i $  with  $ \, i\in I \, $,
 satisfying all the relations in  \eqref{eq: comm-rel's_x_uPhg}  except the commuting relation between  $ E_i $
 and  $ F_j \, $;  in fact, the latter now is replaced simply by  $ \, \big[ E_i \, , F_j \big] = 0 \, $.
 Here enters into the picture the process of  \textit{lifting\/}  or  \textit{deformation\/}:
 \,through this process, one deforms the relations in a specific way, in our case, the element
 $ \, \big[ E_i \, , F_j \big] \in \uhat^- \otimes \uhat^+ \, $  becomes skew-primitive in the bosonization
 $ \, \big( \uhat^- \otimes \uhat^+ \big) \, \# \, \uhat^0 \, $  and one change the relation by setting
 it equal to the difference between the group-like elements appearing in the comultiplication, that is
 $ \; E_i \, F_j \, - \, F_j \, E_i  \; = \;  \delta_{i,j} \, {{\; e^{+\hbar \, T_i^+} - \, e^{-\hbar \, T_i^-} \;} \over {\; q_i^{+1} - \, q_i^{-1} \;}} \; $.
                                                                           \par
   In general, the lifting process can also be described via cocycle deformations.
                                                                           \par
   All these constructions are described explicitly in  \cite{Gar}  in the case of polynomial multiparameter quantum groups.  It is worth noting that through bosonizations and quantum doubles one can generalize triangular decompositions to more general families of Hopf algebras or quantum groups; this implies special features in representation theory, see for example  \cite{PV}.
\end{obs}

\vskip5pt

   Eventually, we can prove the ``triangular decomposition'' for our FoMpQUEAs:

\vskip11pt

\begin{theorem}  \label{thm: triang-decomp.'s}
 (``Triangular Decompositions'' in  $ \uRPhg $)
 \vskip3pt
   There exist natural isomorphisms of topological\/  $ \kh $--algebras
  $$  \displaylines{
   \uRPhnmp \,\widehat{\otimes}_\kh\, U_\hbar(\lieh)  \,\; \cong \;\,  \uRPhbmp  \quad ,   \qquad
   U_\hbar(\lieh) \,\widehat{\otimes}_\kh\, \uRPhnmp  \,\; \cong \;\,  \uRPhbmp  \cr
%
   \uRPhnm \,\widehat{\otimes}_\kh\, U_\hbar(\lieh) \,\widehat{\otimes}_\kh\, \uRPhnp  \,\; \cong \;\,  \uRPhg  \cr
   \uRPhnp \,\widehat{\otimes}_\kh\, U_\hbar(\lieh) \,\widehat{\otimes}_\kh\, \uRPhnm  \,\; \cong \;\,  \uRPhg  }  $$
 (notation as in  Definition \ref{def: Mp-Uhgd})  induced by multiplication in  $ \uRPhg \, $.
\end{theorem}

\begin{proof}
 This is a direct consequence of  Theorem \ref{thm: Uotimes-cong-uRPhg}  above.
\end{proof}

\medskip

\subsection{Hopf structure on FoMpQUEAs}  \label{subsec: Hopf struct x FoMpQUEAs}
 {\ }
 \vskip11pt
   We introduce now on our FoMpQUEAs a structure of topological Hopf algebra.
   Another proof of its existence will follow from an alternative construction
   (cf.\  \S \ref{sec: constr-FoMpQUEAs}).  To begin with, we explain the link between our ``formal'' and the ``polynomial'' one.

\vskip11pt

\begin{obs}  \label{obs: form-MpQUEAs_vs_pol-MpQUEAs}
 The usual, formal QUEA  $ U_\hbar(\lieg) $  by Drinfeld has a ``polynomial'' sibling  $ U_q(\lieg) $
 introduced by Jimbo and Lusztig: the latter is a (Hopf) subalgebra of the former, over the subring
 $ \, \Bbbk\big[\, q \, , q^{-1} \,\big] \, $  of  $ \kh \, $,  \,with  $ \, q^{\pm 1} := e^{\,\pm\hbar} \, $.
%
%
                                                            \par
   Our notion of  {\sl formal\/}  multiparameter QUEA is explicitly tailored so that this parallelism
   extend to the multiparameter setting, linking our formal multiparameter  $ \uRPhg $  with a
   ``polynomial'' multiparameter QUEA  $ U_{\mathbf{q}}(\lieg) $  as in  \cite{HPR}  or  \cite{GG1}.
                                                            \par
   To see this, we consider a matrix  $ P $  of Cartan type, whose associated Cartan matrix is  $ A $,
   and a \textsl{split\/}  realization  $ \, \R = \big(\, \lieh \, , \Pi \, , \Pi^\vee \,\big) \, $  of it: for this the algebra  $ \uRPhg $
   is defined and we begin by modifying the presentation of the latter.
   First, replace each  $ F_i $  by  $ \, \dot{F}_i := q_i F_i \, $  for all  $ \, i \in I \, $.
   Then the fourth relation in  \eqref{eq: comm-rel's_x_uPhg}  reads
$$
   E_i \, \dot{F}_j \, - \, \dot{F}_j \, E_i  \,\; = \;\,
   \delta_{i,j} \, q_{ii} \, {{\; e^{+\hbar \, T_i^+} - \, e^{-\hbar \, T_i^-} \;} \over {\; q_{ii} - \, 1 \;}}
   \quad \qquad  \forall \;\; i \, , j \in I  \quad
$$
                                                                \par
   Second, using the relation
 $ \; {\bigg( {\displaystyle {n \atop k}} \bigg)}_{\!\!q_{ii}} \! = \, q_i^{\,k\,(n-k)} {\bigg[\, {\displaystyle {n \atop k}} \,\bigg]}_{q_i} \; $
 (cf.\ \S \ref{q-numbers})  along with the identity  $ \; q_{ij} \, q_{ji} = q_{ii}^{\,a_{ij}} \; $
 --- which follows from the assumption  $ \, P + P^{\scriptscriptstyle t} = 2\,D\,A \, $  and definitions, see
 Definition \ref{def: params-x-Uphgd},  we can re-write the two last relations in  \eqref{eq: comm-rel's_x_uPhg}  as
$$
\begin{aligned}
   \sum\limits_{k = 0}^{1-a_{ij}} (-1)^k {\left( { 1-a_{ij} \atop k }
\right)}_{\!\!q_{ii}} q_{ii}^{k \choose 2} q_{ij}^{\,k} \, E_i^{1-a_{ij}-k} E_j E_i^k  \; = \;  0   \quad \qquad   (\, i \neq j \,)  \\
   \sum\limits_{k = 0}^{1-a_{ij}} (-1)^k {\left( { 1-a_{ij} \atop k }
\right)}_{\!\!q_{ii}} q_{ii}^{k \choose 2} q_{ij}^{\,k} \, \dot{F}_i^k \dot{F}_j \dot{F}_i^{1-a_{ij}-k}  \; = \;  0
\quad \qquad   (\, i \neq j \,)
\end{aligned}
$$
   \indent   With this reformulation,  $ \uRPhg $  has the following, alternative presentation:
   it is the unital, associative, topological,  $ \hbar $--adically  complete  $ \kh $--algebra
   generated by the  $ \kh $--submodule  $ \lieh $  and the elements
 $ \, E_i \, $,  $ \dot{F}_i \, $  (for all  $ \, i \in I \, $),
  with relations
\begin{equation}  \label{eq: multip-comm-rel's_x_uPhgd_ALTERN}
 \begin{aligned}
   T \, E_j \, - \, E_j \, T  \, = \,  +\alpha_j(T) \, E_j  \;\; ,  \qquad  T \, \dot{F}_j \, - \, \dot{F}_j \, T  \; = \;
   -\alpha_j(T) \, \dot{F}_j   \hskip51pt  \\
   \hskip-13pt   T' \, T''  \; = \;  T'' \, T'  \;\; ,  \;\quad   E_i \, \dot{F}_j \, - \, \dot{F}_j \, E_i  \; = \;
   \delta_{i,j} \, q_{ii} \, {{\; e^{+\hbar \, T_i^+} - \, e^{-\hbar \, T_i^-} \;} \over {\; q_{ii} - \, 1 \;}}   \quad  (\, i, j \in I \,)  \\
   \sum\limits_{k = 0}^{1-a_{ij}} (-1)^k {\left[ { 1-a_{ij} \atop k }
\right]}_{\!q_i} q_{ij}^{+k/2\,} q_{ji}^{-k/2} \, E_i^{1-a_{ij}-k} E_j E_i^k  \; = \;  0   \qquad  (\, i \neq j \,)   \hskip25pt  \\
   \sum\limits_{k = 0}^{1-a_{ij}} (-1)^k {\left[ { 1-a_{ij} \atop k }
\right]}_{\!q_i} q_{ij}^{+k/2\,} q_{ji}^{-k/2} \, \dot{F}_i^k \dot{F}_j \dot{F}_i^{1-a_{ij}-k}  \; = \;  0   \qquad  (\, i \neq j \,)   \hskip25pt
 \end{aligned}
\end{equation}
 for all  $ \, T , T' , T'' \in \lieh \, $,  $ \, i, j \in I \, $.
%
%

%
 \vskip3pt

   Now we set $ \, q^{\pm 1} := e^{\,\pm\hbar} \in \kh \, $ and consider   the  $ \Bbbk $--subalgebra
   $ \, \Bbbk\big[\, q \, , q^{-1} \,\big] \, $  of  $ \kh \, $,  \,and the elements
   $ \; K_i^{\pm 1} := e^{\pm\hbar\,T_i^+} \, $,  $ \; L_i^{\pm 1} := e^{\mp\hbar\,T_i^-} \; $  ($ \, i \in I \, $)  in  $ \, \uRPhg \, $.
   Then in the scalar extension
 $ \; \URPhg \, := \, \kq \operatorname{\otimes}\limits_{\Bbbk[q\,,\,q^{\pm 1}]} \uRPhg \, $,
 \,we slightly abuse notation by writing  $ \, X := 1 \otimes X \, $  for any  $ \, X \in \uRPhg \, $,  \,and we define
 $ U_{\mathbf{q}}(\lieg) $
 to be the unital  $ \Bbbk(q) $--subalgebra  of  $ \URPhg $  generated by
 $ \, {\big\{ K_i^{\pm 1} , L_i^{\pm 1} , E_i \, , \dot{F}_i \,\big\}}_{i \in I} \, $.
 By construction and by  \eqref{eq: multip-comm-rel's_x_uPhgd_ALTERN},
 we can describe  $ U_{\mathbf{q}}(\lieg) $  as being the unital, associative algebra over  $ \kq $
 with generators  $ \, K_i^{\pm 1} \, $,  $ L_i^{\pm 1} \, $,  $ \, E_i \, $,  $ \dot{F}_i \, $  (for all  $ \, i \in I \, $)  and relations
\begin{equation}  \label{eq: comm-rel's_x_Uqgd}
 \begin{aligned}
   K_i^{\pm 1} K_j^{\pm 1}  \, = \,  K_j^{\pm 1} K_i^{\pm 1}  \; ,
 \qquad  K_i^{\pm 1} L_j^{\pm 1}  \, = \,  L_j^{\pm 1} K_i^{\pm 1}  \; ,
 \qquad  L_i^{\pm 1} L_j^{\pm 1}  \, = \,  L_j^{\pm 1} L_i^{\pm 1}  \\
   K_i^{+1} K_i^{-1}  = \,  1  \, = \,  K_i^{-1} K_i^{+1}  \quad  ,
 \qquad  L_i^{+1} L_i^{-1}  = \,  1  \, = \,  L_i^{-1} L_i^{+1}    \quad \qquad  \\
   K_i^{\pm 1} \, E_j \, K_i^{\mp 1}  \, = \;  q_{i,j}^{\,\pm 1} \, E_j  \quad  ,
 \qquad  L_i^{\pm 1} \, E_j \, L_i^{\mp 1}  \, = \;  q_{j,i}^{\,\mp 1} \, E_j    \quad \quad \qquad  \\
   K_i^{\pm 1} \, \dot{F}_j \, K_i^{\mp 1}  \, = \;  q_{i,j}^{\,\mp 1} \, \dot{F}_j  \quad  ,
 \qquad  L_i^{\pm 1} \, \dot{F}_j \, L_i^{\mp 1}  \, = \;  q_{j,i}^{\,\pm 1} \, \dot{F}_j    \;\quad \quad \qquad  \\
   E_i \, \dot{F}_j \, - \, \dot{F}_j \, E_i  \; = \;  \delta_{i,j} \, q_{ii} \, {{\; K_i - \, L_i \;} \over {\; q_{ii} - \, 1 \;}}   \qquad \qquad \qquad  \\
   \sum\limits_{k = 0}^{1-a_{ij}} (-1)^k {\left( { 1-a_{ij} \atop k }
\right)}_{\!\!q_{ii}} q_{ii}^{k \choose 2} q_{ij}^k \, E_i^{1-a_{ij}-k} E_j E_i^k  \; = \;  0   \quad \qquad   (\, i \neq j \,)  \\
   \sum\limits_{k = 0}^{1-a_{ij}} (-1)^k {\left( { 1-a_{ij} \atop k }
\right)}_{\!\!q_{ii}} q_{ii}^{k \choose 2} q_{ij}^k \, \dot{F}_i^k \dot{F}_j \dot{F}_i^{1-a_{ij}-k}  \; = \;  0   \quad \qquad   (\, i \neq j \,)
 \end{aligned}
\end{equation}
   \indent   Next we consider the scalar extension  $ \; \UU_\hbar(\lieh) := \kq \operatorname{\otimes}\limits_{\Bbbk[q\,,\,q^{\pm 1}]} \uhh \; $
   of  $ \, \uhh \, $   --- cf.\ Definition \ref{def: Mp-Uhgd}{\it (b)}  ---   which clearly embeds as a ``Cartan subalgebra'' into
   $ \; \URPhg \, := \, \kq \operatorname{\otimes}\limits_{\Bbbk[q\,,\,q^{\pm 1}]} \uRPhg \; $;  \,let also  $ U_{\mathbf{q}}(\lieh) $  be the  $ \kq $--subalgebra
   ---   inside  $ U_{\mathbf{q}}(\lieg) \, $,  $ U_\hbar(\lieh) $  and  $ \UU_\hbar(\lieh) $  ---   generated by all the  $ K_i^{\pm 1} $'s  and all the  $ L_i^{\pm 1} $'s.
   Note that inside  $ \UU_\hbar(\lieh) $  we have  $ \; \pm T_i^+ = \hbar^{-1} \log\big(K_i^{\pm 1}\big) \; $  and  $ \; \pm T_i^- = \hbar^{-1} \log\big(L_i^{\pm 1}\big) \; $.
   By construction, both  $ \UU_\hbar(\lieh) $  and  $ U_{\mathbf{q}}(\lieg) $  are modules (on the right and on the left, respectively) for the commutative algebra
   $ U_{\mathbf{q}}(\lieh) \, $;  therefore, the  $ U_{\mathbf{q}}(\lieh) $--module
   $ \, \UU_\hbar(\lieh) \hskip-3pt \mathop{\otimes}\limits_{U_{\mathbf{q}}(\lieh)} \hskip-3pt U_{\mathbf{q}}(\lieg) \, $  is well-defined.
%
%
 Finally, the  $ \hbar $--adic  completion of the latter
 actually identifies with its closure inside  $ \, \URPhg \, $,  which is exactly all of  $ \, \URPhg \, $:
 \,in a nutshell, we have a (complete) tensor product factorization
 $ \; \URPhg = \UU_\hbar(\lieh) \hskip-2pt \mathop{\widehat{\otimes}}\limits_{U_{\mathbf{q}}(\lieh)} \hskip-2pt U_{\mathbf{q}}(\lieg) \; $.
\end{obs}

\vskip7pt

   The previous observation is our bridge to achieve the key point
   about the notion of formal multiparameter QUEA, that is the following:

\vskip13pt

\begin{theorem}  \label{thm: form-MpQUEAs_are_Hopf}
 Every FoMpQUEA  $ \uRPhg $  as in  Definition \ref{def: Mp-Uhgd}
 bears a structure of topological Hopf algebra over\/  $ \kh $
 --- with coproduct taking values into the  $ \hbar $--adically  completed tensor product
 $ \; \uRPhg \!\mathop{\widehat{\otimes}}\limits_\kh\! \uRPhg \; $
 --- given by  ($ \; \forall \; T \in \lieh \, $,  $ \, \ell \in I $)
\begin{equation}  \label{eq: coprod_x_uPhg}
 \begin{aligned}
   \Delta \big(E_\ell\big)  \, = \,  E_\ell \otimes 1 \, + \, e^{\hbar \, T_\ell^+} \otimes E_\ell  \\
   \Delta\big(T\big)  \, = \,  T \otimes 1 \, + \, 1 \otimes T   \hskip25pt  \\
   \Delta\big(F_\ell\big)  \, = \, F_\ell \otimes e^{-\hbar \, T_\ell^-} \, + \, 1 \otimes F_\ell
 \end{aligned}   \qquad
\end{equation}
%
%
 \vskip-9pt
\begin{eqnarray}
    \epsilon\big(E_\ell\big) \, = \, 0  \;\;\; ,  \qquad  &   \;\quad
    \epsilon\big(T\big) \, = \, 0  \;\;\; ,  &
    \;\qquad  \epsilon\big(F_\ell\big) \, = \, 0  \;\;\;\;  \label{eq: counit_x_uPhg}  \\
   \SS\big(E_\ell\big)  \, = \,  - e^{-\hbar \, T_\ell^+} E_\ell  \;\; ,  &   \;\;\;
   \SS\big(T\big)  \, = \,  - T   \;\; ,  &
   \;\;\;
   \SS\big(F_\ell\big)  \, = \,  - F_\ell \, e^{+\hbar \, T_\ell^-}   \qquad
  \label{eq: antipode_x_uPhg}
\end{eqnarray}
\end{theorem}

\pf
 We provide hereafter two proofs; a third, independent one will follow from another approach, that is detailed in
 \S \ref{sec: constr-FoMpQUEAs}  later on.
 \vskip5pt
   \textit{$ \underline{\text{First Proof}} $:}\,  Via a direct approach, the proof is a matter of computation.
   First, in the free, topological,  $ \hbar $--adically  complete, unital, associative  $ \kh $--algebra
   $ \mathcal{F}_\R $  generated by the  $ \kh $--submodule  $ \lieh $  together with the  $ E_i $'s  and the  $ F_i $'s,
   the formulas
   \eqref{eq: coprod_x_uPhg}  and  \eqref{eq: counit_x_uPhg}
 define a structure of (topological) bialgebra.
   Then one has to check that the two-sided ideal  $ \I_\R $ in  $ \mathcal{F}_\R $
   generated by relations  \eqref{eq: comm-rel's_x_uPhg}  is a bi-ideal for that bialgebra.
   Second, one has to check that the map  $ \; \SS : \uRPhg \longrightarrow \uRPhg^{\op} \; $
   defined on  $ \; \mathcal{F}_\R \Big/ \I_\R  =: \uRPhg \; $  by the second line in  \eqref{eq: antipode_x_uPhg}
   is an algebra anti-morphism with the ``right'' properties for the antipode map on the generators.
   In all cases, computations are the same as for  \cite[Definition-Proposition 6.5.1]{CP}
   which treats Drinfeld's special case with  $ \lieh $  of minimal rank.
   It is worth stressing, though, a key feature of this generalized result: namely, the assumption that  $ P $
   be of Cartan type is what one uses to prove that the quantum Serre's relations
   --- i.e., the last two relations in  \eqref{eq: comm-rel's_x_uPhg}  ---
   are mapped into
   $ \, \mathcal{F}_{\!\R} \,\widehat{\otimes}\, \I_\R + \I_\R \,\widehat{\otimes}\, \mathcal{F}_\R \, $
   --- where  $ \, \widehat{\otimes} \, $  denotes the  $ \hbar $--adic  completion of the algebraic tensor product ---
   by the given coproduct on  $ \mathcal{F}_\R \, $.  Actually, one shows that the elements  $ E_{i,j}$  and  $ F_{i,j} $
   represented by the left hand side of these equalities are skew-primitives, namely
 $ \; \com(E_{i,j}) = E_{i,j} \otimes 1 + e^{+(1-a_{ij}) \, \hbar\,T_i^+ + \hbar\,T_j^+} \otimes E_{i,j} \; $
 and similarly
 $ \; \com(F_{i,j}) = F_{i,j} \otimes e^{-(1-a_{ij}) \, \hbar\,T_i^- - \hbar\,T_j^-} + 1 \otimes F_{i,j} \; $.
 \vskip5pt
   \textit{$ \underline{\text{Second Proof}} $:}\,  This alternative method goes through an indirect approach.
 \vskip3pt
   First of all,  \textit{we assume the realization  $ \R $  to be  \textsl{split}}.  In this case,
   Observation \ref{obs: form-MpQUEAs_vs_pol-MpQUEAs}  provides a factorization
 $ \; \URPhg \, := \, \kq \operatorname{\otimes}\limits_{\Bbbk[q\,,\,q^{\pm 1}]} \uRPhg \, =
 \, \UU_\hbar(\lieh) \hskip-2pt \mathop{\widehat{\otimes}}\limits_{U_{\mathbf{q}}(\lieh)} \hskip-2pt U_{\mathbf{q}}(\lieg) \; $.
                                                             \par
   Due to its presentation in  Observation \ref{obs: form-MpQUEAs_vs_pol-MpQUEAs},  our  $ \, U_{\mathbf{q}}(\lieg) \, $
   is a  ``multiparameter quantum group'' in the sense of  \cite{HPR}   --- cf.\ also  \cite{GG1},
   where such an example of multiparameter quantum group is referred to as being ``rational'',
   in that each  $ q_{ij} $  is a power of a single, common parameter  $ q \, $.
   The key point then is that any such (``polynomial'') multiparameter quantum group  $ \, U_{\mathbf{q}}(\lieg) \, $
   has a specific Hopf algebra structure   --- cf.  \cite{HPR}  and  \cite{Gar}  ---   given (for all  $ \, \ell \in I \, $)  by
\begin{equation}  \label{eq: Hopf-struct_Uqgd}
   \quad
 \begin{matrix}
   \Delta \big(E_\ell\big) \, = \, E_\ell \otimes 1 \, + \, K_\ell \otimes E_\ell  \;\; ,
   &  \quad  \epsilon\big(E_\ell\big) \, = \, 0
   \;\; ,  \quad   &   \SS\big(E_\ell\big) \, = \, - K_\ell^{-1} E_\ell   \\
   \Delta\big(K_\ell^{\pm 1}\big) \, = \, K_\ell^{\pm 1} \otimes K_\ell^{\pm 1}  \;\; ,
   &   \quad  \epsilon\big(K_\ell^{\pm 1}\big) \, = \, 1
   \;\; ,  \quad   &   \SS\big(K_\ell^{\pm 1}\big) \, = \, K_\ell^{\mp 1}   \\
   \Delta\big(L_\ell^{\pm 1}\big) \, = \, L_\ell^{\pm 1} \otimes L_\ell^{\pm 1}  \;\; ,
   &  \quad  \epsilon\big(L_\ell^{\pm 1}\big) \, = \, 1
   \;\; ,  \quad   &   \SS\big(L_\ell^{\pm 1}\big) \, = \, L_\ell^{\mp 1}   \\
   \Delta\big(\dot{F}_\ell\big) \, = \, \dot{F}_\ell \otimes L_\ell \, + \, 1 \otimes \dot{F}_\ell  \;\; ,   &
   \quad  \epsilon\big(\dot{F}_\ell\big) \, = \, 0  \;\; ,  \quad   &   \SS\big(\dot{F}_\ell\big) \, = \, - \dot{F}_\ell \, L_\ell^{-1}
\end{matrix}
\end{equation}
                                                            \par
   Let now  $ \, \widetilde{U}_{\mathbf{q}}(\lieg) \, $  be the  $ \k\big[q\,,q^{-1}\big] $--subalgebra  of  $ \, U_{\mathbf{q}}(\lieg) \, $
   --- hence of  $ \uRPhg $  ---   generated by  $ \, {\big\{ E_i \, , K_i^{\pm 1} , L_i^{\pm 1} , \dot{F}_i \big\}}_{i \in I} \, $:
   it is  $ \hbar $--adic  dense in  $ \, U_{\mathbf{q}}(\lieg) \, $,
   and restricting the coproduct  $ \Delta $  to
   $ \, \widetilde{U}_{\mathbf{q}}(\lieg) \, $  yields
   $ \; \Delta \big( \widetilde{U}_{\mathbf{q}}(\lieg) \big) \, \subseteq \,
   \widetilde{U}_{\mathbf{q}}(\lieg) \hskip-3pt \mathop{\otimes}\limits_{\k[q,q^{-1}]} \hskip-3pt \widetilde{U}_{\mathbf{q}}(\lieg) \; $.
   Therefore, since  $ \, K_i = \exp\big(\!+\hbar\,T_i^+\big) \, $  and  $ \, L_i = \exp\big(\!-\hbar\,T_i^-\big) \, $,
   \,there exists one and only one way to extend (continuously)  $ \Delta{\Big|}_{\widetilde{U}_{\mathbf{q}}(\lieg)} $  to  $ \uRPhg $,  which gives a map
\begin{equation}  \label{eq: Delta_uRPhg}
   \Delta \,:\, \uRPhg \relbar\joinrel\relbar\joinrel\longrightarrow \uRPhg \hskip3pt \widehat{\otimes} \hskip3pt \uRPhg
\end{equation}
 described by  \eqref{eq: Hopf-struct_Uqgd}  plus the additional constraint that all the  $ T_i^\pm $'s  be primitive.
                                                            \par
   Since the original map  $ \Delta $  on  $ U_{\mathbf{q}}(\lieg) $  obey the axioms of a coproduct,
   the same holds true for the map in  \eqref{eq: Delta_uRPhg}  as well   --- though in a topological framework.
   With similar arguments, we deal with counit and antipode map, so that we end up with a (topological)
   Hopf structure for  $ \, \uRPhg \, $,  uniquely induced from the one on  $ U_{\mathbf{q}}(\lieg) \, $.
   Tracking the whole construction, we eventually see that such a structure is described on generators
   by  \eqref{eq: coprod_x_uPhg},  \eqref{eq: counit_x_uPhg}  and  \eqref{eq: antipode_x_uPhg},  q.e.d.
 \vskip5pt
   Now we consider the case of a realization  $ \; \R \, := \, \big(\, \lieh \, , \Pi \, , \Pi^\vee \,\big) \; $  of any kind.
   By  Lemma \ref{lemma: split-lifting},
   we can then pick a  \textsl{split\/}  realization (of the same matrix  $ P $  as  $ \R \, $),  say
   $ \; \dot{\R} \, := \, \big(\, \dot{\lieh} \, , \dot{\Pi} \, , {\dot{\Pi}}^\vee \,\big) \; $  and an epimorphism of realizations
   $ \; \underline{\pi} : \dot{\R} \relbar\joinrel\twoheadrightarrow \R \; $.
   By functoriality  (cf.\ Proposition \ref{prop: functor_R->uRPhg}),
   such a  $ \underline{\pi} $  induces an epimorphism of FoMp\-QUEAs
   $ \; U_{\underline{\pi}} : U^{\,\dot{\R}}_{\!P,\hbar}(\lieg) \relbar\joinrel\twoheadrightarrow \uRPhg \; $,
   \,whose kernel  $ \Ker\big(U_{\underline{\pi}}\,\big) $  is generated by  $ \Ker(\pi) \, $,
   and the latter is central in  $ U^{\,\dot{\R}}_{\!P,\hbar}(\lieg) \, $;
   moreover, every element in  $ \Ker(\pi) $  is primitive.  Therefore  $ \Ker\big(U_{\underline{\pi}}\,\big) $
   is indeed a  \textsl{Hopf ideal\/}  in the Hopf algebra
   $ U^{\,\dot{\R}}_{\!P,\hbar}(\lieg) \, $,  \,hence  $ U^{\,\R}_{\!P,\hbar}(\lieg) $  automatically inherits via
   $ U_{\underline{\pi}} $  a quotient Hopf algebra structure from  $ U^{\,\dot{\R}}_{\!P,\hbar}(\lieg) \, $,
   which is again described by the formulas in the statement, q.e.d.
\epf

\vskip5pt

   The following is now immediate:

\vskip13pt

\begin{cor}  \label{cor: Hopf-struct_Cartan-&-Borel}
 The Cartan subalgebra  $ U_\hbar(\lieh) $  and the Borel subalgebras  $ \uRPhbp $  and  $ \uRPhbm $
 are actually (topological) Hopf subalgebras of  $ \, \uRPhg \, $,  their Hopf structure being described again via formulas
 \eqref{eq: coprod_x_uPhg}  and  \eqref{eq: antipode_x_uPhg}.
                                                                       \par
   In addition, when  $ \R $  is split we have that
   $ \; \uRPhbpm \, \cong \, U_\hbar(\lieh_\pm) \hskip-2pt \mathop{\widehat{\;\ltimes\,}}\limits_{U_{\mathbf{q}}(\lieh_\pm)}
   \hskip-1pt U_{\mathbf{q}}(\lieb_\pm) \; $,
   \,where  $ U_{\mathbf{q}}(\lieb_\pm) \, $
   is the multiparameter quantum group corresponding to the Borel subalgebras.
\end{cor}

\vskip7pt

   The next result follows at once from the second proof of  Theorem \ref{thm: form-MpQUEAs_are_Hopf}  above.

\vskip9pt

\begin{prop}  \label{prop: central-Hopf-extens_FoMpQUEAs}
 Let  $ \, \underline{\phi} : \R' \!\relbar\joinrel\relbar\joinrel\longrightarrow\! \R'' \, $
 be a morphism between realizations of a same matrix  $ \, P \in M_{n}\big( \kh \big) \, $.
 Then the morphism of unital topological  $ \, \kh $--algebras
 $ \, U_{\underline{\phi}} : U^{\,\R'}_{\!P,\hbar}(\lieg) \relbar\joinrel\relbar\joinrel\longrightarrow U^{\,\R''}_{\!P,\hbar}(\lieg) \, $
 granted by  Proposition \ref{prop: functor_R->uRPhg}  is indeed a morphism of (unital topological)  \textsl{Hopf}  $ \, \kh $--algebras.
 If we set  $ \; \liek := \Ker(\phi) \, $, then  $ \, U_\hbar(\liek) \, $  is a unital,  $ \hbar $--adically  complete  $ \, \kh $--subalgebra  of  $ U^{\,\R'}_{\!P,\hbar}(\lieg) $
 which is a central Hopf subalgebra, isomorphic to a quantum Cartan (in the sense of
 Definition \ref{def: Mp-Uhgd}\textit{(b)\/}), and  $ \; \Ker(U_{\underline{\phi}}) = U^{\,\R'}_{\!P,\hbar}(\lieg) \, {U_\hbar(\liek)}^+ \, $
--- where  $ \, {U_\hbar(\liek)}^+ \, $  is the augmentation ideal of  $ \, U_\hbar(\liek) \, $.
\par
   In particular, if  $ \, U_{\underline{\phi}} $  is an epimorphism, then
   $ \; U^{\,\R''}_{\!P,\hbar}(\lieg) \, \cong \, U^{\,\R'}_{\!P,\hbar}(\lieg) \Big/ U^{\,\R'}_{\!P,\hbar}(\lieg) \, {U_\hbar(\liek)}^+ \; $.
\end{prop}

\vskip9pt

\begin{exa}  \label{example P=DA}
 Fix  $ \, P := DA \, $, $ \, r := \rk\!\big(DA\big) \, $ and let  $ \, \hat{\R} := \big(\, \hat{\lieh} \, , \hat{\Pi} \, , \hat{\Pi}^\vee \big) \, $  and
 $ \, \R := \big(\, \lieh \, , \Pi \, , \Pi^\vee \big) \, $  be realizations of  $ \, DA  \, $,  where  $ \hat{\R} $
 is straight and split with  $ \, \rk\!\big(\hat{\lieh}\big) = 2\,(2\,n-r) \, $  while  $ \R $  is straight and small with
 $ \, \rk(\lieh) = 2\,n-r \, $;  more precisely, we assume  $ \, T_i^+ = T_i^- \, $ in  $ \R \, $,  for all  $ \, i \in I \, $.  With this setup,
 $ U^{\,\R}_{\!DA,\hbar}(\lieg) $  is the usual Drinfeld's QUEA
 $ U_\hbar\big(\lieg_{{}_A}\big) $  for the Kac-Moody algebra  $ \lieg_{{}_A} $
 associated with the Cartan matrix $ A $  as in  \S \ref{root-data_Lie-algs};
 in particular, its semiclassical limit is  $ U(\lieg) \, $.  Instead,  $ U^{\,\hat{\R}}_{\!DA,\hbar}(\lieg) $  has semiclassical limit  $ U(\liegd) \, $,
 with  $ \liegd $  the Manin double of  $ \, \lieg = \lieg_{{}_A} \, $  (cf.\ \S \ref{root-data_Lie-algs}).
 \vskip1pt
   Now, there exists a (non-unique, if  $ \, r \lneqq n \, $)  epimorphism
   $ \, \underline{\phi} : \hat{\R} \relbar\joinrel\relbar\joinrel\twoheadrightarrow \R \, $
   such that  $ \, \phi\big(\hat{T}_i^\pm\big) = T_i^\pm \, $  ($ \, i \in I \, $);  then
   $ \, \liez := \Ker(\phi) \subseteq \bigcap\limits_{j \in I}\Ker(\hat{\alpha}_j) \, $.
   Since every element of  $ \hat{\lieh} $  is primitive inside  $ U^{\,\hat{\R}}_{\!DA,\hbar}(\lieg) \, $,
   the subalgebra  $ U^{\,\hat{\R}}_\hbar(\liez) $  of  $ U^{\,\hat{\R}}_{\!DA,\hbar}(\lieg) $  generated by
   $ \liez $  is indeed a Hopf subalgebra; moreover, it is central in  $ U^{\,\hat{\R}}_{\!DA,\hbar}(\lieg) $
   because  $ \, \liez \subseteq \bigcap\limits_{j \in I}\Ker(\hat{\alpha}_j) \, $.
 Also,  $ U^{\,\hat{\R}}_\hbar(\liez) $  is the  $ \hbar $--adic  completion of the polynomial
  $ \kh $--algebra  over  $ \liez^* $,  so we can loosely think of it as being a ``quantum Cartan algebra'' of ``rank''  $ \, 2\,n-r \, $.
 \vskip1pt
   By Proposition \ref{prop: central-Hopf-extens_FoMpQUEAs},  $ \, \underline{\phi} \, $
   yields a (Hopf) epimorphism  $ \, U_{\underline{\phi}} : U^{\,\hat{\R}}_{\!DA,\hbar}(\lieg)
   \relbar\joinrel\twoheadrightarrow U^{\,\R}_{\!DA,\hbar}(\lieg) \, $;  by construction,
   the kernel of latter is the two-sided ideal generated by  $ \liez \, $,  that is
  $$  \Ker\big( U_{\underline{\phi}} \big)  \; = \;  U^{\,\hat{\R}}_{\!DA,\hbar}(\lieg) \,
  U^{\,\hat{\R}}_\hbar{(\liez)}^+ \, U^{\,\hat{\R}}_{\!DA,\hbar}(\lieg)  \; = \;
  U^{\,\hat{\R}}_{\!DA,\hbar}(\lieg) \, U^{\,\hat{\R}}_\hbar{(\liez)}^+  \, = \;
  U^{\,\hat{\R}}_\hbar{(\liez)}^+ \, U^{\,\hat{\R}}_{\!DA,\hbar}(\lieg),  $$
and we have that  $ \; U^{\,\R}_{\!DA,\hbar}(\lieg) \,\cong\, U^{\,\hat{\R}}_{\!DA,\hbar}(\lieg) \Big/ U^{\,\hat{\R}}_{\!DA,\hbar}(\lieg) \, {U^{\,\hat{\R}}_\hbar(\liez)}^+ \; $.
                                                                          \par
   Finally, if we deal instead with  $ \R $  small and  $ \hat{\R} $  split such that
   $ \, \rk\big(\hat{\lieh}\big) = 2\,n \, $  and  $ \, \rk(\lieh) = n \, $,  \,then  $ \, \lieg = \lieg_{{}_A} \, $
   has to be replaced by the  \textsl{derived\/}  algebra  $ \lieg' $  associated with  $ A \, $,  and  $ \liegd $
   by the Manin double of  $ \lieg' $.  The previous analysis then works again.
\end{exa}

\vskip5pt

\begin{rmk}  \label{rmk: generaliz-centr-Hopf-ext}
 Let us now take  \textsl{any\/}  matrix  $ P $  (of Cartan type), a realization  $ \R $  of it that is
minimal and small with  $ \, \rk(\lieh) = 2\,n-r \, $   --- with  $ \, r := \rk\!\big(P+P^{\,\scriptscriptstyle T}\big) \, $
--- and the associated FoMpQUEA  $ \uRPhg \, $;  then we can still find another realization  $ \dot{\R} $  of  $ P $
that is split with  $ \, \rk(\lieh) = 2\,(2\,n-r) \, $   and an epimorphism of realizations
$ \, \underline{\pi} : \dot{\R} \relbar\joinrel\relbar\joinrel\twoheadrightarrow \R \, $  so that
$ \, \liez_\pi := \Ker(\pi) \, $  is again free of rank  $ \, 2\,n-r \, $   ---
see  Lemma \ref{lemma: split-lifting}  \textit{and\/}  its proof.
Then the previous analysis   --- that was based upon
$ \, \phi : \hat{\R} \relbar\joinrel\relbar\joinrel\twoheadrightarrow \R \, $  and
$ \, U_{\underline{\phi}} : U^{\,\hat{\R}}_{\!DA,\hbar}(\lieg) \relbar\joinrel\relbar\joinrel\twoheadrightarrow
U^{\,\R}_{\!DA,\hbar}(\lieg) \, $
---   can be repeated now, step by step, basing instead upon
$ \, \underline{\pi} : \dot{\R} \relbar\joinrel\relbar\joinrel\twoheadrightarrow \R \, $
and the associated epimorphism  $ \, U_{\underline{\pi}} : U^{\,\dot{\R}}_{\!P,\hbar}(\lieg)
\relbar\joinrel\relbar\joinrel\twoheadrightarrow U^{\,\R}_{\!P,\hbar}(\lieg) \, $
of FoMpQUEAs: this leads to the sequence of Hopf algebra maps
\begin{equation}  \label{eq: centr-ex-seq}
   U^{\,\dot{\R}}_\hbar(\liez_\pi)
  \;\relbar\joinrel\relbar\joinrel\longrightarrow\; U^{\,\dot{\R}}_{\!P,\hbar}(\lieg)
 \;{\buildrel U_{\underline{\pi}} \over {\relbar\joinrel\relbar\joinrel\relbar\joinrel\longrightarrow}}\; U^{\,\R}_{\!P,\hbar}(\lieg)
\end{equation}
 where  $ U^{\,\dot{\R}}_\hbar(\liez_\pi) $  is the (central) subalgebra of  $ \, U^{\,\dot{\R}}_{\!P,\hbar}(\lieg) \, $
 genera\-ted by  $ \liez_\pi \, $,  which is again a ``quantum Cartan algebra'' of ``rank''  $ \, 2\,n-r \, $.
Again, we obtain that
 \vskip-9pt
  $$  U^{\,\R}_{\!P,\hbar}(\lieg)  \, \cong \,  U^{\,\dot{\R}}_{\!P,\hbar}(\lieg) \Big/ U^{\,\dot{\R}}_{\!P,\hbar}(\lieg) \, {U^{\,\dot{\R}}_\hbar(\liez_\pi)}^+  $$
 \vskip-1pt
   Therefore, the situation in general is much similar to what happens in the special,
   ``standard'' case of  $ \, P = DA \, $;
   what does actually change, indeed, is the explicit description of  $ \liez_\pi $
   --- with respect to that of  $ \liez \, $,  that was quite clear ---
   hence of the ``quantum Cartan algebra of rank  $ \, 2\,n-r \, $''  $ U^{\,\dot{\R}}_\hbar(\liez_\pi) \, $.
\end{rmk}

\vskip7pt

\begin{free text}  \label{minimal split}
 {\bf The case of
%
%
split (and) minimal FoMpQUEAs.}
 We consider now the special case of a FoMpQUEA  $ \uRPhg $
 --- as defined in  Definition \ref{def: Mp-Uhgd}  ---   for which the realization  $ \R $  is
%
%
split and minimal   --- in short,  \textsl{split minimal}.
 In this case, it follows by definition that  $ \uRPhg $  can be described as follows:
 it is the unital, associative, topological,  $ \hbar $--adically  complete algebra over  $ \kh $  with generators
 $ \, E_i \, $,  $ T_i^+ $,  $ T_i^- $,  $ F_i \, $  (for all  $ \, i \in I \, $)
  and relations
\begin{equation*}
%
%
 \begin{aligned}
%
%
   T_i^+ \, E_j \, - \, E_j \, T_i^+  \; = \;  +p_{i,j} \, E_j  \quad  ,
 \qquad  T_i^- \, E_j \, - \, E_j \, T_i^-  \; = \;  +p_{j,i} \, E_j   \qquad \qquad \quad  \\
   T_i^+ \, F_j \, - \, F_j \, T_i^+  \; = \;  -p_{i,j} \, F_j  \quad  ,
 \qquad  T_i^- \, F_j \, - \, F_j \, T_i^-  \; = \;  -p_{j,i} \, F_j   \qquad \qquad \quad  \\
   T_i^\pm \, T_j^\pm  \, = \, T_j^\pm \, T_i^\pm  \;\, ,   \!\quad  E_i \, F_j \, - \, F_j \, E_i  \; = \;  \delta_{i,j} \, {{\; e^{+\hbar \, T_i^+} - \, e^{-\hbar \, T_i^-} \;}
   \over {\; q_i^{+1} - \, q_i^{-1} \;}}  \;\; ,  \!\quad  T_i^\pm \, T_j^\mp  \, = \,  T_j^\mp \, T_i^\pm  \\
%
%
   {\textstyle \sum\limits_{k=0}^{1-a_{ij}}} {(-1)}^k {\left[ { 1-a_{ij} \atop k } \right]}_{\!q_i} q_{ij}^{+k/2\,} q_{ji}^{-k/2} \, E_i^{1-a_{ij}-k} E_j E_i^k  \; = \;  0
   \qquad \qquad   \forall \;\; i \neq j   \qquad  \\
   {\textstyle \sum\limits_{k=0}^{1-a_{ij}}} {(-1)}^k {\left[ { 1-a_{ij} \atop k } \right]}_{\!q_i} q_{ij}^{+k/2\,} q_{ji}^{-k/2} \, F_i^k F_j F_i^{1-a_{ij}-k}  \; = \;  0
   \qquad \qquad   \forall \;\; i \neq j   \qquad
 \end{aligned}
\end{equation*}
 and bearing the (topological) Hopf  $ \kh $--algebra  structure given (for all  $ \, \ell \in I \, $)  by
  $$  \displaylines{
   \Delta \big(E_\ell\big)  \, = \,  E_\ell \otimes 1 \, + \, e^{+\hbar \, T_\ell^+} \otimes E_\ell  \;\; ,
   \qquad  \epsilon\big(E_\ell\big) \, = \, 0  \;\; ,  \qquad  \SS\big(E_\ell\big)  \, = \,  - e^{-\hbar \, T_\ell^+} E_\ell  \cr
   \Delta \big(T^\pm_\ell\big)  \, = \,  T^\pm_\ell \otimes 1 \, + \, 1 \otimes T^\pm_\ell  \;\; ,
   \;\;\qquad  \epsilon\big(T^\pm_\ell\big) \, = \, 0  \;\; ,  \;\;\qquad  \SS\big(T^\pm_\ell\big)  \, = \,  - T^\pm_\ell  \;\;\; \cr
   \Delta \big(F_\ell\big)  \, = \,  F_\ell \otimes e^{-\hbar \, T_\ell^-} \, + \, 1 \otimes F_\ell  \;\; ,
   \qquad  \epsilon\big(F_\ell\big) \, = \, 0  \;\; ,  \qquad  \SS\big(F_\ell\big)  \, = \,  - F_\ell \, e^{+\hbar \, T_\ell^-}  }  $$
 \vskip5pt
   \textsl{Note\/}  then that  \textit{in this case  $ \uRPhg $  depends only on the matrix  $ P $}.
   Indeed, in the spirit of  Observation \ref{obs: form-MpQUEAs_vs_pol-MpQUEAs},
   in this special case the  \textsl{formal\/}  MpQUEA  $ \uRPhg $
   is just a ``logarithmic version'' of the  \textsl{polynomial\/}  MpQUEA  $ U_{\mathbf{q}}(\lieg) $
   in  Observation \ref{obs: form-MpQUEAs_vs_pol-MpQUEAs}.
 \vskip5pt
   In addition, in this case the FoMpQUEA  $ \uRPhg $  admits an alternative,
   somewhat significant presentation, as follows. Consider inside  $ \uRPhg $  the vectors
 $ \,  S_i := 2^{-1} \big(\, T_i^+ \! + T_i^- \big) \, $
 and
 $ \, \varLambda_i := 2^{-1} \big(\, T_i^+ \! - T_i^- \big) \, $
 --- for all  $ \, i \in I \, $;  these clearly form yet another  $ \kh $--basis  of  $ \lieh \, $.
 Then  $ \uRPhg $  admits the obvious presentation given by construction and taking into account that
 $ \, {\big\{ T_i^+ , T_i^- \big\}}_{i \in I} \, $  is a  $ \kh $--basis  of  $ \lieh \, $,  but also the following, alternative one:
 it is the  $ \hbar $--adically  complete, unital  $ \kh $--algebra  with generators
 $ E_i $,  $ F_i $,  $ S_i $,  $ \varLambda_i $  ($ \, i \in I \, $)  and relations
  $$  \displaylines{
   [S_i\,,E_j]  \, = \,  \frac{+(\,p_{ij}+p_{ji})}{2} \, E_j  \, = \,  +d_i\,a_{ij} \, E_j  \; ,
  \;\;  [S_i\,,F_j]  \, = \,  \frac{-(\,p_{ij}+p_{ji})}{2} \, F_j  \, = \,  -d_i\,a_{ij} \, F_j  \cr
   [\,\varLambda_i\,,E_j\,]  \, = \,  +(\,p_{ij}-p_{ji}) \, E_j  \;\; ,   \qquad  [\,\varLambda_i\,,F_j\,]  \,
   = \,  -(\,p_{ij}-p_{ji}) \, F_j  \cr
   [\,\varLambda_i\,,S_j\,]  \, = \,  0  \;\; ,   \qquad  E_i \, F_j \, - \, F_j \, E_i  \,\; = \;\,
   \delta_{i,j} \, e^{+\hbar \, \varLambda_i^+} \, {{\; e^{+\hbar \, S_i} - \, e^{-\hbar \, S_i} \;} \over {\; q_i^{+1} - \, q_i^{-1} \;}}  \cr
   {\textstyle \sum\limits_{k=0}^{1-a_{ij}}} {(-1)}^k {\left[ { 1-a_{ij} \atop k } \right]}_{\!q_i} q_{ij}^{+k/2\,} q_{ji}^{-k/2} \, E_i^{1-a_{ij}-k} E_j E_i^k  \; = \;  0
   \qquad \qquad  \forall \;\; i \neq j   \qquad  \cr
   {\textstyle \sum\limits_{k=0}^{1-a_{ij}}} {(-1)}^k {\left[ { 1-a_{ij} \atop k } \right]}_{\!q_i} q_{ij}^{+k/2\,} q_{ji}^{-k/2} \, F_i^k F_j F_i^{1-a_{ij}-k}  \; = \;  0
   \qquad \qquad  \forall \;\; i \neq j   \qquad  }  $$
 Moreover, the Hopf structure of  $ \uRPhg $  is then described (for all  $ \, i \in I \, $)  by
  $$  \begin{aligned}
   \Delta \big(E_\ell\big)  &  \, = \,  E_\ell \otimes 1 \, + \, e^{+\hbar \, \varLambda_\ell} e^{+\hbar \, S_\ell} \! \otimes E_\ell  \\
   \Delta\big(S_\ell\big)  \, = \,  S_\ell \otimes 1 \,  &  + \, 1 \otimes S_\ell  \quad ,
   \qquad  \Delta\big(\varLambda_\ell\big)  \, = \,  \varLambda_\ell \otimes 1 \, + \, 1 \otimes \varLambda_\ell  \\
   \Delta\big(F_\ell\big)  &  \, = \, F_\ell \otimes e^{-\hbar \, S_\ell} e^{+\hbar \, \varLambda_\ell} + \, 1 \otimes F_\ell
 \end{aligned}  $$
   \indent   In particular, this implies that the  $ \hbar $--adically  complete, unital subalgebra
   $ U_\hbar(\lieb_+) \, $,  rep.\  $ U_\hbar(\lieb_-) \, $,  of  $ \uRPhg $  generated by all the
   $ E_i $'s,  resp.\ all the  $ F_i $'s,  and all the  $ S_i $'s  is isomorphic to the ``standard'' positive, resp.\ negative,
   Borel subalgebra in the derived version of Drinfeld's QUEA  $ U_\hbar(\lieg) \, $.
   On the other hand, both subalgebras  $ U_\hbar(\lieb_\pm) $  are  {\sl not Hopf\/}  subalgebras inside
   $ \uRPhg \, $,  contrary to what happens in Drinfeld's setup.
\end{free text}

\bigskip

\subsection{Further results on FoMpQUEAs}  \label{subsec: further-results}
{\ }

\smallskip

   We present now a few more techniques, which provide alternative proofs for our results about the structure of quantum nilpotent, Cartan and Borel subalgebras, as well as the triangular decomposition results.  This mainly follows in the footpath of a standard strategy, already used for one-parameter QUEA's.

\vskip11pt

\begin{free text}
 {\bf Preformal multiparameter QUEAs and special representations.}\,
 We introduce now some ``preliminary versions of FoMpQUEAs'',
 essentially defined like the FoMpQUEAs but for dropping from their definition the quantum Serre relations.
 These ``pre-FoMpQUEAs'' will be a key tool in our analysis, as well as some
 special representations of them that we also introduce presently.
 \vskip7pt
   Let $ U $  be the unital, associative $ \kh $-algebra  generated by the  $ \kh $--submodule  $ \lieh $  together with elements
 $ \, E_i \, $,  $ F_i \, $  (for all  $ \, i \in I \, $),
subject to the same relations as in \eqref{eq: comm-rel's_x_uPhg},  \textsl{except the last two}
--- the  \textsl{quantum Serre relations}.  Let  $ \, \widetilde{U}^{\,\R}_{\!P,\hbar}(\lieg) \, $
be the  $ \hbar $--adic  completion of  $ U $.  From the proof of  Theorem \ref{thm: form-MpQUEAs_are_Hopf},
it follows at once that $ \, \widetilde{U}^{\,\R}_{\!P,\hbar}(\lieg) \, $ is a topological Hopf algebra over  $ \kh \, $.
 \vskip11pt
   Let  $ \, V := \bigoplus_{i \in I} \kh.v_i \, $  be the free  $ \kh $--module  with basis
   $ \, {\{ v_i \}}_{i \in I} \, $,  let  $ \, T_\hbar(V) \, $  be the tensor algebra of  $ \, V $  over  $ \, \kh \, $,
   and let  $ \, \widehat{T}_\hbar(V) \, $  be the  $ \hbar $--adic  completion of  $ \, T_\hbar(V) \, $.
   Then,  $ \, T_\hbar(V) \, $  is a  free  $ \kh $--module  with basis  $ \, {\{ v_J \}}_{J \in \mathcal{J}} \, $,
   \,where  $ \mathcal{J} $  is the set of all finite sequences of elements in  $ I $
  and  $ \, v_J := v_{j_1} \otimes \cdots \otimes v_{j_r} \, $   --- or simply  $ \, v_J := v_{j_1} \cdots v_{j_r} \, $
  ---   is standard monomial notation for all  $ \, J := (j_1,\dots,j_r) \in \mathcal{J} \, $.
 For  $ \, J =(j_1,\ldots,j_r) \in \mathcal{J} \, $  and  $ \, 1 \leq k \leq r \, $,  \,write
 $ \, \widehat{J}_k := (j_1,\dots,j_{k-1},j_{k+1},\ldots,j_r) \, $,  $ \, J_k := (j_{k+1},\ldots, j_r) \, $  and
 $ \, \alpha_J := \sum_{\ell=1}^r \alpha_{j_\ell} \, $.

\end{free text}

\vskip9pt

\begin{lema}  \label{lemma: tensor repr x utildeRPhg}
 For every  $ \, \lambda \in \lieh^* \, $,  there exists a unique representation of
 $ \, U \, $  onto  $ \, T_\hbar(V) $
 --- which is then denoted  $ \, T_\hbar^\lambda(V) $  ---   such that, for all  $ \, J =(j_1,\dots,j_r) \in \mathcal{J} \, $
 \vskip-11pt
  $$  \displaylines{
   F_i\,.\,v_J  \; = \;  v_{(i,\,J)}  \quad ,   \qquad  T\,.\, v_J  \; = \;  \big( \lambda(T) - \alpha_J(T) \big) v_J  \cr
   E_i\,.\,v_J  \; =  \sum_{1\leq \ell\leq r \, , \; j_\ell = i} \hskip-5pt \frac{\; q^{+\lambda(T_i^+)
   - \alpha_{J_\ell}(T_i^+)} \! - q^{-\lambda(T_i^-) + \alpha_{J_\ell}(T_i^-)} \;}{q_i^{+1} - q_i^{-1}} \; v_{\widehat{J}_\ell}  }  $$
 In addition, this representation   --- of  $ \, U $  onto  $ \, T_\hbar(V) $  ---   induces by continuity a unique representation
 $ \, \widetilde{U}^{\,\R}_{\!P,\hbar}(\lieg) \, $  onto  $ \, \widehat{T}_\hbar(V) \, $,  \,which is hereafter denoted by  $ \, \widehat{T}_\hbar^\lambda(V) \, $.
\end{lema}

\begin{proof}
 We must prove that the equalities above do endow  $ T_\hbar(V) $  with a structure of  $ U $--module:
 to this end, let us performe a quick check to show that such an action is well-defined.
 For every  $ \, T, T' \in \lieh \, $  we have
  $$
  T'.\big( T.\,v_J\big)  =  \big( \lambda(T) - \alpha_J(T) \big) T'.v_J  \, =
  \big( \lambda(T) - \alpha_J(T) \big) \big( \lambda(T') - \alpha_J(T') \big) \, v_J  \, = \,  T.\big( T'.v_J \big)
  $$
   \indent   Take now  $ \, T\in \lieh \, $  and  $ \, F_i \, $  with  $ \, i \in I \, $.
   Then
  $$
  \displaylines{
   T.\big( F_i\,.v_J\big) - F_i\,.\big(T.v_J\big)  \; = \;  T.v_{(i,\,J)} - F_i\,.\big( \big( \lambda(T) - \alpha_J(T) \big) \, v_J \big)  \; =   \hfill  \cr
   = \;  \big( \lambda(T) - \alpha_{(i,\,J)}(T) - \lambda(T) + \alpha_{J}(T)\big)v_{(i,J)}  \; = \; \big(\! -\alpha_{(i,J)}(T) + \alpha_J(T) \big) \, v_{(i,J)}  \; =  \cr
   \hfill   = \;  -\alpha_i(T) \, v_{(i,J)}  \; = \; -\alpha_i(T) \, F_i\,.v_J  }
   $$
   \indent   Similarly, for  $ \, T \in \lieh \, $  and  $ E_i $  with  $ \, i \in I \, $  we find
  $$
  \displaylines{
   T.\big( E_i\,.v_J \big)  \; = \;  \sum_{1 \leq \ell\leq r \, , \; i=j_\ell} \hskip-5pt \big(\, \lambda(T) -
   \alpha_{\widehat{J}_\ell}(T) \big) \, \frac{\; q^{+\lambda(T_i^+) -\alpha_{J_\ell}(T_i^+)} -
   q^{-\lambda(T_i^-) + \alpha_{J_\ell}(T_i^-)} \;}{q_i^{+1} \! - q_i^{-1}} \; v_{\widehat{J}_\ell}  \cr
   E_i\,.\big( T.v_J \big)  \; = \;  \sum_{1 \leq \ell \leq r \, , \; i=j_\ell} \hskip-5pt \big(\, \lambda(T) -
   \alpha_J(T) \big) \, \frac{\; q^{+\lambda(T_i^+) - \alpha_{J_\ell}(T_i^+)} - q^{-\lambda(T_i^-) +
   \alpha_{J_\ell}(T_i^-)} \;}{q^{+1}_i \! - q_i^{-1}} \; v_{\widehat{J}_\ell}  }
   $$
 Since  $ \, \alpha_J(T) - \alpha_{\widehat{J}_\ell}(T) = \alpha_{j_\ell}(T) = \alpha_i(T) \, $
 for all  $ \, j_\ell = i \, $,  \,we obtain
  $$
  [T,E_i]\,.v_J  \; =  \sum_{1 \leq \ell\leq r \, , \; i = j_\ell} \hskip-5pt \alpha_i(T) \,
  \frac{\; q^{+\lambda(T_i^+) -\alpha_{J_\ell}(T_i^+)} - q^{-\lambda(T_i^-) + \alpha_{J_\ell}(T_i^-)} \;}{q_i^{+1} \! - q_i^{-1}} \;
  v_{\widehat{J}_\ell}  \; = \;  \alpha_i(T) \, E_i\,.v_J
  $$
   \indent   Finally, to check the commuting relation between  $ E_i $  and  $ F_j $  we note first that
 $ \; \sum_{n=0}^m \frac{\,\hbar^n\,}{\,n!\,} \, T^n.v_J \, = \, \sum_{n=0}^m \frac{\,\hbar^n\,}{\,n!\,} \,
 {\big(\, \lambda(T) - \alpha_J(T) \big)}^{n}.v_J \; $
 for all  $ \, T\in \lieh \, $  and  $ \, m \geq 1 \, $.  Then, by the continuity of the linear action, we get that
  $$  e^{+\hbar\,T^+_i}.\,v_J  \; = \;  q^{+(\lambda(T^+_i) - \alpha_J(T^+_i))} \, v_J
   \qquad  \text{ and }  \qquad
      e^{-\hbar\,T^-_i}.\,v_J  \; = \;  q^{-(\lambda(T^-_i) - \alpha_J(T^-_i))} \, v_J  $$
Then, for  $ \, i , j \in I \, $  with  $ \, i \neq j \, $,  \,we have
  $$  \displaylines{
   E_i\,.\big( F_j\,.v_J \big) \, - \, F_j\,.\big( E_i\,.v_J \big)  \; =   \hfill  \cr
   = \;  E_i\,.v_{(j,J)} \, - \, F_j\,.\left( \sum_{1 \leq \ell \leq r \, , \; i = j_\ell} \hskip-5pt \frac{\; q^{+\lambda(T_i^+) - \alpha_{J_\ell}(T_i^+)} - q^{-\lambda(T_i^-) + \alpha_{J_\ell}(T_i^-)} \;}{q_i^{+1} \! - q_i^{-1}} \; v_{\widehat{J}_\ell} \right)  \; =  \cr
   \hfill   = \;  E_i\,.v_{(j,J)} \, - \, \sum_{1 \leq \ell \leq r \, , \; i = j_\ell} \hskip-5pt \frac{\; q^{+\lambda(T_i^+) - \alpha_{J_\ell}(T_i^+)} - q^{-\lambda(T_i^-) + \alpha_{J_\ell}(T_i^-)} \;}{q_i^{+1} \! - q_i^{-1}} \; v_{(j,\widehat{J}_\ell)}  }  $$
   \indent   First, if  $ \, i \neq j \, $,  \,then  $ \, j_\ell = i \neq j \, $  and
   $ \, \widehat{(\,j\,,J\,)}_\ell = \big(\, j\,,\widehat{J}_\ell \,\big) \, $, $ \, \alpha_{{(\,j,\,J\,)}_\ell} = \alpha_{J_\ell} \, $  for all  $ \, 1 \leq \ell \leq r \, $.  Hence
 \vskip-11pt
  $$  E_i\,.v_{(j,J)}  \;\; =  \sum_{1 \leq \ell \leq r \, , \; i = j_\ell} \hskip-9pt
   {\big( q_i^{+1} \! - q_i^{-1}\big)}^{-1} \big( q^{+\lambda(T_i^+) - \alpha_{J_\ell}(T_i^+)} - q^{-\lambda(T_i^-) + \alpha_{J_\ell}(T_i^-)} \big) \, v_{(j,\widehat{J}_\ell)}
   $$
 which implies that  $ \; E_i\,.\big( F_j\,.v_J \big) \, - \, F_j\,.\big( E_i\,.v_J \big) \, = \, 0 \; $.
                                                                  \par
   Assume now that  $ \, i = j \, $.  Then for  $ \, i = j_\ell \, $  we have
   $ \, \widehat{(i\,,J)}_\ell = J \, $  and  $ \, \alpha_{(i,J)_\ell} = \alpha_J \, $  for
   $ \, \ell = 1 \, $,  \,while  $ \, \widehat{(i\,,J)}_\ell = \big(i\,,\widehat{J}_\ell\big) \, $  and
   $ \, \alpha_{(i,J)_\ell} = \alpha_{J_\ell} \, $  for  $ \, \ell > 1 \, $.  This implies that
  $$  \displaylines{
   E_i\,.v_{(i,J)}  \,\; = \;\,  {\big( q_i^{+1} \! - q_i^{-1}\big)}^{-1} \big( q^{+\lambda(T_i^+) - \alpha_J(T_i^+)} - q^{-\lambda(T_i^-) + \alpha_J(T_i^-)} \big) \, v_J \,\; +   \hfill  \cr
   \hfill   +  \sum_{1 \leq \ell \leq r \, , \; i = j_\ell} \hskip-9pt {\big( q_i^{+1} \! - q_i^{-1}\big)}^{-1}
   \big( q^{+\lambda(T_i^+) - \alpha_{J_\ell}(T_i^+)} - q^{-\lambda(T_i^-) + \alpha_{J_\ell}(T_i^-)} \big) \, v_{(i,\widehat{J}_\ell)}  }
   $$
and consequently
  $$  E_i . \big( F_i . v_J \big) - F_i . \big( E_i.v_J \big)  \, = \,
  \frac{q^{\lambda(T_i^+) - \alpha_J(T_i^+)} - q^{-\lambda(T_i^-) + \alpha_J(T_i^-)}}{q_i^{+1} \! - q_i^{-1}} \,.\, v_J  \, = \,  \frac{\, e^{\hbar\, T_i^+} \! - e^{\hbar\, T_i^-} \,}{q_i^{+1} \! - q_i^{-1}} \,.\, v_J  $$
 Therefore, the formulas above define indeed an action of  $ U $  onto  $ T_\hbar(V) \, $.
                                                              \par
   Finally, since this action is  $ \kh $--linear,  it induces a unique (topological) action of
   $ \widetilde{U}^{\,\R}_{\!P,\hbar}(\lieg) $  on  $ \widehat{T}_\hbar(V) $ by completion.  This completes the proof.
\end{proof}

\vskip7pt

   With an entirely analogous proof, we obtain the following lemma:

\vskip11pt

\begin{lema}\label{lemma: repLT}
 For every  $ \, \lambda \in \lieh^* \, $,  there exists a unique representation of
 $U$  onto  ${T}_\hbar(V) $
 --- which is then denoted  $ \, {}^\lambda T_\hbar(V) $  ---
 such that, for all  $ \, J =(j_1,\dots,j_r) \in \mathcal{J} \, $
%
%
 \vskip-17pt
  $$
  \displaylines{
   E_i\,.\,v_J  \; = \;  v_{(i,\,J)}  \quad ,   \qquad  T\,.\, v_J  \; = \;  \big( \lambda(T) + \alpha_J(T) \big) v_J  \cr
   F_i\,.\,v_J  \; =  \sum_{1\leq \ell\leq r \, , \; j_\ell = i} \hskip-5pt \frac{\; q^{-\lambda(T_i^-) -
   \alpha_{J_\ell}(T_i^-)} \! - q^{+\lambda(T_i^+) + \alpha_{J_\ell}(T_i^+)} \;}{q_i^{+1} - q_i^{-1}} \; v_{\widehat{J}_\ell}  }
   $$
 In addition, this representation, of  $ \, U $  onto  $ \, T_\hbar(V) \, $,
 induces by continuity a unique representation  $ \, \widetilde{U}^{\,\R}_{\!P,\hbar}(\lieg) \, $  onto
 $ \, \widehat{T}_\hbar(V) \, $,  \,which is hereafter denoted by  $ \,{}^\lambda \widehat{T}_\hbar(V) \, $.   \qed
\end{lema}

\vskip9pt

   Denote by  $ U^0 $,  resp.\  $ U^+ $,  resp.\  $ U^- $,  the unital associative  $ \kh $--subalgebras  of
   $ U $  generated by  $ \lieh \, $, resp.\ by all the  $ E_i $'s,  resp.\  by all the  $ F_i $'s  ($ \, i \in I \, $).
   Write  $ \, \widetilde{U}^0 \, \big( = \widetilde{U}_\hbar(\lieh) := \widetilde{U}_{P,\hbar}^{\,\R}(\lieh) \big) \, $,
   resp.\  $\widetilde{U}^+ $,  resp.\  $ \widetilde{U}^- $,  for the $ \hbar $--adic  completion of them: clearly,
   all these are topological  $ \kh $--subalgebras  of  $ \widetilde{U}^{\,\R}_{\!P,\hbar}(\lieg) \, $.
 \vskip9pt
 As a first consequence of  Lemma \ref{lemma: tensor repr x utildeRPhg},  we get the following:

\vskip11pt

\begin{prop}  \label{prop: structure Uh(h)}
 The Cartan subalgebra
$ \, \widetilde{U}^0 = \widetilde{U}_\hbar(\lieh) = \widetilde{U}_{P,\hbar}^{\,\R}(\lieh) \, $
 is the  $ \hbar $--adic  completion of  $ \, S_\hbar(\lieh) \, $,  \,the symmetric algebra  of $ \, \lieh $
 over  $ \, \kh \, $.  In particular,  $ \, \widetilde{U}_\hbar(\lieh) \, $  is  \textsl{independent of  $ \, P $  and  $ \, \R $
 (though not of  $ \, \lieh $)}   --- whence the simplified notation.
\end{prop}

\begin{proof}
 We provide two different, independent proofs.
 \vskip7pt
   \textsl{$ \underline{\text{First Proof}} \, $:}\,  Let  $ T_\hbar(\lieh) \, $,  \,resp.\  $ \, S_\hbar(\lieh) \, $,
   \,be the tensor algebra, resp.\ the symmetric algebra, of  $ \lieh $  over  $ \kh \, $,  and let  $ \widehat{T}_\hbar(\lieh) \, $,
   \,resp.\  $ \, \widehat{S}_{\hbar}(\lieh) $  be the  $ \hbar $--adic  completion of them.
 By the commutation relations among elements of  $ \lieh $  in  $ \widetilde{U}^{\,\R}_{\!P,\hbar}(\lieg) \, $,  \,we have a diagram of morphisms
%
%
  $$
  \xymatrix{
   \; \widehat{T}_\hbar(\lieh) \; \ar@{->>}[rr]^{p_{\hbar,S}}  &  &  \; \widehat{S}_\hbar(\lieh) \; \ar@{->>}[rr]^{p_{\hbar,U}} &  &
   \; \widetilde{U}_\hbar(\lieh) \; \ar@{^{(}->}[rr]^{\iota_{\hbar}}  &  &   \; \widetilde{U}^{\,\R}_{\!P,\hbar}(\lieg)  \\
   \; T_\hbar(\lieh) \; \ar[u]^{} \ar@{->>}[rr]^{p_{S}}  &  &  \; S_\hbar(\lieh) \; \ar[u]^{} \; \ar@{-->}[urr]^{} \ar[rr]^{p_{U^{0}}}  & & \; U^{0} \; \ar[u]^{} & &   \\
   \; \lieh \; \ar[u]^{} \ar[urr]^{} &  &  &  &  &  &  }
   $$
 where the maps  $ p_{\hbar,S} $, $ p_{\hbar,U} $,  $ p_S $  and  $ p_{U^{0}} $
 are the canonical epimorphisms, $ \iota_\hbar $  is the canonical inclusion, all vertical arrows are canonical embeddings,
 and the diagonal arrows  $ \, \lieh \longrightarrow S_\hbar(\lieh) \, $  too.
 We want to show that  $ p_{\hbar,U} $  is in fact an  \textsl{iso}momorphism.
 \vskip5pt
   Let  $ \, V := \bigoplus_{i \in I} \kh.v_i \, $  be the free  $ \kh $--module  with basis  $ \, {\{ v_i \}}_{i \in I} \, $.
Choosing  $ \, \lambda \in \lieh^* \, $  and restricting
 the action defined in  Lemma \ref{lemma: tensor repr x utildeRPhg}  to the image of  $ i_\hbar \, $,
 \,we have that  $ \widetilde{U}_\hbar(\lieh) $  acts on  $ \widehat{T}_\hbar^\lambda(V) $  by the character
 $ \lambda $  via  $ \; T.v_J \, = \, \big( \lambda(T) - \alpha_J(T) \big) \, v_J \, $,  \,for all  $ \, T\in \lieh \, $  and
for all  $ \, J = \big( j_1 \, , \cdots , j_r \big) \in \mathcal{J} \, $.  In particular,
$ \, T.v_\emptyset = T.1 = \lambda(T) \, $  for all  $ \, T \in \lieh \, $.
This action induces an action of  $ \widehat{S}_\hbar(\lieh) \, $,  \,and of its  $ \kh $--subalgebra  $ S_\hbar(\lieh) \, $,  \,on
$ \widehat{T}_\hbar^\lambda(V) $  via the epimorphism  $ p_{\hbar,U} \, $.  By the very definition, this action
coincides with the unique action of  $ S_\hbar(\lieh) $  on  $ \widehat{T}_\hbar^\lambda(V) $  defined by the
character  $ \, \lambda \in \lieh^* \, $  by the universal properties of the symmetric and the tensor algebras.
                                                                           \par
   Let  $ \, t \in S_\hbar(\lieh) \, $  be such that  $ p_{\hbar,U}(t) = 0 \, $:  \,then  $ \, 0 = t.v_\emptyset = \lambda_S(t) \, $
   --- extending  $ \, \lambda \in \lieh^* \, $  to a  $ \kh $--algebra  character  $ \lambda_S $  of  $ S_\hbar(\lieh) \, $.
   Since  $ \lambda $  is arbitrary, we get  $ \, \lambda_S(t) = 0 \, $  for all  $ \, \lambda \in \lieh^* \, $,  so that  $ \, t = 0 \, $.
   Similarly,  $ \lambda_S $  further extends, canonically and uniquely, to a  $ \hbar $--adically
   continuous character of  $ \widehat{S}_\hbar(\lieh) \, $,  \,denoted  $ \lambda_{\widehat{S}} \, $.
   Then for any  $ \, \widehat{t} \in \widehat{S}_\hbar(\lieh) \, $
   such that  $ p_{\hbar,U}\big(\,\widehat{t}\;\big) = 0 \, $  we have  $ \, 0 = t.v_\emptyset = \lambda_{\widehat{S}}\big(\,\widehat{t}\;\big) \, $,
   \,thus  $ \, \lambda_{\widehat{S}}\big(\,\widehat{t}\;\big) = 0 \, $
   for all  $ \, \lambda \in \lieh^* \, $,  \,which implies  $ \, \widehat{t} = 0 \, $.  Hence  $ p_{\hbar,U} $  is injective and
$ \, \widehat{S}_\hbar(\lieh) \cong \widetilde{U}_\hbar(\lieh) \, $.
 \vskip7pt
   \textsl{$ \underline{\text{Second Proof}} \, $:}\,  By definition there is a  $ \kh $--linear  morphism  from
   $ \lieh $  to  $ \uRPhg \, $,  \,whose image we denote by  $ \lieh' \, $;  \,in other words,  $ \lieh' $
   is the  $ \kh $--submodule of  $ \uRPhg $
spanned by the generators  $ \, T \in \lieh \, $.
 By definition,  $ U_{P,\hbar}^{\,\R}(\lieh) $  is (topologically) generated by  $ \lieh' \, $,
   \,which in turn is a Lie subalgebra inside the Lie algebra of  \textsl{primitive\/}  elements of the Hopf algebra
   $ U_{P,\hbar}^{\,\R}(\lieh) \, $. Since we are in characteristic zero, by Milnor-Moore's Theorem
 we deduce that  $ U_{P,\hbar}^{\,\R}(\lieh) $  is indeed nothing but the  $ \hbar $--adic
 completion of the universal enveloping algebra  $ \, U\big(\lieh'\big) $  of  $ \, \lieh' \, $;  \,in turn,
 the latter coincides with the  \textsl{symmetric\/}  $ \kh $--algebra  $ S_\hbar\big(\lieh'\big) $
 --- hence its completion coincides with  $ \widehat{S}_\hbar\big(\lieh'\big) $
 ---   because the multiplication therein is commutative, by construction!
                                                          \par
   Finally, we observe that the built-in epimorphism
   $ \, \lieh \relbar\joinrel\twoheadrightarrow \lieh' \, $  is indeed an  \textsl{iso}mor\-phism,  so that  $ \, \lieh' \cong \lieh \, $.
   This is proved again making use of  Lemma \ref{lemma: tensor repr x utildeRPhg}  above, in particular looking at how  $ \lieh $
   acts on each representation  $ \widehat{T}_\hbar^\lambda(V) \, $,  for all  $ \, \lambda \in \lieh^* \, $,
   \,along the same lines as in the last part of the  \textsl{First Proof\/}  above.
\end{proof}

\vskip9pt

   Let  $ U' $  be the  $ \kh $--subalgebra  of  $ \widetilde{U}^{\,\R}_{\!P,\hbar}(\lieg) $  generated by the
   $ \kh $--subalgebras  $ U^- $,  $ U^+$  and  $ \, \widetilde{U}^0 = \widetilde{U}_\hbar(\lieh) \, $   --- the
$ \hbar $--adic  completion of  $ U^0 $.  By  Theorem \ref{thm: form-MpQUEAs_are_Hopf},
it follows that  $ U' $  is a (topological) Hopf  $ \kh $--subalgebra of  $ \, \widetilde{U}^{\,\R}_{\!P,\hbar}(\lieg) \, $.
 \vskip3pt
   We fix some more notation.  For any finite sequence
   $ \, J = \big( j_1 \, , j_2 \, , \ldots , j_r \big) \, $   --- with  $ \, r \geq 1 \, $  ---   of elements in  $ I $,
 \,we set  $ \, T^\pm_\emptyset := 0 \, $  and  $ \, E_\emptyset := 1 =: F_\emptyset \, $,  \,and in general
  $$
  T^\pm_J \, := \, T^\pm_{j_1} + T^\pm_{j_2} + \cdots + T^\pm_{j_r}  \;\; ,  \;\quad
   E_J \, := \, E_{j_1} E_{j_2} \cdots E_{j_r}  \;\; ,  \;\quad
   F_J \, := \, F_{j_1} F_{j_2} \cdots F_{j_r}  \;\; .
   $$
   \indent   As the coproduct of the elements  $ E_i $'s  and  $ F_i $'s  in  $ U' $  and in  $ \widetilde{U}^{\,\R}_{\!P,\hbar}(\lieg) $
   coincides with the one defined for the one-parameter  \textsl{polynomial\/}  QUEA, the following lemma follows at once,
   \textit{mutatis mutandis},  from  \cite[Lemma 4.12]{Ja}:

\vskip11pt

\begin{lema}  \label{lemma: formulacoprodEiFi}
 Let  $ J $  be a finite sequence as above.  Then there exist Laurent polynomials  $ \, c_{A,B}^{\,J} \in \ZZ\big[x\,,x^{-1}\big] \, $,
 \,indexed by finite sequences of elements of  $ I $,  \,with  $ \, \alpha_J = \alpha_A + \alpha_B \, $
 and such that, both in  $ U' $  and in  $ \, \widetilde{U}^{\,\R}_{\!P,\hbar}(\lieg) \, $,  \,one has
\begin{equation}\label{eq:formulacoprodEiFi}
   \Delta\big(E_J\big) = {\textstyle \sum\limits_{A,B}} c_{A,B}^{\,J}(q) \, E_A \,
   e^{\hbar \, T^+_B} \otimes E_B \;\; , \!\quad \Delta\big(F_J\big) = {\textstyle \sum\limits_{A,B}} c_{A,B}^{\,J}\big(q^{-1}\big) \, F_A \otimes e^{\hbar \, T^-_A} \, F_B
\end{equation}
   \indent   Moreover, one has  $ \; c_{A,\emptyset}^{\,J} = \delta_{A,J} \; $
   and  $ \; c_{\emptyset,B}^{\,J} = \delta_{B,J} \; $.   \qed
\end{lema}

\vskip9pt

   As an intermediate result, we have now a ``triangular decomposition'' for  $ \, \widetilde{U}^{\,\R}_{\!P,\hbar}(\lieg) \, $:

\vskip11pt

\begin{prop}  \label{prop: triangulardecUtilde}
 The  multiplication  maps
   $ \; \overrightarrow{\!\mu_3} : \widetilde{U}^- \otimes \widetilde{U}_\hbar(\lieh) \otimes \widetilde{U}^+
   \!\relbar\joinrel\relbar\joinrel\longrightarrow \widetilde{U}^{\,\R}_{\!P,\hbar}(\lieg) \; $
 and
   $ \; \overleftarrow{\mu_3} : \widetilde{U}^+ \otimes \widetilde{U}_\hbar(\lieh) \otimes \widetilde{U}^-
   \!\relbar\joinrel\relbar\joinrel\longrightarrow \widetilde{U}^{\,\R}_{\!P,\hbar}(\lieg) \; $
 induced by restriction of multiplication in  $ \widetilde{U}^{\,\R}_{\!P,\hbar}(\lieg) \, $
 --- the tensor products here being the  $ \hbar $--adically  completed ones ---
 are both isomorphisms of topological  $ \kh $--modules.
\end{prop}

\begin{proof}
 It is enough to prove one case, say that of
 $ \; \overrightarrow{\!\mu_3} : \widetilde{U}^- \otimes \widetilde{U}_\hbar(\lieh) \otimes
 \widetilde{U}^+ \!\relbar\joinrel\longrightarrow \widetilde{U}^{\,\R}_{\!P,\hbar}(\lieg) \, $,
 \;the other one being similar.
 \vskip3pt
   Consider the map $ \; {\!\mu_3}' : {U}^- \otimes \widetilde{U}_\hbar(\lieh) \otimes {U}^+ \!\relbar\joinrel\relbar\joinrel\longrightarrow U' \, $
   induced by restriction of multiplication in  $ U' $.  We show that this map is bijective.
                                                                \par
   We prove first that  $ {\!\mu_3}' $  is  \textsl{surjective}.  Let  $ \, \{H_{g}\}_{g\in \mathcal{G}} $  be a $ \kh $--basis  of  $ \lieh \, $.
   Thanks to the defining relations in the first two lines of  \eqref{eq: comm-rel's_x_uPhg},  we see at once that
   $ U' $  is  $ \kh $--spanned  (in  $ \hbar $--adic  sense)  by the set of ``monomials''
  $$  \Big\{\, F_{i_1} \cdots F_{i_n} \cdot H_{\underline{i}\,,\,\underline{j}} \cdot E_{j_1} \cdots E_{j_m} \,\Big|\, n, s, m \in \NN \, , \, i_a , j_b \in I \,\; \forall \, a, b \, , \; H_{\underline{i}\,,\,\underline{j}} \in \widetilde{U}_\hbar(\lieh) \,\Big\}  $$
 and then this guarantees that  $ {\!\mu_3}' \, $  is onto, since
  $$  F_{i_i} \cdots F_{i_n} \cdot H_{\underline{i}\,,\,\underline{j}} \cdot E_{j_1} \cdots E_{j_m} =
  {\!\mu_3}'\big( \big( F_{i_i} \cdots F_{i_n} \big) \otimes H_{\underline{i}\,,\,\underline{j}} \otimes \big( E_{j_1} \cdots E_{j_m} \big) \big)  $$
 \vskip5pt
   Our second, last task is to prove that  $ \overrightarrow{\!\mu_3} $  is  \textsl{injective}.
   Let  $ \, \lambda, \theta \in \lieh^* \, $  and consider the  $ \widetilde{U}^{\,\R}_{\!P,\hbar}(\lieg) $--modules
   $ \widehat{T}_\hbar^\lambda(V) $  and  $ {}^\theta\widehat{T}_\hbar(V) $  given by
   Lemmas \ref{lemma: tensor repr x utildeRPhg} and  \ref{lemma: repLT},  respectively.
   Then  $ \widetilde{U}^{\,\R}_{\!P,\hbar}(\lieg) $  acts  on the tensor product  $ \, \widehat{T}_\hbar^\lambda(V) \otimes {}^\theta\widehat{T}_\hbar(V) \, $
   too, and this yields by restriction a  $ U' $--action  as well.  Assume we have a linear dependence relation
   $$  {\textstyle \sum\limits_{J,L}} \; a_{J,L} \, F_J \, H_{J,L} \, E_L  \,\; = \;\,  0  $$
 for finitely many elements  $ \, a_{J,L} \in \kh \, $,  \,where  $ J $,  and  $ L $  are finite sequences of elements
 in  $ I $  and  $ \; H_{J,L} \in  \,  \widetilde{U}_\hbar(\lieh) \, $.
                                                                 \par
   In the set of finite sequences of elements in  $ I $,  we consider the partial order given by
   $ \; J = \big( j_1 \, , j_2 \, , \ldots , j_r \big) > \big( \ell_1 \, , \ell_2 \, , \ldots , \ell_s \big) =: L \; $
   if and only if  $ \; \alpha_J - \alpha_L = \sum_k \alpha_{j_k} - \sum_k \alpha_{\ell_k} = \sum_t \alpha_{i_t} \; $
   for some simple roots $ \alpha_{i_t} \, $.  Choose  $ J_0 $  such that  $ \, a_{J_0\,,L} \neq 0 \, $
   for some  $ L \, $, \,and such that  $ J_0 $  is maximal with respect to the order given above.
                                                                 \par
   Now, for  $ \, v_\emptyset \in \widehat{T}_\hbar^\lambda(V) \setminus \{0\} \, $  and
   $ \, w_\emptyset \in {}^\theta\widehat{T}_\hbar(V) \setminus \{0\} \, $  being (non-zero)
   ``highest weight vector'' and ``lowest weight vector'', respectively,  definitions give
  $$  \displaylines{
   0  \; = \;  \bigg(\, {\textstyle \sum\limits_{J,K,L}} a_{J,L} \, F_J \, H_{J,L} \, E_L \bigg).(v_\emptyset \otimes w_{\emptyset})  \; =
   \;  {\textstyle \sum\limits_{J,L}} \, a_{J,L} \, \big( F_J \, H_{J,L} \big).\big( E_L.(v_\emptyset \otimes w_\emptyset) \big)  \; = \;   \hfill  \cr
   \qquad \quad   = \;
   {\textstyle \sum\limits_{J,L}} \, a_{J,L} \, \big( F_J \, H_{J,L} \big).\bigg(\, {\textstyle \sum\limits_{A,B}} \, c^{\,L}_{A,B}(q) \left( E_A \,
   e^{\hbar \, T_B^+} \right)\!.v_\emptyset \otimes \big( E_B.w_\emptyset \big) \bigg)  \; =   \hfill  \cr
   \qquad \qquad \quad \quad   = \;  {\textstyle \sum\limits_{J,L}} \, a_{J,L} \,
   \big( F_J \, H_{J,L} \big).\bigg(\, {\textstyle \sum\limits_{A,B}} \, c^{\,L}_{A,B}(q) \, q^{\lambda(T_B^+)} \, E_A.v_\emptyset \otimes w_B \bigg)  \; =   \hfill  \cr
   \qquad \qquad \qquad \quad \quad \quad   = \;  {\textstyle \sum\limits_{J,L}} \, a_{J,L} \,
   c^{\,L}_{\emptyset,L}(q) \, q^{\lambda(T_L^+)} \, \big( F_J \, H_{J,L} \big).\big( v_\emptyset \otimes w_L \big)  \; =   \hfill  \cr
   \hfill   = \;  {\textstyle \sum\limits_{J,L}} \, a_{J,L} \, q^{\lambda(T_L^+)} \, \lambda\big({(H_{J,L})}_{(1)}\big) \, (\theta + \alpha_L)\big({(H_{J,L})}_{(2)}\big) \, F_J.\big( v_\emptyset \otimes w_L \big)  \; =   \quad  \cr
   \quad   = \;  {\textstyle \sum\limits_{J,L}} \, a_{J,L} \, q^{\lambda(T_L^+)} \, \lambda\big({(H_{J,L})}_{(1)}\big) \, (\theta + \alpha_L)\big({(H_{J,L})}_{(2)}\big) \; \cdot  \hfill  \cr
   \hfill   \cdot \bigg(\, {\textstyle \sum\limits_{A,B}} \, c_{A,B}^{\,J}\big(q^{-1}\big) \, F_A.v_\emptyset \otimes \! \big( e^{\hbar \, T^-_A} \, F_B \big).w_L \bigg)  \; =   \quad  \cr
   \hfill   = \;  {\textstyle \sum\limits_{J,L}} \, a_{J,L} \, q^{\lambda(T_L^+)} \, \lambda\big({(H_{J,L})}_{(1)}\big) \,
   (\theta + \alpha_L)\big({(H_{J,L})}_{(2)}\big) \cdot {\textstyle \sum\limits_{A,B}} \, c_{A,B}^{\,J}\big(q^{-1}\big) \,
   v_A \otimes \big( e^{\hbar \, T^-_A} \, F_B \big).w_L  }  $$
 where we get third and fifth equality from  Lemma \ref{lemma: formulacoprodEiFi},  and we recall  $ \, c^{L}_{\emptyset,L}(q) = 1 \, $.
 \vskip5pt
   Consider now those coefficients with  $ \, J = J_0 \, $.  Since  $ \, \alpha_{J_0} = \alpha_A + \alpha_B \, $,
   \,we have that  $ \, A \, , B \leq J_0 \, $  and  $ \, A = J_0 \, $  if and only if  $ \, B = \emptyset \, $.  Then we must have
 \vskip5pt
  $$  \displaylines{
   0  \; = \;  {\textstyle \sum_L} \, a_{J_0,L} \, q^{\lambda(T_L^+)} \,
   \lambda\big({(H_{J_{0},L})}_{(1)}\big) \, (\theta+\alpha_L)\big({(H_{J_{0},L})}_{(2)}\big) \cdot
   c_{J_0,\emptyset}^{\,J_0}\big(q^{-1}\big) \, v_{J_0} \otimes e^{\hbar \, T^-_{J_0}}.w_L  \; =   \hfill  \cr
   \hfill   = \; {\textstyle \sum_L} \, a_{J_0,L} \, q^{\lambda(T_L^+)} \,
   \lambda\big({(H_{J_{0},L})}_{(1)}\big) \, (\theta + \alpha_L)\big({(H_{J_{0},L})}_{(2)}\big) \, q^{(\theta + \alpha_L)(T_{J_0}^-)} \, v_{J_0} \otimes w_L  }  $$
 Since  $ \, {\big\{ v_J \otimes w_L \big\}}_{\!J,L} \, $  is a basis  of the free  $ \kh $--module
 $ \, \widehat{T}_\hbar^\lambda(V) \otimes {}^\theta\widehat{T}_\hbar(V) \, $,  \,for all  $ L $  such that  $ \, a_{J_0,L} \neq 0 \, $  we must have
 \vskip-13pt
  $$  \displaylines{
   0  \; = \;  a_{J_0,L} \, q^{\lambda(T_L^+)} \, \lambda\big({(H_{J_{0},L})}_{(1)}\big) \,
   (\theta + \alpha_L)\big({(H_{J_{0},L})}_{(2)}\big) \, q^{(\theta + \alpha_L)(T_{J_0}^-)}  \; =   \hfill  \cr
   \hfill   = \;  a_{J_0,L} \, q^{\lambda(T_L^+) + (\theta + \alpha_L)(T_{J_0}^-)} \,
   \big( \lambda * (\theta + \alpha_L) \big)\big(H_{J_{0},L} \big)  }  $$
 \vskip-3pt
\noindent
 where  \,``$ \, * \, $''\,  denotes the convolution product between characters; this implies that
 $ \; 0 = a_{J_0,L}\big( \lambda * (\theta + \alpha_L) \big)\big(H_{J_0,L} \big) \; $  for all  $ \, \lambda , \theta \in \lieh^* \, $.
 Since  $ \, \widetilde{U}_\hbar(\lieh) \cong \widehat{S}_\hbar(\lieh) \, $,
 \,this holds true if and only if  $ \, a_{J_0,L} = 0 \, $  for all  $ K $,  a contradiction.
 Thus,  $ \mu'_3 $  is injective, q.e.d.
 \vskip7pt
   Finally, as  $ \mu'_3 $  is a  $ \kh $--linear  map, it is an isomomorphism between topological
   algebras which extends uniquely to an isomorphism  $ \overrightarrow{\!\mu_3} $  on their completions.
   Since the completion of $ U' $  is exactly our  $ \widetilde{U}^{\,\R}_{\!P,\hbar}(\lieg) \, $,  \,we eventually obtain
the isomomorphism
 $ \; \overrightarrow{\!\mu_3} : \widetilde{U}^- \otimes \widetilde{U}_\hbar(\lieh) \otimes \widetilde{U}^+
   \!\relbar\joinrel\relbar\joinrel\longrightarrow \widetilde{U}^{\,\R}_{\!P,\hbar}(\lieg) \; $
 which is described just like in the claim.
\end{proof}

\vskip9pt

   For the last steps, we need more notation: for all  $ \, i \, , j \in I \, $  with  $ \, i \neq j \, $,  \,set
%
%
  $$  \displaylines{
  u_{ij}^E  \; := \;  \sum\limits_{k = 0}^{1-a_{ij}} {(-1)}^k
  {\displaystyle {\left[ { 1-a_{ij} \atop k }
\right]}_{\!q_i}} q_{ij}^{+k/2\,} q_{ji}^{-k/2} \, E_i^{1-a_{ij}-k} \, E_j \, E_i^k  \;\; \in \;\; \widetilde{U}^{\,\R}_{\!P,\hbar}(\lieg)  \cr
  u_{ij}^F  \; := \;  \sum\limits_{k = 0}^{1-a_{ij}} {(-1)}^k
  {\displaystyle {\left[ { 1-a_{ij} \atop k }
\right]}_{\!q_i}} q_{ij}^{+k/2\,} q_{ji}^{-k/2} \, F_i^k \, F_j \, F_i^{1-a_{ij}-k}  \;\; \in \;\; \widetilde{U}^{\,\R}_{\!P,\hbar}(\lieg)  }  $$
   \indent   Let  $ \, \E^+ $,  resp.\  $ \, \F^- $,  \,be the closed, two-sided ideal of
   $ \widetilde{U}^+ $,  resp.\ of  $ \widetilde{U}^- $,  \,generated by
 all the  $ u_{ij}^E $'s,  resp.\  $ u_{ij}^F $'s  ($ \, i \not= j \, $).
  Denote by  $ U^{\,\R}_{\!P,\hbar}(\lien_\pm) $  the unital,  $ \hbar $--adically
  complete topological  $ \kh $--subalgebra  of  $ U^{\,\R}_{\!P,\hbar}(\lieg) $  generated by all the  $ E_i $'s,  resp.\  $ F_i $'s  ($ \, i \in I \, $).

\vskip9pt

   Next result describes explicitly the structure of the  \textsl{Cartan\/}  and
   of the  \textsl{(positive/negative) nilpotent\/}  subalgebras in our FoMpQUEAs:

\vskip11pt

\begin{prop}  \label{prop: struct-Cart/nilp-subalg.s}  {\ }
 \vskip3pt
   \textit{(a)}\,  The closed two-sided ideal of  $ \, \widetilde{U}^{\,\R}_{\!P,\hbar}(\lieg) \, $
   generated by
   all the  $ u_{ij}^E $'s,  resp.\ all the  $ u_{ij}^F $'s,
  is equal to the image of  $ \, \widetilde{U}^- \otimes \widetilde{U}_\hbar(\lieh) \otimes \E^+ \, $,
  resp.\  $ \, \F^- \otimes \widetilde{U}_\hbar(\lieh) \otimes \widetilde{U}^+ \, $,  under the multiplication map
  $ \; \overrightarrow{\!\mu_3} : \widetilde{U}^- \otimes \widetilde{U}_\hbar(\lieh)
  \otimes \widetilde{U}^+ \!\relbar\joinrel\relbar\joinrel\longrightarrow \widetilde{U}^{\,\R}_{\!P,\hbar}(\lieg) \; $.
                                                                \par
 An entirely similar claim holds true as well with  $ \overleftarrow{\mu_3} $  replacing $ \overrightarrow{\!\mu_3} \, $.
 \vskip3pt
   \textit{(b)}\,  $ \; \widetilde{U}_\hbar(\lieh) \cong U_{P,\hbar}^{\,\R}(\lieh) \; $
through the obvious canonical epimorphism.  In particular,  $ \, U_\hbar(\lieh) := U_{P,\hbar}^{\,\R}(\lieh) \, $
is the  $ \hbar $--adic  completion of  $ \, S_\hbar(\lieh) \, $,  \,the symmetric algebra  of $ \, \lieh $  over  $ \, \kh \, $,
\,hence it is  \textsl{independent of  $ \, P $  and  $ \, \R $  (though not of  $ \, \lieh $)}.
 \vskip3pt
   \textit{(c)}\,  The algebra  $ \, U^{\,\R}_{\!P,\hbar}(\lien_+) \, $,  \,resp.\  $ U^{\,\R}_{\!P,\hbar}(\lien_-) \, $,
   \,is isomorphic to the unital,  $ \hbar $--adically  complete topological  $ \kh $--algebra  generated by all the
   $ E_i $'s  ($ \, i \in I $),  resp.\ the  $ F_i $'s  ($ \, i \in I $),  \,with relations
$ \, u^E_{ij} = 0 \, $,  \,resp.\  $ u^F_{ij} = 0 \, $,
\,for all $ \, i \neq j \, $.
\end{prop}

\textit{Proof.}
   \textit{(a)}\,  This follows \textit{mutatis mutandis} from \cite[Lemma 4.20]{Ja}.
 \vskip7pt
   \textit{(b)}\,  The first part of the claim follows from the proof of  \cite[Theorem 4.21 d)]{Ja}.
   Indeed, let  $ \mathcal{J} $  be the kernel of the canonical,  $ \hbar $--adically  continuous epimorphism
   $ \, p : \widetilde{U}_{P,\hbar}^{\,\R}(\lieg) \relbar\joinrel\relbar\joinrel\twoheadrightarrow \uRPhg \; $;
   \,by construction, it is the closed two-sided ideal generated
 by the  $ u^E_{ij} $  and  $ u^F_{ij} $  for all
  $ \, i \neq j \, $  in  $ I $.  By  \textit{(a)},  we have
\begin{equation}  \label{eq: descr-J}
  \mathcal{J}  \; = \;  \overrightarrow{\!\mu_3}\Big( \widetilde{U}^{-} \otimes \widetilde{U}_\hbar(\lieh) \otimes \E^+
  + \F^- \otimes \widetilde{U}_\hbar(\lieh) \otimes \widetilde{U}^+ \Big)
\end{equation}
 Now, the kernel of the canonical epimorphism  $ \; p{\big|}_{\widetilde{U}_\hbar(\lieh)} :
 \widetilde{U}_\hbar(\lieh) \relbar\joinrel\relbar\joinrel\twoheadrightarrow U_\hbar(\lieh) := U_{P,\hbar}^{\,\R}(\lieh) \; $
 is nothing but  $ \, \mathcal{J} \cap \widetilde{U}_{P,\hbar}^{\,\R}(\lieh) \, $; \,then by  \eqref{eq: descr-J}
 it is equal to the image under  $ \overrightarrow{\!\mu_3} $  of
  $$
  \Big( \widetilde{U}^- \otimes \widetilde{U}_\hbar(\lieh) \otimes \E^+ + \,
  \F^- \otimes \widetilde{U}_\hbar(\lieh) \otimes \widetilde{U}^+ \Big) \,{\textstyle \bigcap}\,
  \big(\, \kh \otimes \widetilde{U}_{P,\hbar}^{\,\R}(\lieh) \otimes \kh \,\big)
  $$
 and since the latter is obviously trivial, we get that  $ \, \mathcal{J} \cap \widetilde{U}_{P,\hbar}^{\,\R}(\lieh) = 0 \, $.
 Therefore,  $ \; p{\big|}_{\widetilde{U}_\hbar(\lieh)} : \widetilde{U}_\hbar(\lieh) \relbar\joinrel\relbar\joinrel\longrightarrow U_{P,\hbar}^{\,\R}(\lieh) \; $
 is injective, hence it is an isomorphism, q.e.d.
                                                                         \par
   The second part of the claim follows from the first, coupled with
   Proposition \ref{prop: structure Uh(h)}.
%
%
 \vskip7pt
   \textit{(c)}\,  This is proved much like item  \textit{(b)},  but using that
\begin{align*}
  \mathcal{J} \cap \widetilde{U}^+  &  \! =
  \overrightarrow{\!\mu_3}\Big(\! \big( \widetilde{U}^- \otimes \widetilde{U}_\hbar(\lieh) \otimes \E^+ + \,
  \F^- \otimes \widetilde{U}_\hbar(\lieh) \otimes \widetilde{U}^+ \,\big) \,{\textstyle \bigcap}\, \big(\, \kh \otimes \kh \otimes \widetilde{U}^+ \big) \!\Big)  =  \\
     &  \! = \;  \overrightarrow{\!\mu_3}\big(\, \kh \otimes \kh \otimes \E^+ \,\big)  \; = \;  \E^+  \\
  \mathcal{J} \cap \widetilde{U}^-  &  \! = \widetilde{\mu}_3\Big(\! \big( \widetilde{U}^- \otimes \widetilde{U}_\hbar(\lieh) \otimes \E^+ +
  \, \F^- \otimes \widetilde{U}_\hbar(\lieh) \otimes \widetilde{U}^+ \,\big) \,{\textstyle \bigcap}\, \big(\, \widetilde{U}^- \otimes \kh \otimes \kh \,\big) \!\Big)  =  \\
     &  \! = \;  \overrightarrow{\!\mu_3}\big(\, \F^- \otimes \kh \otimes \kh \,\big)  \; = \;  \F^-   \hfill \hskip193pt  \square
\end{align*}
%

\vskip7pt

\begin{obs}  \label{obs: altern-proof x triang-decomp_Utilde}
 An alternative, independent argument which also leads to prove  Proposition \ref{prop: triangulardecUtilde}
 and  Proposition \ref{prop: struct-Cart/nilp-subalg.s}  goes as follows.
                                                          \par
   First, we can state a strict analogue of  Lemma \ref{lemma: alg-struct-on-Uotimes(g)}
   where we replace the  $ \kh $--module  $ \, U^{\R,\otimes}_{\!P,\,\hbar}(\lieg) := U^- \otimes U_\hbar(\lieh) \otimes U^+ \, $
   --- taking  \textsl{complete\/}  tensor product ---   with its ``parent''
   $ \, \widetilde{U}^{\R,\otimes}_{\!P,\,\hbar}(\lieg) := \widetilde{U}^- \otimes \widetilde{U}_\hbar(\lieh) \otimes \widetilde{U}^+ \, $:
   \, the claim will be that  $ \widetilde{U}^{\R,\otimes}_{\!P,\,\hbar}(\lieg) $
   bears a structure of  $ \hbar $--adically  complete, topological  $ \kh $--algebra
   which is uniquely determined by the same recipe (and formulas) as in  Lemma \ref{lemma: alg-struct-on-Uotimes(g)}
   --- the proof will be quite the same, only a bit simpler because there will be less relations to deal with.
                                                          \par
   Second, we provide a strict analogue of  Theorem \ref{thm: Uotimes-cong-uRPhg},
   now concerning  $ \widetilde{U}^{\R,\otimes}_{\!P,\,\hbar}(\lieg) $  and  $ \widetilde{U}^{\,\R}_{\!P,\hbar}(\lieg) \, $
   instead of  $ U^{\R,\otimes}_{\!P,\,\hbar}(\lieg) $  and  $ \uRPhg \, $;  \,here again, the proof will follow in the footsteps
   of the one for the previous case.  This last result eventually will be just a reformulation of
   Proposition \ref{prop: triangulardecUtilde}  and  Proposition \ref{prop: struct-Cart/nilp-subalg.s}  altogether.
\end{obs}

\vskip9pt

%
%
%
%
%
%
%
%
%
%
%
%
%
%
%
%

\begin{free text}  \label{triang-decompos revisited}
 \textbf{Triangular decomposition --- revisited.}
 Using the previous constructions, we present now alternative proofs of ``triangular decomposition'' for FoMpQUEAs.
  Again, it is enough to prove one of the various isomorphisms in
  the statement of  Theorem \ref{thm: triang-decomp.'s},
  so now we shall deal with
  $$  \uRPhnm \,\widehat{\otimes}_\kh\, U_\hbar(\lieh) \,\widehat{\otimes}_\kh\, \uRPhnp  \,\; \cong \;\,  \uRPhg  $$
 where the (would-be) isomorphism is induced by multiplication.
 \vskip7pt
  \textsl{$ \underline{\text{Second Proof}} $:}\,  By construction and by the results reported in
  Proposition \ref{prop: struct-Cart/nilp-subalg.s}, we have isomorphisms of (topological)  $ \kh $--modules
  $$  \displaylines{
   \uRPhg  \; \cong \;  \Big( \widetilde{U}^- \otimes \widetilde{U}_\hbar(\lieh) \otimes \widetilde{U}^+ \Big)
   \bigg/ \Big(\, \F^- \otimes \widetilde{U}_\hbar(\lieh) \otimes \widetilde{U}^+ + \widetilde{U}^- \otimes \widetilde{U}_\hbar(\lieh) \otimes \E^+ \Big)  \; \cong   \hfill  \cr
   \hfill   \cong   \Big( \widetilde{U}^- \! \big/ \F^- \Big) \otimes \widetilde{U}_\hbar(\lieh) \otimes \Big( \widetilde{U}^+ \! \big/ \E^+ \Big)
   \cong \,  U^- \otimes \widetilde{U}_\hbar(\lieh) \otimes U^+  \! =  \uRPhnm \otimes U_\hbar(\lieh) \otimes \uRPhnp  }  $$
 (using simplified notation for the tensor product), where the isomorphism   --- from right to left ---
 is actually induced by multiplication hence we are done.   \hfill  $ \diamond $
 \vskip7pt
   \textsl{$ \underline{\text{Third Proof}} $:}\,  The claim amount to saying that the $\kh$-linear map
\begin{equation}  \label{eq: map mu3}
  \underrightarrow{\mu_3} \, : \, \uRPhnm \,\widehat{\otimes}_\kh\, U_\hbar(\lieh) \,\widehat{\otimes}_\kh\,
  \uRPhnp \relbar\joinrel\relbar\joinrel\relbar\joinrel\relbar\joinrel\longrightarrow \uRPhg
\end{equation}
 induced by multiplication (on three factors) is in fact  \textsl{bijective}.
 \vskip5pt
   To begin with, let  $ U^- $,  resp.\  $ U^0 $,  resp.\  $ U^+ $,  \,be the  $ \kh $--subalgebra  of  $ \uRPhg $
   generated by all the  $ F_i $'  ($ \, i \in I \, $),  resp.\  all of  $ \, \lieh \, $,  resp.\  $ E_i $'s  ($ \, i \in I \, $);
   then let  $ \, U^-_\hbar := \uRPhnm \, $,  resp.\  $ \, U^0_\hbar := U_\hbar(\lieh) \, $,  resp.\
   $ \, U^+_\hbar := \uRPhnp \, $.  Furthermore, let us define
  $ \, U^\downarrow  \, := \,  \textsl{Span}_{\kh\!}\Big( {\big\{\, \underline{F} \cdot \underline{H} \cdot \underline{E} \,\big\}}_{\underline{F} \, ,
  \, \underline{H} \, , \, \underline{E}} \Big) \, $
 where the  $ \, \underline{F} \, $,  resp.\  $ \, \underline{H} \, $,  resp.\  $ \, \underline{E} \, $,
 are all possible monomials in the  $ F_i $'s  ($ \, i \in I \, $),  resp.\ in the  $ H $'s  ($ \, H \in \lieh \, $),
 resp.\ in the  $ E_j $'s  ($ \, j \in I \, $);  note that  $ \, U^\downarrow \, $  is a  $ \kh $--submodule  of  $ \uRPhg \, $,
 \,but  \textsl{not\/}  a  $ \kh $--subalgebra.  Finally, we let  $ U $  be the  $ \kh $--subalgebra  of  $ \uRPhg $
 generated by  $ U^\downarrow \, $.
                                                                      \par
   Clearly, the map  $ \, \underrightarrow{\mu_3} \, $  in  \eqref{eq: map mu3}  restricts to a similar map
\begin{equation}  \label{eq: map mu}
  \mu \, : \, U^- \otimes_\kh U^0 \otimes_\kh U^+ \relbar\joinrel\relbar\joinrel\relbar\joinrel\relbar\joinrel\longrightarrow U^\downarrow
\end{equation}
 which again is induced by multiplication.  We shall presently prove the following
\begin{equation}  \label{eq: Claim}
   \textit{\underline{\textsl{Claim}}:\;  the map  $ \, \mu \, $  in  (\ref{eq: map mu})  is bijective}   \qquad \hskip115pt \qquad
\end{equation}
 Once this is settled, we have that both  $ \mu $  and its inverse  $ \mu^{-1} $  will be (mutually inverse) isomorphisms of  $ \kh $--modules,
 hence in particular continuous for the  $ \hbar $--adic  topology.  Then, taking  $ \hbar $--adic  completion on both sides,
 \textsl{$ \mu $  and  $ \mu^{-1} $  will canonically induce (bicontinuous) mutually inverse isomorphisms of topological  $ \kh $--modules,
 denoted  $ \, \mu_\hbar \, $  and $ \, \mu_\hbar^{-1} \, $,  between the  $ \hbar $--adic  completion of  $ \; U^- \otimes_\kh U^0 \otimes_\kh U^+ $
 and the  $ \hbar $--adic  completion of  $ \, U^\downarrow \, $}.  Now, by construction the  $ \hbar $--adic
 completion of  $ \, U^- \otimes_\kh U^0 \otimes_\kh U^+ \, $  is just  $ \, U^-_\hbar \,\widehat{\otimes}_\kh\, U^0_\hbar \,\widehat{\otimes}_\kh\, U^+_\hbar \, $,
 \,while the  $ \hbar $--adic  completion of  $ \, U^\downarrow \, $  (which coincides with the  $ \hbar $--adic
 completion of  $ \, U \, $)  is nothing but  $ \, \uRPhg \, $.  In particular, again by construction  $ \mu_\hbar $
 coincides with  $ \underrightarrow{\mu}_3 $  in  \eqref{eq: map mu3},  hence the latter in turn is an isomorphism
 of topological  $ \kh $--modules, q.e.d.
 \vskip5pt
   Thus we are left with the task to prove the  \textit{Claim\/}  in  \eqref{eq: Claim}  above.  As it is clear that  $ \mu $  is  \textsl{surjective},  hence we only have to prove that it is  \textit{injective\/}  too.  For this, we use the Hopf structure of  $ \uRPhg \, $,  in particular its coproduct, adapting an argument that does work in the uniparameter case, see e.g.\  \cite{HK}, \S 3.1,  or  \cite{KS}, \S 6.1.5.
 \vskip3pt
   We saw above that  $ U^\downarrow $  is  $ \kh $--spanned by the set of monomials of the form  $ \; \underline{F} \cdot \underline{H} \cdot \underline{E} \; $  where each single factor in turn is of type  $ \, \underline{F} = F_{i_1} \cdots F_{i_n} \, $,  resp.\  $ \, \underline{H} = H_{\ell_1} \cdots H_{\ell_s} \, $,  resp.\  $ \, \underline{E} = E_{j_1} \cdots E_{j_m} \, $,  \,with  $ \, i_1 , \dots i_n , j_1 , \dots , j_m \in I \, $  ($ \, n , m \in \NN \, $)  and  $ H_{\ell_1} , \dots , H_{\ell_s} $  ($ \, s \in \NN \, $)  ranging in some fixed, ordered  $ \kh $--basis  $ \, {\big\{ H_\ell \big\}}_{\ell \in \L} \, $  of  $ \lieh \, $.  Similarly (with same notation),  $ \, U^- \otimes_\kh U^0 \otimes_\kh U^+ \, $  is  $ \kh $--spanned by the set of (tensor) monomial of the form  $ \; \underline{F} \otimes \underline{H} \otimes \underline{E} \; $,  \,with  $ \underline{F} \, $,  $ \underline{H} $  and  $ \underline{E} $  as before.
                                                                         \par
   To begin with, let  $ \G $  be the free Abelian group with basis  $ \, {\{ \varepsilon_i \}}_{i \in I} \, $:
   \, we endow it with the product order, hereafter denoted by  $ \, \preceq \, $,  induced by the standard order in  $ \ZZ \, $.
   We define on  $ U $  a  $ \kh $--algebra  $ \G $--grading  $ \, U = \oplus_{\gamma \in \G} U_\gamma \, $  given on generators by
 \vskip5pt
   \centerline{ $ \, \partial(E_i) := +\varepsilon_i \, ,  \quad  \partial\big(\,\underline{H}\,\big) := 0 \, ,
   \quad  \partial(F_j) := -\varepsilon_j   \qquad  \forall \;\; \, i \, , j \in I \, , \; \underline{H} \in U_\hbar(\lieh) \setminus \{0\} \, $ }
 \vskip5pt
\noindent
 --- the reader can easily check that these formulas on the generators are indeed compatible
 with the relations in  \eqref{eq: comm-rel's_x_uPhg}  among them.  This restricts to similar
 $ \G $--gradings  on  $ U^\downarrow $  as well as on  $ U^- $,  $ U^0 $  and  $ U^+ $
 --- which then are  \textsl{graded\/}  subalgebras of  $ U $  with respect to this  $ \G $--grading  ---
 hence on  $ \, U^- \otimes_\kh U^0 \otimes_\kh U^+ \, $  too.
Note then that each one of the monomials
  $ \; \underline{F} \cdot \underline{H} \cdot \underline{E} = F_{i_1} \cdots F_{i_n} \cdot H_{\ell_1} \cdots H_{\ell_s} \cdot E_{j_1} \cdots E_{j_m} \; $
 and similarly
  $ \; \underline{F} \otimes \underline{H} \otimes \underline{E} \; $
 considered above is  $ \G $--\textsl{homogeneous\/}  of degree
\begin{equation*}  \label{eq: grad-monomials}
   \partial\big( F_{i_1} \cdots F_{i_n} \cdot H_{\ell_1} \cdots H_{\ell_s} \cdot E_{j_1} \cdots E_{j_m} \big) \,
   = \, \varepsilon_{j_1} + \cdots + \varepsilon_{j_m} - \varepsilon_{i_1} - \cdots - \varepsilon_{i_n} \; \in \; \G
\end{equation*}
 \vskip3pt
   Now consider the twofold iteration
  $ \; \Delta^{\!(2)} := (\Delta \otimes \id) \circ \Delta = (\id \otimes \Delta) \circ \Delta \; $
  of the coproduct map  $ \Delta $  of  $ \uRPhg \, $.  By the very definition of  $ \Delta $  we easily see that
\begin{equation}  \label{eq: Delta2_part-monomial}
  \begin{aligned}
     \Delta^{\!(2)}\big( F_{i_1} \cdots F_{i_n} \big)  &  \, = \,  1 \otimes 1 \otimes \big( F_{i_1} \cdots F_{i_n} \big) \, + \, {\textstyle \sum_t u_t \otimes v_t \otimes w_t}  \\
     \Delta^{\!(2)}\big( E_{j_1} \cdots E_{j_m} \big)  &  \, = \,  \big( E_{j_1} \cdots E_{j_n} \big) \otimes 1 \otimes 1 \, + \, {\textstyle \sum_k a_k \otimes b_k \otimes c_k}
  \end{aligned}
\end{equation}
 where the  $ w_t $'s  in the first line of  \eqref{eq: Delta2_part-monomial}  are elements in  $ \uRPhg $
 which are  $ \G $--homogeneous  of degree strictly  \textsl{greater}   --- for the order  $ \, \preceq \, $  in  $ \G $
 ---   than  $ \partial\big( F_{i_1} \cdots F_{i_n} \big) \, $,  \,while similarly the  $ a_k $'s
 in second line are elements in  $ \uRPhg $  which are  $ \G $--homogene\-ous  of degree strictly  \textsl{smaller\/}
 than  $ \, \partial\big( E_{j_1} \cdots E_{j_m} \big) \, $;  \,in short,
\begin{equation}  \label{eq: degree monomials in Delta2}
   \partial(w_t) \succneqq \partial\big( F_{j_1} \cdots F_{j_m} \big)  \quad \forall \;\; t \;\; ,   \qquad
 \partial(a_k) \precneqq \partial\big( E_{i_1} \cdots E_{i_n} \big)  \quad \forall \;\; k \;\; .
\end{equation}
                                                           \par
   Eventually, from  \eqref{eq: Delta2_part-monomial}  and  \eqref{eq: degree monomials in Delta2} together we get
  $$  \displaylines{
   \Delta^{\!(2)}\big( F_{i_1} \cdots F_{i_n} H_{\ell_1} \cdots H_{\ell_s} E_{j_1} \cdots E_{j_m} \big)  \; = \;
   \Delta^{\!(2)}\big( F_{i_1} \cdots F_{i_n} \cdot \underline{H} \cdot E_{j_1} \cdots E_{j_m} \big)  \; =   \hfill  \cr
   \qquad   = \;  \Delta^{\!(2)}\big( F_{i_1} \cdots F_{i_n} \big) \,
   \Delta^{\!(2)}\big( \underline{H} \big) \, \Delta^{\!(2)}\big( E_{j_1} \cdots E_{j_m} \big)  \; =   \hfill  \cr
   \qquad \qquad   = \;  \big( F_{i_1} \cdots F_{i_n} \cdot {\underline{H}}_{(1)} \big) \otimes {\underline{H}}_{(2)}
   \otimes \big(\, {\underline{H}}_{(3)} \cdot E_{j_1} \cdots E_{j_m} \big)  \, + \,  {\textstyle \sum_r} \,
   \varPhi_r \otimes \varXi_r \otimes \varOmega_r   \hfill  }  $$
 where each tensor  $ \, \varPhi_r \otimes \varXi_r \otimes \varOmega_r \, $  lies in
 $ \, U_\phi \otimes U_\xi \otimes U_\omega \, $   --- with  $ \phi \, $,  $ \xi $  and  $ \omega $
 being degrees for the  $ \G $--grading  ---   and obeys the following condition:
\begin{equation}  \label{eq: degree less}
   \partial\big( F_{i_1} \cdots F_{i_n} \big) \precneqq \phi
\qquad  \text{or}  \qquad  \omega \precneqq \partial\big( E_{j_1} \cdots E_{j_m} \big)
\end{equation}
 \vskip3pt
   We have to prove that the map  $ \, \mu \, $  in  \eqref{eq: map mu}  is injective.
   As  $ \, U^- \otimes U^0 \otimes U^+ \, $  is  $ \kh $--spanned by  all the (tensor)
   monomials of the form  $ \, \underline{F} \otimes \underline{H} \otimes \underline{E} \, $
   (notation as above, with the monomial  $ \underline{H} $  being also ordered):  \,so we assume now
\begin{equation}  \label{eq: lc-monom's -> 0}
  \mu\Big(\, {\textstyle \sum_{\sigma \in S}} \, \kappa_\sigma \, \underline{F}_{\,\sigma} \otimes \underline{H}_{\,\sigma} \otimes \underline{E}_{\,\sigma} \Big)  \; = \;  0
\end{equation}
 --- for finitely many  $ \, \kappa_\sigma \in \kh \, $  ---   and we prove that
 $ \; \sum_{\sigma \in S} \kappa_\sigma \, \underline{F}_{\,\sigma} \otimes \underline{H}_{\,\sigma} \otimes \underline{E}_{\,\sigma} \, = \, 0 \; $.
                                                              \par
   First of all,  \eqref{eq: lc-monom's -> 0}  yields
  $$
  {\textstyle \sum_{\sigma \in S}} \, \kappa_\sigma \, \underline{F}_{\,\sigma} \,
  \underline{H}_{\,\sigma} \, \underline{E}_{\,\sigma}  \; = \;
  \mu\Big(\, {\textstyle \sum_{\sigma \in S}} \, \kappa_\sigma \,
  \underline{F}_{\,\sigma} \otimes \underline{H}_{\,\sigma} \otimes \underline{E}_{\,\sigma} \Big)  \; = \;  0
  $$
                                                                  \par
   Second, the previous analysis for  $ \, \Delta^{(2)} \, $  gives, for all indices  $ \sigma \, $,
\begin{equation*}
   \Delta^{\!(2)}\big(\, \underline{F}_{\,\sigma} \, \underline{H}_{\,\sigma} \, \underline{E}_{\,\sigma} \big)  \,  = \,
   \Big(\, \underline{F}_{\,\sigma} \cdot {\big(\, \underline{H}_{\,\sigma} \big)}_{(1)} \Big) \otimes
   {\big(\, \underline{H}_{\,\sigma} \big)}_{(2)} \otimes \Big(\! {\big(\, \underline{H}_{\,\sigma} \big)}_{(3)} \cdot \underline{E}_{\,\sigma} \Big)  \, + \,
   {\textstyle \sum_r} \, \varPhi_r^\sigma \otimes \varXi_r^\sigma \otimes \varOmega_r^\sigma
\end{equation*}
 with the  $ \, \varPhi_r^\sigma \otimes \varXi_r^\sigma \otimes \varOmega_r^\sigma \, $'s  obeying  \eqref{eq: degree less}  above.  Then
  $$
  \displaylines{
   0  \; = \;  \Delta^{\!(2)}(0)  \; = \;  \Delta^{\!(2)}\Big(\, {\textstyle \sum_{\sigma \in S}} \, \kappa_\sigma \, \underline{F}_{\,\sigma} \,
   \underline{H}_{\,\sigma} \, \underline{E}_{\,\sigma} \Big)  \; = \;
   {\textstyle \sum_{\sigma \in S}} \, \kappa_\sigma \,
   \Delta^{\!(2)}\Big(\, \underline{F}_{\,\sigma} \, \underline{H}_{\,\sigma} \, \underline{E}_{\,\sigma} \Big)  \; =   \hfill  \cr
   \hfill   = \;  {\textstyle \sum_{\sigma \in S}} \, \kappa_\sigma \,
   \Big( \Big(\, \underline{F}_{\,\sigma} \cdot {\big(\, \underline{H}_{\,\sigma} \big)}_{(1)} \Big) \otimes
   {\big(\, \underline{H}_{\,\sigma} \big)}_{(2)} \otimes \Big(\! {\big(\, \underline{H}_{\,\sigma} \big)}_{(3)} \cdot \underline{E}_{\,\sigma} \Big)  \, +
   \,  {\textstyle \sum_r} \, \varPhi_r^\sigma \otimes \varXi_r^\sigma \otimes \varOmega_r^\sigma \,\Big)  }
   $$
 Now, we select those  $ \sigma $  in  $ S $  for which  $ \partial\big(\,\underline{F}_{\,\sigma}\big) $
 has some minimal value   --- in  $ \big(\,\G\,;\preceq\!\big) $ ---   say  $ \check{\mu} \, $,
 and simultaneously  $ \partial\big(\,\underline{E}_{\,\sigma}\big) $  has a maximal value, say
 $ \hat{\mu} \, $;  \,we denote by  $ S^{\hat{\mu}}_{\check{\mu}} $  the subset of such indices.
 Then by degree comparison, we see   --- cf.\  \eqref{eq: degree less}  ---   that
                                                         \par
   \centerline{ $ \; {\textstyle \sum_{\sigma \in S^{\hat{\mu}}_{\check{\mu}}}} \,
   \kappa_\sigma \, \Big(\, \underline{F}_{\,\sigma} \cdot {\big(\, \underline{H}_{\,\sigma} \big)}_{(1)} \Big)
   \otimes {\big(\, \underline{H}_{\,\sigma} \big)}_{(2)} \otimes \Big(\! {\big(\, \underline{H}_{\,\sigma} \big)}_{(3)}
   \cdot \underline{E}_{\,\sigma} \Big) \; $ }
\noindent
 is the whole homogeneous summand in
  $$
  {\textstyle \sum_{\sigma \in S}} \, \kappa_\sigma \,
  \Big( \Big(\, \underline{F}_{\,\sigma} \cdot {\big(\, \underline{H}_{\,\sigma} \big)}_{(1)} \Big) \otimes
  {\big(\, \underline{H}_{\,\sigma} \big)}_{(2)} \otimes \Big(\! {\big(\, \underline{H}_{\,\sigma} \big)}_{(3)}
  \cdot \underline{E}_{\,\sigma} \Big) \, + \, {\textstyle \sum\limits_r} \, \varPhi_r^\sigma \otimes \varXi_r^\sigma \otimes \varOmega_r^\sigma \,\Big)
  $$
 of degree  $ \, \big(\, \check{\mu} \, , 0 \, , \hat{\mu} \,\big) \, $  with respect to the grading by
 $ \, \G \times \G \times \G \, $  in  $ \, U \otimes U \otimes U \, $  canonically induced by the
 $ \G $--grading  of  $ U \, $.  For this reason, the identity
  $$
  {\textstyle \sum\limits_{\sigma \in S}} \, \kappa_\sigma \, \Big( \Big(\, \underline{F}_{\,\sigma} \cdot
  {\big(\, \underline{H}_{\,\sigma} \big)}_{(1)} \Big) \otimes {\big(\, \underline{H}_{\,\sigma} \big)}_{(2)} \otimes
  \Big(\! {\big(\, \underline{H}_{\,\sigma} \big)}_{(3)} \cdot \underline{E}_{\,\sigma} \Big)  \, + \,
  {\textstyle \sum_r} \, \varOmega_r^\sigma \otimes \varXi_r^\sigma \otimes \varPhi_r^\sigma \,\Big)  \; = \;  0
  $$
 found above implies at once
\begin{equation}  \label{eq: id=0 x Shatcheck}
  {\textstyle \sum_{\sigma \in S^{\hat{\mu}}_{\check{\mu}}}} \, \kappa_{\sigma} \, \Big(\, \underline{F}_{\,\sigma}
  \cdot {\big(\, \underline{H}_{\,\sigma} \big)}_{(1)} \Big) \otimes {\big(\, \underline{H}_{\,\sigma} \big)}_{(2)} \otimes
  \Big(\! {\big(\, \underline{H}_{\,\sigma} \big)}_{(3)} \cdot \underline{E}_{\,\sigma} \Big)  \; = \;  0
\end{equation}
   \indent   Now observe that  $ U^0 $  admits as  $ \kh $--basis  the set of all  \textsl{ordered\/}  monomials
   $ \underline{H} $  in the  $ H_\ell $'s  (i.e., we assume that  $ \, H_{\ell_1} \preceq \cdots \preceq H_{\ell_s} \, $),
   directly by construction and by  Proposition \ref{prop: structure Uh(h)}.  Let  $ \D $  be the Abelian group of rank
   $ \, t := \rk(\lieh) \, $  with basis  $ \, {\{ \delta_\ell \}}_{\ell \in \L} \, $,  that we endow with the product order, again
   denoted by  $ \, \preceq \, $,  induced by the standard order in  $ \ZZ \, $.  There is a natural  $ \D $--grading  on
   $ U^0 $  such that  $ \, \partial\big(H_\ell\big) = \delta_\ell \, $  for all elements of the fixed basis
   $ \, {\big\{ H_\ell \big\}}_{\ell \in \L} \, $,  \,whence  $ \, \underline{H} = H_{\ell_1} \cdots H_{\ell_s} \, $  is
   $ \D $--homogenenous  of degree  $ \, \delta_{\ell_1} + \cdots +\delta_{\ell_s} \, $.
   Also, as the elements of $ \lieh $  are primitive in  $ \uRPhg \, $,  we have
\begin{equation}  \label{eq: Delta2_H-monomial}
     \Delta^{\!(2)}\big( H_{\ell_1} \cdots H_{\ell_s} \big)  \, = \,
     1 \otimes \big( H_{\ell_1} \cdots H_{\ell_s} \big) \otimes 1 \, + \, {\textstyle \sum_r x_r \otimes y_r \otimes z_r}
\end{equation}
 where the  $ y_r $'s  are elements in  $ U^0 $  which are homogeneous
 --- for the  $ \D $--grading mentioned above ---   of degree strictly  \textsl{lower\/}
 than that of  $ \, H_{\ell_1} \cdots H_{\ell_s} \, $,  \,that is
\begin{equation}  \label{eq: degree H-monomials in Delta3}
   \partial(y_r) \, \precneqq \, \partial\big( H_{\ell_1} \cdots H_{\ell_s} \big) \, = \, s   \qquad \forall \;\; r \; .
\end{equation}
   \indent   The left-hand side of  \eqref{eq: id=0 x Shatcheck}  belongs to  $ \, U^-_{\check{\mu}} \otimes U^0 \otimes U^+_{\hat{\mu}} \, $;
   taking into account the  $ \D $--grading  in  $ U^0 $  mentioned above, the identity
   \eqref{eq: id=0 x Shatcheck}  implies that each homogeneous component
   --- with respect to the obvious grading of  $ \, U^- \otimes U^0 \otimes U^+ \, $  by  $ \, \G \times \D \times \G \, $  ---
   in the left-hand side of  \eqref{eq: id=0 x Shatcheck}  must be zero as well.
   In particular, let us focus on a single monomial  $ \, \underline{H}_{\,\bar{\sigma}} = H_{\ell^{\bar{\sigma}}_1} \cdots H_{\ell^{\bar{\sigma}}_s} \, $
   which actually occurs in  \eqref{eq: id=0 x Shatcheck},  having  \textsl{maximal\/}
   degree in  $ \, \big(\, \D \, ; \preceq \big) \, $:  \,then for the  $ \big(\, \G \times \D \times \G \,\big) $--homogeneous
   component of degree
   $ \, \big(\, \check{\mu} \, , \partial\big(\,\underline{H}_{\,\bar{\sigma}}\big) \, , \hat{\mu} \,\big) = \big(\, \check{\mu} \, ,
   \delta_{\ell^{\bar{\sigma}}_1} + \cdots + \delta_{\ell^{\bar{\sigma}}_s} \, , \hat{\mu} \,\big) \, $
   in  \eqref{eq: id=0 x Shatcheck}  we find, by  \eqref{eq: Delta2_H-monomial}
   and  \eqref{eq: degree H-monomials in Delta3}
\begin{equation}  \label{eq: id=0 x Shatcheck-TOP}
  {\textstyle \sum_{\sigma \in S^{\hat{\mu}}_{\check{\mu}}(\ell)}} \, \kappa_{\sigma} \,
  \underline{F}_{\,\sigma} \otimes \underline{H}_{\,\sigma} \otimes \underline{E}_{\,\sigma}  \; = \;  0
\end{equation}
 where  $ \, S^{\hat{\mu}}_{\check{\mu}}(\ell) := \big\{\, \sigma \in S^{\hat{\mu}}_{\check{\mu}} \,
 \big|\, \underline{H}_{\,\sigma} = \underline{H}_{\,\bar{\sigma}} \,\big\} \, $
 is a non-empty subset of  $ S^{\hat{\mu}}_{\check{\mu}} \, $.
 But now  \eqref{eq: lc-monom's -> 0}  and  \eqref{eq: id=0 x Shatcheck-TOP}  jointly provide the new,  \textsl{shorter\/}  linear combination
  $$
  {\textstyle \sum_{\sigma \in S \setminus S^{\hat{\mu}}_{\check{\mu}}}} \,
  \kappa_\sigma \, \underline{E}_{\,\sigma} \otimes \underline{H}_{\,\sigma} \otimes \underline{F}_{\,\sigma}  \; = \;
  {\textstyle \sum_\sigma} \, \kappa_\sigma \, \underline{E}_{\,\sigma} \otimes \underline{H}_{\,\sigma} \otimes \underline{F}_{\,\sigma}
  \, - \, {\textstyle \sum_{\sigma \in S^{\hat{\mu}}_{\check{\mu}}}} \,
  \kappa_{\sigma} \, \underline{E}_{\,\sigma} \otimes \underline{H}_{\,\sigma} \otimes \underline{F}_{\,\sigma}
  $$
 that still belongs to  $ \, \Ker(\mu) \, $:  \,applying again the same argument, and iterating, we eventually end up with
  $ \; {\textstyle \sum_{\sigma \in S}} \, \kappa_\sigma \,
  \underline{E}_{\,\sigma} \otimes \underline{H}_{\,\sigma} \otimes \underline{F}_{\,\sigma} \, = \, 0 \; $,  \;q.e.d.   \hfill  $ \diamond $
\end{free text}

\bigskip

\subsection{Construction of FoMpQUEAs}  \label{sec: constr-FoMpQUEAs}
{\ }

\smallskip

   In this section we provide
 two new, independent constructions (with respect to what is done in  \S \ref{sec: form-MpQUEAs})
 of FoMpQUEAs.  Namely, when  $ \R $  is  \textsl{split minimal\/}  we construct a FoMpQUEA  $ \uRPhg $
 --- with its whole structure of Hopf algebra ---   first as (a slight variation of) a  \textit{quantum double\/}
 (or  \textit{``Drinfeld's double''\/}),  and them as a  \textit{double cross product}.
 Then from this special case (via  Proposition \ref{prop: exist-realiz's},  Lemma \ref{lemma: split-lifting}
 and  Proposition \ref{prop: central-Hopf-extens_FoMpQUEAs})  we deduce the existence
 --- and explicit presentation ---   of  $ \uRPhg $  for  \textsl{any\/}  realization  $ \R $  as a quotient of
 $ U^{\,\R'}_{\!P,\hskip0,7pt\hbar}(\hskip0,5pt\lieg) $  for a suitable, split realization  $ \R' \, $.
                                                                      \par
   It is worth explaining a bit what is the general scheme beneath our presentation.
   The construction of the quantum double applies to any pair of Hopf algebras (possibly topological) over a ring  $ R \, $,
   with an  $ R $--valued  skew-Hopf pairing between them.
   Typically, this applies to any pair of Hopf algebras which are dual to each other,
   and their canonical (evaluation) pairing.  Now assume we do that for some QUEA, say  $ \uhg \, $,
   together with its dual (in topological sense)  $ \, {\uhg}^* =: F_\hbar[[G]] \, $
   --- the latter being a ``quantum formal series Hopf algebra'' (=:QFSHA), in Drinfeld's terminology.
   Then the corresponding quantum double  $ \, D \big( \uhg \, , F_\hbar[[G]] \big) \, $  will be a ``quantum object''
   --- isomorphic to  $ \, \uhg \otimes F_\hbar[[G]] \, $  as a coalgebra ---   whose specialization at  $ \, \hbar = 0 \, $  will be the Drinfeld's double  $ \, D \big( U(\lieg) \, , F[[G]] \big) \, $;  \,this means that, roughly speaking,  $ \, D \big( \uhg \, , F_\hbar[[G]] \big) \, $  is indeed ``half a QUEA'' and ``half a QFSHA''.  Therefore, if one aims instead to get a new, full QUEA out of the initial QUEA  $ \uhg \, $,  then one has to modify the previous construction; indeed, there exists a general recipe to perform such a modification  (see  \cite{ES}, \S 12.2) which in turn relies on Drinfeld's ``Quantum Duality Principle'' which allows one to ``extract'' a suitable QUEA out of a QFSHA  (cf.\ \cite{Ga2}  and references therein).  In our presentation we will  \textsl{not\/}  formally apply this general recipe: instead, we will present an ad hoc construction, tailored to the specific situation we have at hand.  However, whatever we do is directly dictated, step by step, by the general recipe, only we display our construction in layman's terms just to spare the reader some extra
   theoretical tools that are not needed in full generality.
   Nevertheless, it is worth stressing that we are actually applying the general recipe,
   even we do not show it in full light: yet it is there, standing in the backstage.
                                                                      \par
   As a first goal, we aim to construct a suitable quantum double of Borel-like FoMpQUEAs,
   starting from a pairing among Borel FoMpQUEAs: to this end, we need to step back and
   introduce ``pre-Borel'' FoMpQUEAs instead, a pairing  \textsl{with values in\/}  $ \khp $
   among them, and Borel FoMpQUEAs as quotients of pre-Borel ones.
   As second step, we show that this pairing ``pushes-forward'' to Borel FoMpQUEAs,
   hence can be used to perform a quantum double construction; actually,  \textit{a priori\/}
   this would not be feasible, because the pairing is valued in  $ \khp $  rather than  $ \kh \, $:
   nevertheless, we prove that in the present case the quantum double construction indeed does work
   (in a suitable sense) over  $ \kh $  as well, hence we are done.
                                                                      \par
   Along the way, another obstruction we encounter is that the construction of the pairing
   we would need actually clashes with  $ \hbar $--completeness  of (pre-)Borel FoMpQUEAs;
   therefore, we scale down to constructing a pairing defined on some dense, non-complete subalgebras
   --- of (pre-)Borel FoMpQUEAs ---   and then we manage to carry out the quantum double
   construction we are looking for.
                                                                      \par
   Finally, we will present yet another construction
   --- parallel to that via quantum doubles ---
   of FoMpQUEAs in terms of  \textit{double cross products\/}  of Borel FoMpQUEAs.

\medskip

\begin{free text}  \label{quantum pre-Borel}
 {\bf Pre-Borel FoMpQUEAs and their pairings.}\,
 Our first purpose is to construct quantum doubles of Borel
 FoMpQUEAs of split, minimal type; for this, we need a suitable pairing among such Borel FoMpQUEAs.
 To this end, we need to step back and introduce ``pre-Borel'' FoMpQUEAs instead,
 a pairing with values in  $ \khp $  among them,
 and Borel FoMpQUEAs as quotients of pre-Borel ones.
\end{free text}

\vskip9pt

\begin{definition}  \label{def: pre-Borel_FoMpQUEAs}
 Let  $ \, A := {\big(\, a_{i,j} \big)}_{i, j \in I} \, $  be a generalized symmetrizable Cartan matrix,
 $ \, P := {\big(\, p_{i,j} \big)}_{i, j \in I} \in M_{n}\big( \kh \big) \, $
 a matrix of Cartan type with associated Cartan matrix $ A $  and
 $ \,\R= \big(\, \lieh \, , \Pi \, , \Pi^\vee \,\big) \, $  a  \textsl{split minimal\/} realization of  $ P \, $,
 \,so that  $ \, \lieh = \textsl{Span}_\kh\big( {\big\{ T_i^+ , T_i^- \big\}}_{i \in I} \,\big) \, $
and it has rank  $ \, 2\,n \, $  (cf.\ Definition \ref{def: realization of P}).
                                                                 \par
   We define the  \textsl{positive, resp.\ negative, pre-Borel
   formal multiparameter quantum universal enveloping algebra}
   --- in short  \textsl{positive},  resp.\  \textsl{negative, pre-Borel FoMpQUEA}  ---
   with multiparameter  $ P $  as being the free unital, associative, topological,
   $ \hbar $--adically  complete algebra over  $ \kh \, $,  denoted by  $ \utildeRPhbp \, $,
   resp.\ by  $ \utildeRPhbm \, $,  with generators  $ \, T_i^+ $,  $ E_i \, $  ($ \, i \in I \, $),
 resp.\ $ \, T_i^- $, $ F_i \, $  ($ \, i \in I \, $).
                                                        \par
   Moreover, we give to  $ \utildeRPhbp \, $,  resp.\ to  $ \utildeRPhbm \, $,
   the unique, topological Hopf $ \kh $--algebra  structure uniquely defined by (for all  $ \, i \in I \, $)
  $$  \begin{aligned}
   \qquad  \Delta\big(\,T_i^+\big)  \; = \;  T_i^+ \otimes 1 \, + \, 1 \otimes T_i^+  \;\,  ,
\quad  \SS\big(\,T_i^+\big)  \, = \,  -T_i^+ \;\; ,   \quad  \epsilon\big(\,T_i^+\big) \, = \, 0  \\
   \Delta\big(E_i\big) \, = \, E_i \otimes 1 + e^{+\hbar \, T^+_i} \!\!\otimes E_i  \; ,
\quad  \SS\big(E_i\big) \, = \, - e^{-\hbar \, T^+_i} E_i  \; ,   \quad  \epsilon\big(E_i\big) \, = \, 0
\end{aligned}  $$
 for  $ \utildeRPhbp \, $,  \;and for  $ \utildeRPhbm $  in turn by (for all  $ \, i \in I \, $)
  $$  \begin{aligned}
   \qquad  \Delta\big(\,T_i^-\big)  \; = \;  T_i^- \otimes 1 \, + \, 1 \otimes T_i^-  \;\,  ,  &
\;\;\quad  \SS\big(\,T_i^-\big)  \, = \,  -T_i^-  \;\; ,   \;\;\quad  \epsilon\big(\,T_i^-\big) \, = \, 0  \\
   \hfill   \Delta\big(F_i\big) \, = \, F_i \otimes e^{-\hbar \, T^-_i} + 1 \otimes F_i  \; ,  &
\quad  \SS\big(F_i\big) \, = \, - F_i \, e^{+\hbar \, T^-_i}  \; ,   \quad  \epsilon\big(F_i\big) \, = \, 0
\hfill \qquad \qquad  \diamondsuit
\end{aligned}  $$
\end{definition}

\vskip7pt

   From now on, we work with fixed positive and negative pre-Borel FoMpQUEAs
   $ \utildeRPhbpm $  as above; in the following construction the two will play asymmetric roles,
   but one can also reverse those roles   --- switching  $ \utildeRPhbp $  and  $ \utildeRPhbm $
   among them ---   and eventually get exactly the same outcome.

\vskip9pt

\begin{definition}  \label{def: dual-pre-Borel_FoMpQUEAs}
 Let us consider  $ \, \bar{T}_t^- := \hbar \, T_t^- \, , \, \bar{F}_t := \hbar \, F_t \, \in \, \utildeRPhbm \, $,  \,for $ \, t \in I \, $.
                                                                  \par
   We define  $ \calutildeRPhbm $  to be the  $ \hbar $--adic  closure in  $ \utildeRPhbm $
   of the unital  $ \kh $--subalgebra  generated by  $ \, {\big\{ \bar{T}_\ell^- , \bar{F}_\ell \big\}}_{\ell \in I} \, $.
                                                                  \par
   Similarly, we define  $ \calutildeRPhbp $  to be the  $ \hbar $--adic  closure in  $ \utildeRPhbp $
   of the unital  $ \kh $--subalgebra  generated by
   $ \, {\big\{ \bar{T}_k^+ := \hbar \, T_k^+ , \bar{E}_k := \hbar \, E_k \big\}}_{k \in I} \, $.   \hfill  $ \diamondsuit $
\end{definition}

\vskip9pt

   The following, technical result is obvious  from definitions

\vskip11pt

\begin{lema}  \label{lem: propt.'s_dual-preBorel}
 Let  $ \,\; \widetilde{\E}_\pm^{\,(\hbar)} := \textsl{Ker}\Big( \epsilon_{\utildeRPhbpm} \Big) \bigcap\,
  \calutildeRPhbpm + \hbar \, \calutildeRPhbpm \; $.  Then:
 \vskip5pt
   (a)\;  $ \calutildeRPhbpm $  is complete with respect to the  $ \widetilde{\E}_\pm^{\,(\hbar)} \! $--adic
   topology, and  $ \, {\big\{ \bar{T}_i^+ , \bar{E}_i \big\}}_{i \in I} \, $,  resp.  $ \, {\big\{ \bar{T}_j^- , \bar{F}_j \big\}}_{j \in I} \, $,
   is a set of topological generators of it;
 \vskip5pt
   (b)\;  $ \calutildeRPhbpm $  is a  \textsl{Hopf}\/  $ \kh $--subalgebra  (in topological sense) of  $ \, \utildeRPhbpm \, $.
\qed
\end{lema}

\vskip9pt

   The key point with pre-Borel FoMpQUEAs is our next result, whose proof is more or less standard in Hopf theory.

\vskip13pt

\begin{prop}  \label{prop: pairing_x_pre-Borel}
 There exists a  $ \kh $--linear  skew-Hopf pairing
  $$  \widetilde{\pi} \, : \, \utildeRPhbp \mathop{\widehat\otimes}\limits_\kh \calutildeRPhbm
  \relbar\joinrel\relbar\joinrel\relbar\joinrel\relbar\joinrel\relbar\joinrel\longrightarrow \kh  $$
 uniquely given
   --- for all $ \, i \, , j \in I \, $ ---   by
  $$  \displaylines{
   \widetilde{\pi}\big(\, T_i^+ , \bar{T}_j^- \big) = \alpha_i\big(\,T_j^-\,) = \alpha_j(\,T_i^+\,) = p_{ij}  \,\; ,
   \;\;\;  \widetilde{\pi}\big(\, T_i^+ \, ,1 \,\big) = 0 = \widetilde{\pi}\big(\, 1 \, , \bar{T}_j^- \big)  \cr
  \widetilde{\pi}\big(\, T_i^+ \, , \bar{F}_j \,\big)  \, = \,  0  \, = \,
  \widetilde{\pi}\big(\, E_i \, , \bar{T}_j^- \,\big)  \qquad ,   \quad \qquad  \widetilde{\pi}\big(\, 1 , \bar{F}_j \,\big)  \, = \,  0
  \, = \,  \widetilde{\pi}\big(\, E_i \, , 1 \,\big)  \cr
   \qquad  \widetilde{\pi}\big(\, 1 \, ,1 \,\big)  \, = \,
   1  \quad \qquad ,   \qquad \qquad \quad  \widetilde{\pi}\big( E_i \, , \bar{F}_j \,\big)  \, =
   \,  {{\delta_{i{}j} \, \hbar \;} \over {\; q_i^{+1} - \, q_i^{-1} \;}}  }  $$
\end{prop}

\pf
 Assume first that such a skew-Hopf pairing exists: then it is uniquely determined by its values on the
 (topological) algebra generators of  $ \, \utildeRPhbp \, $  and  $ \, \calutildeRPhbm \, $  chosen in the sets
 $ \, {\big\{ T_i^+ , E_i \big\}}_{i \in I} \, $  and  $ \, {\big\{ \bar{T}_j^- , \bar{F}_j \big\}}_{j \in I} \, $,  \,respectively;
  indeed,
%
%
%
 this follows from repeated applications of formulas  \eqref{eqn:skew-Hpair-1}  and  \eqref{eqn:skew-Hpair-2}
 in  Definition \ref{def_(skew-)Hopf-pairing}  along with the fact that  $ \, \calutildeRPhbm \, $  is a Hopf (sub)algebra.
%
%
 Therefore, once the above mentioned values are specified as in the statement,
 this proves uniqueness.
                                                                               \par
   To show that such a pairing exists, it is equivalent to prove, as usual, that there
   exists an algebra anti-homomorphism
   $ \; \gamma: \utildeRPhbp \relbar\joinrel\relbar\joinrel\relbar\joinrel\longrightarrow
   {\calutildeRPhbm}^{\!*} \, $,  \,where  $ \, {\calutildeRPhbm}^{\!*} \, $  is the linear dual
   $ \; {\calutildeRPhbm}^{\!*} \! := \Hom_\kh\!\Big( \calutildeRPhbm \, , \kh \Big) \, $.
   This is known once it is assigned on the free (topological) generators of  $ \utildeRPhbp $  picked from
   $ \, {\big\{ T_i^+ , E_i \big\}}_{i \in I} \, $,  \,and to define  $ \gamma $  on
   those elements we use the coproduct on them, because  $ \calutildeRPhbm $
   is freely (topologically) generated by  $ \, {\big\{\, \bar{T}_j^- , \bar{F}_j \big\}}_{i \in I} \, $.
                                                                               \par
   Given an augmented  $ R $--algebra  $ (A\,,\epsilon\,) $  over a ring  $ R \, $,  a map
   $ d : A \longrightarrow R \, $ is called a  \textit{derivation\/}  if
   $ \, d(xy) = d(x) \epsilon(y) + \epsilon(x) d(y) \, $  for all  $ \, x , y \in A \, $.
   More generally, for two algebra maps  $ \, \alpha \, , \beta \in \Alg_R(A,R\,) \, $,
   \,an \textit{$ (\alpha,\beta\,) $--derivation\/}  is a map  $ \, d : A \longrightarrow R \, $
   such that  $ \; d(xy) = d(x) \, \alpha(y) + \beta(x) \, d(y) \; $  for all  $ \, x , y \in A \, $.
                                                                               \par
   Taking into account that the  $ T_i^+ $'s  are primitive, the  $ \, K_i := e^{+\hbar\,T_i^+} \, $
   are group-like and the  $ E_i $'s  are  $ (1,K_i) $--primitive,  for all  $ \, i \in I \, $,
   we define in  $ \, {\calutildeRPhbm}^{\!*} $  the derivation  $ \tau_i \, $,
   the algebra morphism  $ \kappa_i $  and the  $ (\epsilon,\kappa_i) $--derivation  $ \eta_i $  by
  $$
  \displaylines{
   \tau_i\big(\bar{T}_j^-\big) \; := \; p_{ij}  \;\; ,  \tau_i\big(\bar{F}_j\big) \; := \; 0
   \quad  \kappa_i\big(\bar{T}_j^-\big) \; := \; p_{ij}  \;\; ,  \quad  \kappa_i\big(\bar{F}_j\big) \; := \; 0   \qquad  \forall \;\; j \in I  \cr
   \eta_i\big(\bar{T}_j^-\big) \; := \; 0  \quad ,  \qquad  \eta_i\big(\bar{F}_j\big)  \;
   := \; \delta_{i{}j} \, \hbar \, {\big( q_i^{+1} - \, q_i^{-1} \big)}^{-1}   \quad \qquad  \forall \;\; j \in I  }
   $$

  \indent   Now consider  $ \, {\calutildeRPhbm}^{\!*} $  as an algebra with the convolution product,
  that is,  $ \; (fg)(x) := f\big(x_{(1)}\big) \, g\big(x_{(2)}\big) \; $  for all  $ \, f , g \in {\calutildeRPhbm}^{\!*} \, $
  and  $ \, x \in \utildeRPhbm \, $.  As  $ \utildeRPhbp $  is the free (topological) algebra generated by
  $ \, {\big\{ T_i^+ , E_i \big\}}_{i \in I} \, $,  \,one has an algebra anti-homomorphism
  $ \; \gamma : \utildeRPhbp \relbar\joinrel\relbar\joinrel\longrightarrow {\calutildeRPhbm}^{\!*} \; $
  given by  $ \, \gamma\big(T_i^+\big) := \tau_i \, $  and  $ \, \gamma(E_i) := \eta_i \, $  for all
  $ \, i \in I \, $.  Let  $ \; \widetilde{\pi} : \utildeRPhbp \otimes_\kh \calutildeRPhbm \relbar\joinrel\relbar\joinrel\longrightarrow \kh \; $
  be the linear map defined by  $ \, \widetilde{\pi}\big(x,y\big) := \big(\gamma(x)\big)(y) \, $  for all  $ x \in \utildeRPhbp \, $
  and  $ \, y \in \calutildeRPhbm \, $;  \,then by the very construction of  $ \gamma \, $,  condition  \eqref{eqn:skew-Hpair-2}  is satisfied.
  On the other hand, condition \eqref{eqn:skew-Hpair-1}  is satisfied because it is satisfied on the generators
  $ T_i^+ $ and $ E_i $  ($ \, i \in I \, $)  and the comultiplication is an algebra map;
  the same holds for the conditions \eqref{eqn:skew-Hpair-3}.
  Finally, one may prove conditions  \eqref{eqn:skew-Hpair-4}
  concerning the antipode using again the values on the generators, as both  $ \SS $  and  $ \SS^{-1} $
  are algebra and coalgebra anti-homomorphisms.
\epf

\vskip7pt

\begin{rmk}  \label{rmk: sym-pairing}
 It is clear by construction that one can also introduce a topological Hopf subalgebra
 $ \, \calutildeRPhbp \, $  of  $ \utildeRPhbp \, $,  \,for which the analog of
 Lemma \ref{lem: propt.'s_dual-preBorel}  holds true, and a suitable skew-Hopf pairing
 $ \; \widetilde{\pi} \, : \, \calutildeRPhbp \mathop{\widehat\otimes}\limits_\kh
 \utildeRPhbm \relbar\joinrel\relbar\joinrel\relbar\joinrel\relbar\joinrel\relbar\joinrel\longrightarrow \kh \; $
 similar to the one in  Proposition \ref{prop: pairing_x_pre-Borel},  and denoted again by  $ \widetilde{\pi} \, $.
 Moreover, both  $ \, \utildeRPhbp \!\mathop{\widehat\otimes}\limits_\kh\! \calutildeRPhbm \, $  and
 $ \, \calutildeRPhbp \!\mathop{\widehat\otimes}\limits_\kh\! \utildeRPhbm \, $  will embed in
 $ \, \utildeRPhbp \!\mathop{\widehat\otimes}\limits_\kh\! \utildeRPhbm \, $,  their intersection will coincide with
 $ \, \calutildeRPhbp \mathop{\widehat\otimes}\limits_\kh \calutildeRPhbm \, $,  and the restrictions to this last submodule
 of the two pairings considered so far will coincide.
\end{rmk}

\vskip13pt

\begin{free text}  \label{pre-Borel ---> Borel}
 {\bf From pre-Borel FoMpQUEAs to Borel FoMpQUEAs.}\,
 We introduce now  \textsl{Borel FoMpQUEAs},  as quotients of pre-Borel FoMpQUEAs: indeed,
the former are carefully devised so to inherit from the latter all possible ``good'' properties.
\end{free text}

\vskip11pt

\begin{definition}  \label{def: Borel_FoMpQUEAs}
  Let pre-Borel FoMpQUEAs  $ \, \utildeRPhbpm $  be given as in  Definition \ref{def: pre-Borel_FoMpQUEAs}.
  We define  $ \widetilde{I}_+ $  to be the closure   --- in the  $ \hbar $--adic  topology ---
  of the two-sided ideal in  $ \, \utildeRPhbp $  generated by the elements
  $$  \displaylines{
   T^+_{i,j} \, := \, T_i^+ \, T_j^+ - \, T_j^+ \, T_i^+  \; ,
\quad  E^{(T)}_{i,j} \, := \, T_i^+ \, E_j \, - \, E_j \, T_i^+ - \, \alpha_j(T_i^+) \, E_j   \hskip21pt  (\, i , \, j \in I \,)
  \cr
   E_{i,j}  \, := \,  \sum\limits_{k = 0}^{1-a_{ij}} (-1)^k {\left[ { 1-a_{ij} \atop k }
\right]}_{\!q_i} q_{ij}^{+k/2\,} q_{ji}^{-k/2} \, E_i^{1-a_{ij}-k} \, E_j \, E_i^k   \quad \hskip31pt   (\, i \neq j \,)  }  $$
 and we define  $ \widetilde{I}_- $  to be the closure   --- in the  $ \hbar $--adic  topology ---
 of the two-sided ideal in  $ \, \utildeRPhbm $  generated by all the elements
  $$  \displaylines{
   T^-_{i,j} \, := \, T_i^- \, T_j^- \, - \, T_j^- \, T_i^-  \; ,
\quad   F^{(T)}_{i,j} \, := \, T_i^- \, F_j \, - \, F_j \, T_i^- \, + \, \alpha_j(T_i^-) \, E_j   \hskip21pt  (\, i , \, j \in I \,)
  \cr
   F_{i,j}  \, := \,  \sum\limits_{k = 0}^{1-a_{ij}} (-1)^k {\left[ { 1-a_{ij} \atop k }
\right]}_{\!q_i} q_{ij}^{+k/2\,} q_{ji}^{-k/2} \, F_i^k \, F_j \, F_i^{1-a_{ij}-k}   \quad \hskip25pt   (\, i \neq j \,)  }  $$
   \indent  Moreover, we define the  \textsl{positive, resp.\ negative,
   Borel formal multiparameter quantum universal enveloping algebra}
   --- in short  \textsl{positive, resp.\ negative, Borel FoMpQUEA}  ---
   with multiparameter  $ P $  as being the quotient  $ \;\uRPhbp := \utildeRPhbp \Big/ \widetilde{I}_+ \; $,  \,resp.\
   $ \; \uRPhbm := \utildeRPhbm \Big/ \widetilde{I}_- \; $.
 With a standard abuse of notation, hereafter  \textsl{we shall denote with the same symbol any element in
 $ \utildeRPhbpm $  as well as its image (via the quotient map) in the quotient algebra}
 $ \,\; \uRPhbpm := \utildeRPhbpm \Big/ \widetilde{I}_\pm \;\, $.   \hfill  $ \diamondsuit $
\end{definition}

\vskip5pt

   We need also similar definitions for the (topological) Hopf subalgebras  $ \, \calutildeRPhbpm \, $:

\vskip11pt

\begin{definition}  \label{def: dual-Borel_FoMpQUEAs}
 Let  $ \, \calutildeRPhbpm $  be defined as in  Definition \ref{def: dual-pre-Borel_FoMpQUEAs},
 and consider in it  $ \; \widetilde{\E}_\pm^{\,(\hbar)} :=
 \textsl{Ker}\Big( \epsilon_{\utildeRPhbpm} \Big) \bigcap\, \calutildeRPhbpm + \hbar \, \calutildeRPhbpm \, $.
We define  $ \widetilde{\I}_+ $  to be the closure, in the  $ \widetilde{\E}_+^{\,(\hbar)} $--adic  topology,
  of the two-sided ideal in  $ \, \calutildeRPhbp $  generated by the elements
  $$  \displaylines{
   \bar{T}^+_{i,j} \, := \, \bar{T}_i^+ \, \bar{T}_j^+ - \, \bar{T}_j^+ \, \bar{T}_i^+  \; ,
\quad  \bar{E}^{(T)}_{i,j} \, := \, \bar{T}_i^+ \, \bar{E}_j \, - \, \bar{E}_j \, \bar{T}_i^+ - \, \hbar \, \alpha_j\big( \bar{T}_i^+ \big) \, \bar{E}_j   \hskip21pt  (\, i , \, j \in I \,)
  \cr
   \bar{E}_{i,j}  \, := \,  \sum\limits_{k = 0}^{1-a_{ij}} (-1)^k {\left[ { 1-a_{ij} \atop k }
\right]}_{\!q_i} q_{ij}^{+k/2\,} q_{ji}^{-k/2} \, \bar{E}_i^{1-a_{ij}-k} \, \bar{E}_j \, \bar{E}_i^k   \quad \hskip31pt   (\, i \neq j \,)  }  $$
 and we define  $ \widetilde{\I}_- $  to be the closure   --- in the  $ \widetilde{\E}_-^{\,(\hbar)} $--adic  topology ---
 of the two-sided ideal in  $ \, \calutildeRPhbm $  generated by the elements
  $$  \displaylines{
   \bar{T}^-_{i,j} \, := \, \bar{T}_i^- \, \bar{T}_j^- \, - \, \bar{T}_j^- \, \bar{T}_i^-  \; ,
\quad   \bar{F}^{(T)}_{i,j} \, := \, \bar{T}_i^- \, \bar{F}_j \, - \, \bar{F}_j \, \bar{T}_i^- \, + \, \hbar \, \alpha_j\big( \bar{T}_i^- \big) \, \bar{F}_j   \hskip21pt  (\, i , \, j \in I \,)
  \cr
   \bar{F}_{i,j}  \, := \,  \sum\limits_{k = 0}^{1-a_{ij}} (-1)^k {\left[ { 1-a_{ij} \atop k }
\right]}_{\!q_i} q_{ij}^{+k/2\,} q_{ji}^{-k/2} \, \bar{F}_i^k \, \bar{F}_j \, \bar{F}_i^{1-a_{ij}-k}   \quad \hskip25pt   (\, i \neq j \,)  }  $$
   \indent  Accordingly, we consider the quotients
 $ \; \caluRPhbpm := \calutildeRPhbpm \Big/ \widetilde{\I}_\pm \; $,
 \;and, with standard abuse of notation,
  \textsl{we shall denote with the same symbol any element in  $ \calutildeRPhbpm $  as well as its coset in the quotient algebra}
  $ \,\; \caluRPhbpm := \calutildeRPhbpm \Big/ \widetilde{\I}_\pm \;\, $.   \hfill  $ \diamondsuit $
\end{definition}

\vskip7pt

   The key point concerning Borel FoMpQUEAs is the following:

\vskip11pt

\begin{prop}  \label{prop: left/right_radical_pairing_pre-Borel}
 {\ }
 \vskip5pt
   {\it (a)}\;  $ \, \widetilde{I}_\pm $  is a  \textsl{Hopf}  ideal of  $ \, \utildeRPhbpm \, $,
   so that  $ \; \uRPhbpm := \utildeRPhbpm \Big/ \widetilde{I}_\pm \; $  is a quotient  \textsl{Hopf}  algebra.
   Similarly,  $ \widetilde{\I}_\pm $  is a  \textsl{Hopf}  ideal of  $ \,  \calutildeRPhbpm \, $,
   therefore the quotient  $ \; \caluRPhbpm := \calutildeRPhbpm \Big/ \widetilde{\I}_\pm \; $  is in fact a  \textsl{Hopf}  algebra.
                                                                           \par
   Moreover,  $ \, \caluRPhbpm $  is a Hopf subalgebra (in topological sense) inside  $ \, \uRPhbpm \, $.
 \vskip5pt
    {\it (b)}\;  The ideal  $ \, \widetilde{I}_+ \, $,  \,resp.  $ \, \widetilde{\I}_- \, $,
    \,is contained in the left, resp.\ right, radical of the pairing
    $ \; \widetilde{\pi} : \utildeRPhbp \,\widehat{\otimes}\, \calutildeRPhbm \relbar\joinrel\relbar\joinrel\longrightarrow \kh \, $
    in  Proposition \ref{prop: pairing_x_pre-Borel}.
                                                                           \par
   Similarly, the ideal  $ \, \widetilde{\I}_+ \, $,  \,resp.  $ \, \widetilde{I}_- \, $,
   \,is contained in the left, resp.\ right, radical of the pairing
   $ \; \widetilde{\pi} : \calutildeRPhbp \,\widehat{\otimes}\, \utildeRPhbm \relbar\joinrel\relbar\joinrel\longrightarrow \kh \, $
   in  Remark \ref{rmk: sym-pairing}.
 \vskip5pt
   {\it (c)}\;  The two skew-Hopf pairings
 $ \; \widetilde{\pi} : \utildeRPhbp \,\widehat{\otimes}\,
 \calutildeRPhbm \relbar\joinrel\relbar\joinrel\relbar\joinrel\relbar\joinrel\longrightarrow \kh \, $
 and
  $ \; \widetilde{\pi} : \calutildeRPhbp \,\widehat{\otimes}\,
  \utildeRPhbm \relbar\joinrel\relbar\joinrel\relbar\joinrel\relbar\joinrel\longrightarrow \kh \, $
 mentioned in (b) uniquely induce skew-Hopf pairings
  $ \; \pi : \uRPhbp \,\widehat{\otimes}\, \caluRPhbm \longrightarrow \kh \; $  and
  $ \; \pi : \caluRPhbp \,\widehat{\otimes}\, \uRPhbm \longrightarrow \kh $
 \; described by obvious formulas as in  Proposition \ref{prop: pairing_x_pre-Borel}.
\end{prop}

\pf
 Claim  \textit{(c)\/}  follows at once from  \textit{(a)\/}  and  \textit{(b)},
 so now we prove the latter ones.
 \vskip3pt
   As to  \textit{(a)},  computations show that the  $ T^+_{i,j} $'s  are primitive, while the  $ E^{(T)}_{i,j} $'s  are skew-primitive, namely
 $ \; \Delta\Big(E^{(T)}_{i,j}\Big) \, = \, E^{(T)}_{i,j} \otimes 1 + e^{+\bar{T}^+_j} \otimes E^{(T)}_{i,j} \; $;
 \,similarly, again direct computations prove also that
 $ \; \Delta(E_{i,j}) = E_{i,j} \otimes 1 + e^{+(1-a_{ij}) \, \hbar\,T_i^+ + \hbar\,T_j^+} \otimes E_{i,j} \; $.
 This implies that  $ \widetilde{I}_+ $  is a Hopf ideal of  $ \, \utildeRPhbp \, $,  as claimed.
                                                                          \par
   With similar arguments, one proves the claim for  $ \widetilde{I}_- $  and for  $ \widetilde{\I}_\pm $  as well.
                                                                          \par
   Finally, the statement about being a Hopf subalgebra follows by construction.
 \vskip3pt
   As to  \textit{(b)},  again direct computation shows that  $ \widetilde{I}_+ $
   lies in the left radical of the pairing  $ \widetilde{\pi} \, $.  For instance, the functional
   $ \, \widetilde{\pi}\big(\, T_i^+ \, T_j^+ , \,-\, \big) \, $  is non-zero only when it is evaluated in
   elements of the form  $ \, \bar{T}_k^- \, \bar{T}_\ell^- \, $  for some  $ \, 1 \leq k , \ell \leq n \, $;  \,in this case,
\begin{align*}
   \widetilde{\pi}\big(\,T_i^+ \, T_j^+ , \bar{T}_k^- \, \bar{T}_\ell^- \big)  &  \; = \;
   \widetilde{\pi}\big(\,T_i^+,\bar{T}_\ell^-\big) \, \widetilde{\pi}\big(\,T_j^+,\bar{T}_k^-\big) \, + \,
   \widetilde{\pi}\big(\,T_i^+,\bar{T}_k^-\big) \, \widetilde{\pi}\big(\,T_j^+,\bar{T}_\ell^-\big)  \; =  \\
      &  \; = \;  p_{i\ell} \, p_{jk} + p_{ik} \, p_{j\ell}  \; = \;  \widetilde{\pi}\big(\,T_j^+ \, T_i^+ , \bar{T}_k^- \, \bar{T}_\ell^-\big)
\end{align*}
 so the generators  $ \, T^+_{i,j} := T_i^- \, T_j^- - \, T_j^- \, T_i^- \, $  of  $ \, \widetilde{I}_+ $  lie in the left radical for all  $ \, i , j \in I \, $.
                                                                        \par
   Similarly, we saw that the generators
 $ \; E^{(T)}_{i,j} \, := \, T_i^+ \, E_j \, - \, E_j \, T_i^+ - \, \alpha_j(T_i^+) \, E_j \; $
 are skew-primitive again, namely
 $ \; \Delta\Big(E^{(T)}_{i,j}\Big) \, = \, E^{(T)}_{i,j} \otimes 1 + e^{+\bar{T}^+_j} \otimes E^{(T)}_{i,j} \; $.
 Thanks to this, in order to prove that the  $ E^{(T)}_{i,j} $'s
 are contained in the left radical it is enough to show that they kill the generators of
 $ \, \calutildeRPhbm \, $,  \,because
  $$
  \widetilde{\pi}\Big( E^{(T)}_{i,j} , \, x \, y \Big)  \; = \;  \widetilde{\pi}\Big( E^{(T)}_{i,j} , \, x \Big) \,
  \widetilde{\pi}(1 \, , \, y) \, + \, \widetilde{\pi}\big( e^{+\bar{T}^+_j} , \, x \big) \, \widetilde{\pi}\Big( E^{(T)}_{i,j} , \, y \Big)   \quad  \forall \; x , y \in \calutildeRPhbm
  $$
   \indent   Now, from  $ \, \widetilde{\pi}\big( E_\ell \, , 1 \big) = 0 \, $  for all  $ \, \ell \in I \, $,
   the properties of the skew-Hopf pairing imply that  $ \, \widetilde{\pi}\Big( E^{(T)}_{i,j} \, , 1 \Big) = 0 \, $  too,
   for all  $ \, i , j \in I \, $.  Similarly, direct computation gives, using notation
   $ \, \widetilde{\pi}_\otimes(a \otimes b \, , u \otimes v) := \widetilde{\pi}(a\,,u) \, \widetilde{\pi}(b\,,v) \, $,
  $$
  \displaylines{
   \widetilde{\pi}\Big( E^{(T)}_{i,j} \, , \bar{T}^-_k \Big)  \; = \;  \widetilde{\pi}\Big( T_i^+ \, E_j \, - \, E_j \, T_i^+ -
   \, \alpha_j(T_i^+) \, E_j \, , \, \bar{T}^-_k \Big)  \; =   \hfill  \cr
   = \;  \widetilde{\pi}_\otimes\Big( T_i^+ \otimes E_j \, -
   \, E_j \otimes T_i^+ - \, \alpha_j(T_i^+) \, E_j \otimes 1 \, , \, \Delta\big(\bar{T}^-_k\big) \Big)  \; =   \qquad \qquad \qquad  \cr
   \qquad   = \;  \widetilde{\pi}_\otimes\big( T_i^+ \otimes E_j \, -
   \, E_j \otimes T_i^+ - \, \alpha_j(T_i^+) \, E_j \otimes 1 \, , \, \bar{T}^-_k \otimes 1 + 1 \otimes \bar{T}^-_k \big)  \; = \;  0  }
   $$
 exactly because  $ \, \widetilde{\pi}\big( E_j \, , 1 \big) = 0 =  \widetilde{\pi}\big( E_j \, , \bar{T}^-_k \big) \, $
 for all  $ \, j , k \in I \, $.  Likewise, we also have
  $$  \displaylines{
   \widetilde{\pi}\Big( E^{(T)}_{i,j} \, , \bar{F}_k \Big)  \; = \;
   \widetilde{\pi}\Big( T_i^+ \, E_j \, - \, E_j \, T_i^+ - \, \alpha_j(T_i^+) \, E_j \, , \, \bar{F}_k \Big)  \; =   \hfill  \cr
   = \;  \widetilde{\pi}_\otimes\Big( T_i^+ \otimes E_j \, - \, E_j \otimes T_i^+ - \, \alpha_j(T_i^+) \, E_j \otimes 1 \, ,
   \, \Delta\big(\bar{F}_k\big) \Big)  \; =   \qquad \qquad \qquad  \cr
   \qquad   = \;  \widetilde{\pi}_\otimes\Big( T_i^+ \otimes E_j \, -
   \, E_j \otimes T_i^+ - \, \alpha_j(T_i^+) \, E_j \otimes 1 \, , \, \bar{F}_k \otimes e^{-\bar{T}^-_k} + 1 \otimes \bar{F}_k \Big)  \; = \;  0  }  $$
 because when we expand the last line the only non-trivial summands are
  $$
  - \widetilde{\pi}\big( E_j \, , \bar{F}_k \big) \cdot \widetilde{\pi}\Big( T^+_i \, , e^{-\bar{T}^-_k} \Big)  \;
  =  \;  +\delta_{j,k} \, \hbar \, {\big(\, q_j^{+1} - \, q_j^{-1} \,\big)}^{-1} \cdot \alpha_k\big(T_i^+\big)
  $$
 \vskip-5pt
 and
  $$
  - \alpha_j\big(T_i^+\big) \, \widetilde{\pi}\big( E_j \, , \bar{F}_k \big) \cdot
  \widetilde{\pi}\Big( 1 \, , e^{-\bar{T}^-_k} \Big)  \; = \;  -\alpha_j\big(T_i^+\big) \, \delta_{j,k} \, \hbar {\big(\, q_j^{+1} - \, q_j^{-1} \,\big)}^{-1} \, \cdot 1
  $$
 which add up to zero, q.e.d.
                                                                        \par
Finally, the  $ E_{i,j} $'s  are skew-primitives too,
   so again it is enough to show that they kill the generators of  $ \, \calutildeRPhbm \, $.
   This follows again by direct calculation, for instance
  $$  \displaylines{
   \widetilde{\pi}\big( E_{i,j} \, , \bar{F}_k \Big)  \,\; = \;\,
   \widetilde{\pi}\Big(\, {\textstyle{\sum_{s = 0}^{1-a_{ij}}}} {(-1)}^s {\left[ { 1-a_{ij} \atop s } \right]}_{\!q_i} q_{ij}^{+s/2\,} q_{ji}^{-s/2} \, E_i^{1-a_{ij}-s} \, E_j \, E_i^s \, ,
   \, \bar{F}_k \Big)  \; =   \hfill  \cr
   = \;  {\textstyle{\sum\limits_{s=0}^{1-a_{ij}}}} {(-1)}^s {\left[ { 1-a_{ij} \atop s }
\right]}_{\!q_i} q_{ij}^{+s/2\,} q_{ji}^{-s/2} \; \widetilde{\pi}_\otimes^{(3)}\Big( E_i^{1-a_{ij}-s} \otimes E_j \otimes E_i^s \, , \, \Delta^{(3)}\big(\bar{F}_k\big) \Big)  \; =  \cr
   = \;  {\textstyle{\sum\limits_{s=0}^{1-a_{ij}}}} {(-1)}^s {\left[ { 1-a_{ij} \atop s } \right]}_{\!q_i} q_{ij}^{+s/2\,} q_{ji}^{-s/2} \; \times   \hfill  \cr
   \hfill   \times \; \widetilde{\pi}_\otimes^{(3)}\Big( E_i^{1-a_{ij}-s} \otimes E_j \otimes E_i^s \, ,
   \, \bar{F}_k \otimes e^{-\bar{T}^-_k} \otimes e^{-\bar{T}^-_k} + 1 \otimes \bar{F}_k \otimes e^{-\bar{T}^-_k} + 1 \otimes 1 \otimes \bar{F}_k \Big)  }
   $$
 so that for all  $ \, s \, $  we get
  $$  \widetilde{\pi}_\otimes^{(3)}\Big( E_i^{1-a_{ij}-s} \otimes E_j \otimes E_i^s \, , \,
  \bar{F}_k \otimes e^{-\bar{T}^-_k} \otimes e^{-\bar{T}^-_k} + 1 \otimes \bar{F}_k \otimes e^{-\bar{T}^-_k} + 1 \otimes 1 \otimes \bar{F}_k \Big)  \; = \;  0  $$
 because  $ \; \widetilde{\pi}\Big( E_j \, , \, e^{-\bar{T}^-_k} \Big) = 0 = \widetilde{\pi}\big( E_j \, , \, 1 \big) \; $
 and  $ \; \widetilde{\pi}\Big( E_j^{1-a_{ij}-s} \, , 1 \Big) = 0 \; $.
\epf

\vskip7pt

\begin{rmks}
 \vskip3pt
   \textit{(a)}\;  Constructions imply that
%
%
 $ \uRPhbpm $  in  Proposition \ref{prop: left/right_radical_pairing_pre-Borel}\textit{(a)\/}
%
%
 coincide with the Borel FoMpQUEAs from  Definition \ref{def: Mp-Uhgd}\textit{(c)},
 with their whole Hopf structure (cf.\  Proposition \ref{cor: Hopf-struct_Cartan-&-Borel}),
 so we use again same notation and terminology.
 \vskip3pt
   \textit{(b)}\;  Denote by  $ \lier_+ $  and  $ \R_- \, $,  respectively  $ \R_+ $  and  $ \lier_- \, $,
   the left and the right radical of the skew-Hopf pairing
 $ \; \widetilde{\pi} : \utildeRPhbp \,\widehat{\otimes}\,
 \calutildeRPhbm \relbar\joinrel\relbar\joinrel\relbar\joinrel\relbar\joinrel\relbar\joinrel\longrightarrow \kh \, $,
 \,respectively
  $ \; \widetilde{\pi} : \calutildeRPhbp \,\widehat{\otimes}\,
  \utildeRPhbm \relbar\joinrel\relbar\joinrel\relbar\joinrel\relbar\joinrel\relbar\joinrel\longrightarrow \kh \, $;
 \,all these are Hopf ideals, and then   --- in both cases ---   the pairings  $ \widetilde{\pi} $
 induce similar skew Hopf pairings between the quotient Hopf algebras
 $ \, \utildeRPhbp \Big/ \lier_+ \, $
 and  $ \, \calutildeRPhbm \Big/ \R_- \, $,  \,resp.\   $ \,\calutildeRPhbp \Big/ \R_+ \, $ and  $ \, \utildeRPhbm \Big/ \lier_- \, $,
 \,which are non-degenerate.  When the matrix  $ P $  is symmetric, hence equal to  $ DA \, $,
 basing on  \cite{Ka}, \S Theorem 9.11,  one can prove that the relations in  Definition \ref{def: Borel_FoMpQUEAs},
 resp.\ in  Definition \ref{def: dual-Borel_FoMpQUEAs},  generate the Hopf ideals  $ \lier_\pm \, $,  resp.\  $ \R_\mp \, $,
 \,hence one has that  $ \; \utildeRPhbpm \Big/ \lier_\pm \, \cong \, \uRPhbpm \; $  and
 $ \; \calutildeRPhbpm \Big/ \R_\pm \, \cong \, \caluRPhbpm \; $.
\end{rmks}

\vskip1pt

   Next result points out a technical property of the Borel FoMpQUEAs.

\vskip13pt

\begin{lema}  \label{lemma: Borel FoMpQUEAs topol.-free}
 The algebras  $ \uRPhbpm $  are topologically free, i.e.\ they are torsion-free as  $ \kh $--modules
 and they are separated and complete for the  $ \hbar $--adic  topology.
\end{lema}

\pf
 It is proved in  Theorem \ref{thm: double FoMpQUEAs_P-P'_mutual-deform.s}\textit{(b)\/}  later on
 --- in a way  \textsl{independent of whatever follows from here to there}  ---
 that in the  \textsl{split minimal\/}  case the FoMpQUEA  $ \uRPhg $  is just a deformation
 (in a proper sense) of Drinfeld's
%
%
 $ \uhg $   --- in ``double version'', i.e.\ with Cartan subalgebra of rank  $ \, 2\,|I| \, $;
 in particular,  $ \uRPhg $  and  $ \uhg $  have the same  $ \kh $--module  structure.
 Now  $ \uhg $  is known to be topologically free, so the same holds for  $ \uRPhg \, $,
 and then this property is inherited by the subalgebras  $ \uRPhbpm $  too.
 Alternatively, the proof of  Theorem \ref{thm: double FoMpQUEAs_P-P'_mutual-deform.s}\textit{(b)\/}
 also applies directly to  $ \uRPhbpm \, $,
%
%
 proving that the former are suitable deformations of Drinfeld's
%
%
  $ U_\hbar(\lieb_\pm) \, $:  the latter are known
%
%
 to be topologically free, so the same holds for  $ \uRPhbpm $  too.
\epf
%
%
%

\vskip11pt

\begin{free text}  \label{FoMpQUEAs = quasi-doubles Borel's}
 {\bf FoMpQUEAs as quasi-doubles of Borel FoMpQUEAs.}\,
 The analysis carried on from  \S \ref{pre-Borel ---> Borel}  on provides skew-Hopf pairings between Borel FoMpQUEAs
 $ \uRPhbpm $  and  $ \caluRPhbmp \, $.  Following the recipe in  Definition \ref{def_(skew-)Hopf-pairing},
 we can then consider the associated (Drinfeld's) quantum doubles, that we denote by
\begin{equation}  \label{eq: def-qdouble_Borel}
  \begin{aligned}
     \calDRPhgd  \,\;  &  := \;\,  D\big(\, \uRPhbp \, , \caluRPhbm \,, \pi \,\big)  \\
     \calDRPhgs  \,\;  &  := \;\,  D\big(\, \caluRPhbp \, , \uRPhbm \,, \pi \,\big)
  \end{aligned}
\end{equation}
   \indent   By the very definition of Drinfeld's quantum double, there exists an isomorphism of (topological)  $ \kh $--coalge\-bras
  $ \,\; \calDRPhgd \,\cong\, \uRPhbp \!\mathop{\widehat{\otimes}}\limits_\kh\! \caluRPhbm \;\, $.
 Even more, both  $ \, \uRPhbp \, $  and  $ \, \caluRPhbm \, $  embed into  $ \, \calDRPhgd \, $   ---
 via  $ \, u \mapsto u \otimes 1 \, $  and  $ \, \bar{v} \mapsto 1 \otimes \bar{v} \, $,  respectively ---
 as  \textsl{Hopf subalgebras},  and these (Hopf) subalgebras actually generate all of  $ \, \calDRPhgd \, $,
 as a topological algebra.  A similar analysis applies to  $ \, \calDRPhgs \, $.
                                                            \par
   Note that here we apply Lemma \ref{lemma: Borel FoMpQUEAs topol.-free}:  by it,
   $ \, \uRPhbpm \, $  is topologically free, thus also  $ \, \caluRPhbpm \, $  is, hence the products
   $ \, \uRPhbpm \!\mathop{\widehat{\otimes}}\limits_\kh\! \caluRPhbmp \, $  are topologically free too.
                                                            \par
   Our next result is an explicit description of these quantum double Hopf algebras.
\end{free text}

\vskip9pt

\begin{prop}  \label{prop: qdouble-Borel_presentation}
%
%
 With assumptions as above, the quantum double Hopf algebra  $ \calDRPhgd $  in
 \eqref{eq: def-qdouble_Borel}  admits the following presentation:
 it is the unital, associative, topological,  $ \hbar $--adically  complete algebra over  $ \kh $  with generators
 $ \, E_i \, $,  $ \, T_i^+ \, $,  $ \bar{T}^-_j \, $,  $ \bar{F}_j \, $,  (for all  $ \, i , j \in I \, $)
 and relations (for all  $ \, i \, , j \in I \, $)
\begin{equation}  \label{eq: comm-rel's_x_DRPhgd}
 \begin{aligned}
     T_i^+ \, T_j^+  \, = \;  T_j^+ \, T_i^+  \quad ,  \qquad   T_i^+ E_j \, - \, E_j \, T_i^+  \, = \,  +p_{ij} \, E_j  \hskip33pt  \\
      \sum\limits_{k = 0}^{1-a_{ij}} (-1)^k {\left[ { 1-a_{ij} \atop k }
\right]}_{\!q_i} q_{ij}^{+k/2\,} q_{ji}^{-k/2} \, E_i^{1-a_{ij}-k} \, E_j \, E_i^k  \; = \;  0   \qquad  (\, i \neq j \,)   \hskip8pt  \\
     \bar{T}_i^- \, \bar{T}_j^-  \, = \;  \bar{T}_j^- \, \bar{T}_i^-  \quad ,  \qquad   \bar{T}_i^- \bar{F}_j \, - \, \bar{F}_j \, \bar{T}_i^-  \, = \,  -\hbar \, p_{ji} \, \bar{F}_j  \hskip31pt  \\
   \sum\limits_{k = 0}^{1-a_{ij}} (-1)^k {\left[ { 1-a_{ij} \atop k }
\right]}_{\!q_i} q_{ij}^{+k/2\,} q_{ji}^{-k/2} \, \bar{F}_i^k \, \bar{F}_j \, \bar{F}_i^{1-a_{ij}-k}  \; = \;  0   \qquad  (\, i \neq j \,)   \hskip6pt  \\
      \bar{T}_i^- E_j \, - \, E_j \, \bar{T}_i^-  \, = \,  +\hbar \, p_{ji} \, E_j  \quad ,  \qquad
   T_i^+ \bar{F}_j \, - \, \bar{F}_j \, T_i^+  \, = \,  -p_{ij} \, \bar{F}_j   \hskip6pt  \\
      T_i^+ \, \bar{T}_j^-  \, = \;  \bar{T}_j^- \, T_i^+  \quad ,  \qquad
   E_i \, \bar{F}_j \, - \, \bar{F}_j \, E_i  \; = \;  \delta_{i,j} \, \hbar \, {{\; e^{+\hbar \, T_i^+} - \, e^{-\bar{T}_i^-} \;} \over {\; q_i^{+1} - \, q_i^{-1} \;}}   \hskip0pt  \\
 \end{aligned}
\end{equation}
 with Hopf structure given on the above generators (for all  $ \, i \in I \, $)  by
\begin{equation}  \label{eq: Hopf struct x qdouble DRPhg}
 \begin{aligned}
   \Delta\big(E_i\big)  \, = \,  E_i \otimes 1 \, + \, e^{+\hbar \, T_i^+} \otimes E_i  \;\; ,  \quad
    \epsilon\big(E_i\big) \, = \, 0  \;\; ,  \quad   \SS\big(E_\ell\big)  \, = \,  - e^{-\hbar \, T_i^+} E_i  \\
   \Delta\big(T_i^+\big)  \, = \,  T_i^+ \otimes 1 \, + \, 1 \otimes T_i^+  \;\; ,  \quad
    \epsilon\big(T_i^+\big) \, = \, 0  \;\; ,  \quad   \SS\big(T_i^+\big)  \, = \,  -T_i^+  \;\;\quad  \\
   \Delta\big(\bar{T}_i^-\big)  \, = \,  \bar{T}_i^- \otimes 1 \, + \, 1 \otimes \bar{T}_i^-  \,\; ,  \quad
    \epsilon\big(\bar{T}_i^-\big) \, = \, 0  \,\; ,  \quad   \SS\big(\bar{T}_i^-\big)  \, = \,  -\bar{T}_i^-  \;\;\;\quad  \\
   \Delta\big(\bar{F}_i\big)  \, = \,  \bar{F}_i \otimes e^{-\bar{T}_i^-} \, + \, 1 \otimes \bar{F}_i  \;\; ,  \quad
    \epsilon\big(\bar{F}_i\big) \, = \, 0  \;\;\; ,  \qquad   \SS\big(\bar{F}_i\big)  \, = \,  - e^{+\bar{T}_i^-} \bar{F}_i
 \end{aligned}
\end{equation}
   \indent   A similar result provides a likewise presentation of  $ \; \calDRPhgs \, $.
\end{prop}

\pf
 Recall that we have an isomorphism
  $ \,\; \calDRPhgd \,\cong\, \uRPhbp \!\mathop{\widehat{\otimes}}\limits_\kh\! \caluRPhbm \;\, $
 as (topological)  $ \kh $--coalgebras,  and through it both  $ \, \uRPhbp \, $  and
 $ \, \caluRPhbm \, $  embed into  $ \, \calDRPhgd \, $ ,  \,via  $ \, u \mapsto u \otimes 1 \, $
 and  $ \, \bar{v} \mapsto 1 \otimes \bar{v} \, $  ---   as  \textsl{Hopf subalgebras},  which generate
 $ \, \calDRPhgd \, $,  as a topological algebra.  In particular, as a matter of notation we shall write
 $ \, u \, $  for  $ \, u \otimes 1 \, $ and  $ \, \bar{v} \, $  for  $ \, 1 \otimes \bar{v} \, $.
 From all this it follows that $ \, \calDRPhgd \, $  admits a presentation with generators
 $ \, E_i \, $,  $ \, T_i^+ \, $,  $ \bar{T}^-_j \, $,  $ \bar{F}_j \, $,  ($ \, i , j \in I \, $)
 --- as these generate  $ \, \uRPhbp \, $  and  $ \, \caluRPhbm \, $  ---
 and relations given by the first two lines in  \eqref{eq: comm-rel's_x_DRPhgd}   ---
 because these are the relations among the  $ E_i $'s  and  $ T_i^+ $'s  inside  $ \uRPhbp $  ---
 and the mid two lines in  \eqref{eq: comm-rel's_x_DRPhgd}   --- since these are those among the
 $ \bar{T}^-_j $'s  and  $ \bar{F}_j $'s  inside  $ \caluRPhbm $  ---
 \textsl{plus\/}  the additional relations, given at the end of  Definition \ref{def_(skew-)Hopf-pairing},
 that link the generators  $ E_i $  and  $ T_i^+ $  ($ \, i \in I \, $)  inside  $ \uRPhbp $  with the generators
 $ \bar{T}^-_j $  and  $ \bar{F}_j $  ($ \, j \in I \, $)  inside  $ \caluRPhbm \, $.
 Concerning these last set of relations, direct computation proves that they are given by the last two lines in  \eqref{eq: comm-rel's_x_DRPhgd},  q.e.d.
                                                              \par
   For example, taking  $ \; x = \bar{T}_j^- \, $  and  $ \, y = T_i^+ \, $  we have that
  $$  \displaylines{
   x_{(1)} \, y_{(1)} \, \pi\big(\, y_{(2)} \, , x_{(2)} \big)  \; =   \hfill  \cr
   \hfill   = \;  \bar{T}^- \, T_i^+ \, \pi(\,1\,,1) \, + \, \bar{T}_j^- \, \pi\big(\,T_i^+,1\big) \, + \, T_i^+
   \, \pi\big(\, 1 \, , \bar{T}_j^- \big) \, + \, \pi\big(\, T_i^+ , \bar{T}_j^- \big)  \,\; = \;\,  \bar{T}_j^- \, T_i^+ \, + \, p_{ij}  \cr
   \pi\big(\, y_{(1)} \, , x_{(1)} \big) \, y_{(2)} \, x_{(2)}  \; =   \hfill  \cr
   \hfill   = \;  \pi\big(\, T_i^+ , \bar{T}_j^- \big) \, + \, \pi\big(\, T_i^+ , 1 \big) \, \bar{T}_j^- \, +
   \, \pi\big(\, 1 \, , \bar{T}_j^- \big) \, T_i^+ \, + \, \pi(\,1\,,1) \, T_i^+ \, \bar{T}_j^-  \,\; = \; \,  p_{ij} \, + \, T_i^+ \, \bar{T}_j^-  }  $$
 which yields  $ \, T_i^+ \, \bar{T}_j^- = \bar{T}_j^- \, T_i^+ \, $  for all  $ \, i , j \in I \, $;  
   \hbox{\,similarly, for  $ \; x = \bar{F}_j \, $,  $ \, y = E_i \, $,  we get}
  $$
  \displaylines{
   x_{(1)} \, y_{(1)} \, \pi\big(\, y_{(2)} \, , x_{(2)} \big)  \; = \;  \bar{F}_j \, E_i \, \pi\big(\, 1 \, , e^{-\bar{T}_j^-} \big) \, + \, \bar{F}_j \, e^{\hbar\,T_i^+} \, \pi\big(\, E_i \, , e^{-\bar{T}_j^-} \big) \, +   \hfill  \cr
   \hfill   + \, E_i \, \pi\big(\, 1 \, , \bar{F}_j \big) \, + \, e^{\hbar\,T_i^+} \, \pi\big( \, E_i \, , \bar{F}_j \big)  \; = \;  \bar{F}_j \, E_i \, + \, \delta_{i,j} \, \hbar \, e^{\hbar\,T_i^+} {\big(\, q_i^{+1} - \, q_i^{-1} \big)}^{-1}  \cr
   \pi\big(\, y_{(1)} \, , x_{(1)} \big) \, y_{(2)} \, x_{(2)}  \,\; = \;\,  \pi\big(\, E_i \, , \bar{F}_j \big) \, e^{-\bar{T}_j^-} + \, \pi\big(\, E_i \, , 1 \big) \, \bar{F}_j \, +   \hfill  \cr
   \hfill   + \, \pi\big(\, e^{+\hbar\,T_i^+} \! , \bar{F}_j \big) \, E_i \, e^{-\bar{T}_j^-} \, + \, \pi\big(\, e^{+\hbar\,T_i^+} \! , 1 \big) \, E_i \, \bar{F}_j  \,\; = \;\,  \delta_{i,j} \, \hbar \, e^{-\bar{T}_j^-} {\big(\, q_i^{+1} - \, q_i^{-1} \big)}^{-1} \, + \, E_i \, \bar{F}_j  }
   $$
 which yields the relation
 $ \; \displaystyle{ E_i \, \bar{F}_j \, - \, \bar{F}_j \, E_i  \, = \,  \delta_{i,j} \, \hbar \, \frac{\; e^{\hbar\,T_i^+} - e^{\hbar\,T_i^-} \;}{\; q_i^{+1} - q_i^{-1} \; } } \; $
 for all  $ \, i , j \in I \, $.
 \vskip3pt
   Finally, the Hopf structure is given once we know how it looks on generators, hence it is given by
   \eqref{eq: Hopf struct x qdouble DRPhg}  because  $ \uRPhbp $  and  $ \caluRPhbm $  are  \textsl{Hopf}  subalgebras.
                                                                                  \par
   A parallel argument yields a similar presentation for  $ \calDRPhgs \, $.
\epf

\vskip3pt

   We still need some auxiliary ingredients:

\vskip9pt

\begin{definition}  \label{def: qdouble-QUEA}   {\ }
 \vskip3pt
   \textit{(a)}\;  We denote by  $ \, \DRPhgd $  the  $ \hbar $--adic  completion of the  $ \kh $--subalgebra
   generated in  $ \calDRPhgd $  by the set
   $ \; {\big\{ E_i \, , \, T_i^+ , \, T_i^- = \hbar^{-1} \bar{T}_i^- , \, F_i = \hbar^{-1} \bar{F}_i \,\big\}}_{i \in I} \, $.
 \vskip3pt
   \textit{(b)}\;  We denote by  $ \, \DRPhgs $  the  $ \hbar $--adic  completion of the  $ \kh $--subalgebra
   generated in  $ \calDRPhgs $  by the set
   $ \; {\big\{ E_i = \hbar^{-1} \bar{E}_i \, , \, T_i^+ = \hbar^{-1} \bar{T}_i^+ , \, T_i^- , \, F_i \,\big\}}_{i \in I} \, $.   \hfill  $ \diamondsuit $
\end{definition}

\vskip7pt

   We are finally ready for the main result we are looking for:

\vskip11pt

\begin{theorem}  \label{thm: double-cong-uRPhg_split}
 Let  $ \, A := {\big(\, a_{i,j} \big)}_{i, j \in I} \, $  be a generalized symmetrizable Cartan matrix, and let
 $ \, P := {\big(\, p_{i,j} \big)}_{i, j \in I} \in M_n\big(\, \kh \,\big) \, $
 be a matrix of Cartan type with associated Cartan matrix $ A \, $.  With assumptions as above, both
 $ \, \DRPhgd $  and  $ \, \DRPhgs $  are topological,  $ \hbar $--adically  complete Hopf\/  $ \kh $--algebras,
 which are isomorphic to the FoMpQUEA  $ \uRPhg $  given in  Definition \ref{def: Mp-Uhgd}.
\end{theorem}

\pf
 The claim follows directly from the construction of  $ \, \DRPhgd \, $  and  $ \, \DRPhgs \, $,
 \,and from  Proposition \ref{prop: qdouble-Borel_presentation}  above: in fact, all this yields a
 presentation for  $ \, \DRPhgd \, $  and one for  $ \, \DRPhgs \, $   --- with generating set
 $ \; {\big\{ E_i \, , \, T_i^+ , \, T_i^- , \, F_i \,\big\}}_{i \in I} \, $,  in both cases ---
 that just coincide, hence these two algebras are isomorphic.
 At the same time, the formulas for the Hopf structure in  $ \, \DRPhgd \, $  and  $ \, \DRPhgs \, $
 show that these algebras inherit the Hopf structure as well.  Comparing this presentation with the one defining
 $ \uRPhg $  one sees that they coincide again, whence the last part of the claim.
\epf

\vskip9pt

\begin{free text}  \label{constr.-double-crossproduct}
 {\bf Construction as double cross products.}\,  In this subsection we implement an alternative construction
 of  $ \uRPhg $  as a subalgebra
 of a  \textit{double cross product}, which is also an alternative way of constructing a quantum double.
 We follow Majid \cite[\S 7.2]{Mj} for the description of the double cross product.
 We begin by introducing the construction in the general context of matched pairs
 of Hopf algebras.
\end{free text}

\vskip7pt

\begin{definition}  \label{def:matched pairs}
 \cite[Definition 7.2.1]{Mj}  Two bialgebras or Hopf algebras  $ A $  and  $ H $  form a
 right-left matched pair if  $ H $  is a right  $ A $--module  coalgebra and  $ A $  is a left  $ H $--module  coalgebra with mutual actions
%
%
  $ \,\; \triangleleft\,: H \otimes A \longrightarrow H \; $,  $ \,\; \triangleright\,: H \otimes A \longrightarrow A \;\, $
 that obey the compatibility conditions
%
%
  $$
  \displaylines{
   (hg) \triangleleft a  \, = \,  \big( h \triangleleft \big( g_{(1)} \triangleright a_{(1)} \big) \big)\big( g_{(2)} \triangleleft a_{(2)} \big)  \quad ,
   \qquad \qquad   1 \triangleleft a  \, = \,  \epsilon(a)  \cr
   h \triangleright (ab)  \, = \,  \big( h_{(1)} \triangleright a_{(1)} \big)\big( (h_{(2)} \triangleleft a_{(2)} \big) \triangleright b \,\big)  \quad ,
   \qquad \qquad  h \triangleright 1  \, = \,  \epsilon(h)  \cr
   \big( h_{(1)} \triangleleft a_{(1)} \big) \otimes \big( h_{(2)} \triangleright a_{(2)} \big)  \; = \;
   \big( h_{(2)} \triangleleft a_{(2)} \big) \otimes \big( h_{(1)} \triangleright a_{(1)} \big)  }
   $$
%
%
\end{definition}

\vskip7pt

\begin{theorem}  \label{thm:crossproduct}
 \cite[Theorem 7.2.2]{Mj}
                                                                    \par
   Given a matched pair of bialgebras  $ (A,H) \, $,  there exists a  \textsl{double cross product bialgebra}
   $ \, A \bowtie H \, $  built on the vector space  $ \, A \otimes H \, $  with product
  $$
  (a\otimes h) \cdot (b \otimes g)  \; := \;  a \, \big( h_{(1)} \triangleleft b_{(1)} \big) \otimes \big( h_{(2)} \triangleright b_{(2)} \big) \, g
  \eqno  \forall \; a, b \in A \, , \, h, g \in H  $$

\noindent
 and tensor product unit, counit and coproduct maps.  Moreover,  $ A $ and  $ H $  are subbialgebras
 via the canonical inclusions, and  $ \, A \bowtie H \, $  is generated by them with relations
  $$
  h \cdot a  \; = \;  \big( h_{(1)} \triangleleft b_{(1)} \big) \otimes \big( h_{(2)} \triangleright b_{(2)} \big) \eqno \forall \; h \in H \, , \, a \in A
  $$
 If  $ A $  and  $ H $  are  \textsl{Hopf algebras},  then so is their double cross product, with antipode
  $$
  \SS(a \otimes h)  \; = \,  \big( 1 \otimes \SS(h) \big) \big( \SS(a) \otimes 1 \big)  \; = \;
  \big( \SS\big(h_{(2)}\big) \triangleleft \SS\big(b_{(2)}\big) \big) \otimes \big( \SS\big(h_{(1)}\big) \triangleright \SS\big(b_{(1)}\big) \big)   \eqno \square
  $$
\end{theorem}

\vskip9pt

\begin{free text}  \label{Hopf-pairings ==> double-cross-products}
 {\bf From skew-Hopf pairings to double cross products.}\,
 Let  $ R $  be a ring, let  $ A \, $,  $ H $  be two  $ R $--bialgebras
 and let  $ \, \eta : H \otimes A \longrightarrow R \, $  be a skew-Hopf pairing which is convolution invertible.
 Then  $ H $  is a right  $ A $--module  coalgebra and  $ A $  is a left  $ H $--module  coalgebra via the actions
  $$
  h \triangleright a  \; := \;  h_{(2)} \, \eta^{-1}\big( h_{(1)} , a_{(1)} \big) \, \eta\big( h_{(3)}, a _{(2)} \big) \;\; ,
  \quad  h \triangleleft a  \; := \;  a_{(2)} \, \eta^{-1}\big( h_{(1)} , a_{(1)} \big) \, \eta\big( h_{(2)} , a_{(3)} \big)
  $$
 for all  $ \, h \in H \, $  and  $ \, a \in A \, $.  In particular, then, there exists a double cross product bialgebra
 $ \, A \bowtie H \, $  built upon  $ \, A \otimes H \, $;  \,as we know, it has the tensor product unit, counit and coproduct,
 while its product now explicitly reads, in terms of the pairing, as follows   --- see \cite[Example 7.2.7]{Mj}:
  $$  (a \otimes h) \cdot (b \otimes g)  \; := \;  \eta^{-1}\big( h_{(1)} , b_{(1)} \big) \,
  a \, b_{(2)} \otimes h_{(2)} \, g \, \eta\big( h_{(3)} , b_{(3)} \big)  $$
 In addition, when both  $ A $  and  $ H $  are Hopf algebras then such is  $ \, A \bowtie H \, $  as well.
\end{free text}

\vskip9pt

\begin{free text}  \label{FoMpQUEAs-double-crossproducts}
 {\bf FoMpQUEAs as double cross products.}\,  Using the skew-Hopf pairing between our
 Borel FoMpQUEAs given by Proposition \ref{prop: left/right_radical_pairing_pre-Borel},  namely
  $$  \pi : \uRPhbp \,\widehat{\otimes}\, \caluRPhbm \relbar\joinrel\relbar\joinrel\longrightarrow \kh
   \quad  \text{and}  \quad
      \pi : \caluRPhbp \,\widehat{\otimes}\, \uRPhbm \relbar\joinrel\relbar\joinrel\longrightarrow \kh  $$
 we may apply the general construction in  \S \ref{Hopf-pairings ==> double-cross-products}  above and define two Hopf algebras
  $$  \caluRPhbm \bowtie \uRPhbp  \qquad \text{and} \qquad  \uRPhbm \bowtie \caluRPhbp  $$
   \indent   Following the recipe in  \S \ref{Hopf-pairings ==> double-cross-products},
   the actions of  $ \caluRPhbm $  on $ \uRPhbp $  and of  $ \uRPhbp $  on  $ \caluRPhbm $
   via the skew-Hopf paring  $ \pi $  are given by (for all  $ \, i , j \in I \, $)
%
%
  $$  \displaylines{
   T_i^+ \triangleright \bar{T}_j^- \, = \, 0  \quad ,   \qquad  T_i^+ \triangleleft \bar{T}_j^- \, = \, 0  \quad ,   \qquad
     T_i^+ \triangleleft \bar{F}_j \, = \, 0  \quad ,   \qquad  E_i \triangleright \bar{T}_j^- \, = \, 0  \cr
   T_i^+ \triangleright \bar{F}_j \, = \, -
%
%
 p_{ij} \bar{F}_j  \quad ,   \qquad \qquad \qquad  E_i \triangleleft \bar{T}_j^- \, = \, -\hbar \, p_{ij} E_i  \cr
   E_i \triangleright \bar{F}_j  \; = \;  \delta_{i{}j} \, {{\; \hbar \, \big( 1 - e^{-\bar{T}_j^-} \big) \;} \over {\; q_i^{+1} - \, q_i^{-1} \;}}  \quad ,
   \quad \qquad  E_i \triangleleft \bar{F}_j  \; = \; \delta_{i{}j} \, {{\; \hbar \, \big( e^{+\hbar \, T_i^+} \! - 1 \big)\;} \over {\; q_i^{+1} - \, q_i^{-1} \;}}  \cr
   \qquad \qquad \qquad   X \triangleright 1 \, = \, 0  \quad ,  \quad \qquad  1 \triangleleft X \, = \, 0
   \qquad \qquad  \forall \;\;\; X \in {\big\{ E_i , T_i^+ , \bar{F}_i , \bar{T}_i^- \big\}}_{i\in I}  }
   $$
 It is clear that these formulae completely define the cross product structure on
 $ \, \caluRPhbm \bowtie \uRPhbp \, $.  For example, let us compute
 $ \, E_i \triangleleft \bar{F}_j \, $  explicitly.  Set  $ \, L_j = e^{-\bar{T}_j^-} \, $  and $ \, K_i = e^{\bar{T}_i^+} \, $:
 \,then computations give

  $$  \displaylines{
   E_i \triangleright \bar{F}_j  \; = \;  {\big( \bar{F}_j \big)}_{(2)} \,\eta^{-1}\big( {(E_i)}_{(1)} ,
   {\big( \bar{F}_j \big)}_{(1)} \big) \, \eta\big( {(E_i)}_{(2)} , {\big( \bar{F}_j \big)}_{(3)} \big)  \; =   \hfill  \cr
   \quad   = \;  L_j \, \eta^{-1}\big( E_i , \bar{F}_j \big) \, \eta(1,L_j) \, + \, L_j \,
   \eta^{-1}\big( K_i , \bar{F}_j \big) \, \eta(E_i,L_j) \, + \, \bar{F}_j \, \eta^{-1}(E_i,1) \, \eta(1,L_j) \, +   \hfill  \cr
   \quad \quad   + \, \bar{F}_j \, \eta^{-1}(K_i,1) \, \eta(E_i,L_j) \, + \, 1 \, \eta^{-1}(E_i, 1) \, \eta\big(1,\bar{F}_j\big) \, +
   \, 1 \, \eta^{-1}(K_i, 1) \, \eta\big( E_i , \bar{F}_j \big)  \; =   \hfill  \cr
   \quad \quad \quad   = \;  L_j \, \eta^{-1}\big(E_i , \bar{F}_j \big) \, + \, \eta\big( E_i , \bar{F}_j \big)  \; = \;
   L_j \, \eta\big( E_i , \SS\big(\bar{F}_j\big) \big) \, + \, \eta\big( E_i , \bar{F}_j \big)  \; =   \hfill  \cr
   \qquad \qquad   = \;  L_j \, \eta\big( E_i , -\bar{F}_j L_j^{-1} \big) \, + \, \eta\big( E_i , \bar{F}_j \big)  \; = \;
   (1-L_j) \, \eta\big( E_i , \bar{F}_j \big)  \; = \;  \delta_{i{}j} \, {{\;\hbar \, (1-L_j)\;} \over {\; q_i^{+1} - \, q_i^{-1} \;}}   \hfill  }  $$
   \indent   Now, the formulae above show that actually even  $ \, \big( \uRPhbm , \uRPhbp \big) \, $
   is indeed a matched pair of Hopf algebras, with actions uniquely induced in the
   obvious way from the actions for the pair  $ \, \big( \caluRPhbm , \uRPhbp \big) \, $
 which are explicitly given by
%
%
  $$  \displaylines{
   T_i^+ \triangleright T_j^- \, = \, 0  \quad ,   \qquad  T_i^+ \triangleleft T_j^- \, = \, 0  \quad ,   \qquad
     T_i^+ \triangleleft F_j \, = \, 0  \quad ,   \qquad  E_i \triangleright T_j^- \, = \, 0  \cr
   T_i^+ \triangleright F_j \, = \, -p_{ij} F_j  \quad ,   \qquad \qquad \qquad  E_i \triangleleft T_j^- \, = \, -p_{ij} E_i  \cr
   E_i \triangleright F_j  \; = \;  \delta_{i{}j} \, {{\; \big( 1 - e^{-\hbar \, T_j^-} \big) \;} \over {\; q_i^{+1} - \, q_i^{-1} \;}}  \quad ,
   \quad \qquad  E_i \triangleleft F_j  \; = \; \delta_{i{}j} \, {{\; \big( e^{+\hbar \, T_i^+} \! - 1 \big)\;} \over {\; q_i^{+1} - \, q_i^{-1} \;}}  \cr
   \qquad \qquad \qquad   Y \triangleright 1 \, = \, 0  \quad ,  \quad \qquad  1 \triangleleft Y \, = \, 0
   \qquad \qquad  \forall \;\;\; Y \in {\big\{ E_i , T_i^+ , F_i , T_i^- \big\}}_{i\in I}  }  $$
   \indent   Therefore, a well-defined double cross product  $ \, \uRPhbm \bowtie \uRPhbp \, $  exists,
   which is a Hopf algebra containing both  $ \uRPhbp $  and  $ \uRPhbm $  as Hopf subalgebras.
                                                              \par
   With a similar situation as for  Theorem \ref{thm: double-cong-uRPhg_split},  we may then obtain our
   FoMpQUEA  $ \uRPhg $ as a double cross product Hopf algebra, namely the following holds:
\end{free text}

\vskip9pt

\begin{theorem}  \label{thm: double-cong-uRPhg_split-crossproduct}
 Let  $ \, A := {\big(\, a_{i,j} \big)}_{i, j \in I} \, $  be a generalized symmetrizable Cartan matrix, and let
 $ \, P := {\big(\, p_{i,j} \big)}_{i, j \in I} \in M_n\big(\, \kh \,\big) \, $  be a matrix of Cartan type with associated Cartan matrix $ A \, $.
 With assumptions as above,  $ \, \uRPhbm \bowtie \uRPhbp \, $  is a topological,  $ \hbar $--adically  complete Hopf\/  $ \kh $--algebra,
 which is isomorphic to the FoMpQUEA  $ \uRPhg $  given in  Definition \ref{def: Mp-Uhgd}.
 \qed
\end{theorem}

\vskip9pt

\begin{obs}
 Here again, it is worth pointing out that the procedure we followed above to construct  $ \, \uRPhbm \bowtie \uRPhbp \, $  follows a general recipe.  Namely, starting with a QUEA  $ U_\hbar $  and its dual QFSHA  $ \, U_\hbar^{\,*} := F_\hbar \, $,  \,one has the right-left matched pair  $ \, \big( U_\hbar \, , F_\hbar \big) \, $,  with   $ U_\hbar $  acting on  $ F_\hbar $  by coadjoint action, and viceversa; thus one can construct  $ \, U_\hbar \bowtie F_\hbar \, $,  which is isomorphic to the quantum double  $ \, D\big( U_\hbar \, , F_\hbar \big) \, $  and, as such, is not yet the kind of object we are looking for.  Then one observes   --- see  \cite{AT}, \S A.5 ---   that the right-left matched pair  $ \, \big( U_\hbar \, , F_\hbar \big) \, $  induces another similar right-left matched pair  $ \, \big( U_\hbar \, , U_\hbar^\vee \big) \, $,  \,where  $ U_\hbar^\vee $  denotes (in notation of  \cite{AT})  the QUEA that is associated by Drinfeld's  \textsl{Quantum Duality Principle\/}  with the QFSHA  $ F_\hbar \, $.  Finally, we can consider the double cross product  $ \, U_\hbar \bowtie U_\hbar^\vee \, $   --- isomorphic to  $ \, \big( U_\hbar \, , U_\hbar^\vee \big) \, $  ---   which is now exactly the kind of QUEA we are looking for.
                                                                                      \par
   Instead of applying  \textit{verbatim\/}  the recipe sketched above, in the previous construction we followed an explicit, concrete approach that seems totally independent; however, it is important to understand that what we did is in fact nothing but a concrete ``realization'' of the general recipe, even though it is not formally apparent.
\end{obs}

\vskip9pt

\begin{free text}
 {\bf The general case: third proof of  Theorem \ref{thm: form-MpQUEAs_are_Hopf}.}\,
   The previous analysis provides an
   explicit construction of any FoMpQUEA
   defined on a realization  $ \R $  which
   is  \textsl{split minimal\/};  now, out
   of this, we can deduce also a
   construction of a FoMpQUEA on  $ \R $
   of  \textsl{any type},  by a process of
   ``extension and quotient''.  In the end,
   we find another proof for  Theorem \ref{thm: form-MpQUEAs_are_Hopf}.
 \vskip5pt
   Let  $ P $  be a multiparameter matrix (of Cartan type), let  $ \; \R \, := \, \big(\, \lieh \, , \Pi \, , \Pi^\vee \,\big) \; $
   be any realization of it, and let  $ \uRPhg $  be the associated (topological, unital, associative)  $ \kh $--algebra,
   as in  Definition \ref{def: Mp-Uhgd}.
                                                                      \par
   By  Lemma \ref{lemma: split-lifting},  we can also pick a  \textsl{split\/}  realization of  $ P $,  say
   $ \; \dot{\R} \, := \, \big(\, \dot{\lieh} \, , \dot{\Pi} \, , {\dot{\Pi}}^\vee \,\big) \; $,  \;and
   $ \, \dot{\lieh}_{{}_T} := \textsl{Span}\Big( {\big\{ T_i^\pm \big\}}_{i \in I} \Big) \, $  inside  $ \dot{\lieh} \, $.
   Then we take the FoMpQUEA  $ U^{\,\dot{\R}}_{\!P,\hbar}(\lieg)  $  associated with  $ \dot{\R} \, $:
   \,inside it, we consider the Cartan subalgebras  $ \, U_\hbar\big(\dot{\lieh}\big) := U^{\,\dot{\R}}_{\!P,\hbar}\big(\dot{\lieh}\big) \, $
   and  $ \, U_\hbar\big(\dot{\lieh}_{{}_T}\big) := U^{\,\dot{\R}}_{\!P,\hbar}\big(\dot{\lieh}_{{}_T}\big) \, $
   --- both independent of  $ \R $  and  $ P $,  as for every Cartan subalgebra ---
   and also the complete, unital  $ \kh $--subalgebra  generated by  $ \dot{\lieh}_{{}_T} $,
   \,the  $ E_i $'s  and the  $ F_i $'s:  \,the latter is clearly yet another FoMpQUEA, namely
   $ U^{\,\mathring{\R}}_{\!P,\hbar}(\lieg) \, $,  \,where
   $ \; \mathring{\R} \, := \, \big(\, \dot{\lieh}_{{}_T} , \mathring{\Pi} \, , \Pi^\vee \,\big) \; $
   --- with  $ \, \mathring{\Pi} := \Big\{\, \mathring{\alpha}_i := \alpha_i{\big|}_{\dot{\lieh}_{{}_T}} \Big\}_{i \in I} \, $  ---
   is again a realization of  $ P $,  now \textsl{split and minimal}.
   Thanks to  Theorem \ref{thm: double-cong-uRPhg_split}  then,
   there exists a Hopf algebra structure on  $ U^{\,\mathring{\R}}_{\!P,\hbar}(\lieg) \, $,
   \,which is described by formulas as in  \eqref{eq: coprod_x_uPhg},  \eqref{eq: counit_x_uPhg},  \eqref{eq: antipode_x_uPhg}.
                                                                      \par
%
%
   On the other hand, definitions imply that
 $ \; U^{\,\dot{\R}}_{\!P,\hbar}(\lieg) \, \cong \,
 U_\hbar\big(\dot{\lieh}\big) \hskip-3pt \mathop{\widehat{\otimes}}\limits_{U_\hbar(\dot{\lieh}_{{}_T}\!)} \hskip-5pt  U^{\,\mathring{\R}}_{\!P,\hbar}(\lieg) \; $.
 Then there exists only one way to extend the (topological) Hopf structure in
 $ U^{\,\mathring{\R}}_{\!P,\hbar}(\lieg) $  mentioned above to a Hopf structure on
%
 $ \; U^{\,\dot{\R}}_{\!P,\hbar}(\lieg) \, \cong \,
 U_\hbar\big(\dot{\lieh}\big) \hskip-3pt \mathop{\widehat{\otimes}}\limits_{U_\hbar(\dot{\lieh}_{{}_T}\!)} \hskip-5pt  U^{\,\mathring{\R}}_{\!P,\hbar}(\lieg) \; $
 so that all elements in  $ \dot{\lieh} $  are primitive; in other words, there exists a unique (topological) Hopf structure in
%
  $ \; U_\hbar\big(\dot{\lieh}\big)  \hskip-3pt \mathop{\widehat{\otimes}}\limits_{U_\hbar(\dot{\lieh}_{{}_T}\!)}
   \hskip-5pt U^{\,\mathring{\R}}_{\!P,\hbar}(\lieg) \, \cong \, U^{\,\dot{\R}}_{\!P,\hbar}(\lieg) \; $
 which coincides with the given one on the right-hand factor and makes all elements of  $ \dot{\lieh} $  primitive in the left-hand factor.
                                                                      \par
   Finally, again by  Lemma \ref{lemma: split-lifting},  there exists an epimorphism of realizations
   $ \; \underline{\pi} : \dot{\R} \relbar\joinrel\twoheadrightarrow \R \; $.
   By functoriality, we get an epimorphism
   $ \; U_{\underline{\pi}} : U^{\,\dot{\R}}_{\!P,\hbar}(\lieg) \relbar\joinrel\twoheadrightarrow \uRPhg \; $
   with  $ \Ker\big(U_{\underline{\pi}}\,\big) $  generated by  $ \Ker(\pi) $   --- cf.\ Proposition \ref{prop: functor_R->uRPhg};  moreover, every element in
   $ \Ker(\pi) $  is primitive and is central in  $ U^{\,\dot{\R}}_{\!P,\hbar}(\lieg) \, $.
   Thus  $ \Ker\big(U_{\underline{\pi}}\,\big) $  is a  \textsl{Hopf ideal\/}  in the Hopf algebra  $ U^{\,\dot{\R}}_{\!P,\hbar}(\lieg) \, $,
   \,hence  $ U^{\,\R}_{\!P,\hbar}(\lieg) $  inherits via  $ U_{\underline{\pi}} $  a quotient Hopf algebra structure from
   $ U^{\,\dot{\R}}_{\!P,\hbar}(\lieg) \, $,  again described by the formulas in
   \eqref{eq: coprod_x_uPhg},  \eqref{eq: counit_x_uPhg},  \eqref{eq: antipode_x_uPhg},  q.e.d.   \hfill   $ \square $
\end{free text}

\vskip9pt

\begin{rmk}
 We expect that our definition (and construction) of FoMpQUEAs,
 and all related results presented hereafter, can be extended to the case when the symmetrizable generalized Cartan matrix  $ A $
 is replaced by a more general symmetrizable Borcherds-Cartan matrix
 --- see  \cite{ApS}  and references therein.
 However, due to additional technical difficulties, we do not pursue such a goal in this paper.
\end{rmk}

\bigskip
 \vskip9pt

\section{Deformations of formal multiparameter QUEAs}  \label{sec: deform's_FoMpQUEAs}

\vskip11pt

   After introducing formal MpQUEAs in the previous section,
   now we go and study their deformations, either by twist or by  $ 2 $--cocycle   --- both of ``toral type'', say.

\vskip13pt

\subsection{Deformations of FoMpQUEAs by toral twists}  \label{subsec: tw-def_FoMpQUEAs}
 {\ }
 \vskip9pt
   We discuss now suitable twist deformations (of ``toral type'') of FoMpQUEAs,
   proving that they are again FoMpQUEAs.
   By the results in  \cite{ESS}  one can show that all possible twist elements  $ \F $  for Drinfeld's  $ \uhg $
   can be constructed from data associated with Belavin-Drinfeld triples which classify classical
   $ r $--matrices  for  $ \lieg $  itself: in this respect, our ``toral'' twists correspond to
   \hbox{the trivial Belavin-Drinfeld triples.}

\vskip9pt

\begin{free text}  \label{twist-uPhgd}
 {\bf Toral twist deformations of  $ \uRPhg \, $.}
 We fix  $ \, P := {\big(\, p_{i,j} \big)}_{i, j \in I} \in M_n\big(\kh\big) \, $
 of Cartan type with associated Cartan matrix  $ A \, $,  a realization
 $ \, \R \, := \, \big(\, \lieh \, , \Pi \, , \Pi^\vee \,\big) \, $  of it and the FoMpQUEA  $ \uRPhg \, $,
 as in  \S \ref{sec: Cartan-data_realiz's}  and  \S \ref{sec: form-MpQUEAs};  in particular,
 $ \, d_i := p_{ii}/2 \, $  ($ \, i \in I \, $),  and  $ \lieh $  is a free  $ \kh $--module  of finite rank
 $ \, t := \rk(\lieh) \, $.  We fix in  $ \lieh $  any  $ \kh $--basis  $ \, {\big\{ H_g \big\}}_{g \in \G} \, $,
 where  $ \, |\G| = \rk(\lieh) = t \, $.  Pick  $ \; \Phi = \big( \phi_{gk} \big)_{g, k \in \G} \in \lieso_t\big(\kh\big)  \; $,  \,and set
  $$
  \JJ_\Phi  \; := \;  {\textstyle \sum\limits_{g,k=1}^t} \phi_{gk} \, H_g \otimes H_k  \; \in \;
  \lieh \otimes \lieh  \; \subseteq \; U^\R_{\!P,\hskip0,7pt\hbar}(\lieh) \otimes U^\R_{\!P,\hskip0,7pt\hbar}(\lieh)
  $$
 By direct check, we see that the element
\begin{equation}  \label{eq: Resh-twist_F-uPhgd}
  \F_\Phi  \,\; := \;\,
e^{\,\hbar \, 2^{-1} \JJ_\Phi}
 \,\; = \;\,  \exp\Big(\hskip1pt \hbar \,  \, 2^{-1} \,
 {\textstyle \sum_{g,k=1}^t} \phi_{gk} \, H_g \otimes H_k \Big)
\end{equation}
 in  $ \, U^\R_{\!P,\hskip0,7pt\hbar}(\lieh) \,\widehat{\otimes}\, U^\R_{\!P,\hskip0,7pt\hbar}(\lieh) \, $
 is actually a  {\sl twist\/}  for  $ \uRPhg $  in the sense of  \S \ref{defs_Hopf-algs}.
 Using it, we construct a new (topological) Hopf algebra  $ \, {\big(\, \uRPhg \big)}^{\F_\Phi} \, $,
 isomorphic to  $ \uRPhg $  as an algebra but with a new, twisted coalgebra structure, as in  \S \ref{defs_Hopf-algs}.
 A direct calculation yields explicit formulas for the new coproduct on generators, namely
  $$  \displaylines{
   \qquad   \Delta^{\scriptscriptstyle \!\Phi}\big(E_\ell\big)  \; = \;
 E_\ell \otimes \L_{\Phi, \ell}^{+1} \, + \, e^{+\hbar \, T^+_\ell} \K_{\Phi, \ell}^{+1} \otimes
E_\ell   \quad \qquad  \big(\, \forall \;\; \ell \in I \,\big)  \cr
   \qquad \qquad \quad   \Delta^{\scriptscriptstyle \!\Phi}\big(T\big)  \; = \;  T \otimes 1 \, + \, 1 \otimes T
   \quad \quad \qquad \qquad  \big(\, \forall \;\; T \in \lieh \,\big)  \cr
   \qquad   \Delta^{\scriptscriptstyle \!\Phi}\big(F_\ell\big)  \; = \;
   F_\ell \otimes \L_{\Phi, \ell}^{-1} \, e^{-\hbar \, T^-_\ell } + \, \K_{\Phi, \ell}^{-1} \otimes F_\ell
   \quad \qquad  \big(\, \forall \;\; \ell \in I \,\big)  }  $$
with
%
%
  $ \;\; \L_{\Phi, \ell} := e^{+ \hbar \, 2^{-1} \sum_{g,k=1}^t \alpha_\ell(H_g) \, \phi_{gk} H_k} \; $,
  $ \;\; \K_{\Phi,\ell} := e^{+ \hbar \, 2^{-1} \sum_{g,k=1}^t \alpha_\ell(H_g) \, \phi_{kg} H_k} \; (\, \forall \;\; \ell \in I \,) \; $.
 Similarly, the ``twisted'' antipode  $ \, \SS^{\scriptscriptstyle \Phi} := \SS^{\F_\Phi} \, $  and
 the counit  $ \, \epsilon^{\scriptscriptstyle \Phi} := \epsilon \, $   are given by
  $$  \begin{matrix}
   \qquad   \SS^{\scriptscriptstyle \Phi}\big(E_\ell\big)  \, = \,
   - e^{-\hbar \, T^+_\ell} \K_{\Phi, \ell}^{-1} \, E_\ell \, \L_{\Phi, \ell}^{-1}  \quad ,
   &  \qquad  \epsilon^{\scriptscriptstyle \Phi}\big(E_\ell\big) \, = \, 0   \qquad \qquad  \big(\, \forall \;\; \ell \in I \,\big)  \\
   \qquad
 \SS^{\scriptscriptstyle \Phi}\big(T\big)  \, = \,  -T  \quad ,  \phantom{\Big|^|}
  &  \qquad
 \epsilon^{\scriptscriptstyle \Phi}\big(T\big) \, = \, 0   \qquad \qquad  \big(\, \forall \;\; T \in \lieh \,\big)
 \\
   \qquad   \SS^{\scriptscriptstyle \Phi}\big(F_\ell\big)  \, = \,
   - \K_{\Phi, \ell}^{+1} \, F_\ell \, \L_{\Phi, \ell}^{+1} \, e^{+\hbar \, T^-_\ell}  \quad ,
   &  \qquad  \epsilon^{\scriptscriptstyle \Phi}\big(F_\ell\big) \, = \, 0   \qquad \qquad  \big(\, \forall \;\; \ell \in I \,\big)
\end{matrix}  $$
\end{free text}

\vskip5pt

\begin{rmk} \label{rmk: Reshetikhin's twist}
 The twist  $ \F_\Phi $  is an example of Reshetikhin's twist as in  \cite{Re},  only
 ``adapted'' to the present case of our more general FoMpQUEA  $ \uRPhg \, $.
 When  $ \lieg $  is a simple Lie algebra, this twist corresponds to empty datum of the Belavin-Drinfeld triple with respect to the
%
%
 classification in  \cite{ESS}  of twists for  $ \uhg \, $.
\end{rmk}

\vskip9pt

\begin{free text}  \label{twisted-generators}
 {\bf Twisted generators.}\,
 From the explicit description of the coproduct  $ \Delta^{\scriptscriptstyle \!\Phi} \, $,
 it follows that  $ \, {\big(\, \uRPhg \big)}^{\F_\Phi} \, $
 is generated by group-likes and skew-primitive elements; in particular, it is a pointed Hopf algebra.  Moreover,
 both Hopf algebras  $ \,  \uRPhg  \, $  and  $ \, {\big(\, \uRPhg \big)}^{\F_\Phi} \, $  have the same coradical and the same space of skew-primitive elements.
 As the coproduct is changed by the twist, one sees that the skew-primitive generators of  $ \, \uRPhg  \, $,
 which are  $ (1,g) $--  or $ (g,1) $--primitive  for some  $ \, g \in G\big(\, \uRPhg \big) \, $,  with respect to  $ \Delta \, $,
 become  $ (h,k) $--primitive  for $ \Delta^{\scriptscriptstyle \!\Phi} \, $.
 Looking at the coradical filtration, and the associated graded Hopf algebra,
 one may find from that set of generators some new  $ (1,\ell) $--  or  $ (\ell,1) $--primitives
 for  $ {\big(\, \uRPhg \big)}^{\F_\Phi} $.
 This leads to devise (new)  \textit{twisted generators\/}  and a corresponding presentation for
 $ {\big(\, \uRPhg \big)}^{\F_\Phi} $,  which yields a Hopf algebra isormorphism between  $ {\big(\, \uRPhg \big)}^{\F_\Phi} $
 and a new FoMpQUEA with suitable multiparameter matrix and realization.
 \par

   Motivated by the above analysis, we introduce now in  $ {\big(\, \uRPhg \big)}^{\F_\phi} $  the ``twisted'' generators  (for all  $ \, \ell \in I \, $)
  $ \; E^{\scriptscriptstyle \Phi}_\ell := \, \L_{\Phi, \ell}^{-1} \, E_\ell \; $,
 $ \; F^{\scriptscriptstyle \Phi}_\ell := \, F_\ell \, \K_{\Phi, \ell}^{+1} \; $
 and the twisted ``distinguished toral elements'' (or ``coroots'') that were already defined in  \eqref{eq: T-phi},  i.e.\
 $ \; T^\pm_{{\scriptscriptstyle \Phi},\ell} := \,
 T^\pm_\ell \, \pm {\textstyle \sum\limits_{g,k=1}^t} \alpha_\ell(H_g) \, \phi_{kg} \, H_k \; $.
 Still from  \S \ref{twist-deform.'s x mpmatr.'s & realiz.'s},  we recall also
  $ \, P_{\scriptscriptstyle \Phi} \, := \, {\big(\, p^{\scriptscriptstyle \Phi}_{i,j} \big)}_{i, j \in I} \, $  and
  $ \, \R_\Phi := \big(\, \lieh \, , \Pi \, , \Pi^\vee_{\scriptscriptstyle \Phi} \,\big) \, $,
  \,the latter being a realization of the former.
 \vskip3pt
   Second, the commutation relations in the algebra  $ \,{\big(\, \uRPhg \big)}^{\F_\Phi} \, $
   give new commutation relations between twisted generators.  Namely, by straightforward computations
   --- for instance using that
   $ \; \K_{\Phi,j} \, E_i = e^{\hbar\, 2^{-1} \, \sum_{g,k=1}^t \alpha_j(H_g) \, \phi_{kg} \alpha_i(H_k)} E_i \, \K_{\Phi,j} \; $
   and that  $ \; e^{+\hbar \, T^{\pm}_{{\scriptscriptstyle \Phi},i}} =
   e^{+\hbar \, T_i^\pm} \big(\, \K_{\Phi,i} \, \L_{\Phi,i}^{-1} \,\big)^{\pm 1} \; $  ---
   one proves that inside  $ \,{\big(\, \uRPhg \big)}^{\F_\Phi} \, $  the following identities hold true (for all $ \, T , T' , T'' \in \lieh \, $,  $ \, i \, , j \in I \, $):
  $$  \begin{aligned}
   T \, E^{\scriptscriptstyle \Phi}_j \, - \, E^{\scriptscriptstyle \Phi}_j \, T  \; = \;  +\alpha_j(T) \, E^{\scriptscriptstyle \Phi}_j  \quad ,
   \qquad  T \, F^{\scriptscriptstyle \Phi}_j \, - \, F^{\scriptscriptstyle \Phi}_j \, T  \; = \;  -\alpha_j(T) \, F^{\scriptscriptstyle \Phi}_j   \hskip45pt  \\
   T' \, T''  \; = \;  T'' \, T'  \quad ,  \qquad
   E^{\scriptscriptstyle \Phi}_i \, F^{\scriptscriptstyle \Phi}_j \, - \
   , F^{\scriptscriptstyle \Phi}_j \, E^{\scriptscriptstyle \Phi}_i  \; = \;
   \delta_{i,j} \, {{\; e^{+\hbar \, T^+_{{\scriptscriptstyle \Phi},i}} \, -
   \, e^{-\hbar \, T^-_{{\scriptscriptstyle \Phi},i}} \;} \over {\; q_i^{+1} - \, q_i^{-1} \;}}   \hskip47pt  \\
   {\textstyle \sum\limits_{k=0}^{1-a_{ij}}} {(-1)}^k {\left[ { 1-a_{ij} \atop k }
\right]}_{\!q_i} {\big( q_{ij}^{\scriptscriptstyle \Phi} \big)}^{+k/2\,}
{\big( q_{ji}^{\scriptscriptstyle \Phi} \big)}^{-k/2} \,
{\big( E^{\scriptscriptstyle \Phi}_i \big)}^{1-a_{ij}-k} \, E^{\scriptscriptstyle \Phi}_j \, {\big( E^{\scriptscriptstyle \Phi}_i \big)}_i^k  \; = \;  0   \quad   (\, i \neq j \,)
  \\
   {\textstyle \sum\limits_{k=0}^{1-a_{ij}}} {(-1)}^k {\left[ { 1-a_{ij} \atop k }
\right]}_{\!q_i} {\big( q_{ij}^{\scriptscriptstyle \Phi} \big)}^{+k/2\,}
{\big( q_{ji}^{\scriptscriptstyle \Phi} \big)}^{-k/2} \,
{\big( F^{\scriptscriptstyle \Phi}_i \big)}^k \, F^{\scriptscriptstyle \Phi}_j \, {\big( F^{\scriptscriptstyle \Phi}_i \big)}_i^{1-a_{ij}-k}  \; = \;  0   \quad   (\, i \neq j \,)
 \end{aligned}  $$
 with  $ \; q^{\scriptscriptstyle \Phi}_{i,j} := e^{\hbar \, p^{\scriptscriptstyle \Phi}_{i,j}} \; (\, i , j \in I \,) \, $
   --- so that  $ \, q^{\scriptscriptstyle \Phi}_{i,i} = e^{\hbar \, p^{\scriptscriptstyle \Phi}_{i,i}} =
 e^{\hbar \, p_{i,i}} = e^{\hbar \, 2 d_i} = q_i^{\,2} \, $.
 \vskip5pt
   Third, the Hopf operations on the ``twisted'' generators read  (for  $ \, \ell \in I \, $,  $ \, T \in \lieh \, $)
  $$  \begin{matrix}
   \Delta^{\scriptscriptstyle \!\Phi}\big(E^{\scriptscriptstyle \Phi}_\ell\big)
\, = \,  E^{\scriptscriptstyle \Phi}_\ell \otimes 1 \, + \, e^{+\hbar\,T^+_{{\scriptscriptstyle \Phi},\ell}}
\otimes E^{\scriptscriptstyle \Phi}_\ell  \;\; ,   &   \hskip5pt
   \SS^{\scriptscriptstyle \Phi}\big(E^{\scriptscriptstyle \Phi}_\ell\big)  \, = \,  - e^{-\hbar\,T^+_{{\scriptscriptstyle \Phi},\ell}} E^{\scriptscriptstyle \Phi}_\ell  \;\; ,   &   \hskip5pt
  \epsilon^{\scriptscriptstyle \Phi}\big(E^{\scriptscriptstyle \Phi}_\ell\big) \, = \, 0  \\
   \Delta^{\scriptscriptstyle \!\Phi}\big(T\big)  \, = \,  T \otimes 1 \, + \, 1 \otimes T  \;\; ,   &   \hskip5pt
   \SS^{\scriptscriptstyle \Phi}\big(T\big)  \, = \,  - T  \;\; ,   &   \hskip5pt
   \epsilon^{\scriptscriptstyle \Phi}\big(T\big)  \, = \,  0  \\
   \Delta^{\scriptscriptstyle \!\Phi}\big(F^{\scriptscriptstyle \Phi}_\ell\big)  \, = \,
   F^{\scriptscriptstyle \Phi}_\ell \otimes e^{-\hbar\,T^-_{{\scriptscriptstyle \Phi},\ell}} \, + \, 1 \otimes F^{\scriptscriptstyle \Phi}_\ell  \;\; ,   &   \hskip5pt
   \SS^{\scriptscriptstyle \Phi}\big(F^{\scriptscriptstyle \Phi}_\ell\big)  \, = \,
   - F^{\scriptscriptstyle \Phi}_\ell \, e^{+\hbar \, T^-_{{\scriptscriptstyle \Phi},\ell}}  \;\; ,   &   \hskip5pt
   \epsilon^{\scriptscriptstyle \Phi}\big(F^{\scriptscriptstyle \Phi}_\ell\big)  \, = \,  0
 \end{matrix}  $$
\end{free text}

\vskip9pt

   In a nutshell, the above analysis proves the following result:

\vskip13pt

\begin{theorem}  \label{thm: twist-uRPhg=new-uRPhg}
 There exists an isomomorphism of topological Hopf algebras
  $$  f^{\scriptscriptstyle \Phi}_{\scriptscriptstyle P} \; :
   \; U_{\!P_\Phi,\,\hbar}^{\R_{\scriptscriptstyle \Phi}}(\lieg)  \,\; {\buildrel \cong \over
   {\lhook\joinrel\relbar\joinrel\relbar\joinrel\relbar\joinrel\twoheadrightarrow}} \;\,
   {\big(\, U_{\!P,\,\hbar}^{\R}(\lieg)\big)}^{\F_\Phi}  $$
 given by
  $ \; E_i \, \mapsto \, E^{\scriptscriptstyle \Phi}_i \, , \; T \, \mapsto \, T \,$ and
  $\; F_i \, \mapsto \, F^{\scriptscriptstyle \Phi}_i \; $
 for all  $ \, i \in I \, $,  $ \, T \in \lieh \, $.
 \vskip3pt
   In particular, the class of all FoMpQUEAs of any fixed Cartan type and of fixed rank
%
%
 is stable by toral twist deformations.  Moreover, inside it the subclass of all such
 FoMpQUEAs associated with  \textsl{straight},  resp.\  \textsl{small},  realizations is stable as well.
 \vskip3pt
   Similar, parallel statements hold true for the Borel FoMpQUEAs, namely there exist isomorphisms
 $ \; f^{\scriptscriptstyle \Phi}_{\!\scriptscriptstyle P,\pm} :
   U_{\!P_\Phi,\,\hbar}^{\R_{\scriptscriptstyle \Phi}}(\lieb_\pm) \, {\buildrel \cong \over
   {\lhook\joinrel\relbar\joinrel\relbar\joinrel\twoheadrightarrow}} \,
   {\big(\, U_{\!P,\,\hbar}^{\R}(\lieb_\pm)\big)}^{\F_\Phi} \; $
 given by formulas as above.
\end{theorem}

\vskip9pt

   In fact,  \textsl{the previous result can be somehow reversed},  as the following shows:
   in particular, loosely speaking, we end up finding that  \textit{every straight small FoMpQUEA
   can be realized as a toral twist deformation of the ``standard'' FoMpQUEA by Drinfeld}
   ---  cf.\ claim  \textit{(c)\/}  in  Theorem \ref{thm: uRPhg=twist-uhg}  here below.

\vskip13pt

\begin{theorem}  \label{thm: uRPhg=twist-uhg}
 With assumptions as above, let  $ P $  and  $ P' $  be two matrices of Cartan type with
 the same associated Cartan matrix $ A \, $.
 \vskip3pt
   (a)\,  Let  $ \R $  be a  \textsl{straight}  realization of  $ P $  and let  $ \uRPhg $  be the associated FoMp\-QUEA.
   Then there exists a  \textsl{straight}  realization  $ \check{\R}' $
   of  $ P' $  and a matrix  $ \, \Phi \in \! \lieso_t\big(\kh\big) \, $
   such that for the associated twist element  $ \F_\Phi $  as in  \eqref{eq: Resh-twist_F-uPhgd}  we have
  $$  U_{\!P'\!,\,\hbar}^{\check{\R}'}(\lieg)  \; \cong \;  {\big(\, U_{\!P,\,\hbar}^\R(\lieg) \big)}^{\F_\Phi}  $$
 In a nutshell, if  $ \, P'_s = P_s \, $  then from any  \textsl{straight}  FoMpQUEA over  $ P $  we can obtain by toral twist deformation a  \textsl{straight}  FoMpQUEA (of the same rank) over  $ P' $.
                                                                  \par
   Conversely, if  $ \R' $  is any  \textsl{straight}  realization of  $ P' $  and  $  U_{\!P'\!,\,\hbar}^{\R'}(\lieg) $
   is the associated FoMpQUEA, then there exists a  \textsl{straight}  realization  $ \hat{\R} $  of  $ P $
   and a matrix  $ \, \Phi \in \! \lieso_t\big(\kh\big) \, $  such that for the associated twist element  $ \F_\Phi $
   as in  \eqref{eq: Resh-twist_F-uPhgd}  we have
  $$  U_{\!P'\!,\,\hbar}^{\R'}(\lieg)  \; \cong \;  {\big(\, U_{\!P,\,\hbar}^{\hat{\R}}(\lieg) \big)}^{\F_\Phi}  $$
 \vskip3pt
   (b)\,  Let  $ \R $  and  $ \R' $  be  \textsl{straight small}  realizations of  $ P $  and  $ P' $  respectively, with
   $ \, \rk(\R) = \rk(\R') = t\, $,  and let  $ \uRPhg $  and  $ U_{\!P'\!,\,\hbar}^{\R'}(\lieg) $  be the associated FoMpQUEAs.
   Then there exists a matrix  $ \, \Phi \in \! \lieso_t\big(\kh\big) \, $  such that for
%
%
 $ \F_\Phi $  as in  \eqref{eq: Resh-twist_F-uPhgd}  we have
  $$  U_{\!P'\!,\,\hbar}^{\R'}(\lieg)  \; \cong \;  {\big(\, U_{\!P,\,\hbar}^\R(\lieg) \big)}^{\F_\Phi}  $$
 In other words, if  $ \, P'_s = P_s \, $  any  \textsl{straight small}  FoMpQUEA over  $ P' $  is
 isomorphic to a toral twist deformation of any  \textsl{straight small}  FoMpQUEA over  $ P $  of same rank.
 \vskip3pt
   (c)\,  Every  \textsl{straight small}  FoMpQUEA is isomorphic to some toral twist deformation of
   Drinfeld's standard FoMpQUEA (over  $ \, DA = P_s \, $)  of the same rank.
 \vskip3pt
   (d)\,  Similar, parallel statements hold true for the Borel FoMpQUEAs.
\end{theorem}

\pf
 \textit{(a)}\,  By  Theorem \ref{thm: twist-uRPhg=new-uRPhg}  it is enough to find  $ \, \Phi \in \lieso_t(\Bbbk)  \, $
 such that  $ \, P' = P_{\scriptscriptstyle \Phi} \, $,  \,that is
 $ \; P' = P - \mathfrak{A} \, \Phi \, \mathfrak{A}^{\,\scriptscriptstyle T} \, $;
 \,but this is guaranteed by  Lemma \ref{lemma: twist=sym},  so we are done.
 \vskip5pt
   \textit{(b)}\,  Like for  Theorem \ref{thm: MpLbA=twist-gD}\textit{(b)},  this follows from claim  \textit{(a)},
   along with the uniqueness of straight small realizations, by  Proposition \ref{prop: exist-realiz's}\textit{(b)},
   and  Proposition \ref{prop: functor_R->liegRP}.
 \vskip5pt
   \textit{(c)}\,  This follows applying  \textit{(b)},  with  $ U_{\!P'\!,\,\hbar}^{\R'}(\lieg) $
   the given straight small FoMpQUEA and  $ U_{\!P,\,\hbar}^\R(\lieg) $  the ``standard'' FoMpQUEA
   $ U_{\!P,\,\hbar}^\R\big(\lieg^\Dpicc_\Ppicc\big) $  over  $ \, P := DA = P'_s \, $
 as in Drinfeld's definition (up to ``taking the double''), which is straight and split minimal.
\epf

\vskip9pt

\begin{obs}  \label{obs: parameter_(P,Phi)}
 Theorems \ref{thm: twist-uRPhg=new-uRPhg}  and  \ref{thm: uRPhg=twist-uhg}  have the following interpretation.
 Our FoMpQUEAs  $ \, \uRPhg \, $  are quantum objects depending on the multiparameter  $ P \, $;
 but when we perform onto  $ \, \uRPhg \, $  a deformation by twist as in  \S \ref{twist-uPhgd},
 the output  $ \; {U_{\!P,\hskip0,7pt\hbar}^{\R,\Phi}(\hskip0,5pt\lieg)} := {\big(\, \uRPhg \big)}^{\F_\Phi} \; $
 depends on  \textsl{two\/}  multiparameters, namely  $ P $  \textsl{and\/}  $ \Phi \, $.
 Thus all these  $ {U_{\!P,\hskip0,7pt\hbar}^{\R,\Phi}(\hskip0,5pt\lieg)} $'s form a seemingly
 \textsl{richer\/}  family of ``twice-multiparametric'' formal QUEAs.
 Nonetheless,  Theorem \ref{thm: twist-uRPhg=new-uRPhg}
 above proves that this family actually  \textsl{coincides\/}  with the family of all FoMpQUEAs,
 although the latter seems a priori smaller.
                                                              \par
   In short,  Theorems \ref{thm: twist-uRPhg=new-uRPhg}  and  \ref{thm: uRPhg=twist-uhg}  show the following.
   The dependence of the Hopf structure of  $ {U_{\!P,\hskip0,7pt\hbar}^{\R,\Phi}(\hskip0,5pt\lieg)} $
   on the ``double parameter''  $ (P\,,\Phi) $  is ``split'' in the algebraic structure (ruled by  $ P $)
   and in the coalgebraic structure  (ruled by  $ \Phi $);  \,now  Theorems \ref{thm: twist-uRPhg=new-uRPhg}
   and  \ref{thm: uRPhg=twist-uhg}  enable us to ``polarize'' this dependence so to codify it either
   entirely within the algebraic structure (while the coalgebraic one is reduced to a ``canonical form''),
   so that the single multiparameter  $ P_\Phi $  is enough to describe it, or entirely within the
   coalgebraic structure (with the algebraic one being reduced to the ``standard'' Drinfeld's one),
   so that the multiparameter  $ \Phi_P $  alone is enough.
\end{obs}

\vskip7pt

\begin{obs}  \label{obs: split FoMpQUEAs not-stable}
 As the subclass of  \textsl{split\/}  realizations is  \textsl{not closed\/}  under twist
 (cf.\ the end of  \S \ref{further_stability}),  the subclass of all ``split'' FoMpQUEAs is not closed too under twist deformation; this is a quantum
 analogue of  Observations \ref{obs: parameter_(P,Phi) x MpLbA's}\textit{(b)}.
\end{obs}

\vskip17pt

\subsection{Deformations of FoMpQUEAs by toral 2-cocycles}  \label{subsec: coc-def_FoMpQUEAs}
 {\ }
 \vskip7pt
   We consider now some  $ 2 $--cocycle  deformations (called ``of toral type'' again) of the formal MpQUEAs  $ \uRPhg \, $,
   and we prove that these are again formal MpQUEAs.

\vskip9pt

\begin{free text}  \label{2-cocycle-U_h(h)}
 {\bf Special  $ 2 $--cocycles  of  $ U_{\!P,\hbar}(\lieh) \, $.}
 Fix again  $ \, P := {\big(\, p_{i,j} \big)}_{i, j \in I} \in M_n\big( \kh \big) \, $  of
 Cartan type with associated Cartan matrix  $ A \, $,  a realization
 $ \; \R \, := \, \big(\, \lieh \, , \Pi \, , \Pi^\vee \,\big) \; $  of it and the (topological) Hopf algebra  $ \uRPhg \, $,
 as in  \S \ref{sec: Cartan-data_realiz's}  and  \S \ref{sec: form-MpQUEAs},  setting  $ \, d_i := p_{ii}/2 \, $
 for all  $ \, i \in I \, $ and  $ \, D_{\scriptscriptstyle P} := \text{\sl diag}\big(d_1\,,\dots,d_n\big) \, $.
 We consider special 2--cocycles of  $ \, \uRPhg \, $,  called ``toral'' as they are induced from the quantum torus.
 To this end, like in  \S \ref{twist-uPhgd},  we fix in  $ \lieh $  a  $ \kh $--basis  $ \, {\big\{ H_g \big\}}_{g \in \G} \, $,
 where  $ \G $  is an index set with  $ \, |\G| = \rk(\lieh) = t \, $.
 \vskip5pt
   Like in  \S \ref{2-cocycle-deform.'s x mpmatr.'s & realiz.'s},  we fix an antisymmetric,  $ \kh $---bilinear  map
   $ \, \chi : \lieh \times \lieh \relbar\joinrel\longrightarrow \kh \, $,  \,that corresponds to some
   $ \, X = {\big(\, \chi_{g{}\gamma} \big)}_{g , \gamma \in \G} \in \lieso_t\big(\kh\big) \, $  via  $ \, \chi_{g{}\gamma} = \chi(H_g\,,H_\gamma) \, $.
   We also consider the antisymmetric matrix
   $ \, \mathring{X} := {\Big(\, \mathring{\chi}_{i{}j} = \chi\big(\,T_i^+\,,T_j^+\big) \Big)}_{i, j \in I} \! \in \lieso_n(\kh) \, $.
   Any such map  $ \chi $  induces uniquely an antisymmetric,  $ \kh $--bilinear  map
  $$  \tilde{\chi}_{\scriptscriptstyle U} \, : \,  U_{\!P,\hbar}^{\R}(\lieh) \times U_{\!P,\hbar}^{\R}(\lieh)
  \relbar\joinrel\relbar\joinrel\relbar\joinrel\longrightarrow \kh  $$
 as follows.  By definition,  $ U_{\!P,\hbar}^{\R}(\lieh) $  is an  $ \hbar $--adically
 complete topologically free Hopf algebra isomorphic to
 $ \, \widehat{S}_\kh(\lieh) := \widehat{\bigoplus\limits_{n \in \NN} S_\kh^{\,n}(\lieh)} \, $
 --- the  $ \hbar $--adic  completion of the symmetric algebra
 $ \, S_\kh(\lieh) = \bigoplus\limits_{n \in \NN} S_\kh^{\,n}(\lieh) \, $  ---
 hence the following makes sense:

\vskip9pt

\begin{definition}  \label{def: tildechi_U}
 We define  $ \tilde{\chi}_{\scriptscriptstyle U} $  as the unique  $ \kh $--linear
 (hence  $ \hbar $--adically  continuous) map
 $ \; U_{\!P,\hbar}^{\R}(\lieh) \otimes U_{\!P,\hbar}^{\R}(\lieh) \,{\buildrel \tilde{\chi}_{\scriptscriptstyle U}
 \over {\relbar\joinrel\longrightarrow}}\, \kh \; $
 such that (with identifications as above)
  $$  \displaylines{
   \tilde{\chi}_{\scriptscriptstyle U}(z,1) := \epsilon(z) =: \tilde{\chi}_{\scriptscriptstyle U}(1,z)  \;\; ,
   \quad  \tilde{\chi}_{\scriptscriptstyle U}(x,y) := \chi(x,y)  \quad
   \forall \;\; z \in  \widehat{S}_\kh(\lieh) \, , \; x, y \in  S_\kh^{\,1}(\lieh)  \cr
   \hfill \qquad  \tilde{\chi}_{\scriptscriptstyle U}(x,y) := 0   \quad \qquad  \forall \;\; x \in S_\k^{\,r}(\lieh) \, ,
   \; y \in S_\k^{\,s}(\lieh) \, : \, r, s \geq 1 \, , \; r+s > 2   \hfill  \diamondsuit  }  $$
\end{definition}

\vskip9pt

   By construction,  \textsl{$ \tilde{\chi}_{\scriptscriptstyle U} $  is a normalized
   Hochschild 2--cocycle on  $ U_{\!P,\hbar}^{\R}(\lieh) \, $},  \,that is
  $$  \epsilon(x) \,\tilde{\chi}_{\scriptscriptstyle U}(y,z) \, - \, \tilde{\chi}_{\scriptscriptstyle U}(xy,z) \,
  + \, \tilde{\chi}_{\scriptscriptstyle U}(x,yz) \,
  - \, \tilde{\chi}_{\scriptscriptstyle U}(x,y) \, \epsilon(z) \; = \; 0
  \qquad  \forall \;\; x, y, z \in U_{\!P,\hbar}^{\R}(\lieh)  $$

\vskip5pt

   Now recall that, given two linear maps
   $ \, \eta , \vartheta : U_{\!P,\hbar}^{\R}(\lieh) \otimes U_{\!P,\hbar}^{\R}(\lieh) \relbar\joinrel\longrightarrow \kh \, $,
   \,one may define the  {\it convolution product\/}  map
   $ \; \eta \ast \vartheta : {U_{\!P,\hbar}^{\R}(\lieh)}^{\otimes 2} \relbar\joinrel\relbar\joinrel\longrightarrow \kh \; $
   --- using the standard coalgebra structure of
   $ \, {U_{\!P,\hbar}^{\R}(\lieh)}^{\otimes 2}\cong {\widehat{S}_\kh(\lieh)}^{\otimes 2}
   \cong \widehat{S}_\kh(\lieh \oplus \lieh) \, $
   ---   by the formula  $ \; (\eta \ast \vartheta)(x \otimes y) \, := \, \eta(x_{(1)},y_{(1)}) \, \vartheta(x_{(2)},y_{(2)}) \; $
   for all  $ \, x, y \in U_{\!P,\hbar}^{\R}(\lieh) \, $.  Then by  $ \, \eta^{\ast\,m} \, $
   we denote the  $ m $--th  power with respect to the convolution
   product of any map  $ \eta $  as above; in particular, we set  $ \, \eta^{\ast\,0} := \epsilon \otimes \epsilon \, $.
                                                              \par
   The following result describes the powers of our map  $ \tilde{\chi}_{\scriptscriptstyle U} \, \, $:
\end{free text}

\vskip5pt

\begin{lema}  \label{lem: tildechi}
For all  $ \, H_+ ,H_- \in \lieh \, $  and  $ \, k, \ell, m \in \NN_+ \, $,  \,we have
  $$  { \tilde{\chi}_{\scriptscriptstyle U}}^{\,\ast\,m}\big(H_+^{\,k},H_-^{\,\ell}\,\big) \; = \;
   \begin{cases}
      \begin{array}{lr}
  \delta_{k,m} \, \delta_{\ell,m} \, {\big(m!\big)}^2 \, \chi(H_+,H_-)^m  &  \quad  \text{for} \;\;\;\; m \geq 1 \; , \\
  \delta_{k,0} \, \delta_{\ell,0} \,  &  \quad  \text{for} \;\;\;\; m = 0 \; .
      \end{array}
   \end{cases}  $$
\end{lema}

\pf
 The proof follows by a direct computation.
\epf

\vskip9pt

\begin{definition}  \label{def: chi_U}
 Keep notation as above.  We define  $ \, \chi_{\scriptscriptstyle U} \, $  as the unique  $ \kh $--linear
%
%
 map from  $ \, U_{\!P,\hbar}^{\R}(\lieh) \!\mathop{\otimes}\limits_\kh\! U_{\!P,\hbar}^{\R}(\lieh) \, $  to  $ \khp $
 given by the exponentiation of  $ \, \hbar^{-1} \, 2^{-1} \, \tilde{\chi}_{\scriptscriptstyle U} \, $,  i.e.
 \vskip-5pt
  $$  \chi_{\scriptscriptstyle U}  \; := \;  e^{\hbar^{-1} 2^{-1} \tilde{\chi}_{\scriptscriptstyle U}}  \; = \;
  {\textstyle \sum_{m \geq 0}} \, \hbar^{-m} \, {\tilde{\chi}_{\scriptscriptstyle U}}^{\,\ast\,m} \! \Big/ 2^m\,m!   \eqno \diamondsuit  $$
\end{definition}

\vskip9pt

\begin{lema}  \label{lem: chi_U(K,K)}
 The map  $ \, \chi_{\scriptscriptstyle U} $  is a well defined, normalized,
 $ \khp $--valued  Hopf 2--cocy\-cle for  $ \, U_{\!P,\hbar}^{\R}(\lieh) \, $,  such that, for all
 $ H_+ \, , H_- \in \lieh \, $,  and setting  $ \, K_\pm := e^{\,\hbar\,H_\pm} \, $,
  $$  \chi_{\scriptscriptstyle U}^{\pm 1}(H_+,H_-) \, = \, \pm \hbar^{-1} \, 2^{-1}  \chi(H_+,H_-)  \quad ,
  \qquad  \chi_{\scriptscriptstyle U}(K_+,K_-) \, = \, e^{\hbar \, 2^{-1} \chi(H_+,H_-)}  $$
\end{lema}

\pf
 The identities follow from  Lemma \ref{lem: tildechi},  hence  $ \, \chi_{\scriptscriptstyle U} \, $  is well defined.
 The other claims follow from the proof of  \cite[Theorem 4.1]{Sw},  see also  \cite[Lemma 4.1]{GM}.
\epf

\vskip11pt

\begin{free text}  \label{toral-cocycles-uPhgd}
 {\bf Toral  $ 2 $--cocycles of  $ \uRPhg \, $.}
 The previous construction provides, starting from  $ \chi \, $,  a normalized Hopf\/  $ 2 $--cocycle
 $ \; \chi_{\scriptscriptstyle U} : U_{\!P,\hbar}^{\R}(\lieh) \times U_{\!P,\hbar}^{\R}(\lieh)
 \relbar\joinrel\relbar\joinrel\longrightarrow \khp \; $.

\vskip5pt

   \textit{We assume now that the map  $ \, \chi $  satisfies the additional requirement\/  \eqref{eq: condition-chi},
%
 in other words we require that  $ \, \chi \in \text{\sl Alt}_{\,\kh}^{\,S}(\lieh) \, $   --- }
 notation of  \S \ref{subsec: 2-cocycle-realiz}.  The latter map canonically induces a  $ \kh $--bilinear  map
 $ \; \overline{\chi} : \overline{\lieh} \times \overline{\lieh}
\relbar\joinrel\longrightarrow \kh \; $,  \;where  $ \, \overline{\lieh} := \lieh \big/ \lies \, $  with
$ \, \lies := \textsl{Span}_\kh\big( {\{\, S_i \,\}}_{i \in I} \big)\, $,  \,given by
$$
\overline{\chi}\,\big(\, T' \! + \lies \, , T'' \! + \lies \,\big)
\, := \, \chi\big(\,T',T''\big)   \qquad  \forall \;\; T' , T'' \in \lieh
$$
   \indent   Now, replaying the construction above but with  $ \overline{\lieh} $  and
   $ \overline{\chi} $  replacing  $ \lieh $  and  $ \chi \, $,
   we can construct a normalized Hopf\/  $ 2 $--cocycle
   $ \; \overline{\chi}_{\scriptscriptstyle U} \!: U_{\!P,\hbar}^\R\big(\,\overline{\lieh}\,\big) \times
   U_{\!P,\hbar}^\R\big(\,\overline{\lieh}\,\big) \relbar\joinrel\longrightarrow \khp \; $;
for the latter, the analogue of  Lemma \ref{lem: chi_U(K,K)}  holds true again.  Moreover,
  note that  $ \, U_{\!P,\hbar}^\R\big(\,\overline{\lieh}\,\big) \cong \widehat{S}_\kh\big(\,\overline{\lieh}\,\big) \, $,  \,and,
thanks to  \eqref{eq: condition-chi},
 there exists a unique Hopf algebra epimorphism
 $ \; \pi : \uRPhg \relbar\joinrel\relbar\joinrel\twoheadrightarrow U_{\!P,\hbar}^\R\big(\,\overline{\lieh}\,\big) \; $
 given by  $ \, \pi(E_i) := 0 \, $,  $ \, \pi(F_i) := 0 \, $   --- for  $ \, i \in I \, $  ---   and
 $ \, \pi(T) := (T + \lies) \, \in \, \overline{\lieh} \subseteq U_{\!P,\hbar}^\R\big(\,\overline{\lieh}\,\big) \, $
 --- for  $ \, T \in \lieh \, $.
 Then we consider
$$
 \sigma_\chi \, := \, \overline{\chi}_{\scriptscriptstyle U} \circ (\pi \times \pi) \, : \,
 \uRPhg \times \uRPhg \relbar\joinrel\relbar\joinrel\relbar\joinrel\relbar\joinrel\twoheadrightarrow \khp
$$
 which is  \textsl{automatically\/}  a normalized,  $ \khp $--valued  Hopf\/  $ 2 $--cocycle  on  $ \uRPhg \, $.
\end{free text}

\vskip7pt

\begin{definition}  \label{def-sigma_formal}
 We shall call all normalized Hopf  $ 2 $--cocycles  $ \, \sigma_\chi $  of  $ \, \uRPhg $  obtained
 --- from all
 $ \, \chi \in \text{\sl Alt}^{\,S}_{\,\kh}(\lieh) \, $
  ---   via the above construction as ``of toral type'', or ``toral  $ 2 $--cocycles'';
  we denote by  $ \, \Z_{\,2}\big( \uRPhg \big) \, $  the set of all of them
  --- which is actually independent of the multiparameter  $ P $,  indeed.   \hfill  $ \diamondsuit $
\end{definition}

\vskip7pt

\begin{free text}  \label{rmk: formulas_deform-prod}
 {\bf Formulas for the  $ \sigma_\chi $--deformed  product.}
 Let  $ \, \sigma_\chi \in \Z_{\,2}\big( \uRPhg \big) \, $  be a toral  $ 2 $--cocycle  as above.
 Following  \S \ref{defs_Hopf-algs},  using  $ \sigma_\chi $  we introduce in  $ \uRPhg $  a ``deformed product'',
 hereafter denoted by   $ \, \scriptstyle \dot\sigma_\chi \, $;  \,then
 $ \, X^{(n)_{\sigma_{\chi}}} = X\raise-1pt\hbox{$ \, \scriptstyle \dot\sigma_\chi $} \cdots \raise-1pt\hbox{$ \, \scriptstyle \dot\sigma_\chi $} X \, $
 will denote the  $ n $--th  power of any  $ \, X \in \uRPhg \, $  with respect to this deformed product.
                                                \par
   Directly from definitions, sheer computation yields the following formulas,
   relating the deformed product with the old one (for all  $ \; T', T'', T \in \lieh \; $,  $ \; i \, , \, j \in I \, $):
  $$  \displaylines{
   T' \raise-1pt\hbox{$ \, \scriptstyle \dot\sigma_\chi $} T''  \, = \;  T' \, T''  \quad ,
     \qquad  E_i \raise-1pt\hbox{$ \, \scriptstyle \dot\sigma_\chi $} F_j  \; = \,  E_i \, F_j  \quad ,
     \qquad  F_j \raise-1pt\hbox{$ \, \scriptstyle \dot\sigma_\chi $} E_i  \; = \,  F_j \, E_i  \cr
   T \raise-1pt\hbox{$ \, \scriptstyle \dot\sigma_\chi $}  E_j  \; = \;  T \, E_j \, + \, 2^{-1} \chi\big(T,T_j^+\big) \, E_j \quad ,
     \quad  E_j \raise-1pt\hbox{$ \, \scriptstyle \dot\sigma_\chi $} T  \; = \;  E_j \, T \, + \, 2^{-1} \chi\big(T_j^+,T\big)\big)  \, E_j  \cr
   T \raise-1pt\hbox{$ \, \scriptstyle \dot\sigma_\chi $} F_j  \; = \;  T \, F_j \, + \, 2^{-1} \chi\big(T,T_j^-\big)  \, F_j  \quad ,
     \quad  F_j \raise-1pt\hbox{$ \, \scriptstyle \dot\sigma_\chi $} T  \; = \;  F_j \, T \, + \, 2^{-1} \chi\big(T_j^-,T\big) \, F_j
  }  $$
  $$  \displaylines{
   E_i^{\,(m)_{\sigma_{\chi}}}  \, = \;  {\textstyle \prod_{\ell=1}^{m-1}} \sigma_\chi\Big( e^{\,+\hbar\,\ell\,T_i^+} \! , \, e^{\,+\hbar\,T_i^+} \Big) \, E_i^{\,m}  \; = \;  E_i^{\,m}  \cr
   E_i^{\,m} \raise-1pt\hbox{$ \, \scriptstyle \dot\sigma_\chi $} E_j^{\,n}  \; = \;  \sigma_\chi\Big(
e^{\,+\hbar\,m\,T_i^+} \! , \, e^{\,+\hbar\,n\,T_j^+} \Big) \, E_i^{\,m} \, E_j^{\,n}  \; = \;  e^{\, +\hbar \,m\,n\, 2^{-1} \mathring{\chi}_{ij} } E_i^{\,m} \, E_j^{\,n}  \cr
%
%
   E_i^{\,(m)_{\sigma_{\chi}}} \raise-1pt\hbox{$ \, \scriptstyle \dot\sigma_\chi $}\, E_j \,\raise-1pt\hbox{$ \, \scriptstyle \dot\sigma_\chi $}\, E_k^{\,(n)_{\sigma_\chi}}  = \,
   \Big( {\textstyle \prod_{\ell=1}^{m-1}} \sigma_\chi\Big( e^{\,+\hbar\,\ell\,T_i^+} \! ,
   \, e^{\,+\hbar\,T_i^+} \Big) \!\Big) \Big( {\textstyle \prod_{t=1}^{n-1}} \sigma_\chi\Big( e^{\,+\hbar\,t\,T_k^+} \! , \, e^{\,+\hbar\,T_k^+} \Big) \!\Big) \, \cdot   \hfill  \cr
  \hfill \cdot \; \sigma_{\chi}\Big( e^{\,+\hbar\,m\,T_i^+} \! , \, e^{\,+\hbar\,T_j^+} \Big) \,
  \sigma_\chi\Big( e^{\,+\hbar\,(m\,T_i^+ + T_j^+)} \! , \, e^{\,+\hbar\,n\,T_k^+} \Big) \, E_i^{\,m} \, E_j \, E_k^{\,n}  \cr
   F_i^{\,(m)_{\sigma_\chi}}  \, = \;  {\textstyle \prod_{\ell=1}^{m-1}} \sigma_\chi^{-1}\Big( e^{\,-\hbar\,\ell\,T_i^-} \! , \,
   e^{\,-\hbar\,T_i^-} \Big) \, F_i^{\,m}  \; = \;  F_i^{\,m}  \cr
   F_i^{\,m} \raise-1pt\hbox{$ \, \scriptstyle \dot\sigma_\chi $} F_j^{\,n}  \; = \;  \sigma_\chi^{-1}\Big( e^{\,-\hbar\,m\,T_i^-} \! , \,
   e^{\,-\hbar\,n\,T_j^-} \Big) \, F_i^{\,m} \, F_j^{\,n}  \; = \;  e^{\, -\hbar \,m\,n\, 2^{-1} \mathring{\chi}_{ij}  } F_i^{\,m} \, F_j^{\,n}  \cr
%
   \qquad   F_i^{\,(m)_{\sigma_{\chi}}} \raise-1pt\hbox{$ \, \scriptstyle \dot\sigma_\chi $}\,
   F_j \,\raise-1pt\hbox{$ \, \scriptstyle \dot\sigma_\chi $}\, F_k^{\,(n)_{\sigma_{\chi}}}  \, =   \hfill  \cr
   = \;  \Big( {\textstyle \prod_{\ell=1}^{m-1}} \sigma_{\chi}^{-1}\Big( e^{\, -\hbar\,\ell\,T_i^-} \! , \,
   e^{\,-\hbar\,T_i^-} \Big) \Big) \Big( {\textstyle \prod_{t=1}^{n-1}} \sigma_{\chi}^{-1}\Big( e^{\,-\hbar\,t\,T_k^-} \! , \, e^{\,-\hbar\,T_k^-} \Big) \Big) \, \cdot  \cr
   \hfill \cdot \; \sigma_{\chi}^{-1}\Big( e^{\,-\hbar\,m\,T_i^-} \! , \, e^{\,-\hbar\,T_j^-} \Big) \,
   \sigma_{\chi}^{-1}\Big( e^{\,-\hbar\,(m\,T_i^- + T_j^-)} \! , \, e^{\,-\hbar\,n\,T_k^-} \Big) \, F_i^{\,m} \, F_j \, F_k^{\,n}   \quad
 \cr
   F_i^{\,(m)_{\sigma_{\chi}}} \raise-1pt\hbox{$ \, \scriptstyle \dot\sigma_\chi $}\, F_j \,\raise-1pt\hbox{$ \, \scriptstyle \dot\sigma_\chi $}\, F_k^{\,(n)_{\sigma_{\chi}}}  =
   \Big( {\textstyle \prod_{\ell=1}^{m-1}} \sigma_{\chi}^{-1}\Big( e^{\, -\hbar\,\ell\,T_i^-} \! , \,
   e^{\,-\hbar\,T_i^-} \Big) \!\Big) \Big( {\textstyle \prod_{t=1}^{n-1}} \sigma_{\chi}^{-1}\Big( e^{\,-\hbar\,t\,T_k^-} \! , \, e^{\,-\hbar\,T_k^-} \Big) \!\Big) \cdot   \hfill  \cr
   \hfill \cdot \; \sigma_{\chi}^{-1}\Big( e^{\,-\hbar\,m\,T_i^-} \! , \, e^{\,-\hbar\,T_j^-} \Big) \,
   \sigma_{\chi}^{-1}\Big( e^{\,-\hbar\,(m\,T_i^- + T_j^-)} \! , \, e^{\,-\hbar\,n\,T_k^-} \Big) \, F_i^{\,m} \, F_j \, F_k^{\,n}  }  $$
   \indent   It is also worth remarking that the identity
   $ \; T' \raise-1pt\hbox{$ \, \scriptstyle \dot\sigma_\chi $} T'' = T' \, T'' \, $  (for  $ \, T' , T'' \in \lieh \, $)
   also implies that  $ \; T^{(n)_{\sigma_{\chi}}} = T^n \; $  (for  $ \, T \in \lieh \, $  and  $ \, n \in \NN \, $)
   --- i.e., each ``deformed'' power of any toral element coincides with the corresponding ``undeformed'' power.
   It follows then that the exponential of any toral element with respect to the deformed product
   $ \raise-1pt\hbox{$ \, \scriptstyle \dot\sigma_\chi $} $  is the same as with respect to the old one.
 \vskip5pt
\begin{obs}
   Observe that the whole procedure of  $ 2 $--cocycle  deformation by  $ \sigma_\chi $
   should apply to the scalar extension  $ \, \khp \otimes_\kh \uRPhg \, $,
   since  \textit{a priori\/}  the  $ 2 $--cocycle  $ \sigma_\chi $
   takes values in  $ \khp $  rather than in  $ \kh \, $.
   Nevertheless, the formulas above show that
%
%
  $ \uRPhg $  is actually  \textsl{closed\/}  for the deformed product  ``$ \, \raise-1pt\hbox{$ \, \scriptstyle \dot\sigma_\chi $} $''
  provided by this procedure, hence the deformation eventually ``restricts'' to  $ \uRPhg $
  itself, so that we eventually end up with a well-defined  $ 2 $--cocycle  deformation  $  {\big( \uRPhg \big)}_{\!\sigma_\chi} \! $.
\end{obs}

\end{free text}

   A first, direct consequence of these formulas is the following:

\vskip11pt

\begin{prop}  \label{prop: gen's-FoMpQUEA => def-FoMpQUEA}
 With notations as above, the deformed algebra  $ {\big( \uRPhg \big)}_{\sigma_\chi} \! $
 is still generated by the elements  $ E_i \, $,  $ F_i $  and  $ T $
 (with  $ \, i \in I \, $  and  $ \, T \in \lieh \, $)  of  $ \, \uRPhg \, $.
\end{prop}

\vskip7pt

\begin{free text}  \label{tor_2-coc_deform's_uPhgd}
 {\bf Toral  $ 2 $--cocycle deformations of  $ \uRPhg \, $.}
 Our key result concerns  $ 2 $--cocycle  deformations by means of toral  $ 2 $--cocycles.
 In order to state it, we need some more notation, which we now settle.
                                                                                    \par
   Let  $ \, P := {\big(\, p_{ij} \,\big)}_{i, j \in I} \in M_n\big( \kh \big) \, $
   be a multiparameter matrix of Cartan type with associated Cartan matrix $ A $
   ---  cf.\  Definition \ref{def: realization of P}  ---   and fix a realization
   $ \R = \big(\, \lieh \, , \Pi \, , \Pi^\vee \big) \, $  of  $ P $.
   Fix an antisymmetric  $ \kh $--bilinear map  $ \, \chi \! : \lieh \times \lieh \! \longrightarrow \! \kh \, $
   enjoying  \eqref{eq: condition-chi}   --- that is,  $ \, \chi \in \text{\sl Alt}_{\,\kh}^{\,S}(\lieh) \, $  ---
   we associate with it the matrix
   $ \, \mathring{X} := \! {\Big( \mathring{\chi}_{i{}j} = \chi\big(\,T_i^+,T_j^+\big) \!\Big)}_{i, j \in I} \, $  as above.
   Note that
 $ \; + \chi\big(\, \text{--} \, , T_j^+ \big)  \, = \,  - \chi\big(\, \text{--} \, , T_j^- \big) \; $
 for all  $ \, j \in I \, $,  as direct consequence of  \eqref{eq: condition-chi}.
 Basing on the above, like in  \S \ref{subsec: 2-cocycle-realiz}  we define
  $$
  P_{(\chi)}  := \,  P \, + \, \mathring{X}  \, = \,  {\Big(\, p^{(\chi)}_{i{}j} := \,  p_{ij} + \mathring{\chi}_{i{}j} \Big)}_{\! i, j \in I}  \;\, ,  \!\quad
   \Pi_{(\chi)}  := \,  {\Big\{\, \alpha_i^{(\chi)}  := \,
   \alpha_i \pm \chi\big(\, \text{--} \, , T_i^\pm \big)  \Big\}}_{i \in I}
   $$
\noindent
 Then, still from  \S \ref{subsec: 2-cocycle-realiz}  we know that
 $ P_{(\chi)} $  is a matrix of Cartan type   --- the same of  $ P $  indeed ---
 and  $ \, \R_{(\chi)} \, = \, \big(\, \lieh \, , \Pi_{(\chi)} \, , \Pi^\vee \,\big) \, $  is a realization of it.
 \vskip3pt
   We are now ready for our result on toral  $ 2 $--cocycle  deformations.
   Note in particular that, though toral  $ 2 $--cocycles  have values in  $ \khp \, $,
   \,the deformed multipli\-cation is still well defined within our initial FoMpQUEA, which is defined over  $ \kh \, $.
\end{free text}

\vskip7pt

\begin{theorem}  \label{thm: 2cocdef-uPhgd=new-uPhgd}
 There exists an isomorphism of topological Hopf algebras
  $$  {\big( \uRPhg \big)}_{\sigma_\chi}  \; \cong \;\,  U_{\!P_{(\chi)},\,\hbar}^{\R_{(\chi)}}(\lieg)  $$
 which is the identity on generators.  In short, every toral 2--cocycle deformation of a
 FoMpQUEA is another FoMpQUEA, whose multiparameter  $ \, P_{(\chi)} $  and realization  $ \R_{(\chi)} $
 depend on the original  $ P $  and  $ \R \, $,  as well as on  $ \chi \, $,  as explained in  \S \ref{subsec: 2-cocycle-realiz}.
 \vskip5pt
   Similar statements hold true for the Borel FoMpQUEAs and their deformations by  $ \sigma_\chi \, $,
   namely there exist isomorphisms
   $ \,\; {\big( U_{\!P,\,\hbar}^{\R}(\lieb_\pm) \big)}_{\sigma_\chi} \cong \; U_{\!P_{(\chi)},\,\hbar}^{\R_{(\chi)}}(\lieb_\pm) \;\, $.
\end{theorem}

\pf
 We begin by noting the following key
 \vskip7pt
   {\sl  $ \underline{\text{Fact}} $:  The generators of  $ \uRPhg \, $,
   when thought of as elements of the deformed algebra  $ {\big( \uRPhg \big)}_{\sigma_\chi} \, $,
   obey the defining relations of the (same name) generators of
   $ \, U_{\!P_{(\chi)},\,\hbar}^{\R_{(\chi)}}(\lieg) \, $   --- with respect to the deformed product
   $ \raise-1pt\hbox{$ \, \scriptstyle \dot\sigma_\chi $} \, $.}
 \vskip9pt
   Indeed, most relations follow at once from the formulas in  \S \ref{rmk: formulas_deform-prod}.
   Namely, the identities  $ \; T' \raise-1pt\hbox{$ \, \scriptstyle \dot\sigma_\chi $} T'' = T' \, T'' \, $  imply
   $ \; T' \raise-1pt\hbox{$ \, \scriptstyle \dot\sigma_\chi $} T'' = T'' \raise-1pt\hbox{$ \, \scriptstyle \dot\sigma_\chi $} T' \, $
   --- for all $ \, T' , T'' \in \lieh \, $.  Also, from
 $ \,\; T \raise-1pt\hbox{$ \, \scriptstyle \dot\sigma_\chi $} E_j  \, = \,  T \, E_j \, + \, 2^{-1} \chi\big(T,T_j^+\big) \, E_j \;\, $
and
 $ \,\; E_j \raise-1pt\hbox{$ \, \scriptstyle \dot\sigma_\chi $} T  \, = \,  E_j \, T \, + \, 2^{-1} \chi\big(T_j^+,T\big) \, E_j \;\, $
   --- for all  $ \, T \in \lieh \, $  and  $ \, j \in I \, $  ---   we get (since  $ \chi $  is antisymmetric)
  $$  T \raise-1pt\hbox{$ \, \scriptstyle \dot\sigma_\chi $} E_j \, -  \, E_j \raise-1pt\hbox{$ \, \scriptstyle \dot\sigma_\chi $} T  \,\;
  = \;  \Big(\, \alpha_j(T) + 2^{-1} \big(\, \chi - \chi^{\,\scriptscriptstyle T} \big)\big(\, T \, , T_j^+ \big) \!\Big) \, E_j  \,\; = \;  +\alpha^{(\chi)}_j(T) \, E_j  $$
 \vskip3pt
\noindent
 A similar, straightforward analysis also yields
 $ \; T \raise-1pt\hbox{$ \, \scriptstyle \dot\sigma_\chi $}  F_j \, -
 \, F_j \raise-1pt\hbox{$ \, \scriptstyle \dot\sigma_\chi $} T \, = \, -\alpha^{(\chi)}_j(T) \, F_j \; $.
                                                                   \par
   The identities  $ \; E_i \raise-1pt\hbox{$ \, \scriptstyle \dot\sigma_\chi $} F_j \, = \, E_i \, F_j \; $
   and  $ \; F_j \raise-1pt\hbox{$ \, \scriptstyle \dot\sigma_\chi $} E_i \, = \, F_j \, E_i \; $  in turn imply
  $$  E_i\raise-1pt\hbox{$ \, \scriptstyle \dot\sigma_\chi $} F_j \, - \,
  F_j \raise-1pt\hbox{$ \, \scriptstyle \dot\sigma_\chi $} E_i  \,\; = \;\,  E_i \, F_j \, - \, F_j \, E_i  \,\; = \;\,
  \delta_{i,j} \, {{\; e^{+\hbar \, T_i^+} \, - \, e^{-\hbar \, T_i^-} \;} \over {\;\; e^{+\hbar\,p_{i{}i}/2} \, - \,
  e^{-\hbar\,p_{i{}i}/2} \;\;}}  $$
 Eventually, since  $ \, p_{i{}i} = p_{i{}i}^{(\chi)} \, $  by definition,
 and the exponential of toral elements with respect to  $ \raise-1pt\hbox{$ \, \scriptstyle \dot\sigma_\chi $} $
 is the same as with respect to the old product, we conclude that
  $$  E_i \raise-1pt\hbox{$ \, \scriptstyle \dot\sigma_\chi $} F_j \, -
  \, F_j \raise-1pt\hbox{$ \, \scriptstyle \dot\sigma_\chi $} E_i  \,\; = \;\,
 \delta_{i,j} \, {{\; e_{\sigma_\chi}^{+\hbar \, T_i^+} \, -
 \, e_{\sigma_\chi}^{-\hbar \, T_i^-} \;} \over {\; e^{+\hbar\,p_{i{}i}^{(\chi)}/2} \, - \, e^{-\hbar\,p_{i{}i}^{(\chi)}/2} \;}}  $$
where  $ \, e_{\sigma_\chi}^{\,X} \, $  denotes the exponential of any  $ X $  with respect to  $ \hbox{$ \, \scriptstyle \dot\sigma_\chi $} \, $.
 \vskip7pt
   What is less trivial instead is proving the quantum Serre relations for the deformed product; we do this only for the relation involving the  $ E_i$'s,  leaving the relation involving the  $ F_i $'s  as an exercise for the reader.  To this end, set
  $$  q_{ij}^{(\chi)}  \; := \;  e^{\hbar\, p_{ij}^{(\chi)}}  \; = \;  e^{\hbar\, (p_{ij} + \mathring{\chi}_{ij} )}   \qquad \qquad \text{ for all }  \, i , j \in I  $$
Note that, as  $ \, \big(P_{(\chi)}\big)_{s} = P_s = DA \, $,  \,we have  $ \, q_{ii}^{(\chi)} = q_{ii} \, $  and
$ \, q_i^{(\chi)} = e^{+\hbar\,p_{i{}i}^{(\chi)}/2} = q_i \, $  for all  $ \, i \in I \, $.  Then, for all  $ \, i \neq j \in I \, $
we have to prove that
  $$  \sum\limits_{k=0}^{1-a_{ij}} (-1)^k {\left[ {{1-a_{ij}} \atop k} \right]}_{\!q_i}
  {\big( q_{ij}^{(\chi)} \big)}^{+k/2\,} {\big( q_{ji}^{(\chi)} \big)}^{-k/2} \, E_i^{(1-a_{ij}-k)_{\sigma_{\chi}}} \raise-1pt\hbox{$ \, \scriptstyle \dot\sigma_\chi $}
  E_j \raise-1pt\hbox{$ \, \scriptstyle \dot\sigma_\chi $} E_i^{(k)_{\sigma_{\chi}}}  \; = \;  0
$$
 \vskip7pt
   In order to prove the equality, we analyze the factors in the summands separately.
 \vskip7pt

\noindent
 $ \underline{\text{\sl Claim 1}} \, $:  For all  $ \, i \neq j \in I \, $,  \,we have
  $ \; {\big( q_{ij}^{(\chi)} \big)}^{+k/2\,} {\big( q_{ji}^{(\chi)} \big)}^{-k/2} \, = \,  q_{ij}^{+k/2}q_{ji}^{-k/2}e^{\hbar \, k \, \mathring{\chi}_{ij}} $
 \vskip7pt
   This follows by direct computation.  Next claim instead follows from  \S \ref{rmk: formulas_deform-prod}:
 \vskip7pt
\noindent
 $ \underline{\text{\sl Claim 2}} \, $:  Fix  $ \, i \neq j \in I \, $  and write  $ \, m := 1-a_{ij} \, $.  Then
  $$  E_i^{(m-k)_{\sigma_{\chi}}} \raise-1pt\hbox{$ \, \scriptstyle \dot\sigma_\chi $}
  E_j \raise-1pt\hbox{$ \, \scriptstyle \dot\sigma_\chi $} E_i^{(k)_{\sigma_{\chi}}}  \;
  = \; \sigma_\chi\big( K_i^{m-k} , K_j \big) \, \sigma_\chi\big( K_i^{m-k} K_j \, , K_i^k \big) \, E_i^{m-k} E_j \, E_i^k  $$
 \vskip3pt
   Now we evaluate the value of the toral  $ 2 $--cocycle  using the exponentials.
 \vskip7pt
\noindent
 $ \underline{\text{\sl Claim 3}} \, $:  For all  $ \, i , j \in I \, $  and  $ \, m , k , \ell \in \NN \, $,  \,we have
 \vskip7pt
\begin{enumerate}
  \item[\textit{(a)}]  $ \;\; \sigma_\chi\big( K_i^\ell , K_j \big)  \; = \;  e^{\hbar \, \ell \,
 2^{-1} \mathring{\chi}_{ij} } $
  \item[\textit{(b)}]  $ \;\; \sigma_\chi\big( K_i^{m-k} K_j \, , K_i^k \big)  \; = \;  e^{\hbar \, k \, %
 2^{-1} \mathring{\chi}_{ji}  } $
  \item[\textit{(c)}]  $ \;\; \sigma_\chi\big( K_i^{m-k} , K_j \big) \,
  \sigma_\chi\big( K_i^{m-k} K_j \, , K_i^k \big)  \; = \;  e^{\hbar \, (m-2k)\, 2^{-1} \mathring{\chi}_{ij} } $
\end{enumerate}
 \vskip7pt
\noindent
 All assertions follow by computation using  Lemma \ref{lem: chi_U(K,K)}. Indeed, for   \textit{(a)\/}  we have
  $$
  \sigma_\chi\big( K_i^\ell , K_j \big)  \; = \;  \overline{\chi}_{\scriptscriptstyle U}\big( K_i^\ell , K_j \big)  \; = \;  e^{\hbar \,
 2^{-1} \chi(\ell \, T_i^+ , T_j^+) }  \; = \;  e^{\hbar \, \ell \,
2^{-1} \mathring{\chi}_{ij}  }  $$
 For item  \textit{(b)},  \,Lemma \ref{lem: chi_U(K,K)}  yields
  $$  \displaylines{
   \qquad   \sigma_\chi\big( K_i^{m-k} K_j \, , K_i^k \big)  \; = \;
   \overline{\chi}_{\scriptscriptstyle U}\big( K_i^{m-k} K_j \, , K_i^k \big)  \; = \;
   \overline{\chi}_{\scriptscriptstyle U}\big(\, e^{\hbar \, ((m-k) \, T_i^+ + T_j^+)} , e^{\hbar\, k\, T_i^+} \,\big)  \; =   \hfill   \cr
   \hfill   = \;  e^{\hbar \, 2^{-1} \chi  ((m-k) \, T_i^+ + T_j^+ , \, k \, T_i^+)}  \; = \;  e^{\hbar\, k \,
   2^{-1} \mathring{\chi}_{ji} }   \qquad  }  $$
 Now, putting altogether  \textit{(a)\/}  and  \textit{(b)\/}  we eventually get  \textit{(c)},  because
  $$
  \sigma_\chi\big( K_i^{m-k} , K_j \big) \, \sigma_\chi\big( K_i^{m-k} K_j \, , K_i^k \big)  \; = \;
   e^{\hbar\, (m-k) \, 2^{-1} \mathring{\chi}_{ij} } \,e^{\hbar\, k \, 2^{-1} \mathring{\chi}_{ji}}  \; = \;e^{\hbar\, (m-2k) \, 2^{-1} \mathring{\chi}_{ij}}
   $$
 \vskip1pt
   Finally,  \textsl{Claims 1},  \textsl{2\/}  and  \textsl{3\/}  altogether yield, for  $ \, m := 1-a_{ij} \, $,
  $$  \displaylines{
   \quad   \sum\limits_{k=0}^m {(-1)}^k {\bigg[ {m \atop k} \bigg]}_{\!q_i} {\big( q_{ij}^{(\chi)} \big)}^{+k/2\,}
   {\big( q_{ji}^{(\chi)} \big)}^{-k/2} \, E_i^{(m-k)_{\sigma_{\chi}}} \raise-1pt\hbox{$ \, \scriptstyle \dot\sigma_\chi $}
   E_j \raise-1pt\hbox{$ \, \scriptstyle \dot\sigma_\chi $} E_i^{(k)_{\sigma_{\chi}}}  \; =   \hfill  \cr
   \quad \quad   = \;  \sum\limits_{k=0}^m {(-1)}^k {\bigg[ {m \atop k} \bigg]}_{\!q_i} q_{ij}^{+k/2} q_{ji}^{-k/2} \,
   e^{\hbar \, k \, \mathring{\chi}_{ij}  } \, e^{\hbar \, (m-2k) \, 2^{-1} \mathring{\chi}_{ij} } E_i^{m-k} E_j \, E_i^k  \; =   \hfill  \cr
   \quad \quad \quad   = \;  e^{\hbar\,  m \, 2^{-1} \mathring{\chi}_{ij}  }  \sum\limits_{k=0}^m {(-1)}^k
   {\bigg[ {m \atop k} \bigg]}_{\!q_i} q_{ij}^{+k/2} q_{ji}^{-k/2} \, E_i^{m-k} E_j \, E_i^k  \; = \;  0   \quad  }  $$
 where the last equality follows from the quantum Serre relation.
 \vskip7pt
   Now, the  $ \underline{\text{\sl Fact}} \, $  above implies that there exists a well-defined
   homomorphism of topological Hopf algebras
 $ \,\; \ell : U_{\!P_{(\chi)},\,\hbar}^{\R_{(\chi)}}(\lieg) \relbar\joinrel\relbar\joinrel\longrightarrow
 {\big(\, \uRPhg \big)}_{\sigma_\chi} \; $
 given on generators by
 $ \; \ell(E_i) := E_i \; $,  $ \; \ell(F_i) := F_i \, $,  $ \; \ell\big(T) := T \; $  ($ \, i \in I \, $,  $ \, T \in \lieh \, $)
 --- in short, it is the identity on generators.  Moreover, thanks to
 Proposition \ref{prop: gen's-FoMpQUEA => def-FoMpQUEA}  this is in fact an  \textsl{epi\/}morphism.
 As an application of this result, there exists also an epimorphism of topological Hopf algebras
 $ \,\; \ell' : \uRPhg \relbar\joinrel\relbar\joinrel\relbar\joinrel\relbar\joinrel\longrightarrow
 {\big(\, U_{\!P_{(\chi)},\,\hbar}^{\R_{(\chi)}}(\lieg) \big)}_{\sigma_{-\chi}} \; $
 which again is the identity on generators   --- just replace  $ \chi $  with  $ -\chi \, $  and  $ P $  with  $ P_{(\chi)} \, $.
                                                                           \par
 Mimicking what we did for  $ \uRPhg \, $,  \,we can construct, out of  $ \chi \, $,
 a normalized Hopf  $ 2 $--cocycle  $ {\dot\sigma}_\chi $
 for  $ {\big(\, U_{\!P_{(\chi)},\,\hbar}^{\R_{(\chi)}}(\lieg) \big)}_{\sigma_{-\chi}} \, $;
 \,then we also have a similar  $ 2 $--cocycle  $ \sigma'_\chi $  on
 $ \uRPhg $  defined as the pull-back of  $ {\dot\sigma}_\chi $  via  $ \ell' \, $,  \,and a unique,
 induced Hopf algebra homomorphism
 $ \,\; \ell'_{{\dot{\sigma}}_\chi} : {\big( \uRPhg \big)}_{\sigma'_\chi}
 \relbar\joinrel\relbar\joinrel\relbar\joinrel\relbar\joinrel\longrightarrow
 {\Big( {\Big(\, U_{\!P_{(\chi)},\,\hbar}^{\R_{(\chi)}}(\lieg) \Big)}_{\sigma_{-\chi}} \Big)}_{{\dot{\sigma}}_\chi} \; $
 between deformed Hopf algebras, which is once more the identity on generators.
 Now, tracking the whole construction one sees at once that  $ \, \sigma'_\chi = \sigma_\chi \, $,  so that
 $ \; {\big( \uRPhg \big)}_{\sigma'_\chi} = {\big( \uRPhg \big)}_{\sigma_\chi} \; $,  \,and
 $ \; {\Big( {\Big(\, U_{\!P_{(\chi)},\,\hbar}^{\R_{(\chi)}}(\lieg) \Big)}_{\sigma_{-\chi}} \Big)}_{{\dot{\sigma}}_\chi} \!\! =
 \, U_{\!P_{(\chi)},\,\hbar}^{\R_{(\chi)}}(\lieg) \; $.  But then composition gives two homomorphisms
  $$  \displaylines{
   \ell'_{{\dot{\sigma}}_\chi} \!\circ \ell \, : \, U_{\!P_{(\chi)},\,\hbar}^{\R_{(\chi)}}(\lieg) \relbar\joinrel\relbar\joinrel\longrightarrow
   {\big( \uRPhg \big)}_{\sigma_\chi} \relbar\joinrel\relbar\joinrel\longrightarrow U_{\!P_{(\chi)},\,\hbar}^{\R_{(\chi)}}(\lieg)  \cr
   \ell \circ \ell'_{{\dot{\sigma}}_\chi} : \, {\big( \uRPhg \big)}_{\sigma_\chi} \relbar\joinrel\relbar\joinrel\longrightarrow
   U_{\!P_{(\chi)},\,\hbar}^{\R_{(\chi)}}(\lieg) \relbar\joinrel\relbar\joinrel\longrightarrow {\big( \uRPhg \big)}_{\sigma_\chi}  }  $$
 which (both) are the identity on generators: hence in the end we get
 $ \; \ell'_{{\dot{\sigma}}_\chi} \!\circ \ell \, = \, \textsl{id} \; $
 and  $ \; \ell \circ \ell'_{{\dot{\sigma}}_\chi} \! = \, \textsl{id} \; $,  \;thus in particular
 $ \, \ell \, $  is an isomorphism, \,q.e.d.
\epf

\vskip9pt

\begin{rmk}  \label{rmk: about cocycle deform.'s}
 With notation of  Theorem \ref{thm: 2cocdef-uPhgd=new-uPhgd}  above, we have
 $ \; P_{(\chi)} - P \, = \, \varLambda \; $
 for some  {\sl antisymmetric\/}  matrix  $ \, \varLambda \in \lieso_n\big(\kh\big) \, $.
 Conversely, under mild assumptions on  $ P $,
 this result can be ``reversed'' as it is shown below.
\end{rmk}

\vskip7pt

\begin{theorem}  \label{thm: double FoMpQUEAs_P-P'_mutual-deform.s}
 Let  $ \, P, P' \in M_n\big(\kh\big) \, $  be two matrices of Cartan type with the same associated Cartan matrix  $ \, A \, $.
 \vskip5pt
   (a)\,  Let  $ \, \R $  be a  \textsl{split}  realization of  $ \, P $  and  $ \, \uRPhg $  be the associated FoMpQUEA.
   Then there exists a \textsl{split}  realization  $ \check{\R}' $  of  $ \, P' $, a matrix
   $ \, \mathring{X} = {\big(\, \mathring{\chi}_{i{}j} \big)}_{i, j \in I} \in \lieso_n\big(\kh\big)\, $  and a toral\/  $ 2 $--cocycle  $ \, \sigma_\chi $ such that
  $$  U_{\!P'\!,\,\hbar}^{\,\R'}(\lieg)  \,\; \cong \;  {\big(\, U_{\!P,\,\hbar}^{\,\R}(\lieg) \big)}_{\sigma_\chi}  $$
   \indent   In a nutshell, if  $ \, P'_s = P_s \, $  then from any  \textsl{split}  FoMpQUEA over  $ P $
   we can obtain a split FoMpQUEA (of the same rank) over  $ P' $  by a toral 2--cocycle deformation.
 \vskip5pt
   (b)\,  Let  $ \, \R $  be a  \textsl{split minimal}  realization of  $ P $.
   Then the FoMpQUEA  $ \uRPhg $  is isomorphic to a toral 2--cocycle deformation of the Drinfeld's standard double QUEA, that is
 there exists some bilinear map  $ \; \chi \in \text{\sl Alt}_{\,\kh}^{\,S}(\lieh) \; $  such that
  $$  \uRPhg  \,\; \cong \;  {\big(\, U_{DA,\hbar}(\lieg) \big)}_{\sigma_\chi}  $$
 \vskip3pt
   (c)\,  Similar, parallel statements hold true for the Borel FoMpQUEAs.
\end{theorem}

\pf
   \textit{(a)}\,  By  Proposition \ref{prop: mutual-2-cocycle-def}\textit{(a)},  there exists
   $ \; \chi \in \text{\sl Alt}^S_{\,\kh}( \lieh ) \; $  such that
   $ \, P' = P_{(\chi)} \, $  and the triple   --- constructed as in  \S \ref{subsec: 2-cocycle-realiz}  ---
 $ \, \R' \, := \, \R_{(\chi)} \, = \, \big(\, \lieh \, , \Pi_{(\chi)} \, , \Pi^\vee \,\big) \, $
 is a split realization of  $ \, P' = P_{(\chi)} \, $.  Then  $ \; U_{\!P'\!,\hbar}^{\,\R'}(\lieg) \, \cong \, {\big(\, \uRPhg \big)}_{\sigma_\chi} \; $
 by  Theorem \ref{thm: 2cocdef-uPhgd=new-uPhgd}.
 \vskip5pt
   \textit{(b)}\,  Drinfeld's  $ U_{DA,\hbar}(\lieg) $  is   --- in our language ---
   nothing but the FoMpQUEA built upon a split minimal realization
   $ \; \R_{st} \,= \, \big(\, \lieh \, , \Pi_{st} \, , \Pi_{st}^\vee \,\big) \; $  of  $ DA \, $,  \,for which we write
   $ \, \Pi_{st}^\vee\, = \, {\big\{\, T_i^\pm \,\big\}}_{i \in I} \, $  and  $ \, \Pi_{st} = {\big\{\, \alpha_i^{(st)} \,\big\}}_{i \in I} \, $.
   From  Proposition \ref{prop: mutual-2-cocycle-def}\textit{(b)\/}  we have a suitable form  $ \, \chi \in  \text{\sl Alt}_{\,\kh}^{\,S}(\lieh) \; $
   such that the realization  $ \big(\R_{st}\big)_{(\chi)} $  obtained as toral
 $ 2 $--cocycle  deformation of  $ \R_{st} $  through  $ \chi $  coincides with  $ \R \, $.
 Then, by  Theorem \ref{thm: 2cocdef-uPhgd=new-uPhgd}  we get  $ \; \uRPhg \, \cong \, {\big(\, U_{DA,\hbar}(\liegd) \big)}_{\sigma_\chi} \, $  as desired.
\epf

\bigskip

\section{Specialization and quantization: FoMpQUEAs vs.\ MpLbA's}
 \label{sec: special-&-quantiz}
{\ }
 \vskip1pt
   This section dwells upon the interplay of specialization   --- applied to quantum objects as our FoMpQUEAs ---
   and, conversely, of quantization   --- performed onto such semiclassical objects as our MpLbA's.
                                                                \par
   First of all, we shall see that the specialization of a FoMpQUEA yields a suitable MpLbA;
   conversely, any MpLbA has at least one quantization,
   in the form of a well defined FoMpQUEA.  Then, we shall investigate the interaction between the process of specialization
   --- at  $ \, \hbar = 0 \, $  ---   of any FoMpQUEA and the process of deformation
   --- either by (toral) twist or by (toral)  $ 2 $--cocycle  ---
   of the same FoMpQUEA or of the MpLbA which is its semiclassical limit.
   In particular we will find out that, in a suitable, natural sense,  \textsl{the two processes commute with each other}.

\medskip

\subsection{Deformation vs.\ specialization for FoMpQUEAs and MpLbA's}
\label{subsec: def-vs.-spec-x-FoMpQUEAs&MpLbA's} {\ }

\medskip

   Recall that a  \textit{deformation algebra\/}  is a topological, unital, associative  $ \kh $--algebra  $ A $
   which is topologically free as a  $ \kh $--module.  Conversely, a deformation of a (unital, associative)
   $ \Bbbk $--algebra  $ A_0 $  is by definition a deformation algebra  $ A $  such that  $ \; A_0 \, \cong \, A \big/ \hbar \, A \; $.
   The same criteria apply to the notion of ``deformation Hopf algebra'', just replacing
   ``topological, unital, associative  $ \kh $--algebra''  with ``topological Hopf  $ \kh $--algebra''.
   Following Drinfeld, we say that a deformation Hopf algebra  $ H $  is a
   \textit{quantized universal enveloping algebra\/}  (or  \textsl{QUEA\/}  in short) if
   $ \; H \big/ \hbar \, H \, \cong \, U(\lieg) \; $  for some Lie algebra  $ \lieg \, $.
   In particular, the Lie bracket in  $ \lieg $  comes from the multiplication in  $ \; U(\lieg) \, \cong \, H \big/ \hbar H \; $.
   Moreover, this  $ \lieg $  inherits a Lie coalgebra structure from the QUEA,
   making it into a Lie \textsl{bialgebra},  thanks to the following result:

\vskip11pt

\begin{theorem}  (cf.\ \cite[Proposition 6.2.7]{CP}, \cite[Theorem 9.1]{ES})\label{thm:Lie-coalgebra-limit}
 Let  $ H $  be a quantized universal enveloping algebra with  $ \; H \big/ \hbar \, H \, \cong \, U(\lieg) \; $.
 Then the Lie algebra  $ \lieg $  is naturally equipped with a Lie bialgebra structure, whose Lie cobracket is defined by
\begin{equation}  \label{def: cocorchete-semiclasico}
 \delta(\text{\rm x})  \, := \,  \frac{\,\Delta(x) - \Delta^{\op}(x)\,}{\,\hbar\,}   \qquad  \big(\hskip-10pt \mod \hbar \,\big)
\end{equation}
where  $ \, x \in H \, $  is any lifting of  $ \; \text{\rm x} \in \lieg \subseteq U(\lieg) \, \cong \, H \big/ \hbar H \, $.   \qed
\end{theorem}

\vskip11pt

\begin{definition}  \cite{CP,ES}
 The semiclassical limit of a quantized universal enveloping algebra  $ H $ is the Lie bialgebra
%
  $ \, \big(\, \lieg \, , \, [-,-] \, , \, \delta \,\big) \, $  where  $ \lieg $  is the Lie algebra such that  $ \, H \big/ \hbar \, H \cong U(\lieg) \, $
 and  $ \delta $  is defined as above.  Conversely, we say that  $ H $  is a quantization of the Lie bialgebra
 $ \big(\, \lieg \, , \, [-,-] \, , \, \delta \,\big) \, $.   \hfill  $ \diamondsuit $
\end{definition}

\vskip9pt

\begin{free text}
 {\bf Formal MpQUEAs vs.\ MpLbA's.}
 In this section we finally compare our FoMpQUEAs with our MpLbA's.
 Mainly, we show that the FoMpQUEAs are indeed quantized universal enveloping algebras;
 in particular, we prove that their specialization at  $ \, \hbar = 0 \, $
 is a universal enveloping algebra of a MpLbA as those in  \S \ref{sec: Mp Lie bialgebras}.
 Thus the specialization of each FoMpQUEA is a MpLbA; conversely, any FoMpQUEA is the quantization of some MpLbA.
 The other way round is true as well:
 every MpLbA admits a FoMpQUEA as its quantization.
                                                                          \par
   Second, we describe the interplay between the process of specializing/quantizing
   (switching between FoMpQUEAs and MpLbA's) and the process of deforming within either family of FoMpQUEAs or MpLbA's, separately
   --- by twist or by  $ 2 $--cocycle.

\vskip5pt

   We fix a matrix  $ \, P := {\big(\, p_{i,j} \big)}_{i, j \in I} \in M_n\big( \kh \big) \, $
   of Cartan type with associated Cartan matrix  $ A \, $,  and a realization
   $ \, \R \, := \, \big(\, \lieh \, , \Pi \, , \Pi^\vee \,\big) \, $  of  $ P $.
   Then we have the associated (topological) Hopf algebra  $ \uRPhg \, $,  as in  Definition \ref{def: realization of P}  and  \S \ref{sec: form-MpQUEAs}.
%
%
 Similarly, we also have the MpLbA  $ \lieg^{\bar{\Rpicc}}_{\bar{\Ppicc}} $  introduced in  \S \ref{MpLieBialgebras-double},  where
%
%
 we use (loose) obvious notation such as  $ \, \, \bar{\R} := \R \; (\,\text{mod} \; \hbar \,) \, $  and  $ \, \bar{P} := P \; (\,\text{mod} \; \hbar \,) \, $.

\vskip5pt

   Our first result points out that, roughly speaking, FoMpQUEAs and MpLbA's are in bijection
   through the specialization/quantization process, as one might expect:
%
%
\end{free text}

\vskip9pt

\begin{theorem}  \label{thm: semicl-limit FoMpQUEA}
 With assumptions as above,  $ \uRPhg $  is a  \textsl{quantized universal enveloping algebra}
 in the sense of  \S \ref{subsec: def-vs.-spec-x-FoMpQUEAs&MpLbA's}, whose semiclassical limit is  $ U\big(\lieg^{\bar{\Rpicc}}_{\bar{\Ppicc}}\big) \, $.
                                                 \par
   In short, for each pair  $ \, (P,\R) \, $  as above   ---  $ \R $  being a realization of  $ P $  ---
   and for the FoMpQUEA  $ \uRPhg $  and the MpLbA  $ \lieg^{\bar{\Rpicc}}_{\bar{\Ppicc}} $
   associated with  $ \, (P,\R) \, $,  we have:  \textsl{$ \lieg^{\bar{\Rpicc}}_{\bar{\Ppicc}} $
   is the specialization of  $ \, \uRPhg \, $,  or   --- equivalently ---   $ \uRPhg $  is a quantization of  $ \lieg^{\bar{\Rpicc}}_{\bar{\Ppicc}} \, $}.
\end{theorem}

\pf
   First of all, we note that  $ \, H := \uRPhg \, $  is topologically free.
   We prove that by reducing the problem to the case of Drinfeld's standard double QUEA.
   Namely, we begin assuming that  $ \uRPhg $  is  \textsl{split minimal},  i.e.\ such is the realization  $ \R \, $.
   Then by  Theorem \ref{thm: double FoMpQUEAs_P-P'_mutual-deform.s}\textit{(c)\/}  we have
  $ \; \uRPhg  \, \cong \, {\big(\, U_{D\!A,\hbar}(\lieg) \big)}_{\!\sigma_\chi} \; $
  as topological Hopf  $ \kh $--algebras,  where  $ \, {\big(\, U_{D\!A,\hbar}(\lieg) \big)}_{\!\sigma_\chi} $
  is a suitable  $ 2 $--cocycle  deformation of Drinfeld's standard double QUEA  $ \, U_{D\!A,\hbar}(\lieg) \, $.
  As the latter is known to be a topologically free  $ \kh $--module, and  $ 2 $--cocycle  deformation does not affect the
  $ \kh $--module  structure, we conclude that  $ \uRPhg $  is topologically free as well, q.e.d.
                                                                         \par
   Now assume that  $ \uRPhg $  is just split (possibly not minimal): then, by definition,
   $ \, \Pi^\vee := {\big\{\, T_i^+ , T_i^- \,\big\}}_{i \in I} \, $  can be completed to a  $ \kh $--basis  of  $ \lieh \, $,
   hence  $ \, \lieh' := \textsl{Span}_{\kh}\big( \Pi^\vee \big) \, $  has a direct sum complement  $ \lieh'' $
   so that  $ \, \lieh = \lieh' \oplus \lieh'' \, $;  \,therefore  $ \, U_\hbar(\lieh) = U_\hbar\big(\lieh'\big) \,\widehat{\otimes}\, U_\hbar\big(\lieh''\big) \, $  as algebras.
   Furthermore, the realization  $ \R $  clearly ``restricts'' to another realization  $ \R' $  of  $ P $
   whose Cartan (sub)algebra is  $ \lieh' \, $,  which in addition is split  \textsl{minimal\/}:
   then as  $ \, U_\hbar(\lieh) = U_\hbar\big(\lieh'\big) \,\widehat{\otimes}\, U_\hbar\big(\lieh''\big) \, $
 one also gets easily
\begin{equation}  \label{eq: tens-fact_UR-UR'Uh''}
  \uRPhg \,\; \cong \;\, U^{\,\R'}_{\!P,\hbar}(\lieg) \;\widehat{\otimes}\; U_\hbar\big(\lieh''\big)   \qquad  \textsl{as\/  $ \kh $--modules}
\end{equation}
 (by construction).
 Indeed, by definition we have a natural monomorphism of realizations
 $ \, \R' \lhook\joinrel\longrightarrow \R \, $  induced by the monomorphism
 $ \, \lieh' \lhook\joinrel\longrightarrow \lieh \, $  of Cartan (sub)algebras;
 by  Propositions \ref{prop: functor_R->uRPhg}  and  \ref{prop: central-Hopf-extens_FoMpQUEAs},
 this induces a monomorphism  $ \, U^{\,\R'}_{\!P,\hbar}(\lieg) \lhook\joinrel\relbar\joinrel\longrightarrow \uRPhg \, $
 between FoMpQUEAs   --- the image of the latter is the (complete)  $ \kh $--subalgebra  generated by the  $ E_i $'s,
 the  $ T_i^\pm $'s  and the  $ F_i $'s  ($ \, i \in I \, $).
 Now, applying twice  Theorem \ref{thm: triang-decomp.'s}
 --- yielding triangular decompositions for  $ \uRPhg $  and  $ U^{\,\R'}_{\!P,\hbar}(\lieg) $  ---   we get
  $$  \displaylines{
   \uRPhg  \; \cong \;  \uRPhnm \,\widehat{\otimes}\, U_\hbar(\lieh) \,\widehat{\otimes}\, \uRPhnp  \; \cong   \hfill  \cr
   \cong \;  \uRPhnm \,\widehat{\otimes}\, U_\hbar\big(\lieh'\big) \,\widehat{\otimes}\, U_\hbar\big(\lieh''\big) \,\widehat{\otimes}\, \uRPhnp  \; \cong   \cr
   \hfill   \cong \;  \uRPhnm \,\widehat{\otimes}\, U_\hbar\big(\lieh'\big) \,\widehat{\otimes}\, \uRPhnp \,\widehat{\otimes}\,
   U_\hbar\big(\lieh''\big)  \; \cong \;  U^{\,\R'}_{\!P,\hbar}(\lieg) \,\widehat{\otimes}\, U_\hbar\big(\lieh''\big)  }  $$
 --- where we applied also  $ \, U_\hbar\big(\lieh''\big) \,\widehat{\otimes}\, \uRPhnp \,\cong\, \uRPhnp \,\widehat{\otimes}\, U_\hbar\big(\lieh''\big) \, $,
 which is clear,  and  $ \, U_\hbar(\lieh) = U_\hbar\big(\lieh'\big) \,\widehat{\otimes}\, U_\hbar\big(\lieh''\big) \, $
 ---   which proves our claim.
                                                            \par
Therefore, as  $ U^{\,\R'}_{\!P,\hbar}(\lieg) $  is topologically free by the previous analysis (as  $ \R' $  is split minimal) and
$ U_\hbar\big(\lieh''\big) $  is also topologically free by construction, from  \eqref{eq: tens-fact_UR-UR'Uh''}
we deduce the same for  $ \uRPhg $  as well.
                                                                  \par
   Finally, let's cope with the general case.  By  Lemma \ref{lemma: split-lifting}  there exists an epimorphism
   $ \; \underline{\pi} : \dot{\R} \relbar\joinrel\relbar\joinrel\twoheadrightarrow \R \, $,  \,where
   $ \; \dot{\R} := \big(\, \dot{\lieh} \, , \dot{\Pi} \, , {\dot{\Pi}}^\vee \,\big) \, $  is a  \textsl{split\/} realization of  $ P \, $:
   \,by  Proposition \ref{prop: central-Hopf-extens_FoMpQUEAs},
   this induces an epimorphism of Hopf algebras (though, for us, it is enough to be one of algebras, indeed)
 $ \; U^{\,\dot{\R}}_{\!P,\hbar}(\lieg) \,{\buildrel U_{\underline{\pi}} \over {\relbar\joinrel\relbar\joinrel\relbar\joinrel\twoheadrightarrow}}\, \uRPhg \; $
 whose kernel is generated by  $ \, {U_\hbar(\liek)}^+ \, $  with  $ \, \liek := \textsl{Ker}(\pi) \, $,
 \,where  $ \, \pi : \dot{\lieh} \,\relbar\joinrel\relbar\joinrel\twoheadrightarrow\, \lieh \, $
 is the epimorphism of Cartan (sub)algebras associated with  $ \underline{\pi} \, $.
 As  $ \lieh $  is free of finite rank, we have  $ \, \dot{\lieh} = \liek \oplus \lieh' \cong \liek \oplus \lieh \, $
 for some free submodule  $ \, \lieh' \cong \lieh \, $  inside  $ \dot{\lieh} \, $;
 \,therefore  $ \, U_\hbar\big(\dot{\lieh}\big) = U_\hbar(\liek) \,\widehat{\otimes}\, U_\hbar(\lieh) \cong U_\hbar(\liek) \otimes U_\hbar(\lieh) \, $
 as algebras,  whence, as in  \eqref{eq: tens-fact_UR-UR'Uh''},  one gets
 $ U^{\,\dot{\R}}_{\!P,\hbar}(\lieg) \; \cong \; U_\hbar(\liek) \;\widehat{\otimes}\; \uRPhg \;\, $
 as\/  $ \kh $--modules.
 As  $ U^{\,\dot{\R}}_{\!P,\hbar}(\lieg) $  is topologically free by the previous analysis and
 $ U_\hbar(\liek) $  is too, we deduce the same for  $ \uRPhg $  as well.
 \vskip5pt
   Second, we must
prove that
 $ \; \uRPhg \big/ \hbar \, \uRPhg \, $,  \,as a co-Poisson Hopf algebra, is isomorphic to  $ U\big( \lieg^{\bar{\Rpicc}}_{\bar{\Ppicc}} \big) \, $.
 Indeed,
 from the presentation of  $ \uRPhg $  we get that  $ \; \uRPhg \big/ \hbar \, \uRPhg \, $
 is generated by the cosets (modulo  $ \hbar \, \uRPhg \, $)  of the  $ F_i $'s,  $ T $'s  and  $ E_i $'s  ($ \, i \in I \, $,  $ \, T \in \lieh \, $);  moreover,
these cosets
   $ \; \overline{X} := X \!\! \mod \hbar\,\uRPhg \; $  obey all relations induced modulo  $ \hbar $  by
   the defining relations among the original generators  $ X $  of  $ \uRPhg \, $.
   On the other hand, by construction the Lie bialgebra  $ \lieg^{\bar{\Rpicc}}_{\bar{\Ppicc}} $
   is endowed with a built-in presentation (as a Lie algebra) by generators
   --- the  $ F_i $'s,  $ T $'s  and  $ E_i $'s
 ---   and relations, and explicit formulas for the value of the Lie cobracket  $ \delta $  on the given generators.
 From this, a presentation of  $ U\big( \lieg^{\bar{\Rpicc}}_{\bar{\Ppicc}} \big) $
is obtained
 in the obvious way, where the generators are again the  $ F_i $'s,  $ T $'s  and  $ E_i $'s  as before,
 as well as explicit formulas for the value of the Poisson cobracket  $ \delta $  on each one of those generators.
                                                            \par
   Comparing, the presentation of  $ U\big( \lieg^{\bar{\Rpicc}}_{\bar{\Ppicc}} \big) $
   with that of $ \; \uRPhg \big/ \hbar \, \uRPhg \; $  we find that all the given relations among
   generators of the latter algebra do correspond to identical relations among the corresponding
   generators in the former: namely, mapping  $ X $  to  $ \overline{X} $
   --- for all  $ \, X \in \big\{\, E_i \, , F_i \,\big|\, i \in I \,\big\} \cup \lieh \, $  ---
   turns every given relation among the  $ X $'s  into a same-look relation among the  $ \overline{X} $'s.
   In the end, this means that a well-defined epimorphism of Hopf algebras
\begin{equation}  \label{eq: iso_Ugbar-U0g}
     \phi \, : \, U\big( \lieg^{\bar{\Rpicc}}_{\bar{\Ppicc}} \big) \,\relbar\joinrel\relbar\joinrel\twoheadrightarrow\, \uRPhg \Big/ \hbar \, \uRPhg  \;\; ,  \qquad
  E_i \,\mapsto\, \overline{E_i} \, ,  \;  T \,\mapsto\, \overline{T} \, ,  \;  F_i \,\mapsto\, \overline{F_i}
\end{equation}
 ($ \, i \in I \, $,  $ \, T \in \lieh \, $)  exists; moreover, comparing the formulas on both sides for the
 co-Poisson bracket on generators we see that this is also a  \textsl{co-Poisson\/}  Hopf epimorphism.
                                                                         \par
   On the other hand, we can make  $ U\big( \lieg^{\bar{\Rpicc}}_{\bar{\Ppicc}} \big) $
   into a  $ \kh $--algebra  by scalar restriction   --- via  $ \; \kh \relbar\joinrel\relbar\joinrel\twoheadrightarrow \kh \big/ \hbar \, \kh \, \cong \, \Bbbk \; $.
   Then the same remark about relations implies that there exists a well-defined  $ \kh $--algebra  epimorphism
\begin{equation*}  \label{eq: iso_Ug-Ugbar}
   \psi\,{}' \, : \, \uRPhg \,\relbar\joinrel\relbar\joinrel\twoheadrightarrow\,
 U\big( \lieg^{\bar{\Rpicc}}_{\bar{\Ppicc}} \big)  \;\; ,  \qquad  E_i \,\mapsto\, E_i \, ,  \;
 T \,\mapsto\, T \, ,  \;  F_i \,\mapsto\, F_i   \qquad  (\, i \in I \, , \, T \in \lieh \,)
\end{equation*}
 whose kernel contains  $ \, \hbar \, \uRPhg \, $;  \,so a  $ \Bbbk $--algebra  epimorphism
\begin{equation}  \label{eq: iso_U0g-Ugbar}
   \psi \, : \, \uRPhg \Big/ \hbar\,\uRPhg \,\relbar\joinrel\relbar\joinrel\relbar\joinrel\twoheadrightarrow\, U\big( \lieg^{\bar{\Rpicc}}_{\bar{\Ppicc}} \big)  \;\; ,  \qquad
   \overline{E_i} \,\mapsto\, E_i \, ,  \; \overline{T} \,\mapsto\, T \, ,  \; \overline{F_i} \,\mapsto\, F_i
\end{equation}
 ($ \, i \in I \, $,  $ \, T \in \lieh \, $)
 is induced too.  Comparing  \eqref{eq: iso_Ugbar-U0g}  and  \eqref{eq: iso_U0g-Ugbar}  shows that
 $ \phi $  and  $ \psi $  are inverse to each other, hence  $ \psi $  is a  \textsl{Hopf\/}  morphism too and we are done.
\epf

\medskip

\subsection{Blending specialization and deformation}
  \label{subsec: blending-spec-deform}  {\ }
 \vskip7pt
In this section we compare the process of deformation at the quantum level or at the semiclassical level.
  The outcome is very neat, and can be put in a nutshell as follows:
  \textit{deformation (by twist or by 2--cocycle)  \textsl{commutes}  with specialization}.

\vskip11pt

\begin{free text}
 \textbf{Blending specialization and twist deformation.}
 Once more, we fix again
  $ \, P := {\big(\, p_{i,j} \big)}_{i, j \in I} \in M_n\big( \kh \big) \, $  of Cartan type, a realization
  $ \, \R := \big(\, \lieh \, , \Pi \, , \Pi^\vee \,\big) \, $  of it, and the associated FoMpQUEA  $ \uRPhg $
and MpLbA  $ \lieg^{\bar{\Rpicc}}_{\bar{\Ppicc}} \, $.
 As  $ \lieh $  is a free  $ \kh $--module  of finite rank  $ \, t \, $,  we fix a  $ \kh $--basis  $ \, {\big\{ H_g \big\}}_{g \in \G} \, $,  with  $ \, |\G| = \rk(\lieh) = t \, $.
 \vskip5pt
   Pick an antisymmetric matrix  $ \; \Psi = \big( \psi_{gk} \big)_{g, k \in \G} \in \lieso_t\big(\kh\big) \; $.  Out of it, we define
  $$  \displaylines{
   j_{\bar{\Psipicc}}  \; := \;  {\textstyle \sum_{g,k=1}^t} \, \overline{\psi_{gk}} \; \overline{H_g} \otimes \overline{H_k}  \,\; \in \;\,
  \overline{\lieh} \otimes \overline{\lieh}   \qquad \qquad \qquad  \text{as in  \eqref{eq: def_twist_Lie-bialg}}  \cr
   \F_\Psipicc  \; := \;  \exp\Big( \hbar \; 2^{-1} \, {\textstyle \sum_{g,k=1}^t} \psi_{gk} \,
   H_g \otimes H_k \Big)  \,\; \in \;\,  \uRPhg \,\widehat{\otimes}\, \uRPhg   \qquad  \text{as in  \eqref{eq: Resh-twist_F-uPhgd}}  }  $$
 with  $ j_{\bar{\Psipicc}} $  being a (toral)  \textsl{twist\/}  for the  \textsl{Lie bialgebra\/}
 $ \lieg^{\bar{\Rpicc}}_{\bar{\Ppicc}} $  and  $ \F_\Psipicc $  is a (toral)  \textsl{twist\/}
 for the  \textsl{Hopf algebra\/}  $ \uRPhg \, $.  Then we consider the deformation
 $ {\big( \lieg^{\bar{\Rpicc}}_{\bar{\Ppicc}} \big)}^{j_\Psipicc} $  of  $ \lieg^{\bar{\Rpicc}}_{\bar{\Ppicc}} $
 by the (Lie) twist  $ j_{\bar{\Psipicc}} $  and the deformation  $ {\big(\,\uRPhg\big)}^{\F_\Psipicc} $  of
 $ \uRPhg $  by the (Hopf) twist  $ \F_\Psipicc \, $.
 \vskip5pt
   Again out of  $ \Psi \, $,  we define also the matrix  $ P_\Psipicc $  and its realization
   $ \, \R_\Psipicc \! := \big(\, \lieh \, , \Pi \, , \Pi^\vee_\Psipicc \big) $,  \,as in
   Proposition \ref{prop: twist-realizations}.  Then we have the FoMpQUEA
   $ U_{\!P_\Psipicc,\hbar}^{\,\R_\Psipicc}(\lieg) $  and the MpLbA
   $ \lieg_{\bar{\Ppicc}_\Psipicc}^{\bar{\Rpicc}_\Psipicc} \, $,
   again linked with each other by a quantization/specialization relationship.
\end{free text}

\vskip3pt

   We can now state our result, which in particular claims (roughly speaking) that  \textsl{``deformation by twist commutes with specialization''}.

\vskip11pt

\begin{theorem}  \label{thm: specializ twisted FoMpQUEA}
 With assumptions as above, we have that  $ {\big(\, \uRPhg \big)}^{\F_\Psipicc} $
 is a quanti\-zed universal enveloping algebra, whose semiclassical limit is
 $ \, U\big( {\big( \lieg^{\bar{\Rpicc}}_{\bar{\Ppicc}} \big)}^{j_{\bar{\Psipicc}}} \big) \, $.
 More precisely, we have  $ \,\; {\big(\, \uRPhg \big)}^{\F_\Psipicc} \cong \, U_{\!P_{\,\Psipicc},\hbar}^{\,\R_\Psipicc}(\lieg) \;\, $
 and  $ \,\; {\big( \lieg^{\bar{\Rpicc}}_{\bar{\Ppicc}} \big)}^{j_{\bar{\Psipicc}}} \cong \, \lieg_{\bar{\Ppicc}_\Psipicc}^{\bar{\Rpicc}_\Psipicc} \;\, $.
\end{theorem}

\pf
 The claim follows, as direct application, from the isomorphisms
  $$  {\big(\, \uRPhg \big)}^{\F_\Psipicc}  \,\;{\buildrel {\ref{thm: twist-uRPhg=new-uRPhg}} \over \cong}\;\,
  U_{\!P_{\,\Psipicc},\hbar}^{\,\R_\Psipicc}(\lieg)  \;\, ,  \quad
   U_{\!P_{\,\Psipicc},\hbar}^{\,\R_\Psipicc}(\lieg) \Big/ \hbar\,U_{\!P_{\,\Psipicc},\hbar}^{\,\R_\Psipicc}(\lieg)
   \,\;{\buildrel {\ref{thm: semicl-limit FoMpQUEA}} \over \cong}\;\,  U_\hbar\big(\lieg_{\bar{\Ppicc}_\Psipicc}^{\bar{\Rpicc}_\Psipicc}\big)  \;\, ,  \quad
   \lieg_{\bar{\Ppicc}_\Psipicc}^{\bar{\Rpicc}_\Psipicc}  \;{\buildrel {\ref{thm: twist-liegRP=new-liegR'P'}}
   \over \cong}\,  {\big( \lieg^{\bar{\Rpicc}}_{\bar{\Ppicc}} \big)}^{j_{\bar{\Psipicc}}}  $$
 which come from  Theorem \ref{thm: twist-uRPhg=new-uRPhg},
 Theorem \ref{thm: semicl-limit FoMpQUEA}  and  Theorem \ref{thm: twist-liegRP=new-liegR'P'}.
\epf

\vskip9pt

\begin{free text}
 \textbf{Blending specialization and 2--cocycle deformation.}
 Now we analyze what happens when combining  \textsl{deformations by 2--cocycle}
 --- for both FoMpQUEAs and MpLbA's ---   with the specialization process (from the former to the latter ones).
                                                                   \par
   We start with  $ \, P := {\big(\, p_{i,j} \big)}_{i, j \in I} \in M_n\big( \kh \big) \, $  of Cartan type, a realization
   $ \, \R := \big(\, \lieh \, , \Pi \, , \Pi^\vee \,\big) \, $  of  $ P $,  and a fixed  $ \kh $--basis  $ \, {\big\{ H_g \big\}}_{g \in \G} \, $
of  $ \lieh \, $,  with  $ \, |\G| = \rk(\lieh) = t \, $.  Then we have also  $ \uRPhg $  and  $ \lieg^{\bar{\Rpicc}}_{\bar{\Ppicc}} \, $,
\,interlocked via quantization/specialization.
 \vskip5pt
   Like in  \S \ref{2-cocycle-U_h(h)},  let  $ \; \chi : \lieh \times \lieh \relbar\joinrel\longrightarrow \kh \; $
   be an antisymmetric  $ \kh $--bilinear  map which obeys  \eqref{eq: condition-eta}.
   Taking everything modulo  $ \hbar \, $,  this  $ \chi $  defines a similar antisymmetric,  $ \Bbbk $--bilinear  map
   $ \; \gamma := \big(\, \chi \!\mod \hbar \,\big) : \lieh_0 \times \lieh_0 \relbar\joinrel\longrightarrow \Bbbk \; $
   --- where  $ \, \lieh_0 := \lieh \Big/ \hbar \, \lieh \, = \, \overline{\lieh} \, $  ---   which obeys  \eqref{eq: condition-eta}
   again, up to replacing  ``$ \chi $''  with  ``$ \gamma $''.  Following  \S \ref{toral-cocycles-uPhgd},  we construct out of
   $ \chi $  a  $ \khp $--valued  toral  $ 2 $--cocycle  $ \sigma_\chi : \uRPhg \otimes \uRPhg \relbar\joinrel\relbar\joinrel\longrightarrow \khp \, $,
   \,and then out of this the  $ 2 $--cocycle  deformed Hopf algebra  $ \, {\big( \uRPhg \big)}_{\sigma_\chi} \, $.  Similarly, out of  $ \gamma $
   we construct, as in  \S \ref{subsec: tor-2cocyc def's_mp-Lie_bialg's}  (but replacing  ``$ \chi $''  with  ``$ \gamma $'')  a toral  $ 2 $--cocycle
   $ \gamma_\lieg $  for the Lie bialgebra  $ \lieg^{\bar{\Rpicc}}_{\bar{\Ppicc}} \, $,  \,and out of it the  $ 2 $--cocycle  deformed Lie bialgebra
   $ \, {\big( \lieg^{\bar{\Rpicc}}_{\bar{\Ppicc}} \big)}_{\gamma\raisebox{-1pt}{$ {}_\lieg $}} \, $.
 \vskip3pt
   Still out of  $ \chi \, $,  we define the matrix  $ P_{(\chi)} $  and its realization  $ \, \R_{(\chi)} := \big(\, \lieh \, , \, \Pi_{(\chi)} \, , \, \Pi^\vee \,\big) \, $,
   \,as in  Proposition \ref{prop: 2cocdef-realiz};  similarly, out of  $ \gamma $  we get the matrix  $ P_{(\gamma)} $  and its realization  $ \, \R_{(\gamma)} \, $:
   then by construction  $ \, P_{(\gamma)} = \bar{P}_{(\chi)} \, $  and  $ \, \R_{(\gamma)} = \bar{\R}_{(\chi)} \, $.
   Attached to these we have  $ \, U_{\!P_{(\chi)},\,\hbar}^{\,\R_{(\chi)}}(\lieg) \, $  and
   $ \, \lieg_{\Ppicc_{(\gamma)}}^{\Rpicc_{(\gamma)}} = \lieg_{\bar{\Ppicc}_{(\chi)}}^{\bar{\Rpicc}_{(\chi)}} \, $,
   \,again connected via quantization/specialization.
\end{free text}

\vskip5pt

 Next result claims that  \textsl{``deformation by 2--cocycle commutes with specialization''}.

\vskip13pt

\begin{theorem}  \label{thm: specializ 2-cocycle FoMpQUEA}
 With assumptions as above, we have that  $ {\big(\, \uRPhg \big)}_{\sigma_\chi} $
 is a quanti\-zed universal enveloping algebra, with semiclassical limit
 $ \, U\Big(\! {(\liegRP)}_{\gamma\raisebox{-1pt}{$ {}_\lieg $}} \Big) \, $.
%
%
\end{theorem}

\pf
 The claim follows, as direct application, from the isomorphisms
  $$  {\big(\, \uRPhg \big)}_{\sigma_\chi}  \!{\buildrel {\ref{thm: 2cocdef-uPhgd=new-uPhgd}} \over \cong} U_{\!P_{(\chi)},\,\hbar}^{\,\R_{(\chi)}}(\lieg)  \; ,  \;\;\;
   U_{\!P_{(\chi)},\,\hbar}^{\,\R_{(\chi)}}(\lieg) \Big/ \hbar \,
   U_{\!P_{(\chi)},\,\hbar}^{\,\R_{(\chi)}}(\lieg) \;{\buildrel {\ref{thm: semicl-limit FoMpQUEA}} \over \cong}\;
   U_\hbar\Big( \lieg_{\Ppicc_{(\gamma)}}^{\Rpicc_{(\gamma)}} \Big)  \; ,  \;\;\;
   \lieg_{\Ppicc_{(\gamma)}}^{\Rpicc_{(\gamma)}}  \;{\buildrel {\ref{thm: 2-cocycle-def-MpLbA}} \over \cong}\,
   {(\liegRP)}_{\gamma\raisebox{-1pt}{$ {}_\lieg $}}  $$
 which come from  Theorem \ref{thm: 2cocdef-uPhgd=new-uPhgd},
 Theorem \ref{thm: semicl-limit FoMpQUEA}  and  Theorem \ref{thm: 2-cocycle-def-MpLbA}.
\epf

\medskip

\subsection{Final overview}  \label{subsec: final overview} \

\medskip

 In this paper we studied multiparametric versions of formal quantum universal
  enveloping algebras and their semiclassical limits.
  As these are presented by generators and relations, their very definition highlights the relation
  between the multiparameters and the action of a fixed commutative subalgebra: like for Kac-Moody algebras,
  this is encoded in the notion of  \textsl{realization\/}  of a multiparameter matrix  $ P $  related to a symmetrizable Cartan matrix.
  This tool allows us to relate the quantum objects with the semiclassical limit, and also the multiparameter objects with the standard ones:
  the latter is done via deformation(s) and an explicit change of generators.
 \vskip7pt
   In conclusion, loosely speaking, one may say that:
\begin{enumerate}
  \item[$(a)$]  multiparameters are encoded in realizations;
  \item[$(b)$]  FoMpQUEAs are quantizations of MpLbAs;
  \item[$(c)$]  multiparameter objects are given by deformation of either the algebra or the coalgebra structure
  (both options being available) of  \textsl{standard\/}  objects.
\end{enumerate}
 \vskip9pt
   \indent   Finally, we provide a ``pictorial sketch'' of the global picture.
   Keeping notation as before, what we have in one single glance is summed up in the following diagram:
 \vskip-3pt
%
  $$  \xymatrix{
   U^{\,\R_\Psi}_{\!P_\Psi,\,\hbar}(\lieg) \cong {\big( \uRPhg \big)}^{\F_\Psi} \
   \ar@{<.>}[dd]|-(0.5){\hskip-61pt \text{Theorem \ref{thm: semicl-limit FoMpQUEA}}}  &  &
   \ar@{<~>}[ll]_{\hskip37pt \text{Theorem \ref{thm: uRPhg=twist-uhg}}}  \ \uRPhg \
   \ar@{<.>}[dd]|-(0.5){\text{Theorem \ref{thm: semicl-limit FoMpQUEA}}}
   \ar@{<~>}[rr]^{\hskip-37pt \text{Theorem \ref{thm: double FoMpQUEAs_P-P'_mutual-deform.s}}}  &  &
   \ar@{<.>}[dd]|-(0.5){\hskip63pt \text{Theorem \ref{thm: semicl-limit FoMpQUEA}}}
   \, {\big( \uRPhg \big)}_{\sigma_\chi} \! \cong U^{\,\R_{(\chi)}}_{\!P_{(\chi)},\,\hbar}(\lieg)  \\
   &  &  &  &  \\
   U\Big( \lieg^{\bar{\R}_\Psi}_{\bar{P}_\Psi} \!\Big) \cong \, U\Big(\! {\big( \lieg^{\bar{\R}}_{\bar{P}} \big)}^{j_{\bar{\Psi}}} \Big) \
   &  &  \ar@{<~>}[ll]^{\hskip41pt \text{Theorem \ref{thm: MpLbA=twist-gD}}}  \ U\big( \lieg^{\bar{\R}}_{\bar{P}} \big) \
   \ar@{<~>}[rr]_{\hskip-31pt \text{Theorem \ref{thm: 2-cocycle-def-MpLbA}}}  &  &
   U\Big(\! {\big( \lieg^{\bar{\R}}_{\bar{P}} \big)}_{\bar{\chi}} \Big) \cong \, \ U\Big( \lieg^{\bar{\R}_{(\chi)}}_{\bar{P}_{(\chi)}} \Big)  }  $$
%
%
 \vskip17pt
   \textsl{Note\/}  that in this diagram each vertical arrow (with dotted shaft) denotes a
   ``quantization/specialization (upwards/downwards) relationship''   --- which involves the ``continuous parameter'' $ \hbar $  ---
   whereas each horizontal arrow (having a waving shaft) denotes a relationship ``via deformation''   --- which involves ``discrete parameters''.

\vskip13pt

\end{document}